\documentclass{article}
\usepackage{graphicx}
\usepackage{geometry}
\usepackage[dvipsnames]{xcolor}
\usepackage{amssymb}
\usepackage{amsthm}
\usepackage{dynkin-diagrams}
\usepackage{mathtools}
\usepackage[utf8]{inputenc}
\usepackage{amsmath}
\usepackage{accents}
\usepackage{leftidx}
\usepackage{graphicx}
\newtheorem{theorem}{Theorem}[section]
\newtheorem{construction}{Construction}[theorem]

\newtheorem{lemma}{Lemma}[theorem]
\newtheorem{definition}[theorem]{Definition}
\newtheorem{proposition}[theorem]{Proposition}

\newtheorem{remark}{Remark}[theorem]
\usepackage{tikz-cd}
\usepackage{tikz}
\usepackage{hyperref}
\usepackage[backend=biber, style       = alphabetic, sortlocale  = auto, sortcites = true]{biblatex}
\usepackage{mathdots}
\usepackage[shortlabels]{enumitem}
\usetikzlibrary{shapes.geometric}
\usepackage{adjustbox}
\usetikzlibrary{shapes.misc}
\usetikzlibrary{arrows}
\usetikzlibrary {quotes}
\usetikzlibrary{automata,positioning}
\usetikzlibrary{decorations.markings,arrows}

\tikzset{cross/.style={cross out, draw=black, minimum size=2*(#1-\pgflinewidth), inner sep=0pt, outer sep=0pt},
cross/.default={1pt}}
\bibliography{refs}

\title{Maximal Green Sequences for Cluster Algebras Associated to Closed Orbifolds}

\author{Hin Chung Henry TSANG}

\date{21 January, 2026}

\begin{document}

\maketitle

\begin{abstract}
    It is known that the existence of a maximal green sequence for a quiver associated to surfaces is equivalent to the equality of the cluster algebra and upper cluster algebra generated by the quiver \cite{mills2017maximal}. 
    
    This paper makes the first steps in investigating this behavior in the generalised case of cluster algebras from orbifolds; determining when such surfaces admit a diagram with a maximal green sequence.
    
    Specifically, we will provide a triangulation for the orientable surfaces of genus $n$ with an arbitrary number of orbifold points and arbitrary number of punctures, determine when it has a maximal green sequence, and construct one if it exists.
\end{abstract}

\section{Introduction}

Cluster algebras were introduced by Fomin and Zelevinsky in 2002 in \cite{FZ}. To define a cluster algebra, we start with a set of initial (algebraically independent) variables, called the initial cluster, together with a skew-symmetrisable matrix. We then replace any one of the variables with a certain Laurent polynomial expression in the cluster's variables to create a new cluster using variables from the original cluster, and change the entries in the corresponding row and column; this process is called mutation at that variable. By iterating this process, we get a set of variables, called the cluster variables, and the algebra generated by all cluster variables is defined as the cluster algebra. Even though, a priori, the cluster variables are just rational expressions in terms of the initial cluster variables by definition, Fomin and Zelevinsky showed that they are actually Laurent polynomials in any cluster (therefore so is everything in the cluster algebra) in \cite{FZ2}.

The upper cluster algebra is defined to be the set of all expressions that can be expressed as Laurent polynomials in any cluster. Using this definition, the cluster algebra is contained in the upper cluster algebra, a question studied by many people in the field is that under what conditions this will be an equality. In general, the existence of a maximal green sequence for a quiver (which corresponds to the cluster algebra defined by a skew-symmetric matrix) is not equivalent to the equality of the cluster algebra and upper cluster algebra generated by the quiver, a counterexample is given in \cite{Mil}. However, when the quiver is generated from a surface, the statement is true in all known cases. It therefore makes sense to determine the existence of maximal green sequences of such quivers and the generalisation of quiver, which is diagram.

The search for maximal green sequences for quivers associated to triangulated surfaces without orbifold points was done in \cite{ACCER}, \cite{Lad}, and \cite{Buc}, and has been concluded by Bucher and Mills in \cite{BM}. For skew-symmetrisable matrices, Ahmet Seven studied $3\times3$ matrices in \cite{Sev}. We shall continue with orientable surfaces of arbitrary genus $n$ with arbitrary orbifold points $q$ and arbitrary punctures $p$, which correspond to a class of skew-symmetrisable matrices.

In this paper, we start with preliminaries on cluster algebras in Section \ref{prelim} and some notations in Section \ref{notations}. After that, we will consider three cases: $n=1$, $n=0$, and $n>1$. In all cases, there is no maximal green sequence when $p=1$, a short proof is given in Section \ref{puncture}. In each case we start with constructing triangulations, then give mutation sequences, and after that we prove that the mutation sequences are in fact maximal green sequences.

In Section \ref{n=1}, we consider the case where $n=1$, which will be further divided into six cases due to the differences in the mutation sequences and diagrams: ($p=2,q=1$ in Section \ref{2,1}), ($p=2,q>1$ in Section \ref{2,>1}), ($p=4$ in Section \ref{p=4}), ($p>4$ and is even in Section \ref{even}), ($p=3$ in Section \ref{p=3}), and ($p>4$ and is odd in Section \ref{odd}). The case $p=2,q=1$ gives a detailed explanation of the proving strategy used throughout this paper.

With the case $n=1$ done, we shall move on to the case $n=0$ in Section \ref{n=0}. This case will be further divided into two cases: ($p=2$ in Section \ref{0,2}) and ($p>2$ in Section \ref{0,>2}). In the former case the proof does not require Section \ref{n=1}, but in the latter case Section \ref{n=1} is required.

For $n>1$ in Section \ref{n>1}, it will be divided into two cases due to the differences in the diagrams: ($n=2$ in Section \ref{n=2}) and ($n>2$ in Section \ref{n>2}). In both cases Section \ref{n=1} will be required for the proof.

Combining the above cases we have our \textbf{main result}:

\begin{theorem}[Theorem \ref{main_1}, \ref{main_0}, \ref{main_n}]
    For an orbifold $\mathcal{O}$ of genus $n$ with $p$ punctures and $q$ orbifold points, the diagram $D(T_{n,p,q})$ associated to the triangulation $T_{n,p,q}$ has a maximal green sequence if $p\geq 2$. Moreover, if $p=1$ then $D(T)$ does not admit a maximal green sequence for any triangulation $T$ of $\mathcal{O}$.
\end{theorem}

\section{Preliminaries and definitions}\label{prelim}

\subsection{Cluster algebras associated to triangulated orbifolds}

In this subsection we refer to the works of Felikson, Shapiro, and Tumarkin in \cite{FST} and the works of Fomin, Shapiro, and Thurston in \cite{FST2}.

We shall start with introducing orbifolds:

\begin{definition}[Orbifold]
    An orbifold is an ordered triple $\mathcal{O}=(S,M,Q)$, where $S$ is a bordered surface, $M\subset S$ is a finite set of marked points, $Q\subset S$ is a finite set of orbifold points, with $M\cap Q=\phi$. Moreover, every boundary component of $S$ contains at least one marked point (and interior marked points are called punctures), while all orbifold points are interior points of $S$. For technical reasons, if $S$ has no boundary component and has a genus of $0$, we require that $|M|+|Q|>3$.
\end{definition}

A marked point is denoted by a dot while an orbifold point is denoted by a cross. Figure \ref{O_0} shows an orbifold $\hat{\mathcal{O}}$ with one boundary component, four marked points (two of them are punctures) and three orbifold points:

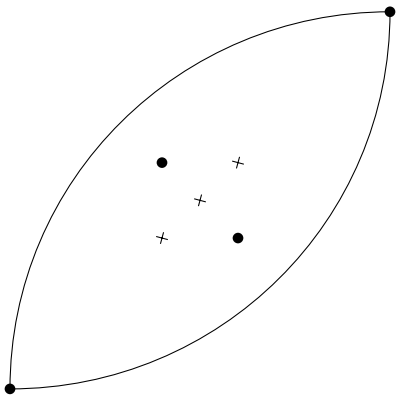
\begin{figure}
    \[\begin{tikzpicture}[every edge quotes/.style={auto=right}]
        \begin{scope}[every node/.style={sloped,allow upside down}][every edge quotes/.style={auto=right}]
            \draw (0,0) to [bend right=45] (5,5);
            \draw (5,5) to [bend right=45] (0,0);
            \fill (3,3) node[cross=2pt,rotate=30] {};
            \fill (2,2) node[cross=2pt,rotate=30] {};
            \fill (2.5,2.5) node[cross=2pt,rotate=30] {};
            \fill (3,2) circle (2pt);
            \fill (2,3) circle (2pt);
            \fill (5,5) circle (2pt);
            \fill (0,0) circle (2pt);
        \end{scope}
    \end{tikzpicture}\]
    \caption{An example of an orbifold $\hat{\mathcal{O}}$ with $2$ marked points on the bounary, $2$ punctures, and $3$ orbifold points} \label{O_0}
\end{figure}

Within an orbifold we can connect points to form arcs:

\begin{definition}[Arc]

An arc $\gamma$ in $\mathcal{O}$ is a curve in $S$ considered up to relative isotopy of $S\setminus\{M\cup Q\}$ modulo endpoints satisfying the following conditions:

\begin{enumerate}
    \item one of the following holds:
    \begin{itemize}
        \item both endpoints of $\gamma$ are in $M$ (in which case $\gamma$ is called an ordinary arc)
        \item one endpoint of $\gamma$ is in $M$ and another is in $Q$ (in which case $\gamma$ is called a pending arc);
    \end{itemize}
    \item $\gamma$ does not intersect with itself except possibly at its endpoints;
    \item $\gamma$ does not intersect with $M\cup Q\cup \partial S$ except at its endpoints;
    \item if $\gamma$ cuts out a monogon then this monogon contains either one point in $M$ or two points in $Q$;
    \item $\gamma$ is not homotopic to a boundary segment.
\end{enumerate}
    
\end{definition}

Using the orbifold $\hat{\mathcal{O}}$ as an example, in the leftmost illustration of Figure \ref{arcs}, the blue curves are ordinary arcs, the red curves are pending arcs, and the green curves in the centre are not arcs.

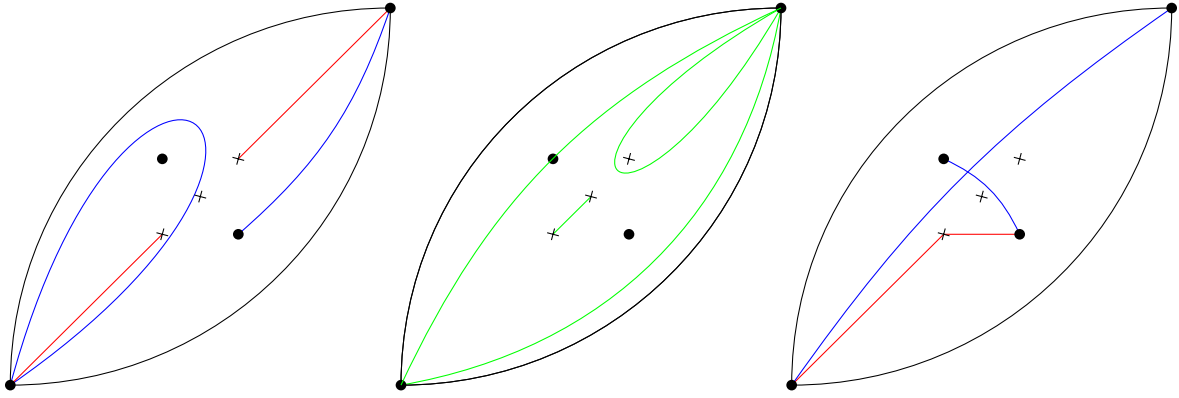
\begin{figure}\[\begin{tikzpicture}[every edge quotes/.style={auto=right}]
    \begin{scope}[every node/.style={sloped,allow upside down}][every edge quotes/.style={auto=right}]
        \draw (0,0) to [bend right=45] (5,5);
        \draw (5,5) to [bend right=45] (0,0);
        \draw[blue] (0,0) to [loop,in=35,out=75,min distance=60mm] (0,0);
        \draw[blue] (3,2) to [bend right=15] (5,5);
        \draw[red] (0,0) to (2,2);
        \draw[red] (5,5) to (3,3);
        \fill (3,3) node[cross=2pt,rotate=30] {};
        \fill (2.5,2.5) node[cross=2pt,rotate=30] {};
        \fill (2,2) node[cross=2pt,rotate=30] {};
        \fill (3,2) circle (2pt);
        \fill (2,3) circle (2pt);
        \fill (5,5) circle (2pt);
        \fill (0,0) circle (2pt);
    \end{scope}
\end{tikzpicture}
\begin{tikzpicture}[every edge quotes/.style={auto=right}]
    \begin{scope}[every node/.style={sloped,allow upside down}][every edge quotes/.style={auto=right}]
        \draw (0,0) to [bend right=45] (5,5);
        \draw (5,5) to [bend right=45] (0,0);
        \draw[green] (0,0) to [bend right=35] (5,5);
        \draw[green] (5,5) to [loop,in=210,out=240,min distance=42mm] (5,5);
        \draw[green] (2,2) to (2.5,2.5);
        \draw (0,0) to [bend right=45] (5,5);
        \draw (5,5) to [bend right=45] (0,0);
        \fill (3,3) node[cross=2pt,rotate=30] {};
        \fill (2,2) node[cross=2pt,rotate=30] {};
        \fill (2.5,2.5) node[cross=2pt,rotate=30] {};
        \fill (3,2) circle (2pt);
        \fill (2,3) circle (2pt);
        \fill (5,5) circle (2pt);
        \fill (0,0) circle (2pt);
        \draw[green] (0,0) to [bend left=20] (5,5);
    \end{scope}
\end{tikzpicture}
\begin{tikzpicture}[every edge quotes/.style={auto=right}]
        \begin{scope}[every node/.style={sloped,allow upside down}][every edge quotes/.style={auto=right}]
            \draw (0,0) to [bend right=45] (5,5);
            \draw (5,5) to [bend right=45] (0,0);
            \draw[red] (0,0) to (2,2);
            \draw[red] (2,2) to (3,2);
            \draw[blue] (2,3) to [bend left=20] (3,2);
            \draw[blue] (0,0) to [bend left=10] (5,5);
            \fill (3,3) node[cross=2pt,rotate=30] {};
            \fill (2,2) node[cross=2pt,rotate=30] {};
            \fill (2.5,2.5) node[cross=2pt,rotate=30] {};
            \fill (3,2) circle (2pt);
            \fill (2,3) circle (2pt);
            \fill (5,5) circle (2pt);
            \fill (0,0) circle (2pt);
        \end{scope}
    \end{tikzpicture}\]\caption{In the leftmost illustration we show valid arcs - ordinary arcs are blue, and pending arcs are red. In the centre we illustrate curves which are not arcs. The rightmost image shows two pairs of non-compatible arcs; one in blue, and one in red.} \label{arcs}
\end{figure}

\begin{definition}[Compatibility of arcs]

Two arcs $\gamma$ and $\gamma'$ are compatible if the following conditions hold:

\begin{enumerate}
    \item There are curves in the isotopy classes of $\gamma$ and $\gamma'$ that do not intersect in the interior of $\mathcal{S}$.
    \item If both $\gamma$ and $\gamma'$ are pending arcs, they do not share an orbifold point as endpoint.
\end{enumerate}
    
\end{definition}

In Figure \ref{arcs}, the arcs in the leftmost illustration are pairwise compatible. On the other hand, the pairs of same-colour arcs in the rightmost illustration are non-compatible.

By choosing a suitable set of arcs we can triangulate an orbifold.

\begin{definition}[Ideal triangulation of an orbifold]

An ideal triangulation of $\mathcal{O}$ is a maximal collection of distinct pairwise compatible arcs. The arcs of a triangulation partition $\mathcal{O}$ into triangles which can possibly be self-folded triangles or have pending arcs as edges.
    
\end{definition}

As an example, in Figure \ref{triangulation} we shall add arcs to the picture on the left of Figure \ref{arcs} to make it an ideal triangulation. Note that some points are moved to make the arcs clearer.

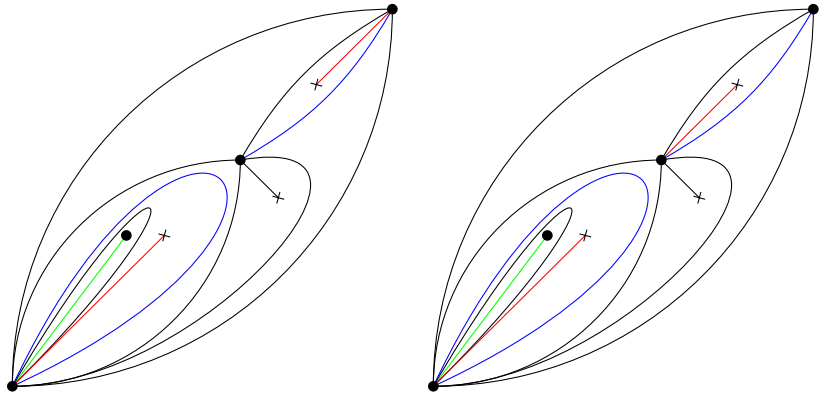
\begin{figure}\[\begin{tikzpicture}[every edge quotes/.style={auto=right}]
    \begin{scope}[every node/.style={sloped,allow upside down}][every edge quotes/.style={auto=right}]
        \draw (0,0) to [bend right=45] (5,5);
        \draw (5,5) to [bend right=45] (0,0);
        \draw (0,0) to [bend right=45] (3,3);
        \draw[blue] (3,3) to [bend right=15] (5,5);
        \draw (3,3) to [bend right=45] (0,0);
        \draw[green] (0,0) to (1.5,2);
        \draw (0,0) to [loop,in=45,out=60,min distance=40mm] (0,0);
        \draw[blue] (0,0) to [loop,in=25,out=65,min distance=55mm] (0,0);
        \draw (3,3) to [bend left=15] (5,5);
        \draw[red] (0,0) to (2,2);
        \draw[red] (5,5) to (4,4);
        \draw (0,0) to [in=10, out=0,min distance=25mm] (3,3);
        \draw (3,3) to (3.5,2.5);
        \fill (4,4) node[cross=2pt,rotate=30] {};
        \fill (2,2) node[cross=2pt,rotate=30] {};
        \fill (3.5,2.5) node[cross=2pt,rotate=30] {};
        \fill (3,3) circle (2pt);
        \fill (1.5,2) circle (2pt);
        \fill (5,5) circle (2pt);
        \fill (0,0) circle (2pt);
    \end{scope}
\end{tikzpicture}\begin{tikzpicture}[every edge quotes/.style={auto=right}]
    \begin{scope}[every node/.style={sloped,allow upside down}][every edge quotes/.style={auto=right}]
        \draw (0,0) to [bend right=45] (5,5);
        \draw (5,5) to [bend right=45] (0,0);
        \draw (0,0) to [bend right=45] (3,3);
        \draw[blue] (3,3) to [bend right=15] (5,5);
        \draw (3,3) to [bend right=45] (0,0);
        \draw[green] (0,0) to (1.5,2);
        \draw (0,0) to [loop,in=45,out=60,min distance=40mm] (0,0);
        \draw[blue] (0,0) to [loop,in=25,out=65,min distance=55mm] (0,0);
        \draw (3,3) to [bend left=15] (5,5);
        \draw[red] (0,0) to (2,2);
        \draw[red] (3,3) to (4,4);
        \draw (0,0) to [in=10, out=0,min distance=25mm] (3,3);
        \draw (3,3) to (3.5,2.5);
        \fill (4,4) node[cross=2pt,rotate=30] {};
        \fill (2,2) node[cross=2pt,rotate=30] {};
        \fill (3.5,2.5) node[cross=2pt,rotate=30] {};
        \fill (3,3) circle (2pt);
        \fill (1.5,2) circle (2pt);
        \fill (5,5) circle (2pt);
        \fill (0,0) circle (2pt);
    \end{scope}
\end{tikzpicture}\]\caption{An example of two ideal triangulations of $\hat{\mathcal{O}}$ which are related by a flip. However, the green arc cannot be flipped in either triangulation} \label{triangulation}
\end{figure}

Given any ideal triangulation, we can create a new ideal triangulation.

\begin{definition}[Flip]
    A flip on an arc $\gamma$ of an ideal triangulation $T$ replaces $\gamma$ with a unique arc $\gamma'\neq\gamma$ that is compatible with all arcs in $T\setminus\{\gamma\}$, the resulting collection of arcs is also an ideal triangulation.
\end{definition}

In fact, we can generate all triangulations by flipping arcs:

\begin{proposition}
    Given any two ideal triangulations $T$, $T'$ of an orbifold $\mathcal{O}$, there is a sequence of flips on arcs that transforms $T$ to $T'$.
\end{proposition}

In Figure \ref{triangulation} the upper red pending arc (enclosed in an orbifold with $2$ marked points on the bounary and $1$ orbifold point) of either triangulation can be flipped to obtain the other triangulation.

However, not all arcs can be flipped. Indeed, in Figure \ref{triangulation}, in either triangulation, the green arc enclosed in the self-folded triangle cannot be flipped.

To make things complete, we shall define a "tagged" version of arcs, compatibility of arcs, triangulations and flips (and from now on they are "tagged" by default).

\begin{definition}[Tagged arc]
    Near each endpoint of an (untagged) arc we may tag it in one of two ways, plain or notched. A tagged arc is an (untagged) arc tagged in this way, subject to the following conditions:
    \begin{enumerate}
        \item the arc does not cut out a once-punctured monogon;
        \item an endpoint lying on the boundary or is an orbifold point is always tagged plain;
        \item both ends of a loop are tagged in the same way.
    \end{enumerate}
\end{definition}

In figures, we omit plain tags and denote notched tags with a $\bowtie$ symbol.

\begin{construction}
    Each untagged arc $\gamma$ can be represented by a tagged arc $\tau(\gamma)$ defined as follows:

    \begin{itemize}
        \item if $\gamma$ does not cut out a once-punctured monogon, then $\tau(\gamma)$ is simply $\gamma$ with both ends tagged plain;
        \item if $\gamma$ cuts out a once-punctured monogon (which implies that $\gamma$ is a loop based at a marked point, say $m$), enclosing puncutre $p$, then $\tau(\gamma)$ is and arc with end points $m$ and $p$, tagged plain near $m$ and notched near $p$.
    \end{itemize}
\end{construction}

\begin{remark}    
    This representation will be used to transform an ideal triangulation to its tagged version, which will be mentioned later.
\end{remark}

\begin{definition}[Compatibility of tagged arcs]

Two tagged arcs $\gamma$ and $\gamma'$ are compatible if the following conditions hold:

\begin{enumerate}
    \item the untagged versions of $\gamma$ and $\gamma'$ are compatible;
    \item if the untagged versions of $\gamma$ and $\gamma'$ are different, then they must be tagged the same at any endpoint where they coincide;
    \item if the untagged versions of $\gamma$ and $\gamma'$ are the same, then they are tagged the same at one of the endpoints.
\end{enumerate}
    
\end{definition}

\begin{remark}
    If two untagged arcs $\gamma$ and $\gamma'$ are compatible, then the tagged arcs $\tau(\gamma)$ and $\tau(\gamma')$ are compatible.
\end{remark}

For ideal triangulations, the tagged version is simply called triangulations and the definition is the same except that we do not allow self-folded triangles.

\begin{definition}[Triangulation of an orbifold]

A triangulation of $\mathcal{O}$ is a maximal collection of distinct pairwise compatible tagged arcs. The arcs of a triangulation partition $\mathcal{O}$ into triangles which can possibly have pending arcs as edges.
    
\end{definition}

\begin{remark}\label{tau}
    Given an ideal triangulation $T$, we get a triangulation by replacing each arc $\gamma$ in $T$ by $\tau(\gamma)$.
\end{remark}

\begin{definition}[Flipping of a tagged arc]
    A flip on a tagged arc $\gamma$ of a triangulation $T$ replaces $\gamma$ with a unique tagged arc $\gamma'\neq\gamma$ that is compatible with all arcs in $T\setminus\{\gamma\}$, the resulting collection of arcs is also a triangulation.
\end{definition}

By transforming the ideal triangulation in Figure \ref{triangulation} in the way mentioned in Remark \ref{tau}, we get a triangulation shown on the left in Figure \ref{true_triangulation}, where every arc can be flipped, including the arc that could not be flipped in the ideal triangulation in Figure \ref{triangulation}. The resulting flipped triangulation is shown on the right in Figure \ref{true_triangulation}.

\begin{figure}\[\begin{tikzpicture}[every edge quotes/.style={auto=right}]
    \begin{scope}[every node/.style={sloped,allow upside down}][every edge quotes/.style={auto=right}]
        \draw (0,0) to [bend right=45] (5,5);
        \draw (5,5) to [bend right=45] (0,0);
        \draw (0,0) to [bend right=45] (3,3);
        \draw[blue] (3,3) to [bend right=15] (5,5);
        \draw (3,3) to [bend right=45] (0,0);
        \draw[green] (0,0) to [bend left=10] (1.5,2);
        \draw (0,0) to [bend right=10] node[near end,rotate=90] {$\bowtie$} (1.5,2);
        \draw[blue] (0,0) to [loop,in=25,out=65,min distance=55mm] (0,0);
        \draw (3,3) to [bend left=15] (5,5);
        \draw[red] (0,0) to (2.3,2);
        \draw[red] (5,5) to (4,4);
        \draw (0,0) to [in=10, out=0,min distance=25mm] (3,3);
        \draw (3,3) to (3.5,2.5);
        \fill (4,4) node[cross=2pt,rotate=30] {};
        \fill (2.3,2) node[cross=2pt,rotate=30] {};
        \fill (3.5,2.5) node[cross=2pt,rotate=30] {};
        \fill (3,3) circle (2pt);
        \fill (1.5,2) circle (2pt);
        \fill (5,5) circle (2pt);
        \fill (0,0) circle (2pt);
    \end{scope}
\end{tikzpicture}\begin{tikzpicture}[every edge quotes/.style={auto=right}]
    \begin{scope}[every node/.style={sloped,allow upside down}][every edge quotes/.style={auto=right}]
        \draw (0,0) to [bend right=45] (5,5);
        \draw (5,5) to [bend right=45] (0,0);
        \draw (0,0) to [bend right=45] (3,3);
        \draw[blue] (3,3) to [bend right=15] (5,5);
        \draw (3,3) to [bend right=45] (0,0);
        \draw (0,0) to [bend left=15] node[near end,rotate=90] {$\bowtie$} (2.4,2.4);
        \draw[green] (0,0) to [bend right=15] node[near end,rotate=90] {$\bowtie$} (2.4,2.4);
        \draw[blue] (0,0) to [loop,in=25,out=65,min distance=55mm] (0,0);
        \draw (3,3) to [bend left=15] (5,5);
        \draw[red] (0,0) to (2,2);
        \draw[red] (5,5) to (4,4);
        \draw (0,0) to [in=10, out=0,min distance=25mm] (3,3);
        \draw (3,3) to (3.5,2.5);
        \fill (4,4) node[cross=2pt,rotate=30] {};
        \fill (2,2) node[cross=2pt,rotate=30] {};
        \fill (3.5,2.5) node[cross=2pt,rotate=30] {};
        \fill (3,3) circle (2pt);
        \fill (2.4,2.4) circle (2pt);
        \fill (5,5) circle (2pt);
        \fill (0,0) circle (2pt);
    \end{scope}
\end{tikzpicture}\]\caption{An example of two triangulations of $\hat{\mathcal{O}}$ related by a flip}\label{true_triangulation}
\end{figure}
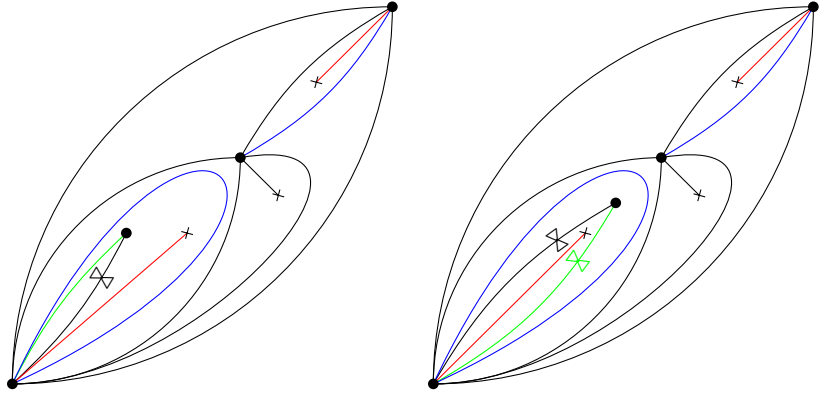

Just like in the case of ideal triangulation, in most of the cases we can generate all triangulations by flipping arcs:

\begin{proposition}
    Given any two ideal triangulations $T$, $T'$ of an orbifold $\mathcal{O}$ \textbf{that is not a closed surface with one puncture with any number of orbifold points}, there is a sequence of flips on arcs that transforms $T$ to $T'$. In the case where $\mathcal{O}$ is a closed surface with one puncture with any number of orbifold points, such sequence does not exist.
\end{proposition}

Now we shall take a pause with triangulations and introduce objects to define cluster algebras.

Let $n\leq m$ be positive integers and let $\mathcal{F}$ be a field of rational functions in $m$ independent variables. Fix a collection $X_1,...,X_n,x_{n+1},...,x_m$ of algebraically independent variables in $\mathcal{F}$. We define the coefficient ring to be $\mathbb{ZP}=\mathbb{Z}[x_{n+1},...,x_m]$

\pagebreak

\begin{definition}[Seed]

A (labelled) seed consists of a pair, $(x,B)$, where
\begin{itemize}
    \item $x=(x_1,...,x_n)$ is a collection of variables in $\mathcal{F}$ which are algebraically independent over $\mathbb{ZP}$.
    \item $B=(b_{jk})_{j,k\in\{1,...,n\}}$ is an $n\times n$ skew-symmetrisable integer matrix.
\end{itemize}
The variables in any seed are called cluster variables. The variables $x_{n+1},...,x_m$ are called frozen variables.

\end{definition}

Just like flipping an arc gives a new triangulation, we can mutate a cluster variable to get a new seed:

\begin{definition}[Seed mutation]
    Let $i\in\{1,...,n\}$, we define a new seed $\mu_i(x,B)=(x',B')$ given by the following:
    \begin{itemize}
        \item $x'_i=\frac{\prod_{b_{ki}>0}x_k^{b_{ki}}+\prod_{b_{ki}<0}x_k^{-b_{ki}}}{x_i}$
        \item $b'_{jk}=\begin{cases}
            -b_{jk},&\text{if }j=i \text{ or }k=i\\
            b_{jk}+\operatorname{max}(0,-b_{ji})b_{ik}+\operatorname{max}(0,b_{ik})b_{ji},&\text{otherwise.}
        \end{cases}$
    \end{itemize}
\end{definition}

\begin{remark}\label{involution}
    Fix $i\in\{1,...,n\}$, $\mu_i$ is an involution.
\end{remark}

The new seed obtained can be mutated at any cluster variable to again to obtain another new seed. 

We are then ready to define the cluster algebra associated to a seed.

\begin{definition}[Cluster algebra]

Fix an initial seed $(x,B)$. Consider the set $\mathcal{X}$ of all cluster variables in each seed obtained via mutating $(x,B)$ arbitrarily many times, the cluster algebra of the seed $(x,B)$ is defined to be $\mathbb{ZP}[\mathcal{X}]$, denoted by $\mathcal{A}(x,B)$.
    
\end{definition}

For properties of cluster algebras, please read \cite{FZ}.

We shall connect triangulated orbifolds and cluster algebras:

\begin{definition}[Skew-symmetrisable matrix associated to a triangulated orbifold]

Let $T=\{\gamma_1,...,\gamma_n\}$ be a triangulation of an orbifold $\mathcal{O}$. The skew-symmetrisable matrix $B(T)=(b_{ij})_{i,j\in\{1,...,n\}}$ associated to $T$ is defined as follows:

\begin{itemize}
    \item in case the untagged versions of two arcs are the same, we treat them as the same arc;
    \item if $\gamma_i$ and $\gamma_j$ are the same or are not in the same triangle, $b_{ij}=0$;
    \item if $\gamma_i$ and $\gamma_j$ are in precisely one triangle and $\gamma_i$ is before $\gamma_j$ in clockwise direction, define $b_{ij}$ to be
    \begin{itemize}
        \item $1$ if both $\gamma_i$ and $\gamma_j$ are ordinary arcs;
        \item  $2$ if exactly one of $\gamma_i$ and $\gamma_j$ is a pending arc;
        \item $4$ if both $\gamma_i$ and $\gamma_j$ are pending arcs;
    \end{itemize}
    \item if $\gamma_i$ and $\gamma_j$ are in two distinct triangles, define $b_{ij}$ to be
    \begin{itemize}
        \item $0$ if $\gamma_i$ is before $\gamma_j$ in clockwise direction in one triangle and $\gamma_j$ is before $\gamma_i$ in clockwise direction in another triangle;
        \item $4$ otherwise.
    \end{itemize}
\end{itemize}
    
\end{definition}

We have introduced the skew-symmetrisable matrix associated to a triagulated orbifold in order to associate a cluster algebra to a triagulated orbifold. However, we have a simpler notation to arrive our main result:

\begin{definition}[Diagram]
    A diagram is a weighted directed graph. The set of vertices and arrows of a diagram $D$ are denoted by $D_0$ and $D_1$ respectively. In case the weight of an arrow is $1$, the weight is omitted.
\end{definition}

We can associate diagrams to two objects defined previously:

\begin{definition}[Diagram assocaited to a skew-symmetrisable matrix]

Let $B=(b_{ij})$ be an $n\times n$ skew-symmetrisable matrix. The diagram $D(B)$ associated to $B$ has vertices $\{1,...,n\}$ with arrows given by the following: There is an arrow from $i$ to $j$ iff $b_{ij}>0$, in such case the arrow is weighted $-b_{ij}b_{ji}$.
    
\end{definition}

\begin{definition}[Diagram assocaited to a triangulated orbifold]

Let $T=\{\gamma_1,...,\gamma_n\}$ be a triangulation of an orbifold $\mathcal{O}$. The diagram $D(T)$ associated to $T$ is a weighted directed graph with $n$ vertices $\{1,...,n\}$ given by the following:
\begin{itemize}
    \item in case the untagged versions of two arcs are the same, we treat them as the same arc;
    \item for each triangle, if $\gamma_i$ is before $\gamma_j$ in clockwise direction, draw an arrow from vertex $i$ to $j$, with weight:
    \begin{itemize}
        \item $1$ if both $\gamma_i$ and $\gamma_j$ are ordinary arcs;
        \item  $2$ if exactly one of $\gamma_i$ and $\gamma_j$ is a pending arc;
        \item $4$ if both $\gamma_i$ and $\gamma_j$ are pending arcs;
    \end{itemize}
    \item if there is a pair of arrows from vertex $i$ to vertex $j$, replace them with an arrow with weight $4$ from vertex $i$ to vertex $j$;
    \item if there is an arrow from vertex $i$ to vertex $j$ and vice versa, remove them.
\end{itemize}
    
\end{definition}

\begin{remark}
    Let $T$ be a triangulation of an orbifold $\mathcal{O}$, $D(T)=D(B(T))$.
\end{remark}

\begin{definition}[Mutation of a diagram]
    Let $D$ be a diagram with $n$ vertices $\{1,...,n\}$. Mutating $D$ at vertex $i$ does the following:
    
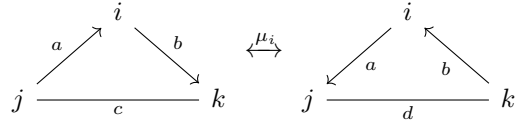
\[\begin{tikzcd}
        &i\arrow[dr,"b"]&\\
        j\arrow[ur,"a"]&&k\arrow[dash,ll,"c"]
    \end{tikzcd}
    \xleftrightarrow[]{\mu_i}
    \begin{tikzcd}
        &i\arrow[dl,"a"]&\\
        j&&k\arrow[dash,ll,"d"]\arrow[ul,"b"]
    \end{tikzcd}\]
    where $\pm\sqrt{c}\pm\sqrt{d}=\sqrt{ab}$, the sign before $c$ (resp. $d$) is positive if the three vertices form an oriented cycle and negative otherwise. The resulting diagram is denoted by $\mu_i(D)$.
\end{definition}

The set of all diagrams we could obtain via mutating $D$ arbitrarily many times is denoted by $\operatorname{Mut}(D)$.

\begin{remark}
    Let $B$ be a skew-symmetrisable matrix, then $D(\mu_i(B))=\mu_i(D(B))$.
\end{remark}

\begin{remark}
    Let $T$ be a triangulation and $\gamma_i$ be an arc in $T$. Let $T'$ be the triangulation obtained by flipping $\gamma_i$, then $\mu_i(D(T))=D(T')$.
\end{remark}

\subsection{Maximal green sequences}

In this subsection we refer to the works of Brüstle, Dupont, and Pérotin in \cite{BDP} which defined maximal green sequences for quivers, a brief definition of maximal green sequences for skew-symmetrisable matrices was given in \cite{Sev}.

\begin{definition}[Framed diagram]
    Let $D$ be a diagram, the framed diagram is the diagram $\hat{D}$ with
    \[\hat{D}_0=D_0\sqcup\{i':i\in D_0\}\]
    \[\hat{D}_1=D_1\sqcup\{i\rightarrow i':i\in D_0\}\]
    The vertices in $D'_0:=\{i':i\in D_0\}$ are called frozen vertices, which means that mutations at vertices in $D'_0$ are not allowed.
\end{definition}

\begin{definition}[Green and red vertices]
    Let $E\in\operatorname{Mut}(\hat{D})$. A non-frozen vertex $i\in E_0$ is called green if
    \[\{j'\in D'_0:\exists j'\rightarrow i\in E_1\}=\phi.\]
    It is called red if
    \[\{j'\in D'_0:\exists i\rightarrow j'\in E_1\}=\phi.\]
\end{definition}

\begin{definition}[Maximal green sequences]
    Let $D$ be a diagram, a maximal green sequence for $D$ is a sequence $(i_1,i_2,...,i_l)$, $1\leq i_1,i_2,...,i_l\leq n$, such that $i_k$ is green in $\mu_{i_{k-1}}\circ...\circ\mu_{i_1}(\hat{D})$ for any $2\leq k\leq l$, and all non-frozen vertices in $\mu_{i_{l}}\circ...\circ\mu_{i_1}(\hat{D})$ are red.
\end{definition}

To illustrate mutation of diagrams and maximal green sequences, consider the following diagram:

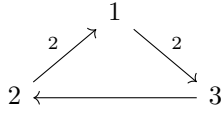
\[\begin{tikzcd}
    & 1\arrow[dr,"2"]&\\
    2\arrow[ur,"2"]&&3\arrow[ll]
\end{tikzcd}\]

We start with the framed diagram and mutate at vertices and end up with a diagram where all non-frozen vertices are red.

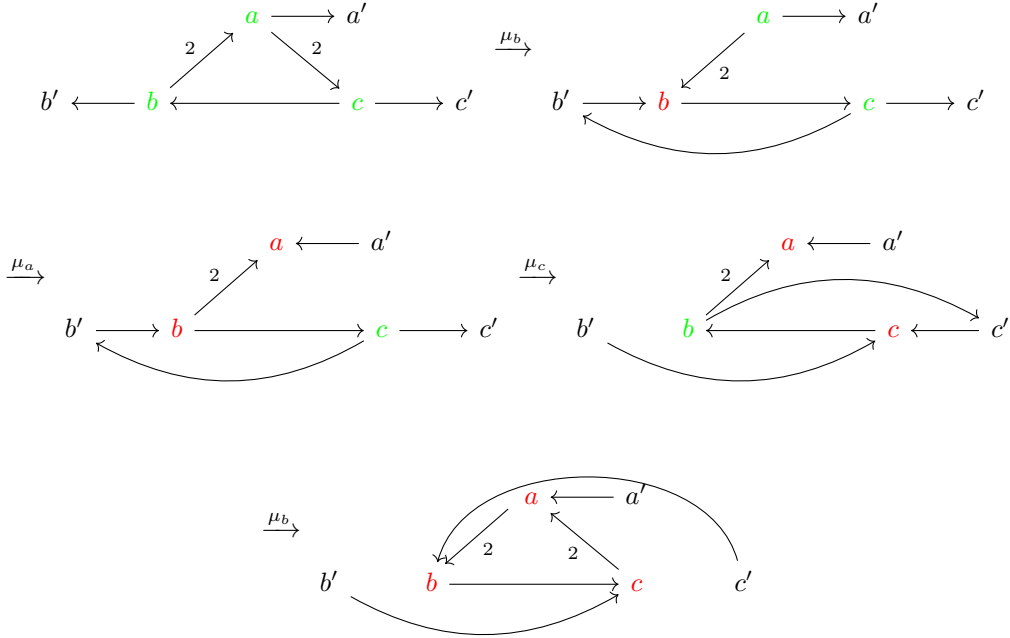
\[\begin{tikzcd}
    &&\color{green}a\arrow[dr,"2"]\arrow[r]&a'\\
    b'&\color{green}b\arrow[ur,"2"]\arrow[l]&&\color{green}c\arrow[ll]\arrow[r]&c'
\end{tikzcd}\xrightarrow{\mu_b}
\begin{tikzcd}
    &&\color{green}a\arrow[dl,"2"]\arrow[r]&a'\\
    b'\arrow[r]&\color{red}b\arrow[rr]&&\color{green}c\arrow[lll,bend left]\arrow[r]&c'
\end{tikzcd}\]\\\[\xrightarrow{\mu_a}
\begin{tikzcd}
    &&\color{red}a&a'\arrow[l]\\
    b'\arrow[r]&\color{red}b\arrow[rr]\arrow[ur,"2"]&&\color{green}c\arrow[lll,bend left]\arrow[r]&c'
\end{tikzcd}\xrightarrow{\mu_c}
\begin{tikzcd}
    &&\color{red}a&a'\arrow[l]\\
    b'\arrow[rrr,bend right]&\color{green}b\arrow[rrr, bend left]\arrow[ur,"2"]&&\color{red}c\arrow[ll]&c'\arrow[l]
\end{tikzcd}\]\\\[\xrightarrow{\mu_b}
\begin{tikzcd}
    &&\color{red}a\arrow[dl,"2",shift right]&a'\arrow[l]\\
    b'\arrow[rrr,bend right]&\color{red}b\arrow[rr]&&\color{red}c\arrow[ul,"2"]&c'\arrow[lll, bend right=75]
\end{tikzcd}
\]

The corresponding maximal green sequence is $(b,a,c,b)$.

It is possible to define maximal green sequences for skew-symmetrisable matrices, but since maximal green sequences exist for a skew-symmetrisable matrix $B$ if and only if maximal green sequences exist for $D(B)$, and it is easier to compute for diagrams, we shall stick with diagrams.

\section{Notations}\label{notations}

\subsection{Framed diagrams}

We shall introduce a notation for framed diagrams. Let $D$ be a diagram. Observe that for any diagram $E\in\operatorname{Mut}(\hat{D})$, the full subdiagram of $E$ given by removing all frozen vertices belongs to $\operatorname{Mut}(D)$.

Therefore, when drawing diagrams in $\operatorname{Mut}(\hat{D})$, we can simply draw the full subdiagram without the frozen vertices, colour the vertices, and for each vertex add the frozen vertices it is connected to as superscript. We shall illustrate this with the maximal green sequence above:

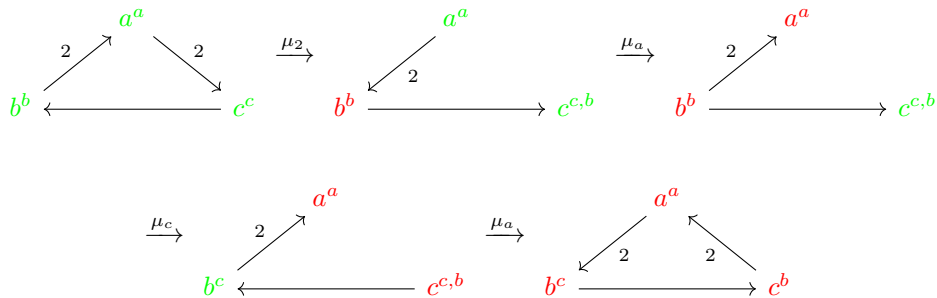
\[\begin{tikzcd}
    &\color{green}a^a\arrow[dr,"2"]&\\
    \color{green}b^b\arrow[ur,"2"]&&\color{green}c^c\arrow[ll]
\end{tikzcd}\xrightarrow{\mu_2}
\begin{tikzcd}
    &\color{green}a^a\arrow[dl,"2"]&\\
    \color{red}b^b\arrow[rr]&&\color{green}c^{c,b}
\end{tikzcd}\xrightarrow{\mu_a}
\begin{tikzcd}
    &\color{red}a^a&\\
    \color{red}b^b\arrow[ur,"2"]\arrow[rr]&&\color{green}c^{c,b}
\end{tikzcd}\]\\\[\xrightarrow{\mu_c}
\begin{tikzcd}
    &\color{red}a^a&\\
    \color{green}b^c\arrow[ur,"2"]&&\color{red}c^{c,b}\arrow[ll]
\end{tikzcd}\xrightarrow{\mu_a}
\begin{tikzcd}
    &\color{red}a^a\arrow[dl,"2"]&\\
    \color{red}b^c\arrow[rr]&&\color{red}c^{b}\arrow[ul,"2"]
\end{tikzcd}
\]

We omitted the $'$ in $a',b',c'$ for the superscripts as it is clear from context that they represent the corresponding frozen vertices and we shall continue to do so. Notice that there is no need to point out the direction of arrows between mutable vertices and frozen vertices as the colour of a mutable vertex determines the possible direction.

In case there is a weight of an arrow between a mutable vertex and a frozen vertex, it is written before the frozen vertex, as shown in the following diagram:

$
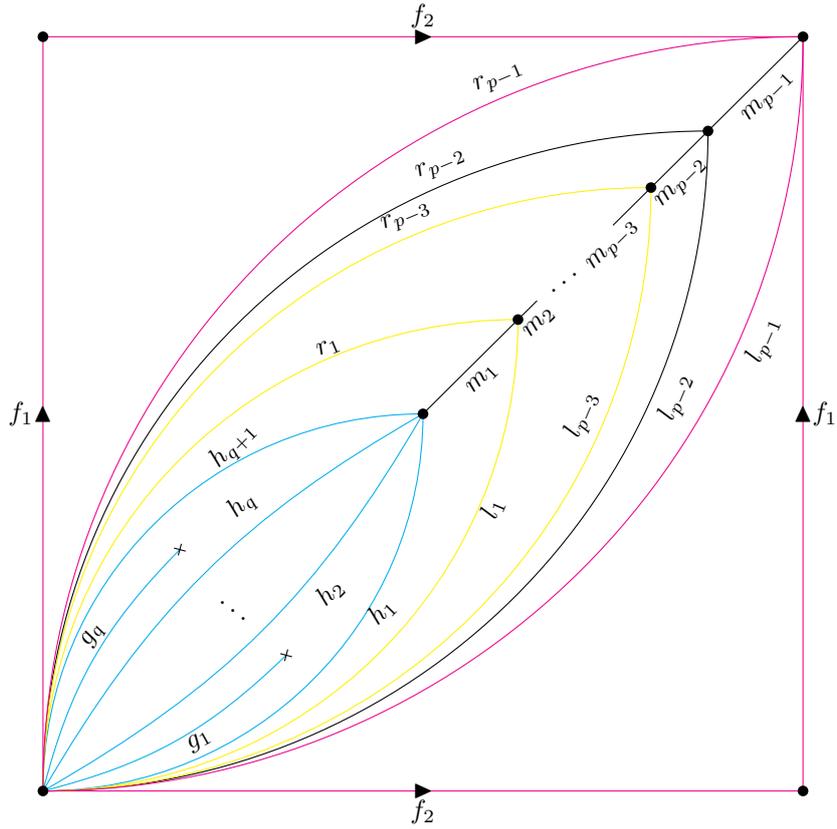
\begin{tikzcd}
    \color{green}a^{b,2c}
\end{tikzcd}$ corresponds to 
$\begin{tikzcd}
    b'&\color{green}a\arrow[l]\arrow[r,"2"]&c'
\end{tikzcd}$

We will label vertices using subscripts. Suppose we label $k$ vertices $a_1,a_2,...,a_k$. When we write $\begin{tikzcd}
    \color{green}b^{a_{i\nearrow j}}
\end{tikzcd}$, we mean that there are $j-i+1$ arrows of weight $1$ from $b$ to the frozen version of vertices $a_i$ to $a_j$ if $i\leq j$, and no arrows otherwise.

We will also use capital letters to represent a set of frozen vertices. Suppose we have sets of frozen vertices $A,B$ with weights for each vertex, we say $A\subset B$ if all vertices in $A$ are in $B$ with weight in $A$ not greater than weight in $B$. Consider the following diagram:

\[\begin{tikzcd}
    \color{green}a^A&\color{red}b^B\arrow[l]
\end{tikzcd}\]

If we mutate the vertex $a$, then $b$ will remain red, and if $A$ is non-empty, the superscript of $b$ will change. We use $A\setminus B$ to denote the superscript after the mutation.

\[\begin{tikzcd}
    \color{red}a^A\arrow[r]&\color{red}b^{B\setminus A}
\end{tikzcd}\]

For example, if $A=c,d$, $B=4c,d$, then $B\setminus A=c$ instead of $3c$.

Also note that we may use a single vertex to represent the set containing the vertex, for example, $A\setminus c=d$.

\subsection{Triangulations}

In a triangulation, to represent an end of an arc (which must be a puncture) is tagged notched, instead of drawing the $\bowtie$ symbol at that end, we colour the puncture blue unless that puncture has both plain tag and notched tag at it, in which case we revert to the original notation. An example is shown in Figure \ref{bowtie}.

\begin{figure}
    \[\begin{tikzpicture}[every edge quotes/.style={auto=right}]
    \begin{scope}[every node/.style={sloped,allow upside down}][every edge quotes/.style={auto=right}]
        \fill (0,0) circle (2pt);
        \fill (0,1) circle (2pt);
        \fill (-0.866,-0.5) circle (2pt);
        \fill (0.866,-0.5) circle (2pt);
        \draw (0,0) to node[near start,rotate=90] {$\bowtie$} (0,1);
        \draw (0,0) to node[near start,rotate=90] {$\bowtie$} (-0.866,-0.5);
        \draw (0,0) to node[near start,rotate=90] {$\bowtie$} (0.866,-0.5);
        \draw (0,1) to (-0.866,-0.5);
        \draw (0,1) to (0.866,-0.5);
        \draw (-0.866,-0.5) to (0.866,-0.5);
    \end{scope}
\end{tikzpicture}
\rightarrow
\begin{tikzpicture}[every edge quotes/.style={auto=right}]
    \begin{scope}[every node/.style={sloped,allow upside down}][every edge quotes/.style={auto=right}]
        \fill (0,1) circle (2pt);
        \fill (-0.866,-0.5) circle (2pt);
        \fill (0.866,-0.5) circle (2pt);
        \draw (0,0) to (0,1);
        \draw (0,0) to (-0.866,-0.5);
        \draw (0,0) to (0.866,-0.5);
        \draw (0,1) to (-0.866,-0.5);
        \draw (0,1) to (0.866,-0.5);
        \draw (-0.866,-0.5) to (0.866,-0.5);
        \fill[blue] (0,0) circle (2pt);
    \end{scope}
\end{tikzpicture}\]
\caption{Replacing the $\bowtie$ symbol with a puncture coloured blue}
    \label{bowtie}
\end{figure}

\section{The case where $p=1$}\label{puncture}

Just like in surfaces without orbifold points where the quiver associated to a triangulation of a once-punctured closed surface has no maximal green sequences (insert citation), the diagram associated to a once-punctured closed surface with any number of orbifold points has no maximal green sequences. A short explanation is given below:

Given any triangulation of any orbifold and any maximal green sequence, if we flip the arcs corresponding to the vertices in the sequence, we will end up with the same triangulation except that for each arc, all ends at a puncture must be tagged notched. In order to tag an end notched there must be a triangulated once-punctured digon before tagging, and this is impossible for once-punctured closed surfaces.

\section{The case where $n=1$}\label{n=1}

\subsection{Constructing the triangulation for $n=1$}\label{construction_1}

An orientable surface of genus $1$ with $p$ punctures and $q$ orbifold points, has triangulation $T_{1,p,q}$, as shown below, with an anatomy of the triangulation to facilitate the presentation of the proof.

\[\begin{tikzpicture}[every edge quotes/.style={auto=right}]
    \begin{scope}[every node/.style={sloped,allow upside down}][every edge quotes/.style={auto=right}]

\node at (6.85,6.85) {$\iddots$};
        \node at ($(0,5)!.5!(5,0)$) {$\ddots$};
        \draw[cyan] (0,0) to [bend right=15,black] node [below=.15,right] {$g_1$} (3.2,1.8);
        \draw[cyan] (0,0) to [bend left=15,black] node [above=.15,right] {$g_q$} (1.8,3.2);
        \draw[magenta] (0,0) -- node[rotate=-90,left,black] {$f_1$} node [rotate=-90,black] {$\blacktriangle$} (0,10);
        \draw[magenta] (0,10) -- node[above,black] {$f_2$} node [rotate=-90,black] {$\blacktriangle$} (10,10);
        \draw[magenta] (10,0) -- node[rotate=-90,right,black] {$f_1$} node [rotate=-90,black] {$\blacktriangle$} (10,10);
        \draw[magenta] (0,0) -- node[below,black] {$f_2$} node [rotate=-90,black] {$\blacktriangle$} (10,0);
        \draw[cyan] (0,0) to [bend right=45,black] node[below=.15,right=1] {$h_1$} (5,5);
        \draw[cyan] (0,0) to [bend left=45,black] node[above=.15,right=1] {$h_{q+1}$} (5,5);
        \draw[cyan] (0,0) to [bend right=15,black] node[right=1,below] {$h_2$} (5,5);
        \draw[cyan] (0,0) to [bend left=15,black] node[right=1,above] {$h_q$} (5,5);
        \draw[yellow] (0,0) to [bend right=45,black] node[near end,right,below=-.1] {$l_1$} (6.25,6.25);
        \draw[yellow] (0,0) to [bend left=45] node[near end,right,above=-.1,black] {$r_1$} (6.25,6.25);
        \draw (0,0) to [bend right=45,black] node[near end,below] {$l_{p-2}$} (8.75,8.75);
        \draw (0,0) to [bend left=45,black] node[near end,above] {$r_{p-2}$} (8.75,8.75);
        \draw[yellow] (0,0) to [bend right=45,black] node[near end,above] {$l_{p-3}$} (8,8);
        \draw[yellow] (0,0) to [bend left=45,black] node[near end,above] {$r_{p-3}$} (8,8);
        \draw[magenta] (0,0) to [bend right=45,black] node[near end,below,black] {$l_{p-1}$} (10,10);
        \draw[magenta] (0,0) to [bend left=45,black] node[near end,above] {$r_{p-1}$} (10,10);
        \draw (5,5) to node[below] {$m_1$} (6.25,6.25);
        \draw (8.75,8.75) to node[below] {$m_{p-1}$} (10,10);
        \draw (6.25,6.25) to node[below] {$m_2$} (6.5,6.5);
        \draw (8,8) to node[left=.2,below] {$m_{p-2}$} (8.75,8.75);
        \draw (7.5,7.5) to node[below=.2,left] {$m_{p-3}$} (8,8);
        \fill (0,10) circle (2pt);
        \fill (10,0) circle (2pt);
        \fill (0,0) circle (2pt);
        \fill (10,10) circle (2pt);
        \fill (5,5) circle (2pt);
        \fill (3.2,1.8) node[cross=2pt,rotate=30] {};
        \fill (1.8,3.2) node[cross=2pt,rotate=30] {};
        \fill (6.25,6.25) circle (2pt);
        \fill (8.75,8.75) circle (2pt);
        \fill (8,8) circle (2pt);
    \end{scope}
\end{tikzpicture}\]

The opposite sides of the squares are glued together, therefore the four corners are the same puncture. The remaining $p-1$ punctures are the interior punctures, ordered from the one incident to the $h_i$'s and $m_1$ to the one incident to $l_{p-2}, m_{p_2},r_{p-2},$ and $m_{p-1}$. The light blue arcs ($g_i$'s and $h_i$'s) are called the core arcs; the yellow arcs ($l_1$ up to $l_{p-3}$ and $r_1$ up to $r_{p-3}$ are called the inner arcs; the magenta arcs (the remaining arcs other than the $m_i$'s) are called the outer arcs. An exception is that when $p=2$, $h_1$ and $h_{q+1}$ are no longer core arcs.

For example, when $p=5$, $q=3$, the triangulation looks like this:

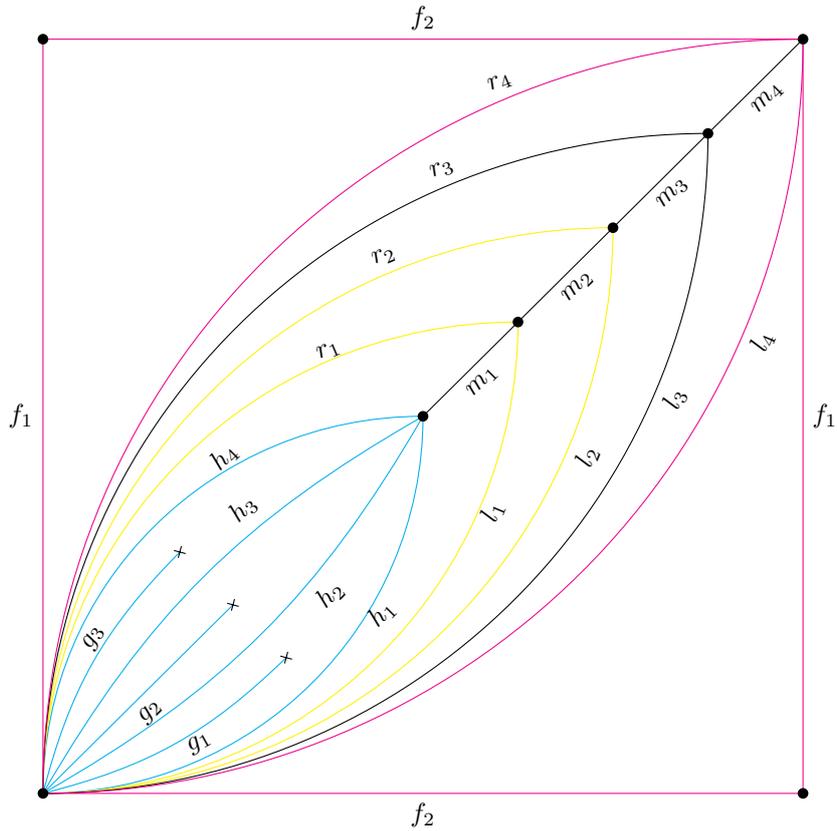
\[\begin{tikzpicture}[every edge quotes/.style={auto=right}]
    \begin{scope}[every node/.style={sloped,allow upside down}][every edge quotes/.style={auto=right}]
        \draw[cyan] (0,0) to node[below,black] {$g_2$} (2.5,2.5);
        \draw[yellow] (0,0) to [bend right=45] node[near end, below,black] {$l_2$} (7.5,7.5);
        \draw[yellow] (0,0) to [bend left=45,black] node[near end, above] {$r_2$} (7.5,7.5);
        \draw (6.25,6.25) to node[below] {$m_2$} (7.5,7.5);
        \draw (7.5,7.5) to node[below] {$m_3$} (8.75,8.75);
        \draw[cyan] (0,0) to [bend right=15,black] node [below=.15,right] {$g_1$} (3.2,1.8);
        \draw[cyan] (0,0) to [bend left=15,black] node [above=.15,right] {$g_3$} (1.8,3.2);
        \draw[magenta] (0,0) -- node[rotate=-90,left,black] {$f_1$} (0,10);
        \draw[magenta] (0,10) -- node[above,black] {$f_2$} (10,10);
        \draw[magenta] (10,0) -- node[rotate=-90,right,black] {$f_1$} (10,10);
        \draw[magenta] (0,0) -- node[below,black] {$f_2$} (10,0);
        \draw[cyan] (0,0) to [bend right=45,black] node[below=.15,right=1] {$h_1$} (5,5);
        \draw[cyan] (0,0) to [bend left=45,black] node[above=.15,right=1] {$h_{4}$} (5,5);
        \draw[cyan] (0,0) to [bend right=15,black] node[right=1,below] {$h_2$} (5,5);
        \draw[cyan] (0,0) to [bend left=15,black] node[right=1,above] {$h_3$} (5,5);
        \draw[yellow] (0,0) to [bend right=45] node[near end,right,below=-.1,black] {$l_1$} (6.25,6.25);
        \draw[yellow] (0,0) to [bend left=45,black] node[near end,right,above=-.1] {$r_1$} (6.25,6.25);
        \draw (0,0) to [bend right=45] node[near end,below] {$l_{3}$} (8.75,8.75);
        \draw (0,0) to [bend left=45] node[near end,above] {$r_{3}$} (8.75,8.75);
        \draw[magenta] (0,0) to [bend right=45,black] node[near end,below] {$l_{4}$} (10,10);
        \draw[magenta] (0,0) to [bend left=45,black] node[near end,above] {$r_{4}$} (10,10);
        \draw (5,5) to node[below] {$m_1$} (6.25,6.25);
        \draw (8.75,8.75) to node[below] {$m_{4}$} (10,10);
        \fill (0,10) circle (2pt);
        \fill (10,0) circle (2pt);
        \fill (0,0) circle (2pt);
        \fill (10,10) circle (2pt);
        \fill (5,5) circle (2pt);
        \fill (3.2,1.8) node[cross=2pt,rotate=30] {};
        \fill (1.8,3.2) node[cross=2pt,rotate=30] {};
        \fill (2.5,2.5) node[cross=2pt,rotate=30] {};
        \fill (6.25,6.25) circle (2pt);
        \fill (8.75,8.75) circle (2pt);
        \fill (7.5,7.5) circle (2pt);
    \end{scope}
\end{tikzpicture}\]

The following diagram is associated to $T_{1,p,q}$:

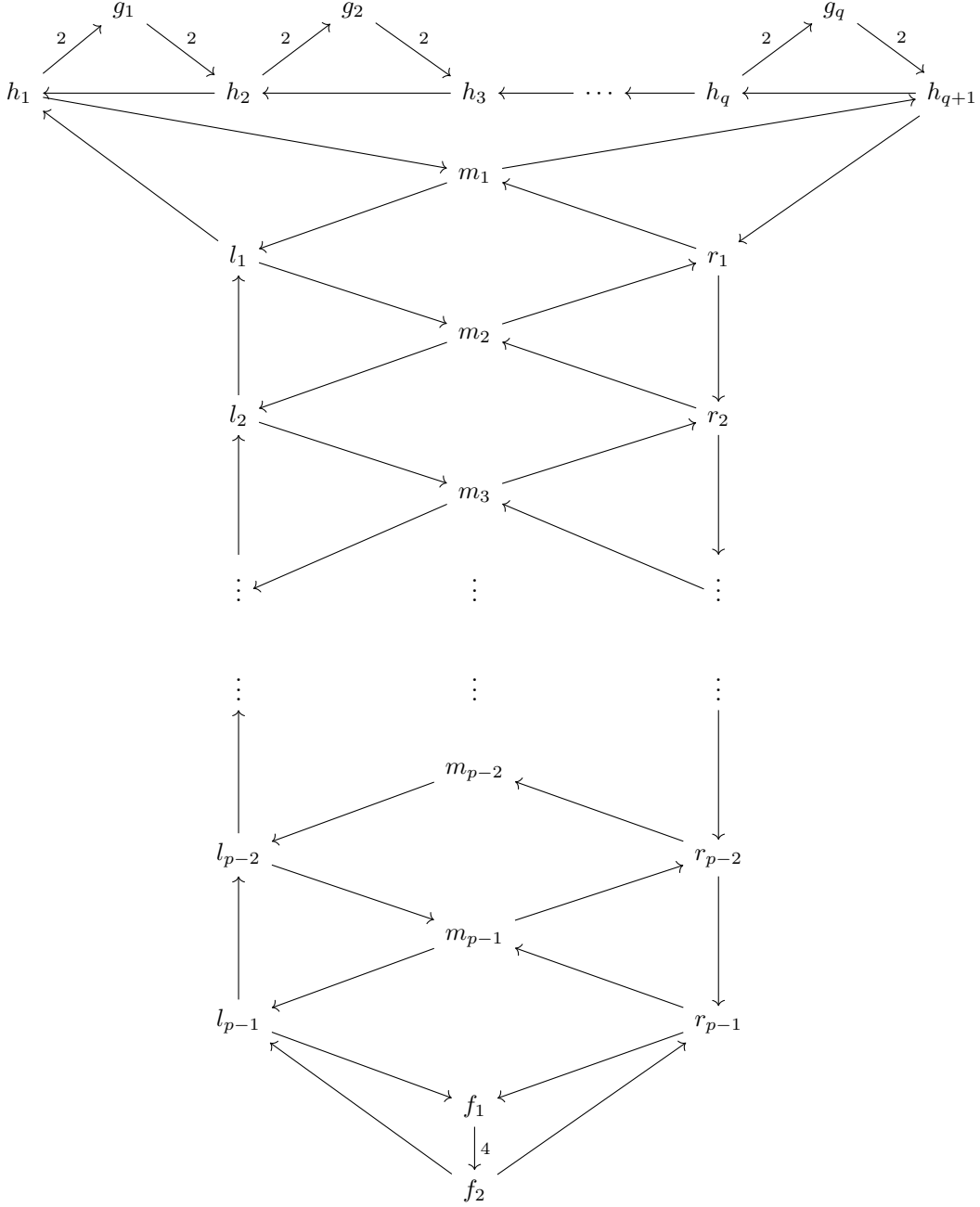
\[\begin{tikzcd}
 & g_1\arrow[dr,"2"] & & g_2\arrow[dr,"2"]&&&&g_q\arrow[dr,"2"]\\
 h_1\arrow[drrrr]\arrow[ur,"2"] && h_2\arrow[ur,"2"]\arrow[ll] && h_3\arrow[ll]&\cdots\arrow[l]&h_q\arrow[l]\arrow[ur,"2"]&&h_{q+1}\arrow[ll]\arrow[ddll]\\
 &&&&m_1\arrow[dll]\arrow[urrrr]&&&&\\
 &&l_1\arrow[uull]\arrow[drr]&&&&r_1\arrow[ull]\arrow[dd]&&\\
 &&&&m_2\arrow[dll]\arrow[urr]&&&&\\
 &&l_2\arrow[uu]\arrow[drr]&&&&r_2\arrow[ull]\arrow[dd]&&\\
 &&&&m_3\arrow[dll]\arrow[urr]&&&&\\
 &&\vdots\arrow[uu]&&\vdots&&\vdots\arrow[ull]\\
 &&\vdots&&\vdots&&\vdots\arrow[dd]\\
 &&&&m_{p-2}\arrow[dll]&&&&\\
 &&l_{p-2}\arrow[uu]\arrow[drr]&&&&r_{p-2}\arrow[ull]\arrow[dd]\\
 &&&&m_{p-1}\arrow[urr]\arrow[dll]&&&&\\
 &&l_{p-1}\arrow[uu]\arrow[drr]&&&&r_{p-1}\arrow[ull]\arrow[dll]\\
 &&&&f_1\arrow[d,"4"]&&&&\\
 &&&&f_2\arrow[uull]\arrow[uurr]&&&&
\end{tikzcd}\]

\subsection{Proof of the main result for $n=1$}\label{proof_1}

\begin{theorem}\label{main_1}
    For an orbifold $\mathcal{O}$ of genus $1$ with $p$ punctures and $q$ orbifold points, the diagram $D(T_{1,p,q})$ associated to the triangulation $T_{1,p,q}$ has a maximal green sequence if $p\geq 2$. Moreover, if $p=1$ then $D(T)$ does not admit a maximal green sequence for any triangulation $T$ of $\mathcal{O}$.
\end{theorem}

The case where $p=1$ is handled in Section \ref{puncture}. For other cases, we have a mutation sequence below for demonstration. To have a clearer view for each subcase, refer to the corresponding subsection.

\begin{proposition}

The mutation sequence $\Delta_1$ formed by concatenating the following steps is a maximal green sequence of $D(T_{1,p,q})$:

\begin{enumerate}
    \item Tagging alternate punctures notched. $(\alpha_{p-2},\alpha_{p-4},...,\alpha_{p-2\lfloor \frac{p-1}{2}\rfloor})$
    \item Tagging the first interior puncture notched. $(h_{q+1},h_q,...,h_1,m_{p+1-2\lfloor\frac{p}{2}\rfloor},h_2,h_3,...,h_{q+1-p+2\lfloor \frac{p}{2}\rfloor})$
    \item Tagging the remaining interior punctures notched and moving half the inner arcs away from the corner. $(\beta_{p-3},\beta_{p-5},...,\beta_{p+1-2\lfloor\frac{p}{2}\rfloor})$
    \item Moving arcs away from the corner.
    \begin{enumerate}
        \item Moving core arcs away from the corner. $(g_1,g_2,...,g_q,h_{q+\operatorname{sgn}(p-2)},h_{q+\operatorname{sgn}(p-2)-1},...,h_2,\mu)$
        \item Moving the remaining inner arcs away from the corner. $(l_{p-2\lfloor\frac{p-1}{2}\rfloor},r_{p-2\lfloor\frac{p-1}{2}\rfloor},l_{p+2-2\lfloor\frac{p-1}{2}\rfloor},r_{p+2-2\lfloor\frac{p-1}{2}\rfloor},\\...,l_{p-4},r_{p-4})$
        \item Moving outer arcs away from the corner. $(l_{p-1},f_1,r_{p-1},f_2,l_{p-1},f_1)$
    \end{enumerate}
    \item Tagging the corner notched. $(\gamma)$
    \item Moving arcs back to the corner.
    \begin{enumerate}
        \item Moving outer arcs back to the corner. 
        $(f_1,l_{p-1},f_2,r_{p-1},f_1,l_{p-1})$
        \item Moving inner arcs back to the corner. 
        $(\delta_{p-4},\delta_{p-6},...,\delta_{p-2\lfloor\frac{p-1}{2}\rfloor})$
        \item Moving core arcs back to the corner. $(\mu,\varepsilon,g_1,g_2,...,g_{q-1},g_q,\omega)$
        \item (If $p>2$ and is even) Moving the remaining pair of inner arc back to the corner. $(l_1,r_1)$
    \end{enumerate}
\end{enumerate}

where

\begin{itemize}
    \item $\alpha_k=r_k,l_k,m_{k+1},m_{k},l_k,r_k$
    \item $\beta_k=\begin{cases}
        r_k,l_k,m_{k+2},m_{k-1}&\text{if }k>1,\\
        r_k,l_k,m_{k+2},h_1&\text{otherwise}
    \end{cases}$
    \item $\mu=\begin{cases}
        m_{p+1-2\lfloor\frac{p}{2}\rfloor}&\text{if }p>2,\\
        \phi & \text{if }p=2
    \end{cases}$
    \item $\gamma=\begin{cases}
        m_{p-2},l_{p-2},r_{p-2},m_{p-2}&\text{if }p>2,\\
        h_1,m_1,h_{q+1},h_1 & \text{if }p=2
    \end{cases}$
    \item $\delta_k=r_k,r_{k+1},l_k,l_{k+1}$
    \item $\varepsilon=\begin{cases}
        h_2,h_3,...,h_{q+1}&\text{if }p\text{ is even and}>2\\
        h_2,h_3,...,h_q &\text{otherwise}
    \end{cases}$
    \item $\omega=\begin{cases}
        h_{q+1},g_q&\text{if }p\text{ is odd},\\
        \phi & \text{otherwise}
    \end{cases}$
\end{itemize}
    
\end{proposition}

\subsubsection{The case where $p=2, q=1$}\label{2,1}

When $p=2$ and $q=1$, $\Delta_1$ is different from any other cases, so we shall consider this case separately. This also serves as an introduction to the structure of $\Delta_1$.

\begin{proposition}
    $\Delta_1$ is a maximal green sequence when $p=2$ and $q=1$.
\end{proposition}

The triangulation $T_{2,1}$ and its associated diagram look like this:

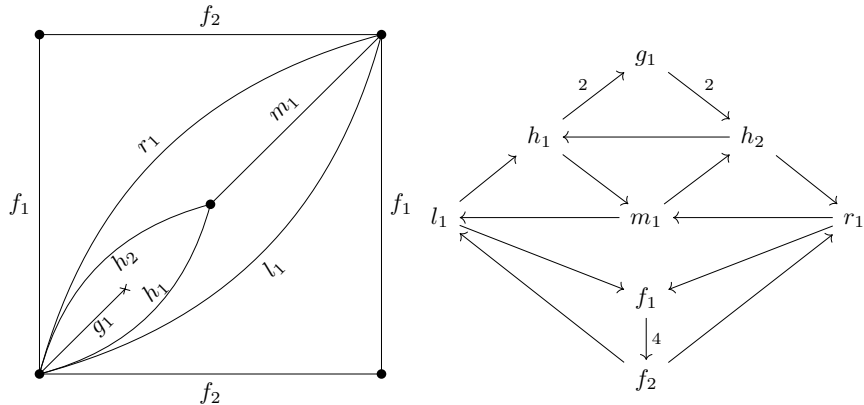
\[\adjustbox{scale=0.9}{\begin{tikzpicture}[every node/.style={draw}][every edge quotes/.style={auto=right}][baseline=(current bounding box.center)]
    \begin{scope}[every node/.style={sloped,allow upside down}][every edge quotes/.style={auto=right}]
        \fill (0,5) circle (2pt);
        \fill (5,0) circle (2pt);
        \fill (0,0) circle (2pt);
        \fill (5,5) circle (2pt);
        \fill (2.5,2.5) circle (2pt);
        \fill (1.25,1.25) node[cross=2pt,rotate=30] {};
        \draw (0,0) -- node[rotate=-90,left] {$f_1$} (0,5);
        \draw (0,5) -- node[above] {$f_2$} (5,5);
        \draw (5,0) -- node[rotate=-90,right] {$f_1$} (5,5);
        \draw (0,0) -- node[below] {$f_2$} (5,0);
        \draw (0,0) -- node[below=0.15,right] {$g_1$} (1.25,1.25);
        \draw (0,0) to [bend right] node[above=.2,right] {$h_1$} (2.5,2.5);
        \draw (0,0) to [bend left] node[below=.2,right] {$h_2$} (2.5,2.5);
        \draw (0,0) to [bend right] node[below]{$l_1$} (5,5);
        \draw (0,0) to [bend left] node[above]{$r_1$} (5,5);
        \draw (2.5,2.5) -- node[above]{$m_1$} (5,5);
    \end{scope}
\end{tikzpicture}}\adjustbox{scale=0.9}{\begin{tikzcd}[baseline=-80pt]
    &&g_1\arrow[dr,"2"]&&\\
    &h_1\arrow[ur,"2"]\arrow[dr]&&h_2\arrow[ll]\arrow[dr]&\\
    l_1\arrow[ur]\arrow[drr]&&m_1\arrow[ur]\arrow[ll]&&r_1\arrow[ll]\arrow[dll]\\
    &&f_1\arrow[d,"4"]&&\\
    &&f_2\arrow[uull]\arrow[uurr]&&
\end{tikzcd}}\]

This diagram has maximal green sequence $(h_2,h_1,m_1,h_2,g_1,l_1,f_1,r_1,f_2,l_1,f_1,h_1,m_1,h_2,h_1,f_1,l_1,f_2,r_1,f_1,l_1,g_1)$. In this case the maximal green sequence consists of step 2, 4(a), 4(c), 5, 6(a) and 6(c) only.

To understand how this works, we shall have a closer look of the triangulation throughout the flips.

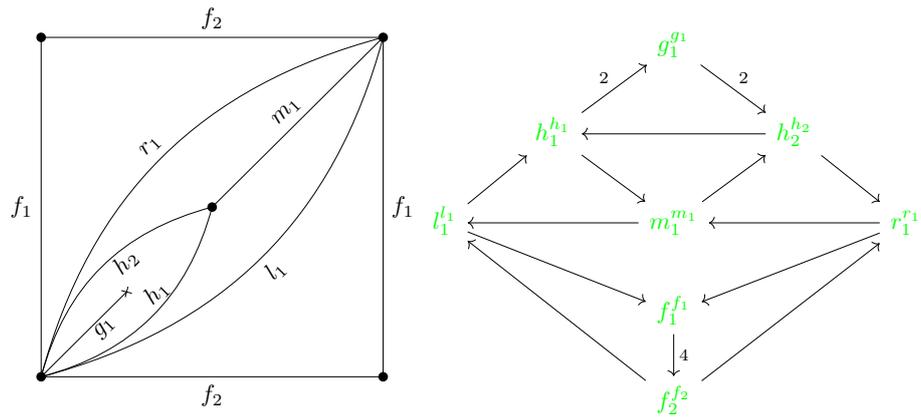
\[\adjustbox{scale=0.9}{\begin{tikzpicture}[every node/.style={draw}][every edge quotes/.style={auto=right}][baseline=(current bounding box.center)]
    \begin{scope}[every node/.style={sloped,allow upside down}][every edge quotes/.style={auto=right}]
        \fill (0,5) circle (2pt);
        \fill (5,0) circle (2pt);
        \fill (0,0) circle (2pt);
        \fill (5,5) circle (2pt);
        \fill (2.5,2.5) circle (2pt);
        \fill (1.25,1.25) node[cross=2pt,rotate=30] {};
        \draw (0,0) -- node[rotate=-90,left] {$f_1$} (0,5);
        \draw (0,5) -- node[above] {$f_2$} (5,5);
        \draw (5,0) -- node[rotate=-90,right] {$f_1$} (5,5);
        \draw (0,0) -- node[below] {$f_2$} (5,0);
        \draw (0,0) -- node[below=0.15,right] {$g_1$} (1.25,1.25);
        \draw (0,0) to [bend right] node[above=.2,right] {$h_1$} (2.5,2.5);
        \draw (0,0) to [bend left] node[below=.2,right] {$h_2$} (2.5,2.5);
        \draw (0,0) to [bend right] node[below]{$l_1$} (5,5);
        \draw (0,0) to [bend left] node[above]{$r_1$} (5,5);
        \draw (2.5,2.5) -- node[above]{$m_1$} (5,5);
    \end{scope}
\end{tikzpicture}}\adjustbox{scale=0.9}{\begin{tikzcd}[baseline=-80pt]
    &&\color{green}g_1^{g_1}\arrow[dr,"2"]&&\\
    &\color{green}h_1^{h_1}\arrow[ur,"2"]\arrow[dr]&&\color{green}h_2^{h_2}\arrow[ll]\arrow[dr]&\\
    \color{green}l_1^{l_1}\arrow[ur]\arrow[drr]&&\color{green}m_1^{m_1}\arrow[ur]\arrow[ll]&&\color{green}r_1^{r_1}\arrow[ll]\arrow[dll]\\
    &&\color{green}f_1^{f_1}\arrow[d,"4"]&&\\
    &&\color{green}f_2^{f_2}\arrow[uull]\arrow[uurr]&&
\end{tikzcd}}\]

Step 2 consists of ($h_2, h_1, m_1, h_2$). After flipping (mutating) at $h_2$, the triangulation (diagram) becomes:

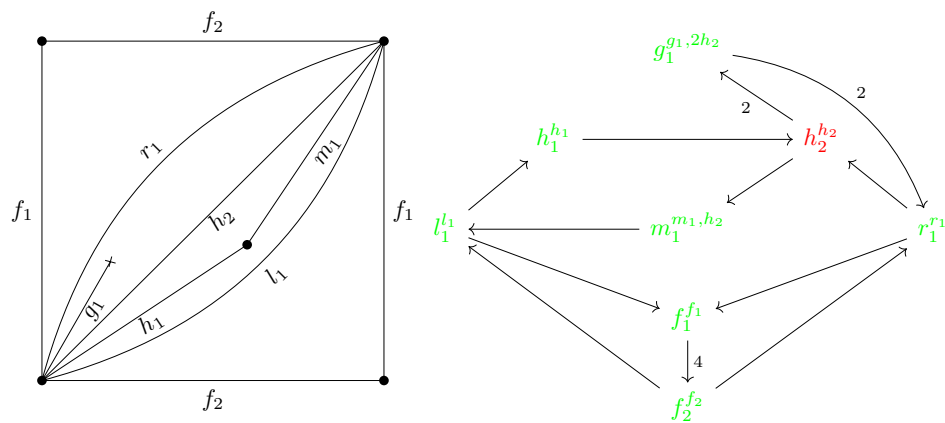
\[\adjustbox{scale=0.9}{\begin{tikzpicture}[every node/.style={draw}][every edge quotes/.style={auto=right}][baseline=(current bounding box.center)]
    \begin{scope}[every node/.style={sloped,allow upside down}][every edge quotes/.style={auto=right}]
        \fill (0,5) circle (2pt);
        \fill (5,0) circle (2pt);
        \fill (0,0) circle (2pt);
        \fill (5,5) circle (2pt);
        \fill (3,2) circle (2pt);
        \fill (1,1.75) node[cross=2pt,rotate=30] {};
        \draw (0,0) -- node[rotate=-90,left] {$f_1$} (0,5);
        \draw (0,5) -- node[above] {$f_2$} (5,5);
        \draw (5,0) -- node[rotate=-90,right] {$f_1$} (5,5);
        \draw (0,0) -- node[below] {$f_2$} (5,0);
        \draw (0,0) -- node[below=0.15,right] {$g_1$} (1,1.75);
        \draw (0,0) to node[below=-.1] {$h_1$} (3,2);
        \draw (0,0) to node[below=-.1] {$h_2$} (5,5);
        \draw (0,0) to [bend right] node[below]{$l_1$} (5,5);
        \draw (0,0) to [bend left] node[above]{$r_1$} (5,5);
        \draw (3,2) -- node[below]{$m_1$} (5,5);
    \end{scope}
\end{tikzpicture}}\adjustbox{scale=0.9}{\begin{tikzcd}[baseline=-80pt]
    &&\color{green}g_1^{g_1,2h_2}\arrow[ddrr,"2",bend left]&&\\
    &\color{green}h_1^{h_1}\arrow[rr]&&\color{red}h_2^{h_2}\arrow[ul,"2"]\arrow[dl]&\\
    \color{green}l_1^{l_1}\arrow[ur]\arrow[drr]&&\color{green}m_1^{m_1,h_2}\arrow[ll]&&\color{green}r_1^{r_1}\arrow[ul]\arrow[dll]\\
    &&\color{green}f_1^{f_1}\arrow[d,"4"]&&\\
    &&\color{green}f_2^{f_2}\arrow[uull]\arrow[uurr]&&
\end{tikzcd}}\]

Observe that the arcs $h_2,l_1,h_1,m_1$ form a triangulated once-punctured digon. When flipping at $h_1$ and then $m_1$, they switch positions with their ends at the middle puncture tagged notched. Because of this, we shall switch the positions of $h_1$ and $m_1$ in the diagram as well:

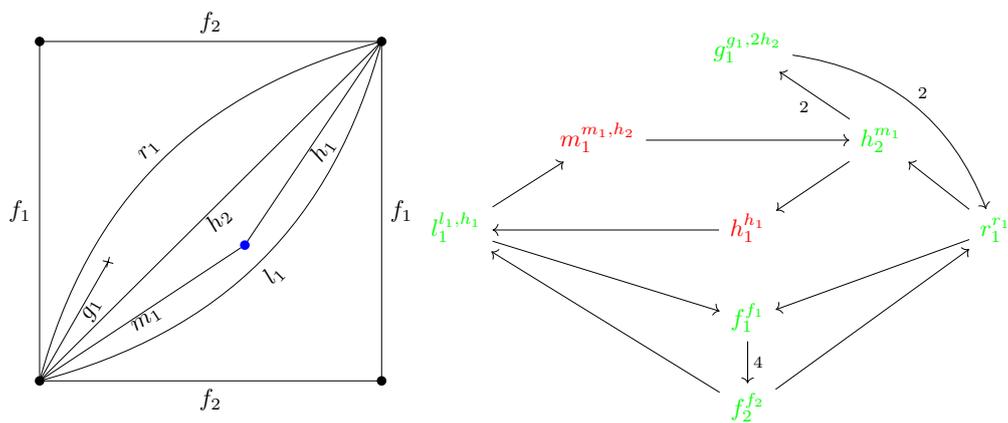
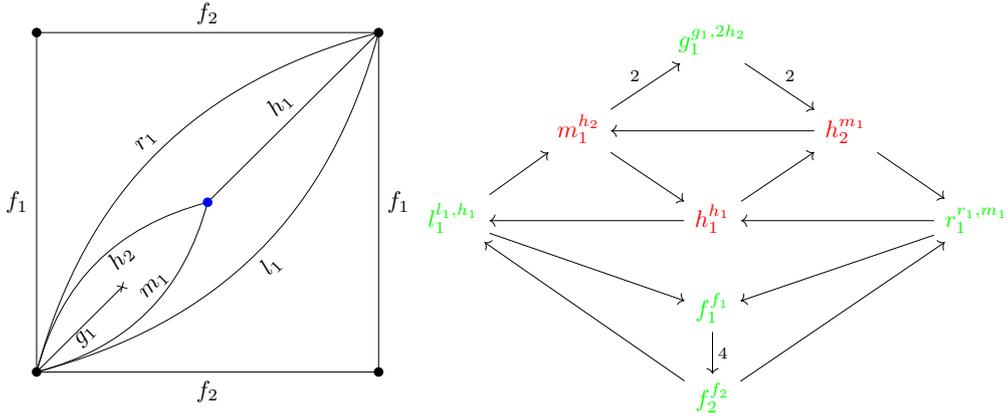
\[\adjustbox{scale=0.9}{\begin{tikzpicture}[every node/.style={draw}][every edge quotes/.style={auto=right}][baseline=(current bounding box.center)]
    \begin{scope}[every node/.style={sloped,allow upside down}][every edge quotes/.style={auto=right}]
        \fill (0,5) circle (2pt);
        \fill (5,0) circle (2pt);
        \fill (0,0) circle (2pt);
        \fill (5,5) circle (2pt);
        \fill (1,1.75) node[cross=2pt,rotate=30] {};
        \draw (0,0) -- node[rotate=-90,left] {$f_1$} (0,5);
        \draw (0,5) -- node[above] {$f_2$} (5,5);
        \draw (5,0) -- node[rotate=-90,right] {$f_1$} (5,5);
        \draw (0,0) -- node[below] {$f_2$} (5,0);
        \draw (0,0) -- node[below=0.15,right] {$g_1$} (1,1.75);
        \draw (0,0) to node[below=-.1] {$m_1$}  (3,2);
        \draw (0,0) to node[below=-.1] {$h_2$} (5,5);
        \draw (0,0) to [bend right] node[below]{$l_1$} (5,5);
        \draw (0,0) to [bend left] node[above]{$r_1$} (5,5);
        \draw (3,2) -- node[below=-.1]{$h_1$}  (5,5);
        \fill[blue] (3,2) circle (2pt);
    \end{scope}
\end{tikzpicture}}\adjustbox{scale=0.9}{\begin{tikzcd}[baseline=-80pt]
    &&\color{green}g_1^{g_1,2h_2}\arrow[ddrr,"2",bend left]&&\\
    &\color{red}m_1^{m_1,h_2}\arrow[rr]&&\color{green}h_2^{m_1}\arrow[ul,"2"]\arrow[dl]&\\
    \color{green}l_1^{l_1,h_1}\arrow[ur]\arrow[drr]&&\color{red}h_1^{h_1}\arrow[ll]&&\color{green}r_1^{r_1}\arrow[ul]\arrow[dll]\\
    &&\color{green}f_1^{f_1}\arrow[d,"4"]&&\\
    &&\color{green}f_2^{f_2}\arrow[uull]\arrow[uurr]&&
\end{tikzcd}}\]

Flip $h_2$:

\[\adjustbox{scale=0.9}{\begin{tikzpicture}[every node/.style={draw}][every edge quotes/.style={auto=right}][baseline=(current bounding box.center)]
    \begin{scope}[every node/.style={sloped,allow upside down}][every edge quotes/.style={auto=right}]
        \fill (0,5) circle (2pt);
        \fill (5,0) circle (2pt);
        \fill (0,0) circle (2pt);
        \fill (5,5) circle (2pt);
        \fill (1.25,1.25) node[cross=2pt,rotate=30] {};
        \draw (0,0) -- node[rotate=-90,left] {$f_1$} (0,5);
        \draw (0,5) -- node[above] {$f_2$} (5,5);
        \draw (5,0) -- node[rotate=-90,right] {$f_1$} (5,5);
        \draw (0,0) -- node[below] {$f_2$} (5,0);
        \draw (0,0) -- node[below=-.1] {$g_1$}  (1.25,1.25);
        \draw (0,0) to [bend right] node[above=.2,right] {$m_1$}  (2.5,2.5);
        \draw (0,0) to [bend left] node[below=.2,right] {$h_2$}  (2.5,2.5);
        \draw (0,0) to [bend right] node[below]{$l_1$} (5,5);
        \draw (0,0) to [bend left] node[above]{$r_1$} (5,5);
        \draw (2.5,2.5) -- node[above]{$h_1$}  (5,5);
        \fill[blue] (2.5,2.5) circle (2pt);
    \end{scope}
\end{tikzpicture}}\adjustbox{scale=0.9}{\begin{tikzcd}[baseline=-80pt]
    &&\color{green}g_1^{g_1,2h_2}\arrow[dr,"2"]&&\\
    &\color{red}m_1^{h_2}\arrow[ur,"2"]\arrow[dr]&&\color{red}h_2^{m_1}\arrow[ll]\arrow[dr]&\\
    \color{green}l_1^{l_1,h_1}\arrow[ur]\arrow[drr]&&\color{red}h_1^{h_1}\arrow[ll]\arrow[ur]&&\color{green}r_1^{r_1,m_1}\arrow[ll]\arrow[dll]\\
    &&\color{green}f_1^{f_1}\arrow[d,"4"]&&\\
    &&\color{green}f_2^{f_2}\arrow[uull]\arrow[uurr]&&
\end{tikzcd}}\]

Now the triangulation is same as the start except that we tagged the arcs $h_1,h_2,m_1$ notched. In other words, the first (only) interior puncture is tagged notched (in the sense that every arc attached to it is tagged notched at this end).

Step 4(a) consists of only one mutation, $g_1$ (remember that since $p=2$, $h_1,h_2$ are not core arcs).

Flip $g_1$:

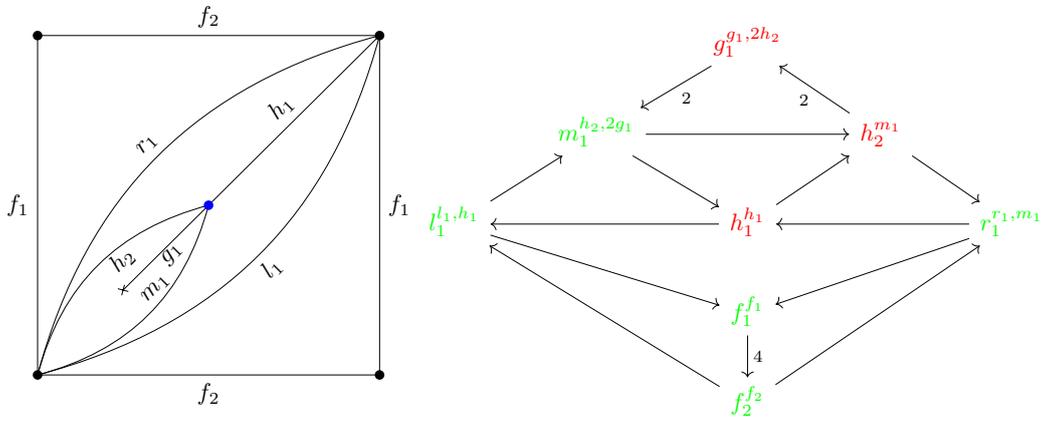
\[\adjustbox{scale=0.9}{\begin{tikzpicture}[every node/.style={draw}][every edge quotes/.style={auto=right}][baseline=(current bounding box.center)]
    \begin{scope}[every node/.style={sloped,allow upside down}][every edge quotes/.style={auto=right}]
        \fill (0,5) circle (2pt);
        \fill (5,0) circle (2pt);
        \fill (0,0) circle (2pt);
        \fill (5,5) circle (2pt);
        \fill (1.25,1.25) node[cross=2pt,rotate=30] {};
        \draw (0,0) -- node[rotate=-90,left] {$f_1$} (0,5);
        \draw (0,5) -- node[above] {$f_2$} (5,5);
        \draw (5,0) -- node[rotate=-90,right] {$f_1$} (5,5);
        \draw (0,0) -- node[below] {$f_2$} (5,0);
        \draw (1.25,1.25) -- node[below=-.1] {$g_1$}  (2.5,2.5);
        \draw (0,0) to [bend right] node[above=.2,right] {$m_1$}  (2.5,2.5);
        \draw (0,0) to [bend left] node[below=.2,right] {$h_2$}  (2.5,2.5);
        \draw (0,0) to [bend right] node[below]{$l_1$} (5,5);
        \draw (0,0) to [bend left] node[above]{$r_1$} (5,5);
        \draw (2.5,2.5) -- node[above]{$h_1$}  (5,5);
        \fill[blue] (2.5,2.5) circle (2pt);
    \end{scope}
\end{tikzpicture}}\adjustbox{scale=0.9}{\begin{tikzcd}[baseline=-80pt]
    &&\color{red}g_1^{g_1,2h_2}\arrow[dl,"2"]&&\\
    &\color{green}m_1^{h_2,2g_1}\arrow[rr]\arrow[dr]&&\color{red}h_2^{m_1}\arrow[ul,"2"]\arrow[dr]&\\
    \color{green}l_1^{l_1,h_1}\arrow[ur]\arrow[drr]&&\color{red}h_1^{h_1}\arrow[ll]\arrow[ur]&&\color{green}r_1^{r_1,m_1}\arrow[ll]\arrow[dll]\\
    &&\color{green}f_1^{f_1}\arrow[d,"4"]&&\\
    &&\color{green}f_2^{f_2}\arrow[uull]\arrow[uurr]&&
\end{tikzcd}}\]

Now the only core arc is moved away from the corner.

Step 4(c) consists of $(l_1,f_1,r_1,f_2,l_1,f_1)$. Flip them:

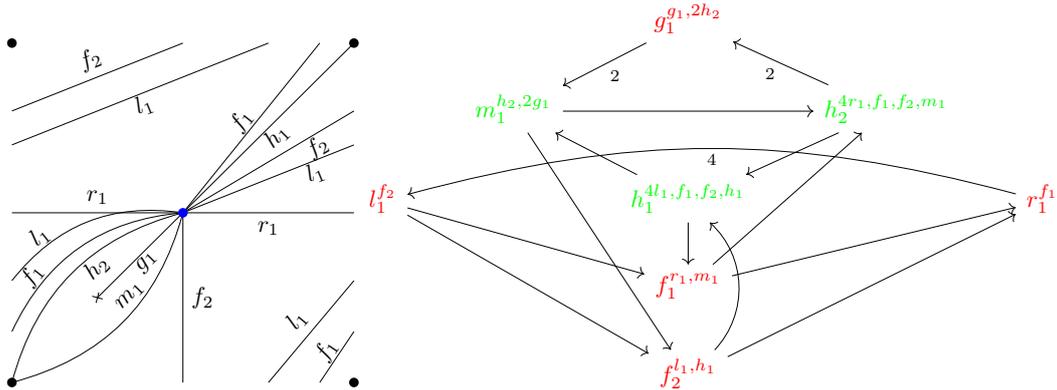
\[\adjustbox{scale=0.9}{\begin{tikzpicture}[every node/.style={draw}][every edge quotes/.style={auto=right}][baseline=(current bounding box.center)]
    \begin{scope}[every node/.style={sloped,allow upside down}][every edge quotes/.style={auto=right}]
        \fill (0,5) circle (2pt);
        \fill (5,0) circle (2pt);
        \fill (0,0) circle (2pt);
        \fill (5,5) circle (2pt);
        \fill (1.25,1.25) node[cross=2pt,rotate=30] {};
        \draw (2.5,2.5) -- node[above=-.1] {$f_1$}  (4.5,5);
        \draw (0,4) -- node[above] {$f_2$} (2.5,5);
        \draw (0,0.75) to [bend left] node[near start, above=-.1] {$f_1$}  (2.5,2.5);
        \draw (4.5,0) -- node[above=-.1] {$f_1$} (5,0.75);
        \draw (2.5,2.5) -- node[right,rotate=90] {$f_2$}  (2.5,0);
        \draw (2.5,2.5) -- node[below=-.1,near end] {$f_2$}  (5,4);
        \draw (1.25,1.25) -- node[below=-.1] {$g_1$}  (2.5,2.5);
        \draw (0,0) to [bend right] node[above=.2,right] {$m_1$}  (2.5,2.5);
        \draw (0,0) to [bend left] node[below=.2,right] {$h_2$}  (2.5,2.5);
        \draw (0,1.5) to [bend left] node[near start,above=-.1]{$l_1$}  (2.5,2.5);
        \draw (3.75,0) to node[above]{$l_1$} (5,1.5);
        \draw (0,3.5) to node[below=-.1]{$l_1$} (3.75,5);
        \draw (2.5,2.5) to node[near end,below=-.1]{$l_1$}  (5,3.5);
        \draw (2.5,2.5) to node[below]{$r_1$}  (5,2.5);
        \draw (0,2.5) to node[above]{$r_1$}  (2.5,2.5);
        \draw (2.5,2.5) -- node[below=-.1]{$h_1$}  (5,5);
        \fill[blue] (2.5,2.5) circle (2pt);
    \end{scope}
\end{tikzpicture}}\adjustbox{scale=0.9}{\begin{tikzcd}[baseline=-80pt]
    &&\color{red}g_1^{g_1,2h_2}\arrow[dl,"2"]&&\\
    &\color{green}m_1^{h_2,2g_1}\arrow[rr]\arrow[dddr]&&\color{green}h_2^{4r_1,f_1,f_2,m_1}\arrow[ul,"2"]\arrow[dl]&\\
    \color{red}l_1^{f_2}\arrow[ddrr]\arrow[drr]&&\color{green}h_1^{4l_1,f_1,f_2,h_1}\arrow[ul]\arrow[d]&&\color{red}r_1^{f_1}\arrow[llll,"4",bend right=15]\\
    &&\color{red}f_1^{r_1,m_1}\arrow[urr]\arrow[uur]&&\\
    &&\color{red}f_2^{l_1,h_1}\arrow[uurr]\arrow[uu,bend right=45]&&
\end{tikzcd}}\]

Now all outer arcs are moved away from the corner.

Step 5 consists of $(h_1,m_1,h_2,h_1)$.

Flip $h_1$:

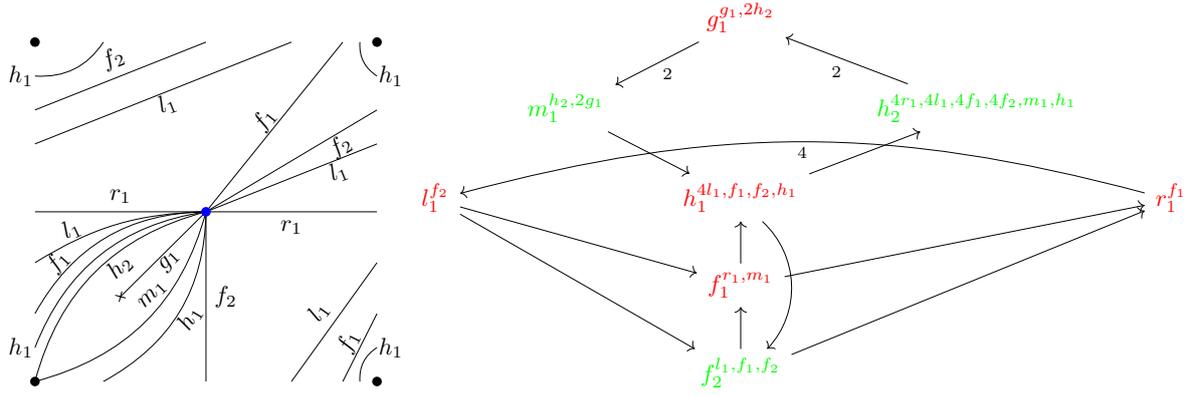
\[\adjustbox{scale=0.9}{\begin{tikzpicture}[every node/.style={draw}][every edge quotes/.style={auto=right}]
    \begin{scope}[every node/.style={sloped,allow upside down}][every edge quotes/.style={auto=right}]
        \fill (0,5) circle (2pt);
        \fill (5,0) circle (2pt);
        \fill (0,0) circle (2pt);
        \fill (5,5) circle (2pt);
        \fill (1.25,1.25) node[cross=2pt,rotate=30] {};
        \draw (2.5,2.5) -- node[above=-.1] {$f_1$}  (4.5,5);
        \draw (0,4) -- node[above] {$f_2$} (2.5,5);
        \draw (0,1) to [bend left] node[near start, above=-.11] {$f_1$}  (2.5,2.5);
        \draw (4.5,0) -- node[above=-.1] {$f_1$} (5,1);
        \draw (2.5,2.5) -- node[right,rotate=90] {$f_2$}  (2.5,0);
        \draw (2.5,2.5) -- node[below=-.1,near end] {$f_2$}  (5,4);
        \draw (1.25,1.25) -- node[below=-.1] {$g_1$}  (2.5,2.5);
        \draw (0,0) to [bend right] node[above=.2,right] {$m_1$}  (2.5,2.5);
        \draw (0,0) to [bend left] node[below=.2,right] {$h_2$}  (2.5,2.5);
        \draw (0,1.75) to [bend left=15] node[near start,above=-.1]{$l_1$}  (2.5,2.5);
        \draw (3.75,0) to node[above]{$l_1$} (5,1.75);
        \draw (0,3.5) to node[below=-.1]{$l_1$} (3.75,5);
        \draw (2.5,2.5) to node[near end,below=-.1]{$l_1$}  (5,3.5);
        \draw (2.5,2.5) to node[below]{$r_1$}  (5,2.5);
        \draw (0,2.5) to node[above]{$r_1$}  (2.5,2.5);
        \draw (0,0.5) to [bend left]  (2.5,2.5);
        \draw (4.75,0) to [bend left] (5,0.5);
        \draw (4.75,5) to [bend right] (5,4.5);
        \draw (0,4.5) to [bend right] (1,5);
        \draw (1,0) to [bend right] node[below=-.1] {$h_1$}  (2.5,2.5);
        \draw (0,0.5) node[left=-.1] {$h_1$};
        \draw (5,0.5) node[right=-.1] {$h_1$};
        \draw (5,4.5) node[right=-.1] {$h_1$};
        \draw (0,4.5) node[left=-.1] {$h_1$};
        \fill[blue] (2.5,2.5) circle (2pt);
    \end{scope}
\end{tikzpicture}}\adjustbox{scale=0.9}{\begin{tikzcd}[baseline=-80pt]
    &&\color{red}g_1^{g_1,2h_2}\arrow[dl,"2"]&&\\
    &\color{green}m_1^{h_2,2g_1}\arrow[dr]&&\color{green}h_2^{4r_1,4l_1,4f_1,4f_2,m_1,h_1}\arrow[ul,"2"]&\\
    \color{red}l_1^{f_2}\arrow[ddrr]\arrow[drr]&&\color{red}h_1^{4l_1,f_1,f_2,h_1}\arrow[ur]\arrow[dd,bend left=45]&&\color{red}r_1^{f_1}\arrow[llll,"4",bend right=15]\\
    &&\color{red}f_1^{r_1,m_1}\arrow[urr]\arrow[u]&&\\
    &&\color{green}f_2^{l_1,f_1,f_2}\arrow[uurr]\arrow[u]&&
\end{tikzcd}}\]

Observe that the arcs $h_1,g_1,h_2,m_1$ form a triangulated once-punctured digon:

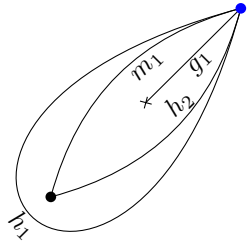
\[\begin{tikzpicture}[every node/.style={draw}][every edge quotes/.style={auto=right}][baseline=(current bounding box.center)]
    \begin{scope}[every node/.style={sloped,allow upside down}][every edge quotes/.style={auto=right}]
        \fill (0,0) circle (2pt);
        \fill (1.25,1.25) node[cross=2pt,rotate=30] {};
        \draw (1.25,1.25) -- node[below=-.1] {$g_1$}  (2.5,2.5);
        \draw (0,0) to [bend right] node[above=.2,right] {$h_2$}  (2.5,2.5);
        \draw (0,0) to [bend left] node[below=.2,right] {$m_1$}  (2.5,2.5);
        \draw (2.5,2.5) to [loop,in=255,out=195,min distance=60mm] node[below=-.1] {$h_1$} (2.5,2.5);
        \fill[blue] (2.5,2.5) circle (2pt);
    \end{scope}
\end{tikzpicture}\]

Now flip $m_1$ and then $h_2$ (and again we switch the positions of $h_2$ and $m_1$):

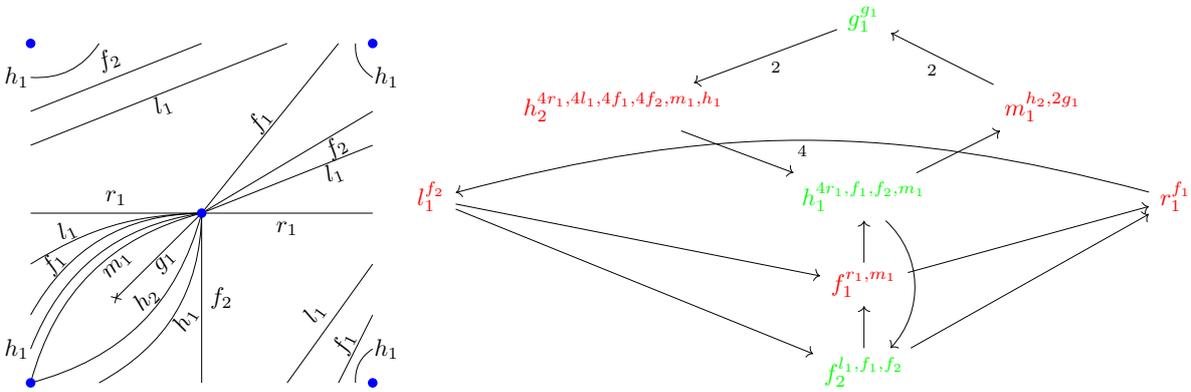
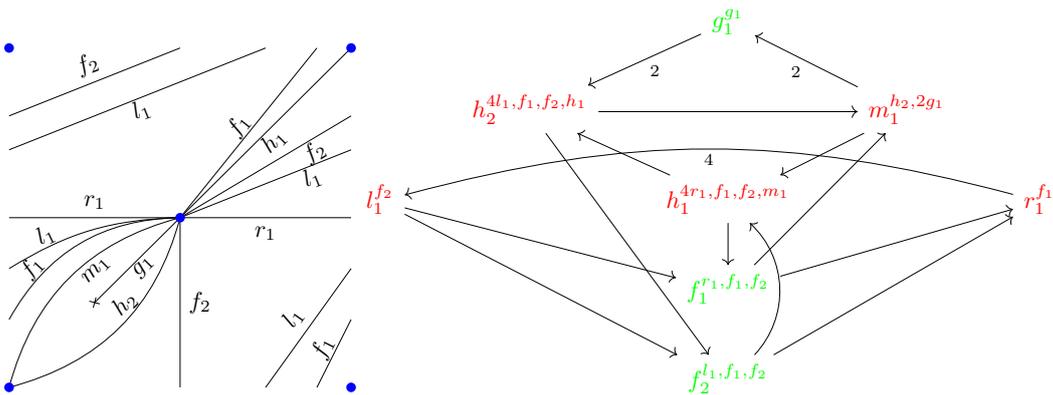
\[\adjustbox{scale=0.9}{\begin{tikzpicture}[every node/.style={draw}][every edge quotes/.style={auto=right}][baseline=(current bounding box.center)]
    \begin{scope}[every node/.style={sloped,allow upside down}][every edge quotes/.style={auto=right}]
        \fill (1.25,1.25) node[cross=2pt,rotate=30] {};
        \draw (2.5,2.5) -- node[above=-.1] {$f_1$}  (4.5,5);
        \draw (0,4) -- node[above] {$f_2$} (2.5,5);
        \draw (0,1) to [bend left] node[near start, above=-.11] {$f_1$}  (2.5,2.5);
        \draw (4.5,0) -- node[above=-.1] {$f_1$} (5,1);
        \draw (2.5,2.5) -- node[right,rotate=90] {$f_2$}  (2.5,0);
        \draw (2.5,2.5) -- node[below=-.1,near end] {$f_2$}  (5,4);
        \draw (1.25,1.25) -- node[below=-.1] {$g_1$}  (2.5,2.5);
        \draw (0,0) to [bend right] node[above=.2,right] {$h_2$}   (2.5,2.5);
        \draw (0,0) to [bend left] node[below=.2,right] {$m_1$}    (2.5,2.5);
        \draw (0,1.75) to [bend left=15] node[near start,above=-.1]{$l_1$}  (2.5,2.5);
        \draw (3.75,0) to node[above]{$l_1$} (5,1.75);
        \draw (0,3.5) to node[below=-.1]{$l_1$} (3.75,5);
        \draw (2.5,2.5) to node[near end,below=-.1]{$l_1$}  (5,3.5);
        \draw (2.5,2.5) to node[below]{$r_1$}  (5,2.5);
        \draw (0,2.5) to node[above]{$r_1$}  (2.5,2.5);
        \draw (0,0.5) to [bend left]  (2.5,2.5);
        \draw (4.75,0) to [bend left] (5,0.5);
        \draw (4.75,5) to [bend right] (5,4.5);
        \draw (0,4.5) to [bend right] (1,5);
        \draw (1,0) to [bend right] node[below=-.1] {$h_1$}  (2.5,2.5);
        \draw (0,0.5) node[left=-.1] {$h_1$};
        \draw (5,0.5) node[right=-.1] {$h_1$};
        \draw (5,4.5) node[right=-.1] {$h_1$};
        \draw (0,4.5) node[left=-.1] {$h_1$};
        \fill[blue] (0,5) circle (2pt);
        \fill[blue] (5,0) circle (2pt);
        \fill[blue] (0,0) circle (2pt);
        \fill[blue] (5,5) circle (2pt);
        \fill[blue] (2.5,2.5) circle (2pt);
    \end{scope}
\end{tikzpicture}}\adjustbox{scale=0.9}{\begin{tikzcd}[baseline=-80pt]
    &&\color{green}g_1^{g_1}\arrow[dl,"2"]&&\\
    &\color{red}h_2^{4r_1,4l_1,4f_1,4f_2,m_1,h_1}\arrow[dr]&&\color{red}m_1^{h_2,2g_1}\arrow[ul,"2"]&\\
    \color{red}l_1^{f_2}\arrow[ddrr]\arrow[drr]&&\color{green}h_1^{4r_1,f_1,f_2,m_1}\arrow[ur]\arrow[dd,bend left=45]&&\color{red}r_1^{f_1}\arrow[llll,"4",bend right=15]\\
    &&\color{red}f_1^{r_1,m_1}\arrow[urr]\arrow[u]&&\\
    &&\color{green}f_2^{l_1,f_1,f_2}\arrow[uurr]\arrow[u]&&
\end{tikzcd}}\]

Flip $h_1$:

\[\adjustbox{scale=0.9}{\begin{tikzpicture}[every node/.style={draw}][every edge quotes/.style={auto=right}][baseline=(current bounding box.center)]
    \begin{scope}[every node/.style={sloped,allow upside down}][every edge quotes/.style={auto=right}]
        \fill (1.25,1.25) node[cross=2pt,rotate=30] {};
        \draw (2.5,2.5) -- node[above=-.1] {$f_1$}  (4.5,5);
        \draw (0,4) -- node[above] {$f_2$} (2.5,5);
        \draw (0,1) to [bend left] node[near start, above=-.11] {$f_1$}  (2.5,2.5);
        \draw (4.5,0) -- node[above=-.1] {$f_1$} (5,1);
        \draw (2.5,2.5) -- node[right,rotate=90] {$f_2$}  (2.5,0);
        \draw (2.5,2.5) -- node[below=-.1,near end] {$f_2$}  (5,4);
        \draw (1.25,1.25) -- node[below=-.1] {$g_1$}  (2.5,2.5);
        \draw (0,0) to [bend right] node[above=.2,right] {$h_2$}   (2.5,2.5);
        \draw (0,0) to [bend left] node[below=.2,right] {$m_1$}    (2.5,2.5);
        \draw (0,1.75) to [bend left=15] node[near start,above=-.1]{$l_1$}  (2.5,2.5);
        \draw (3.75,0) to node[above]{$l_1$} (5,1.75);
        \draw (0,3.5) to node[below=-.1]{$l_1$} (3.75,5);
        \draw (2.5,2.5) to node[near end,below=-.1]{$l_1$}  (5,3.5);
        \draw (2.5,2.5) to node[below]{$r_1$}  (5,2.5);
        \draw (0,2.5) to node[above]{$r_1$}  (2.5,2.5);
        \draw (2.5,2.5) -- node[below=-.1]{$h_1$}   (5,5);
        \fill[blue] (0,5) circle (2pt);
        \fill[blue] (5,0) circle (2pt);
        \fill[blue] (0,0) circle (2pt);
        \fill[blue] (5,5) circle (2pt);
        \fill[blue] (2.5,2.5) circle (2pt);
    \end{scope}
\end{tikzpicture}}\adjustbox{scale=0.9}{\begin{tikzcd}[baseline=-80pt]
    &&\color{green}g_1^{g_1}\arrow[dl,"2"]&&\\
    &\color{red}h_2^{4l_1,f_1,f_2,h_1}\arrow[rr]\arrow[dddr]&&\color{red}m_1^{h_2,2g_1}\arrow[ul,"2"]\arrow[dl]&\\
    \color{red}l_1^{f_2}\arrow[ddrr]\arrow[drr]&&\color{red}h_1^{4r_1,f_1,f_2,m_1}\arrow[ul]\arrow[d]&&\color{red}r_1^{f_1}\arrow[llll,"4",bend right=15]\\
    &&\color{green}f_1^{r_1,f_1,f_2}\arrow[urr]\arrow[uur]&&\\
    &&\color{green}f_2^{l_1,f_1,f_2}\arrow[uurr]\arrow[uu,bend right=45]&&
\end{tikzcd}}\]

Again the triangulation is same as the start of step 5, except that the corner is tagged notched just like the interior puncture.

Step 6(a) consists of $(f_1,l_1,f_2,r_1,f_1,l_1)$. Flip them:

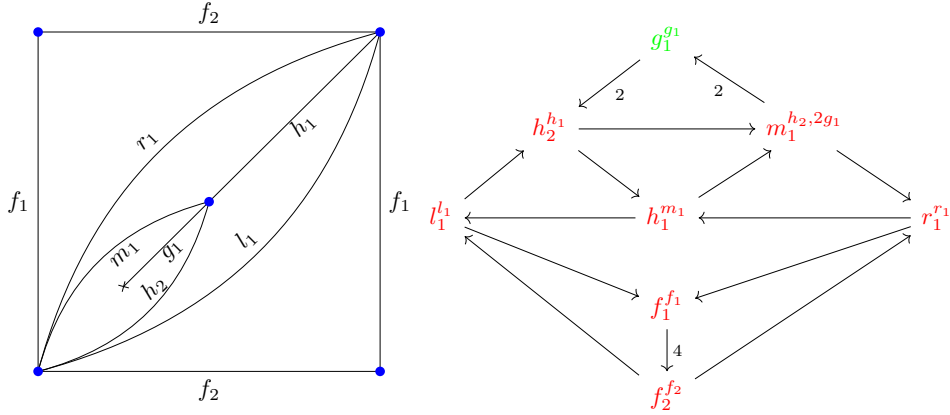
\[\adjustbox{scale=0.9}{\begin{tikzpicture}[every node/.style={draw}][every edge quotes/.style={auto=right}][baseline=(current bounding box.center)]
    \begin{scope}[every node/.style={sloped,allow upside down}][every edge quotes/.style={auto=right}]
        \fill (1.25,1.25) node[cross=2pt,rotate=30] {};
        \draw (0,5) -- node[above] {$f_2$}   (5,5);
        \draw (0,0) -- node[below] {$f_2$}   (5,0);
        \draw (5,0) to node[right,rotate=-90] {$f_1$}   (5,5);
        \draw (0,0) to node[left,rotate=-90] {$f_1$}   (0,5);
        \draw (1.25,1.25) -- node[below=-.1] {$g_1$}  (2.5,2.5);
        \draw (0,0) to [bend right] node[above=.2,right] {$h_2$}   (2.5,2.5);
        \draw (0,0) to [bend left] node[below=.2,right] {$m_1$}    (2.5,2.5);
        \draw (0,0) to [bend right] node[above]{$l_1$}   (5,5);
        \draw (0,0) to [bend left] node[above]{$r_1$}   (5,5);
        \draw (2.5,2.5) -- node[below=-.1]{$h_1$}   (5,5);
        \fill[blue] (0,5) circle (2pt);
        \fill[blue] (5,0) circle (2pt);
        \fill[blue] (0,0) circle (2pt);
        \fill[blue] (5,5) circle (2pt);
        \fill[blue] (2.5,2.5) circle (2pt);
    \end{scope}
\end{tikzpicture}}\adjustbox{scale=0.9}{\begin{tikzcd}[baseline=-80pt]
    &&\color{green}g_1^{g_1}\arrow[dl,"2"]&&\\
    &\color{red}h_2^{h_1}\arrow[rr]\arrow[dr]&&\color{red}m_1^{h_2,2g_1}\arrow[ul,"2"]\arrow[dr]&\\
    \color{red}l_1^{l_1}\arrow[drr]\arrow[ur]&&\color{red}h_1^{m_1}\arrow[ll]\arrow[ur]&&\color{red}r_1^{r_1}\arrow[ll]\arrow[dll]\\
    &&\color{red}f_1^{f_1}\arrow[d,"4"]&&\\
    &&\color{red}f_2^{f_2}\arrow[uull]\arrow[uurr]&&
\end{tikzcd}}\]

The outer arcs are back to their original position.

Step 6(c) consists of just $g_1$. Flip it:

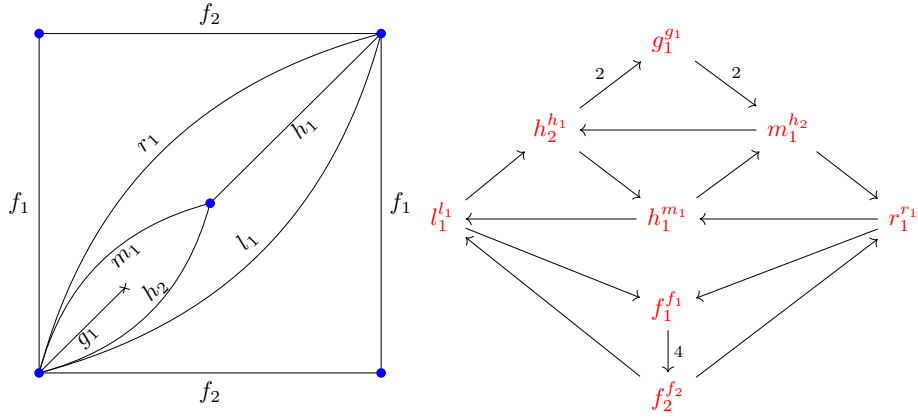
\[\adjustbox{scale=0.9}{\begin{tikzpicture}[every node/.style={draw}][every edge quotes/.style={auto=right}][baseline=(current bounding box.center)]
    \begin{scope}[every node/.style={sloped,allow upside down}][every edge quotes/.style={auto=right}]
        \fill (1.25,1.25) node[cross=2pt,rotate=30] {};
        \draw (0,5) -- node[above] {$f_2$}   (5,5);
        \draw (0,0) -- node[below] {$f_2$}   (5,0);
        \draw (5,0) to node[right,rotate=-90] {$f_1$}   (5,5);
        \draw (0,0) to node[left,rotate=-90] {$f_1$}   (0,5);
        \draw (0,0) -- node[below=-.1] {$g_1$}  (1.25,1.25);
        \draw (0,0) to [bend right] node[above=.2,right] {$h_2$}   (2.5,2.5);
        \draw (0,0) to [bend left] node[below=.2,right] {$m_1$}    (2.5,2.5);
        \draw (0,0) to [bend right] node[above]{$l_1$}   (5,5);
        \draw (0,0) to [bend left] node[above]{$r_1$}   (5,5);
        \draw (2.5,2.5) -- node[below=-.1]{$h_1$}   (5,5);
        \fill[blue] (0,5) circle (2pt);
        \fill[blue] (5,0) circle (2pt);
        \fill[blue] (0,0) circle (2pt);
        \fill[blue] (5,5) circle (2pt);
        \fill[blue] (2.5,2.5) circle (2pt);
    \end{scope}
\end{tikzpicture}}\adjustbox{scale=0.9}{\begin{tikzcd}[baseline=-80pt]
    &&\color{red}g_1^{g_1}\arrow[dr,"2"]&&\\
    &\color{red}h_2^{h_1}\arrow[ur,"2"]\arrow[dr]&&\color{red}m_1^{h_2}\arrow[ll]\arrow[dr]&\\
    \color{red}l_1^{l_1}\arrow[drr]\arrow[ur]&&\color{red}h_1^{m_1}\arrow[ll]\arrow[ur]&&\color{red}r_1^{r_1}\arrow[ll]\arrow[dll]\\
    &&\color{red}f_1^{f_1}\arrow[d,"4"]&&\\
    &&\color{red}f_2^{f_2}\arrow[uull]\arrow[uurr]&&
\end{tikzcd}}\]

Since all vertices are now red, $\Delta_1$ is indeed a maximal green sequence. $\blacksquare$

To summarise the above, we start with creating a once-punctured digon at the middle puncture, which allows us to tag the ends at the puncture notched. After that we flip $h_2$ back, this allows us to flip $g_1$ and end up having the middle puncture instead of the corner puncture as an endpoint of $g_1$. We then flip arcs to create a once-punctured digon at the corner puncture by making both ends of most of the arcs the middle puncture, and then tag the ends at the corner puncture notched. After this, we undo the flips (which is to redo the flips in reverse order) to turn the triangulation back to the starting one except both ends of all ordinary arcs and one end of the pending arc are tagged notched.

As flipping an arc (mutating at a vertex) is a local operation, instead of drawing the entire triangulation (diagram), we shall draw only the arcs (vertices) involved and adjacent arcs (vertices).

\subsubsection{The case where $p=2,q>1$}\label{2,>1}
\begin{proposition}
    $\Delta_1$ is a maximal green sequence when $p=2$ and $q>1$.
\end{proposition}

\textbf{Proof} When $p=2,q>1$, the mutation sequence $\Delta_1$ translates to

\begin{enumerate}
    \item Tagging alternate punctures notched. $(\phi)$
    \item Tagging the first interior puncture notched. $(h_{q+1},h_q,...,h_1,m_1,h_2,h_3,...,h_{q+1})$
    \item Tagging the remaining interior punctures notched and moving half the inner arcs away from the corner. $(\phi)$
    \item Moving arcs away from the corner.
    \begin{enumerate}
        \item Moving core arcs away from the corner. $(g_1,g_2,...,g_q,h_{q},h_{q-1},...,h_2)$
        \item Moving the remaining inner arcs away from the corner. $(\phi)$
        \item Moving outer arcs away from the corner. $(l_1,f_1,r_1,f_2,l_1,f_1)$
    \end{enumerate}
    \item Tagging the corner notched. $(h_1,m_1,h_{q+1},h_1)$
    \item Moving arcs back to the corner.
    \begin{enumerate}
        \item Moving outer arcs back to the corner. 
        $f_1,l_1,f_2,r_1,f_1,l_1$
        \item Moving inner arcs back to the corner. 
        $(\phi)$
        \item Moving core arcs back to the corner. $(h_2,h_3,...,h_q,g_1,g_2,...,g_{q-1},g_q)$
        \item (If $p>2$ and is even) Moving the remaining pair of inner arc back to the corner. $(\phi)$
    \end{enumerate}
\end{enumerate}

Again we only have steps 2, 4(a), 4(c), 5, 6(a) and 6(c).

The triangulation $T_{2,q}$ looks like this:

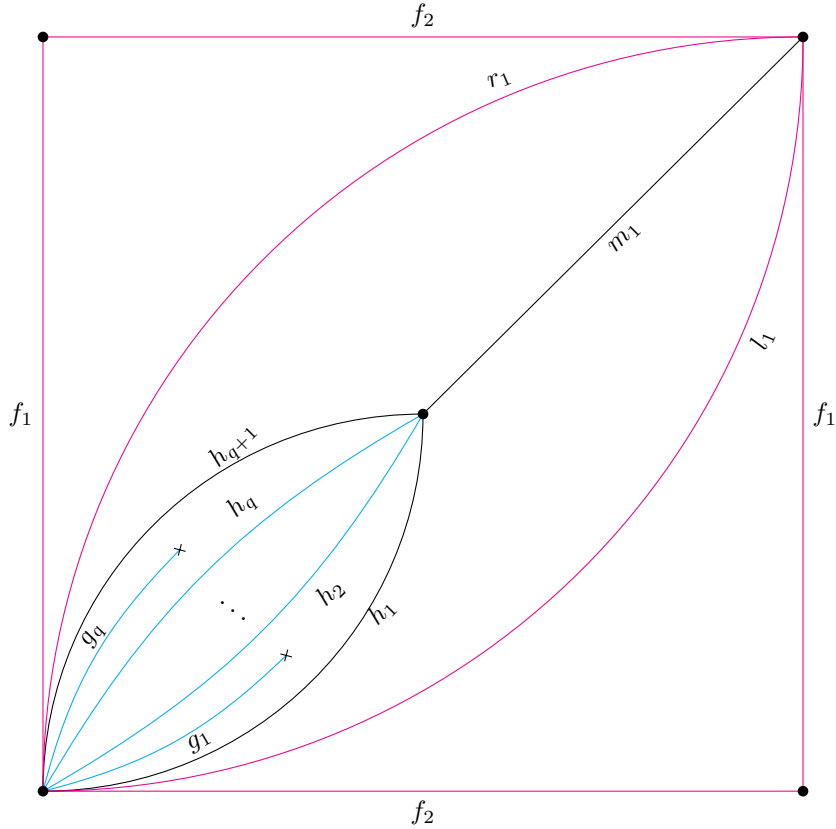
\[\begin{tikzpicture}[every edge quotes/.style={auto=right}]
    \begin{scope}[every node/.style={sloped,allow upside down}][every edge quotes/.style={auto=right}]
        \node at ($(0,5)!.5!(5,0)$) {$\ddots$};
        \draw[cyan] (0,0) to [bend right=15] node [below=.15,right,black] {$g_1$} (3.2,1.8);
        \draw[cyan] (0,0) to [bend left=15] node [above=.15,right,black] {$g_q$} (1.8,3.2);
        \draw[magenta] (0,0) -- node[rotate=-90,left,black] {$f_1$} (0,10);
        \draw[magenta] (0,10) -- node[above,black] {$f_2$} (10,10);
        \draw[magenta] (10,0) -- node[rotate=-90,right,black] {$f_1$} (10,10);
        \draw[magenta] (0,0) -- node[below,black] {$f_2$} (10,0);
        \draw (0,0) to [bend right=45] node[below=.15,right=1] {$h_1$} (5,5);
        \draw (0,0) to [bend left=45] node[above=.15,right=1] {$h_{q+1}$} (5,5);
        \draw[cyan] (0,0) to [bend right=15] node[right=1,below,black] {$h_2$} (5,5);
        \draw[cyan] (0,0) to [bend left=15] node[right=1,above,black] {$h_q$} (5,5);
        \draw[magenta] (0,0) to [bend right=45] node[near end,below,black] {$l_{1}$} (10,10);
        \draw[magenta] (0,0) to [bend left=45] node[near end,above,black] {$r_{1}$} (10,10);
        \draw (5,5) to node[below] {$m_1$} (10,10);
        \fill (0,10) circle (2pt);
        \fill (10,0) circle (2pt);
        \fill (0,0) circle (2pt);
        \fill (10,10) circle (2pt);
        \fill (5,5) circle (2pt);
        \fill (3.2,1.8) node[cross=2pt,rotate=30] {};
        \fill (1.8,3.2) node[cross=2pt,rotate=30] {};
    \end{scope}
\end{tikzpicture}\]

Step 2 consists of $(h_{q+1},h_q,...,h_1,m_1,h_2,h_3,...,h_{q+1})$, look at them and adjacent vertices: 

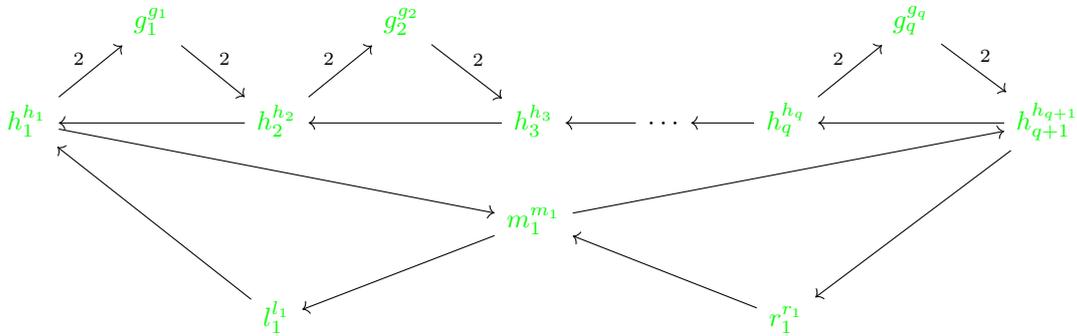
\[\begin{tikzcd}
    &|[color=green]|g_{1}^{g_1}\arrow[dr,"2"]&&|[color=green]|g_{2}^{g_2}\arrow[dr,"2"]&&&&|[color=green]|g_{q}^{g_q}\arrow[dr,"2"]\\
    |[color=green]|h_{1}^{h_1}\arrow[ur,"2"]\arrow[drrrr]&&|[color=green]|h_{2}^{h_2}\arrow[ur,"2"]\arrow[ll]&&|[color=green]|{h_{3}^{h_{3}}}\arrow[ll]&\cdots\arrow[l]&|[color=green]|h_{q}^{h_q}\arrow[ur,"2"]\arrow[l]&&|[color=green]|{h_{q+1}^{h_{q+1}}}\arrow[ll]\arrow[ddll]\\
    &&&&|[color=green]|m_1^{m_1}\arrow[urrrr]\arrow[dll]\\
    &&|[color=green]|l_1^{l_1}\arrow[uull]&&&&|[color=green]|r_1^{r_1}\arrow[ull]
\end{tikzcd}\]

For convenience, we shall show the result of step 2 in a lemma:

\begin{lemma}\label{p>2}
    Mutating the following diagram:
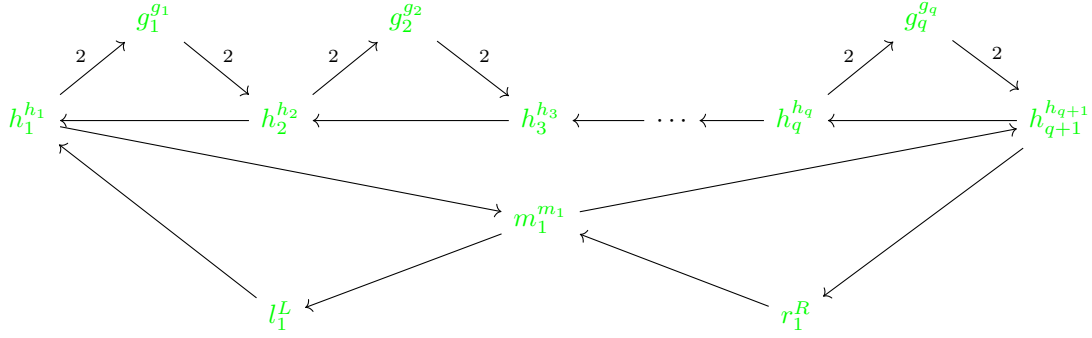
\[\begin{tikzcd}
    &|[color=green]|g_{1}^{g_1}\arrow[dr,"2"]&&|[color=green]|g_{2}^{g_2}\arrow[dr,"2"]&&&&|[color=green]|g_{q}^{g_q}\arrow[dr,"2"]\\
    |[color=green]|h_{1}^{h_1}\arrow[ur,"2"]\arrow[drrrr]&&|[color=green]|h_{2}^{h_2}\arrow[ur,"2"]\arrow[ll]&&|[color=green]|{h_{3}^{h_{3}}}\arrow[ll]&\cdots\arrow[l]&|[color=green]|h_{q}^{h_q}\arrow[ur,"2"]\arrow[l]&&|[color=green]|{h_{q+1}^{h_{q+1}}}\arrow[ll]\arrow[ddll]\\
    &&&&|[color=green]|m_1^{m_1}\arrow[urrrr]\arrow[dll]\\
    &&|[color=green]|l_1^{L}\arrow[uull]&&&&|[color=green]|r_1^{R}\arrow[ull]
\end{tikzcd}\]
    , where $L,R$ are sets of (frozen) vertices, with the mutation sequence $(h_{q+1},h_q,...,h_1,m_1,h_2,h_3,...,h_{q+1})$ gives the following diagram (after relocating some vertices):
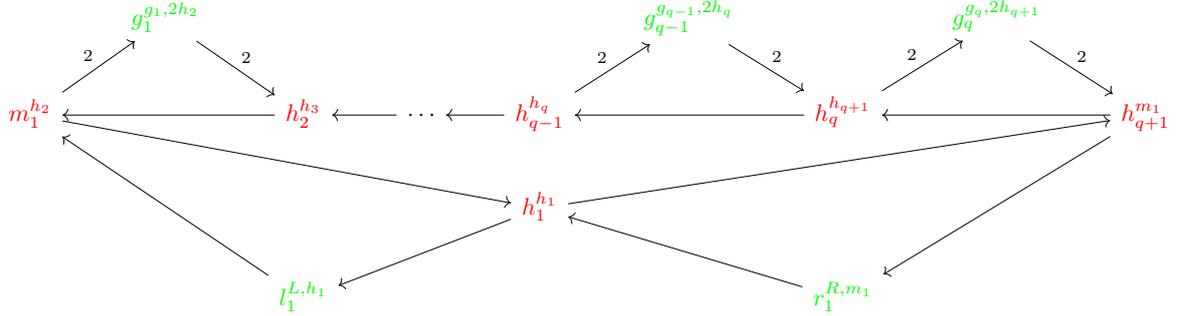
\[\adjustbox{scale=0.9,center}{
\begin{tikzcd}
    &\color{green}g_1^{g_1,2h_2}\arrow[dr,"2"]&&&&\color{green}g_{q-1}^{g_{q-1},2h_q}\arrow[dr,"2"]&&\color{green}g_q^{g_q,2h_{q+1}}\arrow[dr,"2"]\\
    \color{red}m_1^{h_2}\arrow[drrrr]\arrow[ur,"2"]&&\color{red}h_2^{h_3}\arrow[ll]&\cdots\arrow[l]&\color{red}h_{q-1}^{h_q}\arrow[ur,"2"]\arrow[l]&&\color{red}h_{q}^{h_{q+1}}\arrow[ll]\arrow[ur,"2"]&&\color{red}h_{q+1}^{m_1}\arrow[ll]\arrow[ddll]\\
    &&&&\color{red}h_1^{h_1}\arrow[urrrr]\arrow[dll]\\
    &&\color{green}l_1^{L,h_1}\arrow[uull]&&&&\color{green}r_1^{R,m_1}\arrow[ull]
\end{tikzcd}}\]
\end{lemma}

\begin{proof}
    We start with the initial mutations to observe the pattern of the diagram and the frozen vertices each vertex is connected to.

    At $h_{q+1},h_q,h_{q-1}$ and adjacent vertices, the full subdiagram looks like this:

    
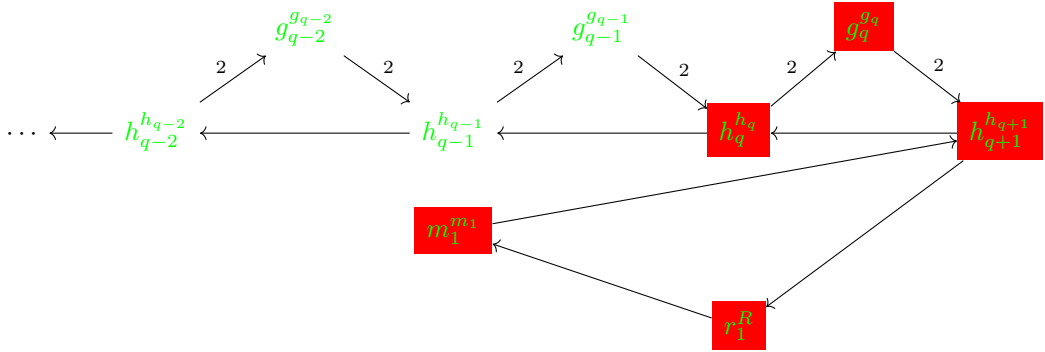
\[\begin{tikzcd}
    &&\color{green}g_{q-2}^{g_{q-2}}\arrow[dr,"2"]&&\color{green}g_{q-1}^{g_{q-1}}\arrow[dr,"2"]&&|[color=green,fill=red]|g_{q}^{g_q}\arrow[dr,"2"]\\
    \cdots&\color{green}h_{q-2}^{h_{q-2}}\arrow[ur,"2"]\arrow[l]&&\color{green}h_{q-1}^{h_{q-1}}\arrow[ur,"2"]\arrow[ll]&&|[color=green,fill=red]|h_{q}^{h_q}\arrow[ur,"2"]\arrow[ll]&&|[color=green,fill=red]|{h_{q+1}^{h_{q+1}}}\arrow[ll]\arrow[ddll]\\
    &&&|[color=green,fill=red]|m_1^{m_1}\arrow[urrrr]\\
    &&&&&|[color=green,fill=red]|r_1^{R}\arrow[ull]
\end{tikzcd}\]

Mutate at $h_{q+1}$:

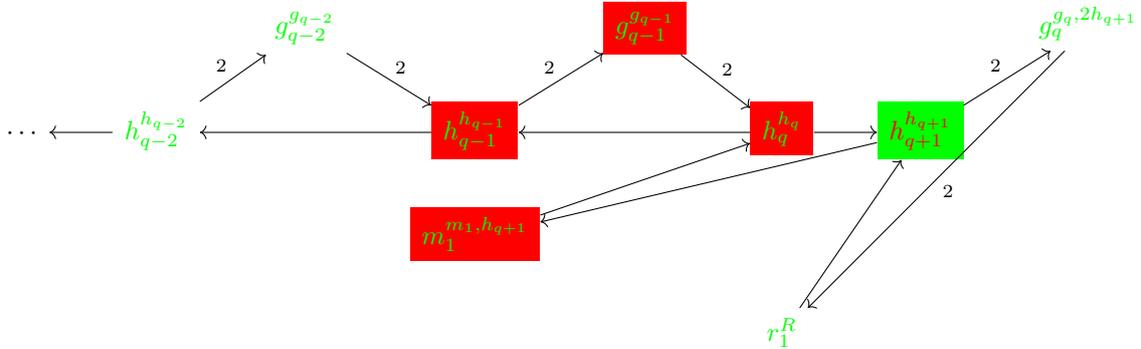
\[\begin{tikzcd}
    &&\color{green}g_{q-2}^{g_{q-2}}\arrow[dr,"2"]&&|[color=green,fill=red]|g_{q-1}^{g_{q-1}}\arrow[dr,"2"]&&&\color{green}g_{q}^{g_q,2h_{q+1}}\arrow[dddll,"2"]\\
    \cdots&\color{green}h_{q-2}^{h_{q-2}}\arrow[ur,"2"]\arrow[l]&&|[color=green,fill=red]|h_{q-1}^{h_{q-1}}\arrow[ur,"2"]\arrow[ll]&&|[color=green,fill=red]|h_{q}^{h_q}\arrow[ll]\arrow[r]&|[color=red,fill=green]|h_{q+1}^{h_{q+1}}\arrow[dlll]\arrow[ur,"2"]\\
    &&&|[color=green,fill=red]|m_1^{m_1,h_{q+1}}\arrow[urr]\\
    &&&&&\color{green}r_1^{R}\arrow[uur]
\end{tikzcd}\]

Mutate at $h_q$:

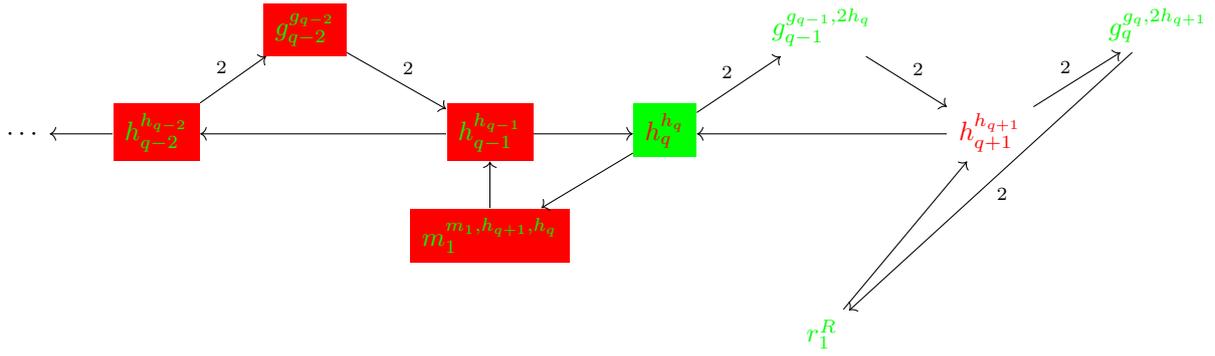
\[\begin{tikzcd}
    &&|[color=green,fill=red]|g_{q-2}^{g_{q-2}}\arrow[dr,"2"]&&&\color{green}g_{q-1}^{g_{q-1},2h_q}\arrow[dr,"2"]&&\color{green}g_{q}^{g_q,2h_{q+1}}\arrow[dddll,"2"]\\
    \cdots&|[color=green,fill=red]|h_{q-2}^{h_{q-2}}\arrow[ur,"2"]\arrow[l]&&|[color=green,fill=red]|h_{q-1}^{h_{q-1}}\arrow[ll]\arrow[r]&|[color=red,fill=green]|h_{q}^{h_q}\arrow[ur,"2"]\arrow[dl]&&\color{red}h_{q+1}^{h_{q+1}}\arrow[ur,"2"]\arrow[ll]\\
    &&&|[color=green,fill=red]|m_1^{m_1,h_{q+1},h_q}\arrow[u]\\
    &&&&&\color{green}r_1^{R}\arrow[uur]
\end{tikzcd}\]

Observe that in the diagrams above, the full subdiagrams formed by the highlighted vertices are isomorphic (not counting frozen vertices attached). So we may conclude that after mutating at $h_3$, we have this (the notation $h_{q+1\searrow3}$ in the superscript of $m_1$ means $h_{q+1},h_q,...,h_3$):

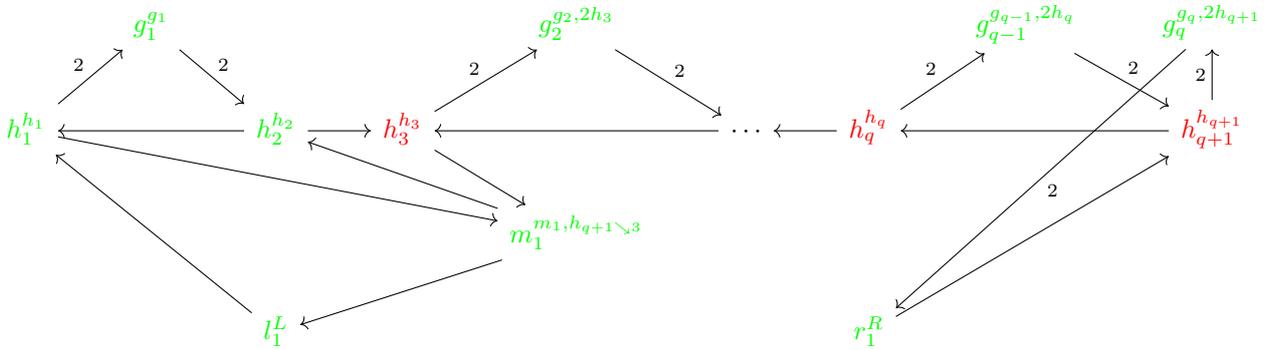
\[\begin{tikzcd}
 &\color{green}g_1^{g_1}\arrow[dr,"2"] &&&\color{green}g_2^{g_2,2h_3}\arrow[dr,"2"]&&&\color{green}g_{q-1}^{g_{q-1},2h_{q}}\arrow[dr,"2"]&\color{green}g_q^{g_q,2h_{q+1}}\arrow[dddll,"2"]\\
 \color{green}h_1^{h_1}\arrow[drrrr]\arrow[ur,"2"] && \color{green}h_2^{h_2}\arrow[ll]\arrow[r]& \color{red}h_3^{h_3}\arrow[ur,"2"]\arrow[dr]&&\cdots\arrow[ll]&\color{red}h_q^{h_q}\arrow[l]\arrow[ur,"2"]&&\color{red}h_{q+1}^{h_{q+1}}\arrow[ll]\arrow[u,"2"]\\
 &&&&\color{green}m_1^{m_1,h_{q+1\searrow3}}\arrow[dll]\arrow[ull]&&&&\\
 &&\color{green}l_1^{L}\arrow[uull]&&&&\color{green}r_1^{R}\arrow[uurr]&&
\end{tikzcd}\]

Mutate at $h_2$:

\[\adjustbox{scale=0.9,center}
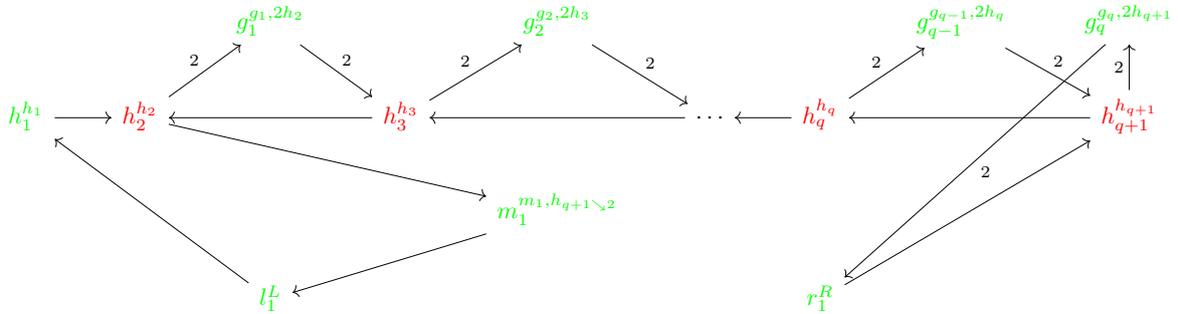
{\begin{tikzcd}
 &&\color{green}g_1^{g_1,2h_2}\arrow[dr,"2"] &&\color{green}g_2^{g_2,2h_3}\arrow[dr,"2"]&&&\color{green}g_{q-1}^{g_{q-1},2h_{q}}\arrow[dr,"2"]&\color{green}g_q^{g_q,2h_{q+1}}\arrow[dddll,"2"]\\ \color{green}h_1^{h_1}\arrow[r]&\color{red}h_2^{h_2}\arrow[ur,"2"]\arrow[drrr]&&\color{red}h_3^{h_3}\arrow[ur,"2"]\arrow[ll]&&\cdots\arrow[ll]&\color{red}h_q^{h_q}\arrow[l]\arrow[ur,"2"]&&\color{red}h_{q+1}^{h_{q+1}}\arrow[ll]\arrow[u,"2"]\\
 &&&&\color{green}m_1^{m_1,h_{q+1\searrow2}}\arrow[dll]&&&&\\
 &&\color{green}l_1^{L}\arrow[uull]&&&&\color{green}r_1^{R}\arrow[uurr]&&
\end{tikzcd}}\]

The corresponding subtriangulation looks like this:

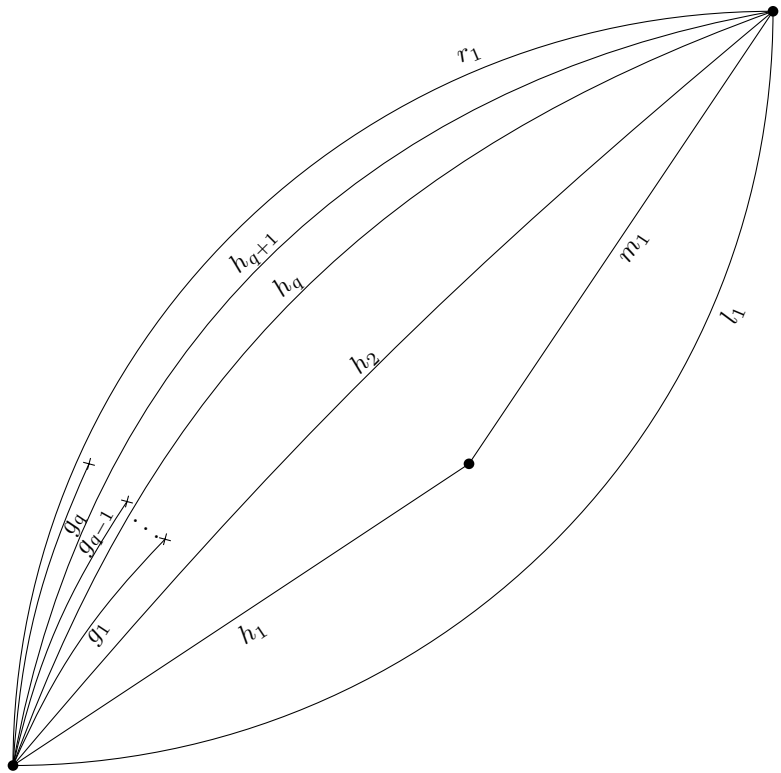
\[\begin{tikzpicture}[every edge quotes/.style={auto=right}]
    \begin{scope}[every node/.style={sloped,allow upside down}][every edge quotes/.style={auto=right}]
        \fill (0,0) circle (2pt);
        \fill (10,10) circle (2pt);
        \fill (6,4) circle (2pt);
        \fill (2,3) node[cross=2pt,rotate=30] {};
        \fill (1,4) node[cross=2pt,rotate=30] {};
        \fill (1.5,3.5) node[cross=2pt,rotate=30] {};
        \node at ($(0,5)!.35!(5,0)$) {$\ddots$};
        \draw (0,0) to [bend left=10] node [below=.15,right] {$g_1$} (2,3);
        \draw (0,0) to [bend left=10] node [below=.15,right,near end] {$g_q$} (1,4);
        \draw (0,0) to [bend left=10] node [above=.1,right,near end] {$g_{q-1}$} (1.5,3.5);
        \draw (0,0) to node[below] {$h_1$} (6,4);
        \draw (0,0) to [bend left=35] node[above=-.1] {$h_{q+1}$} (10,10);
        \draw (0,0) to [bend left=5] node[right=1,above=-.1] {$h_2$} (10,10);
        \draw (0,0) to [bend left=25] node[right=1,above=-.1] {$h_q$} (10,10);
        \draw (0,0) to [bend right=45] node[near end,below] {$l_{1}$} (10,10);
        \draw (0,0) to [bend left=45] node[near end,above] {$r_{1}$} (10,10);
        \draw (6,4) to node[below] {$m_1$} (10,10);
    \end{scope}
\end{tikzpicture}\]

Observe that the arcs $h_2,l_1,h_1,m_1$ form a triangulated once-punctured digon, mutate at $h_1$ and $m_1$ (and switch their positions):

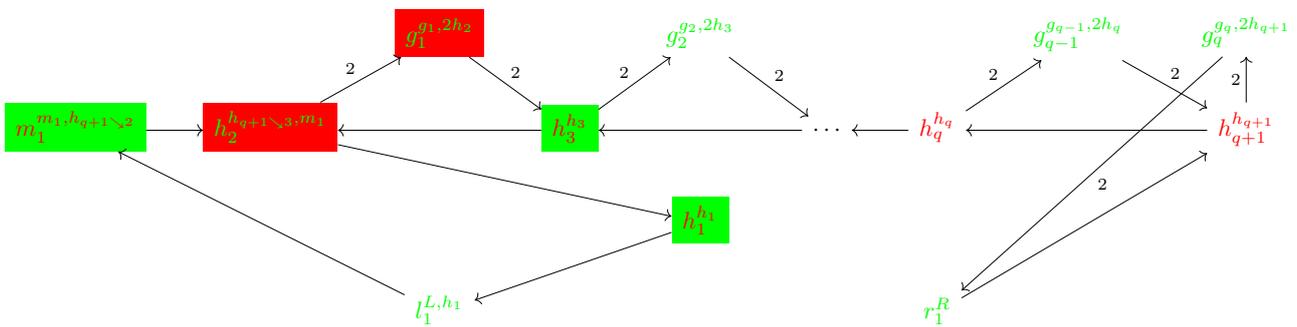
\[\adjustbox{scale=0.9,center}{\begin{tikzcd}
 &&|[color=green,fill=red]|g_1^{g_1,2h_2}\arrow[dr,"2"] &&\color{green}g_2^{g_2,2h_3}\arrow[dr,"2"]&&&\color{green}g_{q-1}^{g_{q-1},2h_{q}}\arrow[dr,"2"]&\color{green}g_q^{g_q,2h_{q+1}}\arrow[dddll,"2"]\\ |[color=red,fill=green]|m_1^{m_1,h_{q+1\searrow2}}\arrow[r]&|[color=green,fill=red]|h_2^{h_{q+1\searrow3},m_1}\arrow[ur,"2"]\arrow[drrr]&&|[color=red,fill=green]|h_3^{h_3}\arrow[ur,"2"]\arrow[ll]&&\cdots\arrow[ll]&\color{red}h_q^{h_q}\arrow[l]\arrow[ur,"2"]&&\color{red}h_{q+1}^{h_{q+1}}\arrow[ll]\arrow[u,"2"]\\
 &&&&|[color=red,fill=green]|h_1^{h_1}\arrow[dll]&&&&\\
 &&\color{green}l_1^{L,h_1}\arrow[uull]&&&&\color{green}r_1^{R}\arrow[uurr]&&
\end{tikzcd}}\]

By observing similar pattern one can argue by induction that mutating at $h_2,h_3,...,h_q$ gives:

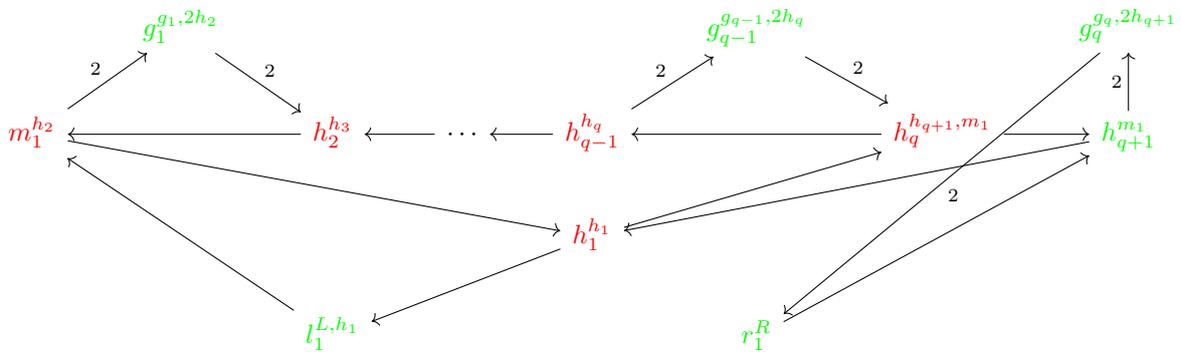
\[\begin{tikzcd}
    &\color{green}g_1^{g_1,2h_2}\arrow[dr,"2"]&&&&\color{green}g_{q-1}^{g_{q-1},2h_q}\arrow[dr,"2"]&&\color{green}g_q^{g_q,2h_{q+1}}\arrow[dddll,"2"]\\
    \color{red}m_1^{h_2}\arrow[drrrr]\arrow[ur,"2"]&&\color{red}h_2^{h_3}\arrow[ll]&\cdots\arrow[l]&\color{red}h_{q-1}^{h_q}\arrow[ur,"2"]\arrow[l]&&\color{red}h_{q}^{h_{q+1},m_1}\arrow[ll]\arrow[r]&\color{green}h_{q+1}^{m_1}\arrow[u,"2"]\arrow[dlll]\\
    &&&&\color{red}h_1^{h_1}\arrow[urr]\arrow[dll]\\
    &&\color{green}l_1^{L,h_1}\arrow[uull]&&&\color{green}r_1^{R}\arrow[uurr]
\end{tikzcd}\]

Now mutate at $h_{q+1}$:

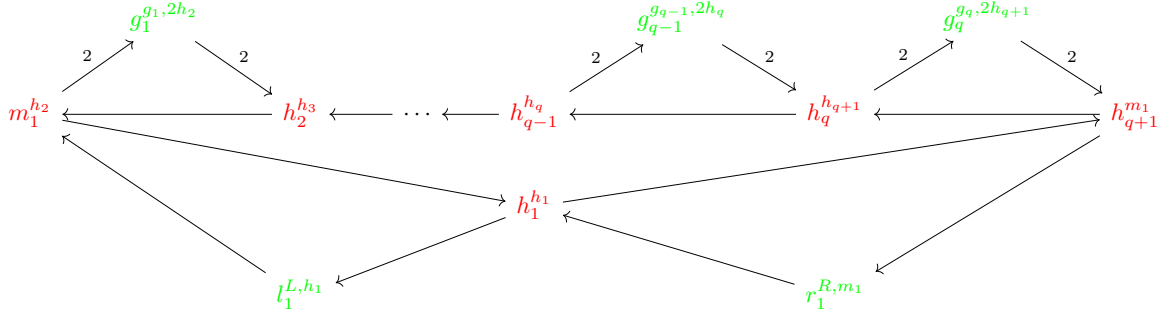
\[\adjustbox{scale=0.9,center}{
\begin{tikzcd}
    &\color{green}g_1^{g_1,2h_2}\arrow[dr,"2"]&&&&\color{green}g_{q-1}^{g_{q-1},2h_q}\arrow[dr,"2"]&&\color{green}g_q^{g_q,2h_{q+1}}\arrow[dr,"2"]\\
    \color{red}m_1^{h_2}\arrow[drrrr]\arrow[ur,"2"]&&\color{red}h_2^{h_3}\arrow[ll]&\cdots\arrow[l]&\color{red}h_{q-1}^{h_q}\arrow[ur,"2"]\arrow[l]&&\color{red}h_{q}^{h_{q+1}}\arrow[ll]\arrow[ur,"2"]&&\color{red}h_{q+1}^{m_1}\arrow[ll]\arrow[ddll]\\
    &&&&\color{red}h_1^{h_1}\arrow[urrrr]\arrow[dll]\\
    &&\color{green}l_1^{L,h_1}\arrow[uull]&&&&\color{green}r_1^{R,m_1}\arrow[ull]
\end{tikzcd}}\]

In terms of triangulation, we get back the starting triangulation with the middle puncture being tagged notched.
\end{proof}

By applying the above lemma, after step 2, we have:

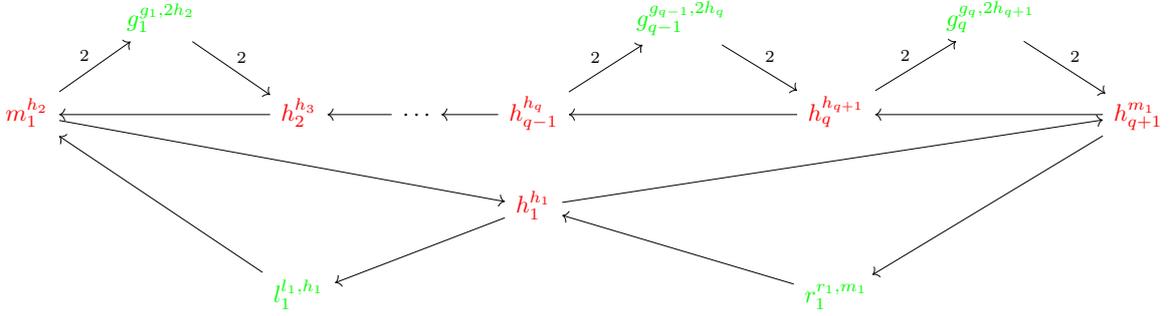
\[\adjustbox{scale=0.9,center}{
\begin{tikzcd}
    &\color{green}g_1^{g_1,2h_2}\arrow[dr,"2"]&&&&\color{green}g_{q-1}^{g_{q-1},2h_q}\arrow[dr,"2"]&&\color{green}g_q^{g_q,2h_{q+1}}\arrow[dr,"2"]\\
    \color{red}m_1^{h_2}\arrow[drrrr]\arrow[ur,"2"]&&\color{red}h_2^{h_3}\arrow[ll]&\cdots\arrow[l]&\color{red}h_{q-1}^{h_q}\arrow[ur,"2"]\arrow[l]&&\color{red}h_{q}^{h_{q+1}}\arrow[ll]\arrow[ur,"2"]&&\color{red}h_{q+1}^{m_1}\arrow[ll]\arrow[ddll]\\
    &&&&\color{red}h_1^{h_1}\arrow[urrrr]\arrow[dll]\\
    &&\color{green}l_1^{l_1,h_1}\arrow[uull]&&&&\color{green}r_1^{r_1,m_1}\arrow[ull]
\end{tikzcd}}\]

Step 4(a) consists of $(g_1,g_2,...,g_q,h_{q},h_{q-1},...,h_2)$.

Mutating at $g_1$ does the following:

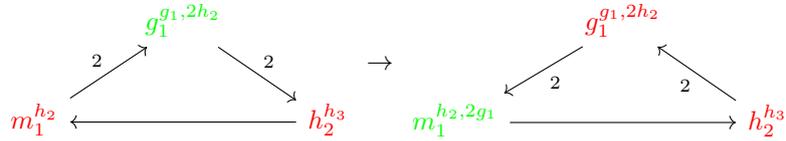
\[
\begin{tikzcd}
    &\color{green}g_{1}^{g_1,2h_2}\arrow[dr,"2"]&\\
    \color{red}m_1^{h_2}\arrow[ur,"2"]&&\color{red}h_2^{h_3}\arrow[ll]
\end{tikzcd}
\rightarrow
\begin{tikzcd}
    &\color{red}g_{1}^{g_1,2h_2}\arrow[dl,"2"]&\\
    \color{green}m_1^{h_2,2g_1}\arrow[rr]&&\color{red}h_2^{h_3}\arrow[ul,"2"]
\end{tikzcd}
\]

Observe that for all $g_k$, $k=1,2,...,q$, the full subdiagram containing $g_k$ and adjancent vertices are isomorphic (again except the frozen vertices attached), so after mutating at $g_1,g_2,...,g_q$, we have:

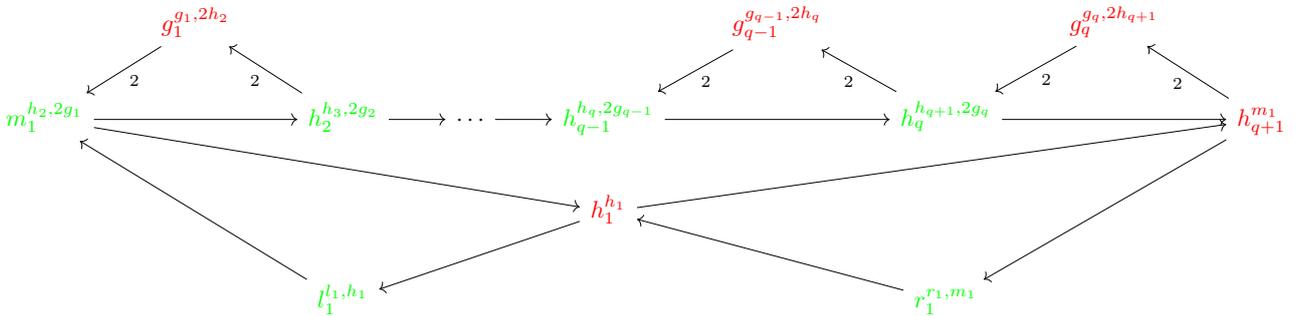
\[\adjustbox{scale=0.9,center}{
\begin{tikzcd}
    &\color{red}g_1^{g_1,2h_2}\arrow[dl,"2"]&&&&\color{red}g_{q-1}^{g_{q-1},2h_q}\arrow[dl,"2"]&&\color{red}g_q^{g_q,2h_{q+1}}\arrow[dl,"2"]\\
    \color{green}m_1^{h_2,2g_1}\arrow[drrrr]\arrow[rr]&&\color{green}h_2^{h_3,2g_2}\arrow[r]\arrow[ul,"2"]&\cdots\arrow[r]&\color{green}h_{q-1}^{h_q,2g_{q-1}}\arrow[rr]&&\color{green}h_{q}^{h_{q+1},2g_q}\arrow[rr]\arrow[ul,"2"]&&\color{red}h_{q+1}^{m_1}\arrow[ul,"2"]\arrow[ddll]\\
    &&&&\color{red}h_1^{h_1}\arrow[urrrr]\arrow[dll]\\
    &&\color{green}l_1^{l_1,h_1}\arrow[uull]&&&&\color{green}r_1^{r_1,m_1}\arrow[ull]
\end{tikzcd}}\]

In terms of triangulation, all pending arcs are now connected to the middle puncture:

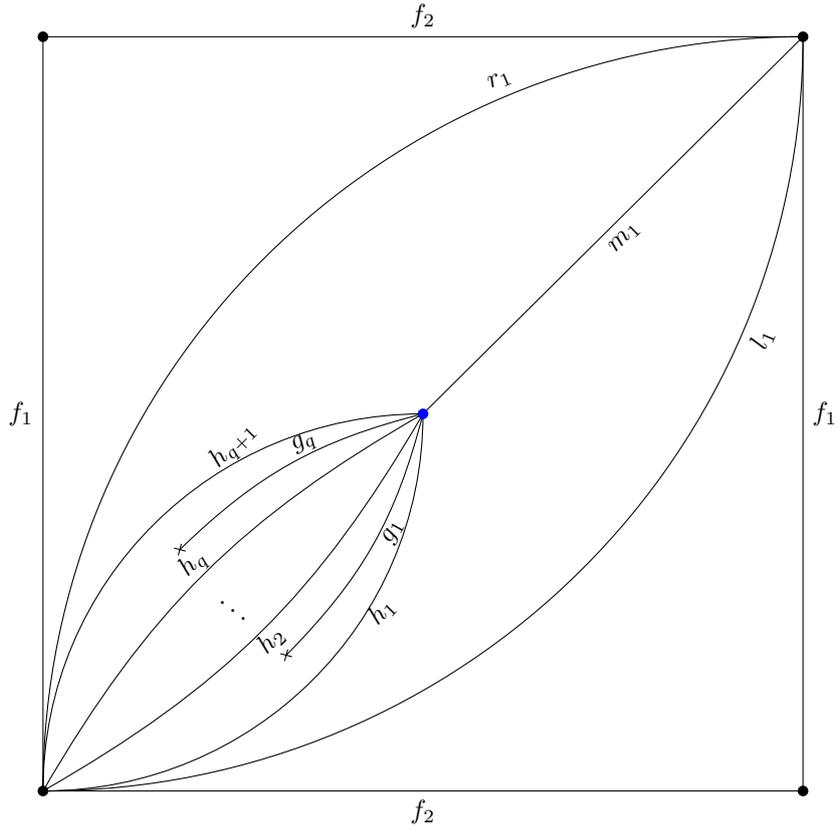
\[\begin{tikzpicture}[every edge quotes/.style={auto=right}]
    \begin{scope}[every node/.style={sloped,allow upside down}][every edge quotes/.style={auto=right}]
        \fill (0,10) circle (2pt);
        \fill (10,0) circle (2pt);
        \fill (0,0) circle (2pt);
        \fill (10,10) circle (2pt);
        \fill (3.2,1.8) node[cross=2pt,rotate=30] {};
        \fill (1.8,3.2) node[cross=2pt,rotate=30] {};
        \node at ($(0,5)!.5!(5,0)$) {$\ddots$};
        \draw (3.2,1.8) to [bend right=15] node [below=.15,right] {$g_1$} (5,5);
        \draw (1.8,3.2) to [bend left=15] node [above=.15,right] {$g_q$} (5,5);
        \draw (0,0) -- node[rotate=-90,left] {$f_1$} (0,10);
        \draw (0,10) -- node[above] {$f_2$} (10,10);
        \draw (10,0) -- node[rotate=-90,right] {$f_1$} (10,10);
        \draw (0,0) -- node[below] {$f_2$} (10,0);
        \draw (0,0) to [bend right=45] node[below=.15,right=1] {$h_1$} (5,5);
        \draw (0,0) to [bend left=45] node[above=.15,right=1] {$h_{q+1}$} (5,5);
        \draw (0,0) to [bend right=15] node[right=1,below=-.1] {$h_2$} (5,5);
        \draw (0,0) to [bend left=15] node[right=1,above=-.1] {$h_q$} (5,5);
        \draw (0,0) to [bend right=45] node[near end,below] {$l_{1}$} (10,10);
        \draw (0,0) to [bend left=45] node[near end,above] {$r_{1}$} (10,10);
        \draw (5,5) to node[below] {$m_1$} (10,10);
        \fill[blue] (5,5) circle (2pt);
    \end{scope}
\end{tikzpicture}\]

Now look at $h_q$, $h_{q-1}$, and adjacent vertices:

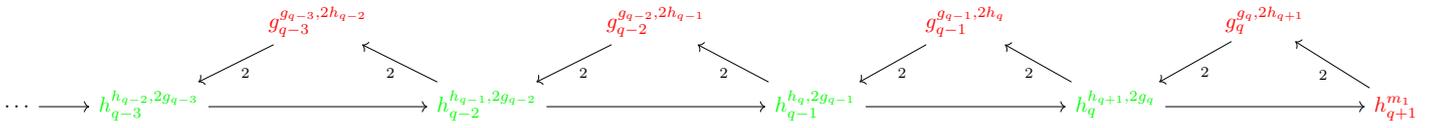
\[\adjustbox{scale=0.8,center}{
\begin{tikzcd}
    &&\color{red}g_{q-3}^{g_{q-3},2h_{q-2}}\arrow[dl,"2"]&&\color{red}g_{q-2}^{g_{q-2},2h_{q-1}}\arrow[dl,"2"]&&\color{red}g_{q-1}^{g_{q-1},2h_q}\arrow[dl,"2"]&&\color{red}g_q^{g_q,2h_{q+1}}\arrow[dl,"2"]\\\cdots\arrow[r]&
    \color{green}h_{q-3}^{h_{q-2},2g_{q-3}}\arrow[rr]&&\color{green}h_{q-2}^{h_{q-1},2g_{q-2}}\arrow[rr]\arrow[ul,"2"]&&\color{green}h_{q-1}^{h_q,2g_{q-1}}\arrow[rr]\arrow[ul,"2"]&&\color{green}h_{q}^{h_{q+1},2g_q}\arrow[rr]\arrow[ul,"2"]&&\color{red}h_{q+1}^{m_1}\arrow[ul,"2"]\
\end{tikzcd}}\]

Mutate at $h_q$:

\[\adjustbox{scale=0.8,center}
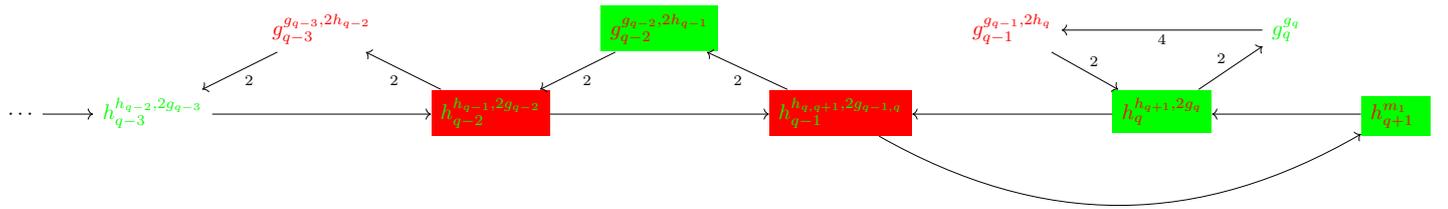
{
\begin{tikzcd}
    &&\color{red}g_{q-3}^{g_{q-3},2h_{q-2}}\arrow[dl,"2"]&&|[color=red,fill=green]|g_{q-2}^{g_{q-2},2h_{q-1}}\arrow[dl,"2"]&&\color{red}g_{q-1}^{g_{q-1},2h_q}\arrow[dr,"2"]&&\color{green}g_q^{g_q}\arrow[ll,"4"]\\\cdots\arrow[r]&
    \color{green}h_{q-3}^{h_{q-2},2g_{q-3}}\arrow[rr]&&|[color=green,fill=red]|h_{q-2}^{h_{q-1},2g_{q-2}}\arrow[rr]\arrow[ul,"2"]&&|[color=green,fill=red]|h_{q-1}^{h_{q,q+1},2g_{q-1,q}}\arrow[rrrr,bend right]\arrow[ul,"2"]&&|[color=red,fill=green]|h_{q}^{h_{q+1},2g_q}\arrow[ll]\arrow[ur,"2"]&&|[color=red,fill=green]|h_{q+1}^{m_1}\arrow[ll]\
\end{tikzcd}}\]

Mutate at $h_{q-1}$:

\[\adjustbox{scale=0.8,center}
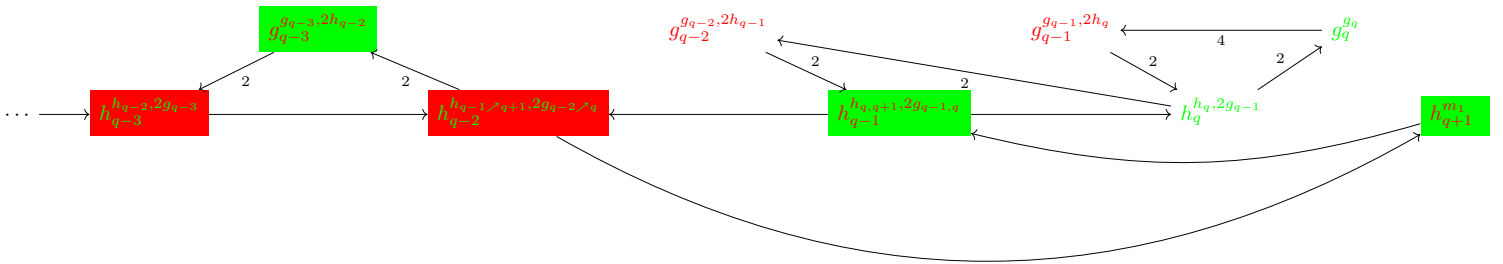
{
\begin{tikzcd}
    &&|[color=red,fill=green]|g_{q-3}^{g_{q-3},2h_{q-2}}\arrow[dl,"2"]&&\color{red}g_{q-2}^{g_{q-2},2h_{q-1}}\arrow[dr,"2"]&&\color{red}g_{q-1}^{g_{q-1},2h_q}\arrow[dr,"2"]&&\color{green}g_q^{g_q}\arrow[ll,"4"]\\\cdots\arrow[r]&
    |[color=green,fill=red]|h_{q-3}^{h_{q-2},2g_{q-3}}\arrow[rr]&&|[color=green,fill=red]|h_{q-2}^{h_{q-1\nearrow q+1},2g_{q-2\nearrow q}}\arrow[rrrrrr,bend right]\arrow[ul,"2"]&&|[color=red,fill=green]|h_{q-1}^{h_{q,q+1},2g_{q-1,q}}\arrow[rr]\arrow[ll]&&\color{green}h_{q}^{h_q,2g_{q-1}}\arrow[ulll,"2"]\arrow[ur,"2"]&&|[color=red,fill=green]|h_{q+1}^{m_1}\arrow[llll,bend left=15]\
\end{tikzcd}}\]

Observe the pattern at each $h_k$ and adjacent vertices, we can argue by induction that after mutating at $h_2$, we have:

$\adjustbox{scale=0.9,center}
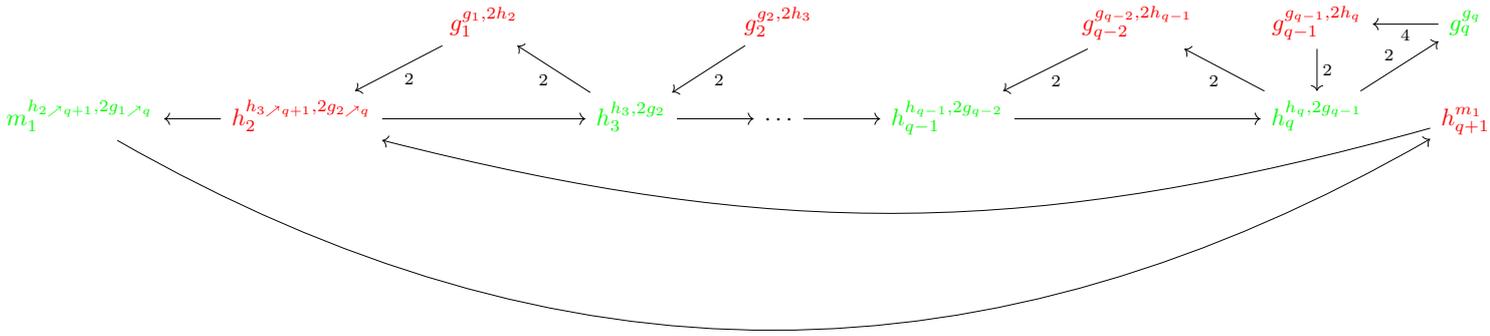
{
\begin{tikzcd}
    &&\color{red}g_1^{g_1,2h_2}\arrow[dl,"2"]&&\color{red}g_2^{g_2,2h_3}\arrow[dl,"2"]&&\color{red}g_{q-2}^{g_{q-2},2h_{q-1}}\arrow[dl,"2"]&\color{red}g_{q-1}^{g_{q-1},2h_q}\arrow[d,"2"]&\color{green}g_q^{g_q}\arrow[l,"4"]\\
    \color{green}m_1^{h_{2\nearrow q+1},2g_{1\nearrow q}}\arrow[rrrrrrrr,bend right]&\color{red}h_{2}^{h_{3\nearrow q+1},2g_{2\nearrow q}}\arrow[l]\arrow[rr]&&\color{green}h_3^{h_3,2g_2}\arrow[ul,"2"]\arrow[r]&\cdots\arrow[r]&\color{green}h_{q-1}^{h_{q-1},2g_{q-2}}\arrow[rr]&&\color{green}h_{q}^{h_q,2g_{q-1}}\arrow[ur,"2"]\arrow[ul,"2"]&\color{red}h_{q+1}^{m_1}\arrow[lllllll,bend left=15]
\end{tikzcd}}$

In terms of triangulation, we have wrapped the arcs $g_1,g_2,...,g_q,h_3,h_4,...,h_q$ inside $h_2$ (and thus removed all core arcs away from the corner).

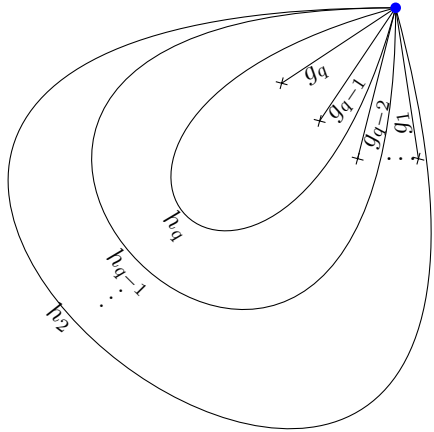
\begin{flushleft}
\hspace*{-0.15\linewidth}
\begin{tikzpicture}[every node/.style={draw}][every edge quotes/.style={auto=right}]
    \begin{scope}[every node/.style={sloped,allow upside down}][every edge quotes/.style={auto=right}]
        \fill (1,1.5) node[cross=2pt,rotate=30] {};
        \fill (1.5,1) node[cross=2pt,rotate=30] {};
        \fill (2,0.5) node[cross=2pt,rotate=30] {};
        \fill (2.8,0.5) node[cross=2pt,rotate=30] {};
        \node[font=\small] at (-1.25,-1.25) {$\iddots$};
        \node at (2.6,0.5) {$\dots$};
        \draw (1,1.5) to node[near start,below=-.1] {$g_q$} (2.5,2.5);
        \draw (1.5,1) to node[near start,below=-.1] {$g_{q-1}$} (2.5,2.5);
        \draw (2,0.5) to node[near start,below=-.1] {$g_{q-2}$} (2.5,2.5);
        \draw (2.8,0.5) to node[near start,above=-.1] {$g_{1}$} (2.5,2.5);
        \draw (2.5,2.5) to [loop,in=255,out=195,min distance=60mm] node[below=-.1] {$h_q$} (2.5,2.5);
        \draw (2.5,2.5) to [loop,in=270,out=180,min distance=90mm] node[below=-.1] {$h_{q-1}$} (2.5,2.5);
        \draw (2.5,2.5) to [loop,in=285,out=180,min distance=130mm] node[below=0.1,left=.7] {$h_{2}$} (2.5,2.5);
        \fill[blue] (2.5,2.5) circle (2pt);
    \end{scope}
\end{tikzpicture}
\end{flushleft}

\vspace*{-0.4\linewidth}

Step 4(c) consists of $(l_1,f_1,r_1,f_2,l_1,f_1)$. We look at them and adjacent vertices:

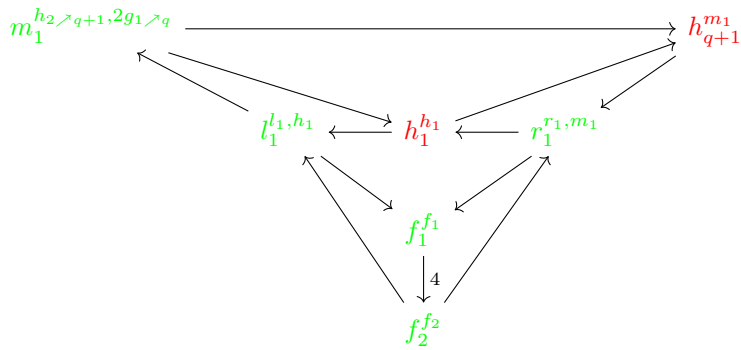
\[\begin{tikzcd}
    \color{green}m_1^{h_{2\nearrow q+1},2g_{1\nearrow q}}\arrow[rrrr]\arrow[drr]&&&&\color{red}h_{q+1}^{m_1}\arrow[dl]\\
    &\color{green}l_1^{l_1,h_1}\arrow[ul]\arrow[dr]&\color{red}h_1^{h_1}\arrow[l]\arrow[urr]&\color{green}r_1^{r_1,m_1}\arrow[l]\arrow[dl]\\
    &&\color{green}f_1^{f_1}\arrow[d,"4"]&&\\
    &&\color{green}f_2^{f_2}\arrow[uul]\arrow[uur]&&
\end{tikzcd}\]

By computation we see that after mutating at $l_1,f_1,r_1,f_2,l_1,f_1$, we have:

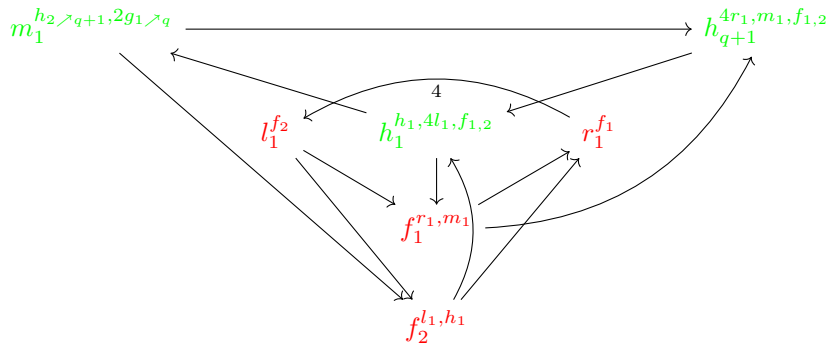
\[\begin{tikzcd}
    \color{green}m_1^{h_{2\nearrow q+1},2g_{1\nearrow q}}\arrow[rrrr]\arrow[dddrr]&&&&\color{green}h_{q+1}^{4r_1,m_1,f_{1,2}}\arrow[dll]\\
    &\color{red}l_1^{f_2}\arrow[ddr]\arrow[dr]&\color{green}h_1^{h_1,4l_1,f_{1,2}}\arrow[d]\arrow[ull]&\color{red}r_1^{f_1}\arrow[ll,bend right,"4"]\\
    &&\color{red}f_1^{r_1,m_1}\arrow[ur]\arrow[uurr,bend right]&&\\
    &&\color{red}f_2^{l_1,h_1}\arrow[uur]\arrow[uu,bend right]&&
\end{tikzcd}\]

In terms of triangulation, we just made both ends of most of the arcs the middle punture like in the case of $p=2,q=1$.

Step 5 consists of $h_1,m_1,h_{q+1},h_1$. Look at them and adjacent vertices:

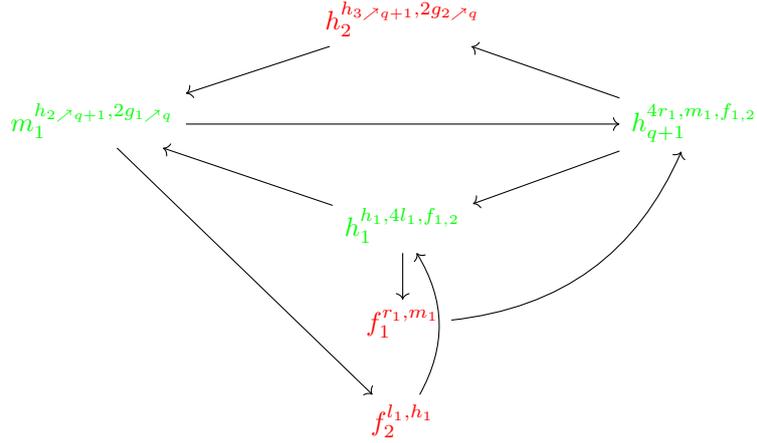
\[\begin{tikzcd}
    &&\color{red}h_{2}^{h_{3\nearrow q+1},2g_{2\nearrow q}}\arrow[dll]&&\\
    \color{green}m_1^{h_{2\nearrow q+1},2g_{1\nearrow q}}\arrow[rrrr]\arrow[dddrr]&&&&\color{green}h_{q+1}^{4r_1,m_1,f_{1,2}}\arrow[dll]\arrow[ull]\\
    &&\color{green}h_1^{h_1,4l_1,f_{1,2}}\arrow[d]\arrow[ull]&\\
    &&\color{red}f_1^{r_1,m_1}\arrow[uurr,bend right]&&\\
    &&\color{red}f_2^{l_1,h_1}\arrow[uu,bend right]&&
\end{tikzcd}\]

By computation we see that after mutating at $h_1,m_1,h_{q+1},h_1$, we have:

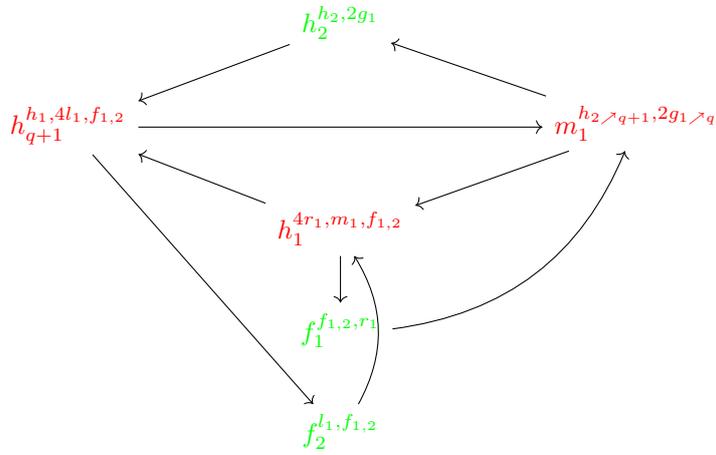
\[\begin{tikzcd}
    &&\color{green}h_{2}^{h_2,2g_1}\arrow[dll]&&\\
    \color{red}h_{q+1}^{h_1,4l_1,f_{1,2}}\arrow[rrrr]\arrow[dddrr]&&&&\color{red}m_1^{h_{2\nearrow q+1},2g_{1\nearrow q}}\arrow[ull]\arrow[dll]\\
    &&\color{red}h_1^{4r_1,m_1,f_{1,2}}\arrow[d]\arrow[ull]&\\
    &&\color{green}f_1^{f_{1,2},r_1}\arrow[uurr,bend right]&&\\
    &&\color{green}f_2^{l_1,f_{1,2}}\arrow[uu,bend right]&&
\end{tikzcd}\]

In terms of triangulation, we tagged the corner puncture notched (like in the previous case).

Step 6(a) consists of $f_1,l_1,f_2,r_1,f_1,l_1$. Look at them and adjacent vertices:

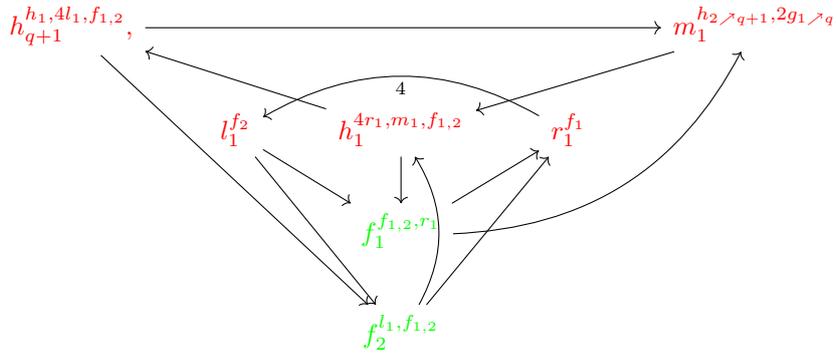
\[\begin{tikzcd}
    \color{red}h_{q+1}^{h_1,4l_1,f_{1,2}}\arrow[rrrr],\arrow[dddrr]&&&&\color{red}m_1^{h_{2\nearrow q+1},2g_{1\nearrow q}}\arrow[dll]\\
    &\color{red}l_1^{f_2}\arrow[ddr]\arrow[dr]&\color{red}h_1^{4r_1,m_1,f_{1,2}}\arrow[d]\arrow[ull]&\color{red}r_1^{f_1}\arrow[ll,bend right,"4"]\\
    &&\color{green}f_1^{f_{1,2},r_1}\arrow[ur]\arrow[uurr,bend right]&&\\
    &&\color{green}f_2^{l_1,f_{1,2}}\arrow[uur]\arrow[uu,bend right]&&
\end{tikzcd}\]

By computation we see that after mutating at $f_1,l_1,f_2,r_1,f_1,l_1$ (undo the flips of the arcs), we have:

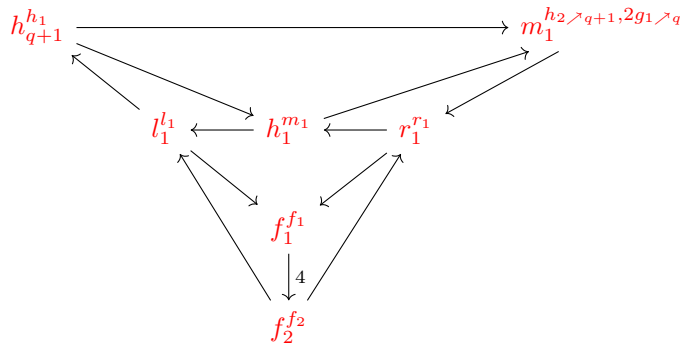
\[\begin{tikzcd}
    \color{red}h_{q+1}^{h_1}\arrow[rrrr]\arrow[drr]&&&&\color{red}m_1^{h_{2\nearrow q+1},2g_{1\nearrow q}}\arrow[dl]\\
    &\color{red}l_1^{l_1}\arrow[dr]\arrow[ul]&\color{red}h_1^{m_1}\arrow[urr]\arrow[l]&\color{red}r_1^{r_1}\arrow[l]\arrow[dl]\\
    &&\color{red}f_1^{f_1}\arrow[d,"4"]&&\\
    &&\color{red}f_2^{f_2}\arrow[uur]\arrow[uul]&&
\end{tikzcd}\]

All outer arcs are moved back to their original position.

The mutations in steps 4(c), 5, and 6(a) can be combined into the following lemma which we will use in later cases:

\begin{lemma}\label{corner}
When mutating the following diagram using the mutation sequence $l_{p-1},f_1,r_{p-1},f_2,l_{p-1},f_1,$ $x,y,z,x,f_1,l_{p-1},f_2,r_{p-1},f_1,l_{p-1}$:
    
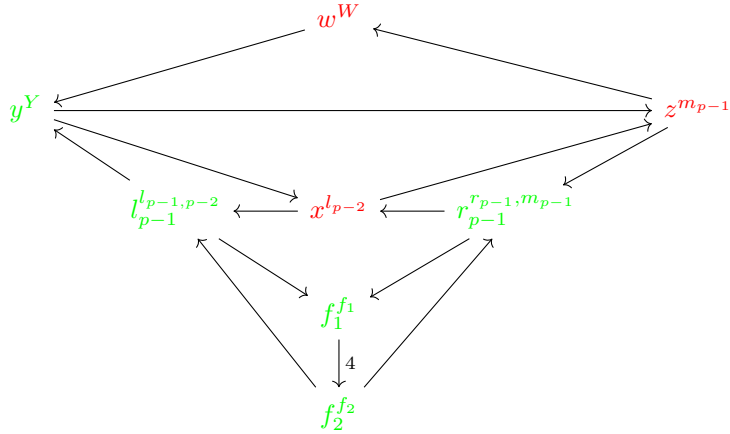
\[\begin{tikzcd}
    &&\color{red}w^{W}\arrow[dll]&&\\
    \color{green}y^{Y}\arrow[rrrr]\arrow[drr]&&&&\color{red}z^{m_{p-1}}\arrow[dl]\arrow[ull]\\
    &\color{green}l_{p-1}^{l_{p-1,p-2}}\arrow[ul]\arrow[dr]&\color{red}x^{l_{p-2}}\arrow[l]\arrow[urr]&\color{green}r_{p-1}^{r_{p-1},m_{p-1}}\arrow[l]\arrow[dl]\\
    &&\color{green}f_1^{f_1}\arrow[d,"4"]&\\
    &&\color{green}f_2^{f_2}\arrow[uul]\arrow[uur]&
\end{tikzcd}\]

where $W,Y$ are sets of (frozen) vertices with $W\subset Y$, we will end up with the following diagram:

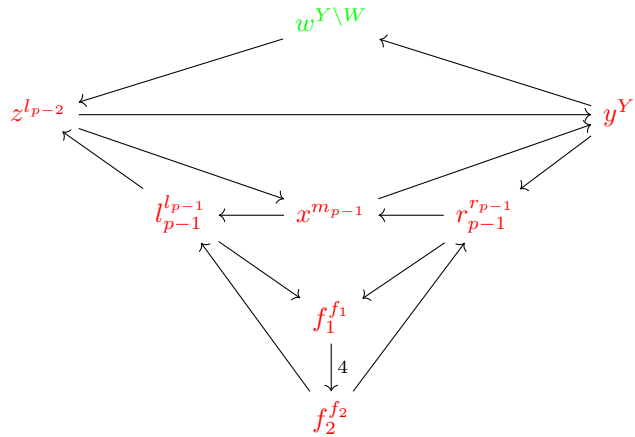
\[\begin{tikzcd}
    &&\color{green}w^{Y\setminus W}\arrow[dll]&&\\
    \color{red}z^{l_{p-2}}\arrow[rrrr]\arrow[drr]&&&&\color{red}y^{Y}\arrow[dl]\arrow[ull]\\
    &\color{red}l_{p-1}^{l_{p-1}}\arrow[ul]\arrow[dr]&\color{red}x^{m_{p-1}}\arrow[l]\arrow[urr]&\color{red}r_{p-1}^{r_{p-1}}\arrow[l]\arrow[dl]\\
    &&\color{red}f_1^{f_1}\arrow[d,"4"]&\\
    &&\color{red}f_2^{f_2}\arrow[uul]\arrow[uur]&
\end{tikzcd}\]
\end{lemma}

\begin{proof}
    Direct computation.
\end{proof}

Step 6(c) consists of $(h_2,h_3,...,h_q,g_1,g_2,...,g_{q-1},g_q)$. Look at them and adjacent vertices:

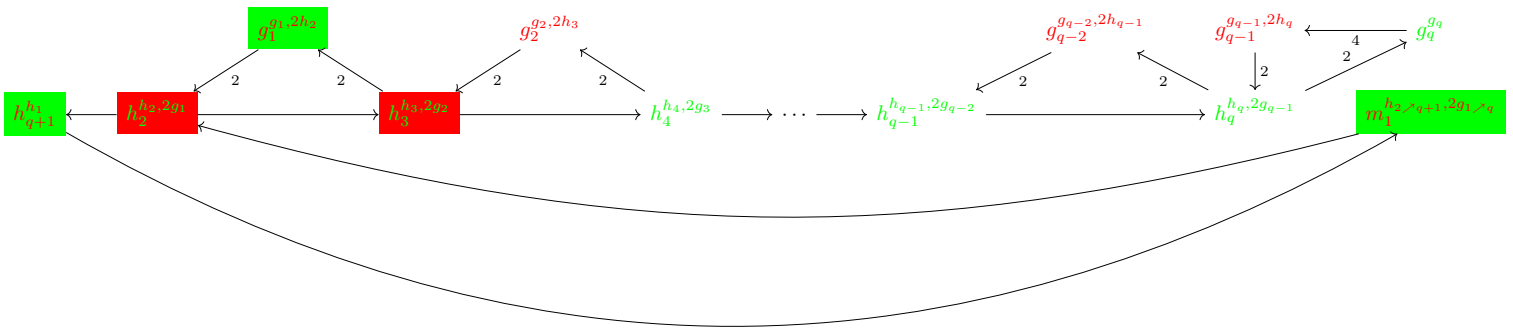
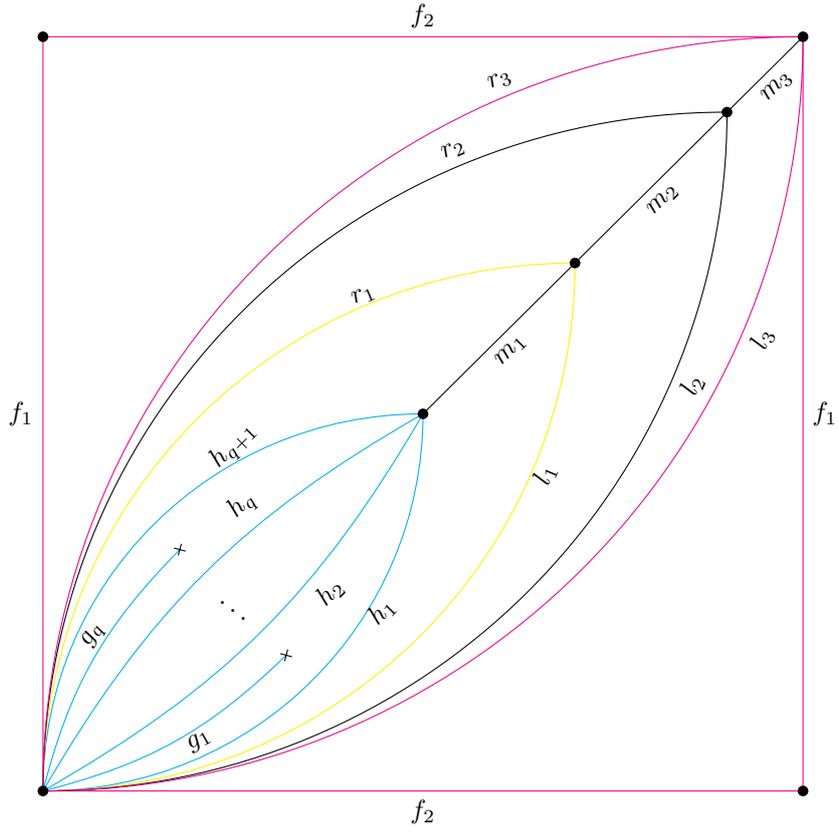
\[\adjustbox{scale=0.8,center}{
\begin{tikzcd}
    &&|[color=red,fill=green]|g_1^{g_1,2h_2}\arrow[dl,"2"]&&\color{red}g_2^{g_2,2h_3}\arrow[dl,"2"]&&&&\color{red}g_{q-2}^{g_{q-2},2h_{q-1}}\arrow[dl,"2"]&\color{red}g_{q-1}^{g_{q-1},2h_q}\arrow[d,"2"]&\color{green}g_q^{g_q}\arrow[l,"4"]\\
    |[color=red,fill=green]|h_{q+1}^{h_1}\arrow[rrrrrrrrrr,bend right]&|[color=green,fill=red]|h_{2}^{h_2,2g_1}\arrow[l]\arrow[rr]&&|[color=green,fill=red]|h_3^{h_3,2g_2}\arrow[ul,"2"]\arrow[rr]&&\color{green}h_4^{h_4,2g_3}\arrow[ul,"2"]\arrow[r]&\cdots\arrow[r]&\color{green}h_{q-1}^{h_{q-1},2g_{q-2}}\arrow[rr]&&\color{green}h_{q}^{h_q,2g_{q-1}}\arrow[ur,"2"]\arrow[ul,"2"]&|[color=red,fill=green]|m_1^{h_{2\nearrow q+1},2g_{1\nearrow q}}\arrow[lllllllll,bend left=15]
\end{tikzcd}}\]

Mutate at $h_2$:

\[\adjustbox{scale=0.8,center}{
\begin{tikzcd}
    &\color{green}g_1^{g_1}\arrow[dl,"2"]&&&|[color=red,fill=green]|g_2^{g_2,2h_3}\arrow[dl,"2"]&&&&\color{red}g_{q-2}^{g_{q-2},2h_{q-1}}\arrow[dl,"2"]&\color{red}g_{q-1}^{g_{q-1},2h_q}\arrow[d,"2"]&\color{green}g_q^{g_q}\arrow[l,"4"]\\
    \color{red}h_{q+1}^{h_1}\arrow[rr]&&|[color=red,fill=green]|h_{2}^{h_2,2g_1}\arrow[ul,"2"]\arrow[rrrrrrrr,bend right]&|[color=green,fill=red]|h_3^{h_3,2g_2}\arrow[l]\arrow[rr]&&|[color=green,fill=red]|h_4^{h_4,2g_3}\arrow[ul,"2"]\arrow[r]&\cdots\arrow[r]&\color{green}h_{q-1}^{h_{q-1},2g_{q-2}}\arrow[rr]&&\color{green}h_{q}^{h_q,2g_{q-1}}\arrow[ur,"2"]\arrow[ul,"2"]&|[color=red,fill=green]|m_1^{h_{3\nearrow q+1},2g_{2\nearrow q}}\arrow[lllllll,bend left=15]
\end{tikzcd}}\]

Again observe that the highlighted subdiagrams are isomorphic. Therefore after mutating at $h_{q-1}$, we have:

$\adjustbox{scale=0.9,center}{
\begin{tikzcd}
    &\color{green}g_1^{g_1}\arrow[dl,"2"]&&&&\color{green}g_{q-2}^{g_{q-2}}\arrow[dl,"2"]&&\color{red}g_{q-1}^{g_{q-1},2h_q}\arrow[d,"2"]&\color{green}g_q^{g_q}\arrow[l,"4"]\\
    \color{red}h_{q+1}^{h_1}\arrow[rr]&&\color{red}h_{2}^{h_2,2g_1}\arrow[ul,"2"]\arrow[r]&\cdots\arrow[r]&\color{red}h_{q-2}^{h_{q-2},2g_{q-3}}\arrow[rr]&&\color{red}h_{q-1}^{h_{q-1},2g_{q-2}}\arrow[ul,"2"]\arrow[rr,bend right]&\color{green}h_{q}^{h_q,2g_{q-1}}\arrow[ur,"2"]\arrow[l]&\color{red}m_1^{h_{q,q+1},2g_{q-1,q}}\arrow[l]
\end{tikzcd}}$

Mutating at $h_q$ gives:

$\adjustbox{scale=0.9,center}{\begin{tikzcd}
    &\color{green}g_1^{g_1}\arrow[dl,"2"]&&&&\color{green}g_{q-1}^{g_{q-1}}\arrow[dl,"2"]&&\color{green}g_{q}^{g_q}\arrow[dl,"2"]&\\
    \color{red}h_{q+1}^{h_1}\arrow[rr]&&\color{red}h_{2}^{h_2,2g_1}\arrow[ul,"2"]\arrow[r]&\cdots\arrow[r]&\color{red}h_{q-1}^{h_{q-1},2g_{q-2}}\arrow[rr]&&\color{red}h_{q}^{h_q,2g_{q-1}}\arrow[ul,"2"]\arrow[rr]&&\color{red}m_1^{h_{q+1},2g_{q}}\arrow[ul,"2"]
\end{tikzcd}}$

In terms of triangulation, we unwrapped the arcs from $h_2$.

Now we see that mutating at all $g_k$ makes all vertices red (and flips the arcs $g_k$'s back to their original position). Therefore, $\Delta_1$ is a maximal green sequence. $\blacksquare$

\subsubsection{The case where $p=4$}\label{p=4}

\begin{proposition}
    $\Delta_1$ is a maximal green sequence when $p=4$.
\end{proposition}

\textbf{Proof} When $p=4$, the mutation sequence $\Delta_1$ translates to

\begin{enumerate}
    \item Tagging alternate punctures notched. $(\alpha_2)$
    \item Tagging the first interior puncture notched. $(h_{q+1},h_q,...,h_1,m_1,h_2,h_3,...,h_{q+1})$
    \item Tagging the remaining interior punctures notched and moving half the inner arcs away from the corner. $(\beta_1)$
    \item Moving arcs away from the corner.
    \begin{enumerate}
        \item Moving core arcs away from the corner. $(g_1,g_2,...,g_q,h_{q+1},h_{q},...,h_2,m_1)$
        \item Moving the remaining inner arcs away from the corner. $(\phi)$
        \item Moving outer arcs away from the corner. $(l_1,f_1,r_1,f_2,l_1,f_1)$
    \end{enumerate}
    \item Tagging the corner notched. $(m_2,l_2,r_2,m_2)$
    \item Moving arcs back to the corner.
    \begin{enumerate}
        \item Moving outer arcs back to the corner. 
        $f_1,l_1,f_2,r_1,f_1,l_1$
        \item Moving inner arcs back to the corner. 
        $(\phi)$
        \item Moving core arcs back to the corner. $(m_1,h_2,h_3,...,h_{q+1},g_1,g_2,...,g_{q-1},g_q)$
        \item (If $p>2$ and is even) Moving the remaining pair of inner arc back to the corner. $(l_1,r_1)$
    \end{enumerate}
\end{enumerate}

Now we have all the steps except 4(b) and 6(b).

The triangulation $T_{4,q}$ looks like this:

\[\begin{tikzpicture}[every edge quotes/.style={auto=right}]
    \begin{scope}[every node/.style={sloped,allow upside down}][every edge quotes/.style={auto=right}]
        \node at ($(0,5)!.5!(5,0)$) {$\ddots$};
        \draw[cyan] (0,0) to [bend right=15,black] node [below=.15,right] {$g_1$} (3.2,1.8);
        \draw[cyan] (0,0) to [bend left=15,black] node [above=.15,right] {$g_q$} (1.8,3.2);
        \draw[magenta] (0,0) -- node[rotate=-90,left,black] {$f_1$} (0,10);
        \draw[magenta] (0,10) -- node[above,black] {$f_2$} (10,10);
        \draw[magenta] (10,0) -- node[rotate=-90,right,black] {$f_1$} (10,10);
        \draw[magenta] (0,0) -- node[below,black] {$f_2$} (10,0);
        \draw[cyan] (0,0) to [bend right=45,black] node[below=.15,right=1] {$h_1$} (5,5);
        \draw[cyan] (0,0) to [bend left=45,black] node[above=.15,right=1] {$h_{q+1}$} (5,5);
        \draw[cyan] (0,0) to [bend right=15,black] node[right=1,below] {$h_2$} (5,5);
        \draw[cyan] (0,0) to [bend left=15,black] node[right=1,above] {$h_q$} (5,5);
        \draw[yellow] (0,0) to [bend right=45,black] node[near end,right,below=-.1] {$l_1$} (7,7);
        \draw[yellow] (0,0) to [bend left=45] node[near end,right,above=-.1,black] {$r_1$} (7,7);
        \draw (0,0) to [bend right=45,black] node[near end,below] {$l_{2}$} (9,9);
        \draw (0,0) to [bend left=45,black] node[near end,above] {$r_{2}$} (9,9);
        \draw[magenta] (0,0) to [bend right=45,black] node[near end,below,black] {$l_{3}$} (10,10);
        \draw[magenta] (0,0) to [bend left=45,black] node[near end,above] {$r_{3}$} (10,10);
        \draw (5,5) to node[below] {$m_1$} (7,7);
        \draw (9,9) to node[below] {$m_{3}$} (10,10);
        \draw (7,7) to node[below] {$m_2$} (9,9);
        \fill (0,10) circle (2pt);
        \fill (10,0) circle (2pt);
        \fill (0,0) circle (2pt);
        \fill (10,10) circle (2pt);
        \fill (5,5) circle (2pt);
        \fill (3.2,1.8) node[cross=2pt,rotate=30] {};
        \fill (1.8,3.2) node[cross=2pt,rotate=30] {};
        \fill (7,7) circle (2pt);
        \fill (9,9) circle (2pt);
    \end{scope}
\end{tikzpicture}\]

Step 1 consists of $\alpha_2$.

We first look at the vertices involved in $\alpha_2$ and adjacent vertices.

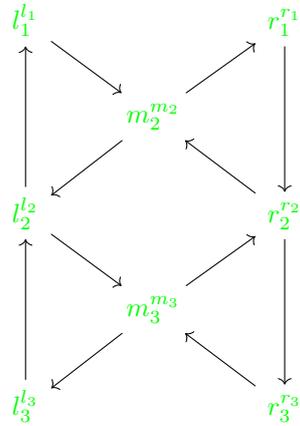
\[\begin{tikzcd}
    \color{green}l_{1}^{l_{1}}\arrow[dr]&&\color{green}r_{1}^{r_{1}}\arrow[dd]\\
    &\color{green}m_2^{m_2}\arrow[ur]\arrow[dl]&\\
    \color{green}l_{2}^{l_2}\arrow[uu]\arrow[dr]&&\color{green}r_{2}^{r_2}\arrow[ul]\arrow[dd]\\
    &\color{green}m_{3}^{m_{3}}\arrow[ur]\arrow[dl]&\\
    \color{green}l_{3}^{l_{3}}\arrow[uu]&&\color{green}r_{3}^{r_{3}}\arrow[ul]
\end{tikzcd}\]

By computation, mutating at $\alpha_2$ changes the subdiargam like this:

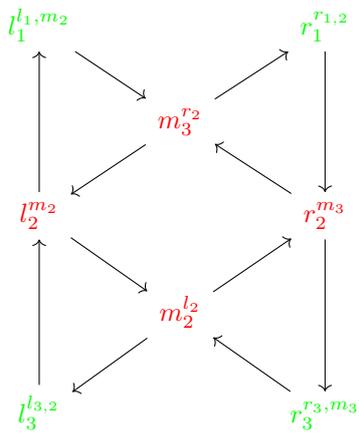
\[\begin{tikzcd}
    \color{green}l_{1}^{l_{1},m_2}\arrow[dr]&&\color{green}r_{1}^{r_{1,2}}\arrow[dd]\\
    &\color{red}m_{3}^{r_2}\arrow[ur]\arrow[dl]&\\
    \color{red}l_{2}^{m_2}\arrow[uu]\arrow[dr]&&\color{red}r_{2}^{m_{3}}\arrow[ul]\arrow[dd]\\
    &\color{red}m_2^{l_2}\arrow[ur]\arrow[dl]&\\
    \color{green}l_{3}^{l_{3,2}}\arrow[uu]&&\color{green}r_{3}^{r_{3},m_{3}}\arrow[ul]
\end{tikzcd}\]

In terms of triangulation, the corresponding subtriangulation at the start and the triangulation after flipping $r_2$ and $l_2$ look like the following:

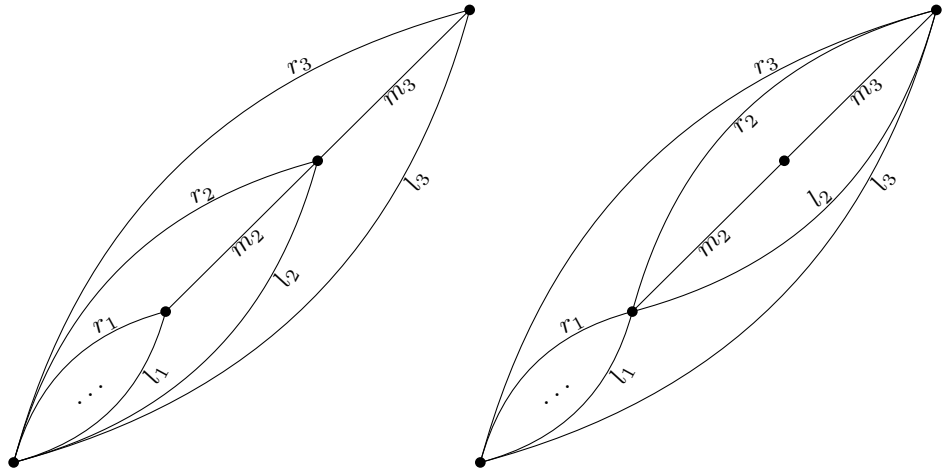
\[\begin{tikzpicture}[every edge quotes/.style={auto=right}]
    \begin{scope}[every node/.style={sloped,allow upside down}][every edge quotes/.style={auto=right}]
        \fill (0,0) circle (2pt);
        \node at ($(1,1)$) {$\iddots$};
        \fill (2,2) circle (2pt);
        \fill (4,4) circle (2pt);
        \fill (6,6) circle (2pt);
        \draw (0,0) to [bend right] node[near end, below=-0.1] {$l_{1}$} (2,2);
        \draw (0,0) to [bend left] node[near end, above=-0.1] {$r_{1}$} (2,2);
        \draw (2,2) to node[below=-0.1] {$m_2$} (4,4);
        \draw (4,4) to node[below=-0.1] {$m_{3}$} (6,6);
        \draw (0,0) to [bend right] node[near end, below=-0.1] {$l_{2}$} (4,4);
        \draw (0,0) to [bend left] node[near end, above=-0.1] {$r_{2}$} (4,4);
        \draw (0,0) to [bend right] node[near end, below=-0.1] {$l_{3}$} (6,6);
        \draw (0,0) to [bend left] node[near end, above=-0.1] {$r_{3}$} (6,6);
    \end{scope}
\end{tikzpicture}
\begin{tikzpicture}[every edge quotes/.style={auto=right}]
    \begin{scope}[every node/.style={sloped,allow upside down}][every edge quotes/.style={auto=right}]
        \fill (0,0) circle (2pt);
        \node at ($(1,1)$) {$\iddots$};
        \fill (2,2) circle (2pt);
        \fill (4,4) circle (2pt);
        \fill (6,6) circle (2pt);
        \draw (0,0) to [bend right] node[near end, below=-0.1] {$l_{1}$} (2,2);
        \draw (0,0) to [bend left] node[near end, above=-0.1] {$r_{1}$} (2,2);
        \draw (2,2) to node[below=-0.1] {$m_2$} (4,4);
        \draw (4,4) to node[below=-0.1] {$m_{3}$} (6,6);
        \draw (2,2) to [bend right] node[above=-0.1] {$l_{2}$} (6,6);
        \draw (2,2) to [bend left] node[below=-0.1] {$r_{2}$} (6,6);
        \draw (0,0) to [bend right] node[near end, below=-0.1] {$l_{3}$} (6,6);
        \draw (0,0) to [bend left] node[near end, above=-0.1] {$r_{3}$} (6,6);
    \end{scope}
\end{tikzpicture}\]

Observe that the arcs $r_2,l_2,m_{3},m_2$ form a triangulated once-punctured digon. Flipping $m_{3},m_2$ tags the third interior puncture notched, and flipping $l_2$ and $r_2$ moves them back to the original positions:

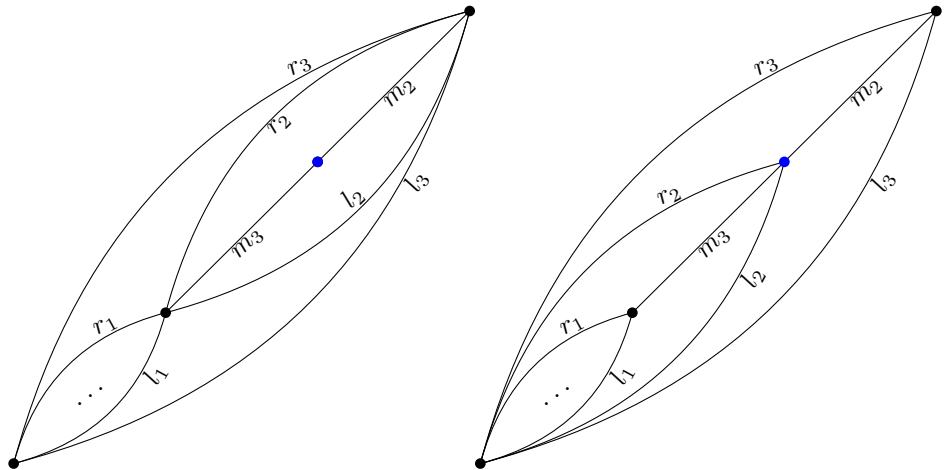
\[\begin{tikzpicture}[every edge quotes/.style={auto=right}]
    \begin{scope}[every node/.style={sloped,allow upside down}][every edge quotes/.style={auto=right}]
        \fill (0,0) circle (2pt);
        \node at ($(1,1)$) {$\iddots$};
        \fill (2,2) circle (2pt);
        \fill (4,4) circle (2pt);
        \fill (6,6) circle (2pt);
        \draw (0,0) to [bend right] node[near end, below=-0.1] {$l_{1}$} (2,2);
        \draw (0,0) to [bend left] node[near end, above=-0.1] {$r_{1}$} (2,2);
        \draw (2,2) to node[below=-0.1] {$m_3$} (4,4);
        \draw (4,4) to node[below=-0.1] {$m_{2}$} (6,6);
        \draw (2,2) to [bend right] node[above=-0.1] {$l_{2}$} (6,6);
        \draw (2,2) to [bend left] node[below=-0.1] {$r_{2}$} (6,6);
        \draw (0,0) to [bend right] node[near end, below=-0.1] {$l_{3}$} (6,6);
        \draw (0,0) to [bend left] node[near end, above=-0.1] {$r_{3}$} (6,6);
        \fill[blue] (4,4) circle (2pt);
    \end{scope}
\end{tikzpicture}
\begin{tikzpicture}[every edge quotes/.style={auto=right}]
    \begin{scope}[every node/.style={sloped,allow upside down}][every edge quotes/.style={auto=right}]
        \fill (0,0) circle (2pt);
        \node at ($(1,1)$) {$\iddots$};
        \fill (2,2) circle (2pt);
        \fill (6,6) circle (2pt);
        \draw (0,0) to [bend right] node[near end, below=-0.1] {$l_{1}$} (2,2);
        \draw (0,0) to [bend left] node[near end, above=-0.1] {$r_{1}$} (2,2);
        \draw (2,2) to node[below=-0.1] {$m_3$} (4,4);
        \draw (4,4) to node[below=-0.1] {$m_{2}$} (6,6);
        \draw (0,0) to [bend right] node[near end, below=-0.1] {$l_{2}$} (4,4);
        \draw (0,0) to [bend left] node[near end, above=-0.1] {$r_{2}$} (4,4);
        \draw (0,0) to [bend right] node[near end, below=-0.1] {$l_{3}$} (6,6);
        \draw (0,0) to [bend left] node[near end, above=-0.1] {$r_{3}$} (6,6);
        \fill[blue] (4,4) circle (2pt);
    \end{scope}
\end{tikzpicture}\]

After step 1, the third interior puncture is tagged notched.

Step 2 consists of $(h_{q+1},h_q,...,h_1,m_{1},h_2,h_3,...,h_{q+1})$. Look at them and adjacent vertices:

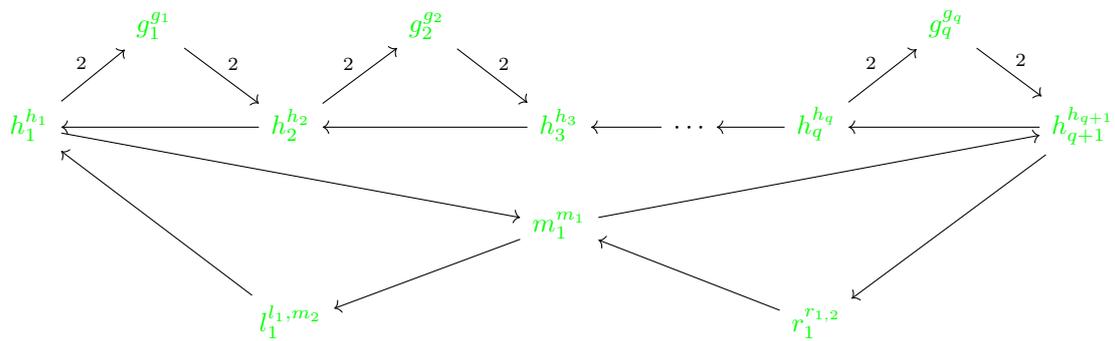
\[\begin{tikzcd}
    &|[color=green]|g_{1}^{g_1}\arrow[dr,"2"]&&|[color=green]|g_{2}^{g_2}\arrow[dr,"2"]&&&&|[color=green]|g_{q}^{g_q}\arrow[dr,"2"]\\
    |[color=green]|h_{1}^{h_1}\arrow[ur,"2"]\arrow[drrrr]&&|[color=green]|h_{2}^{h_2}\arrow[ur,"2"]\arrow[ll]&&|[color=green]|{h_{3}^{h_{3}}}\arrow[ll]&\cdots\arrow[l]&|[color=green]|h_{q}^{h_q}\arrow[ur,"2"]\arrow[l]&&|[color=green]|{h_{q+1}^{h_{q+1}}}\arrow[ll]\arrow[ddll]\\
    &&&&|[color=green]|m_1^{m_1}\arrow[urrrr]\arrow[dll]\\
    &&|[color=green]|l_1^{l_1,m_2}\arrow[uull]&&&&|[color=green]|r_1^{r_{1,2}}\arrow[ull]
\end{tikzcd}\]

By applying Lemma \ref{p>2}, we have:

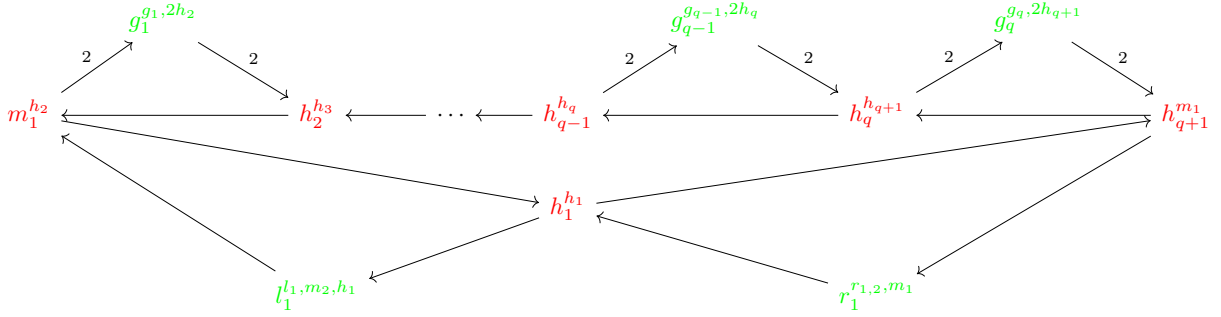
\[\adjustbox{scale=0.9,center}{
\begin{tikzcd}
    &\color{green}g_1^{g_1,2h_2}\arrow[dr,"2"]&&&&\color{green}g_{q-1}^{g_{q-1},2h_q}\arrow[dr,"2"]&&\color{green}g_q^{g_q,2h_{q+1}}\arrow[dr,"2"]\\
    \color{red}m_1^{h_2}\arrow[drrrr]\arrow[ur,"2"]&&\color{red}h_2^{h_3}\arrow[ll]&\cdots\arrow[l]&\color{red}h_{q-1}^{h_q}\arrow[ur,"2"]\arrow[l]&&\color{red}h_{q}^{h_{q+1}}\arrow[ll]\arrow[ur,"2"]&&\color{red}h_{q+1}^{m_1}\arrow[ll]\arrow[ddll]\\
    &&&&\color{red}h_1^{h_1}\arrow[urrrr]\arrow[dll]\\
    &&\color{green}l_1^{l_1,m_2,h_1}\arrow[uull]&&&&\color{green}r_1^{r_{1,2},m_1}\arrow[ull]
\end{tikzcd}}\]

In terms of triangulation, we have tagged the first interior puncture notched.

Step 3 consists of $\beta_1$.

Now look at the vertices involved in $\beta_k$ for $k>1$ and adjacent vertices. Note that as a result of $\alpha_2$ and step 2, $m_{3}$ (instead of $m_{2}$) is adjacent to $l_1$ and $r_1$, and $m_1,h_1$ switched position.

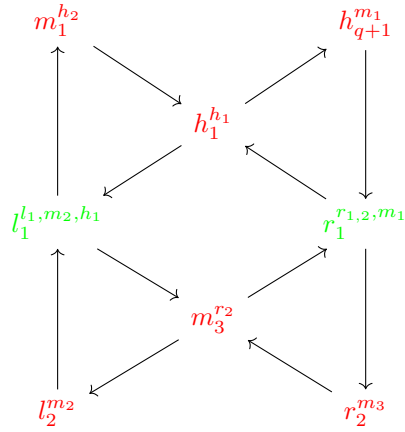
\[\begin{tikzcd}
    \color{red}m_1^{h_2}\arrow[dr]&&\color{red}h_{q+1}^{m_1}\arrow[dd]\\
    &\color{red}h_1^{h_1}\arrow[ur]\arrow[dl]&\\
    \color{green}l_{1}^{l_{1},m_{2},h_1}\arrow[uu]\arrow[dr]&&\color{green}r_{1}^{r_{1,2},m_1}\arrow[ul]\arrow[dd]\\
    &\color{red}m_{3}^{r_{2}}\arrow[ur]\arrow[dl]&\\
    \color{red}l_{2}^{m_{2}}\arrow[uu]&&\color{red}r_{2}^{m_{3}}\arrow[ul]
\end{tikzcd}\]

By computation, mutating at $\beta_1$ gives:

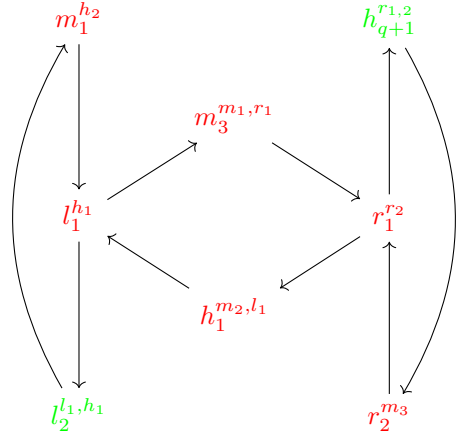
\[\begin{tikzcd}
    \color{red}m_1^{h_2}\arrow[dd]&&\color{green}h_{q+1}^{r_{1,2}}\arrow[dddd,bend left]\\
    &\color{red}m_{3}^{m_1,r_1}\arrow[dr]&\\
    \color{red}l_{1}^{h_1
    }\arrow[dd]\arrow[ur]&&\color{red}r_{1}^{r_{2}}\arrow[dl]\arrow[uu]\\
    &\color{red}h_1^{m_{2},l_1}\arrow[ul]&\\
    \color{green}l_{2}^{l_{1},h_1}\arrow[uuuu,bend left]&&\color{red}r_{2}^{m_{3}}\arrow[uu]
\end{tikzcd}\]

After step 3, the second interior puncture is tagged notched. Together with step 1 and step 2, the remaining puncture not yet tagged notched is the corner. Also the arcs $l_1$ and $r_1$ are moved away from the corner.

Step 4(a) consists of $(g_1,g_2,...,g_q,h_{q+1},h_{q},...,h_2,m_1)$.

Mutating at $g_1$ does the following:

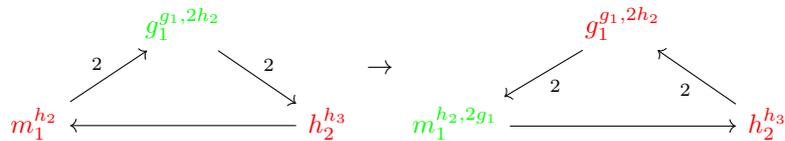
\[
\begin{tikzcd}
    &\color{green}g_{1}^{g_1,2h_2}\arrow[dr,"2"]&\\
    \color{red}m_1^{h_2}\arrow[ur,"2"]&&\color{red}h_2^{h_3}\arrow[ll]
\end{tikzcd}
\rightarrow
\begin{tikzcd}
    &\color{red}g_{1}^{g_1,2h_2}\arrow[dl,"2"]&\\
    \color{green}m_1^{h_2,2g_1}\arrow[rr]&&\color{red}h_2^{h_3}\arrow[ul,"2"]
\end{tikzcd}
\]

Observe that for all $g_k$, $k=1,2,...,q$, the full subdiagram containing $g_k$ and adjancent vertices are isomorphic (again except the frozen vertices attached), so after mutating at $g_1,g_2,...,g_q$, we have:

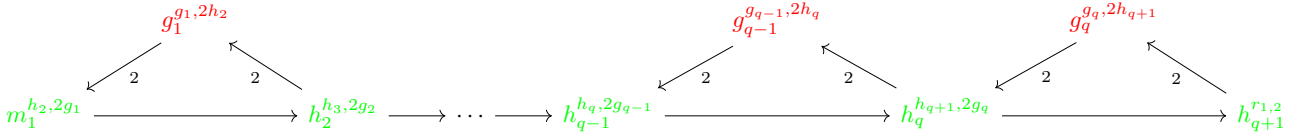
\[\adjustbox{scale=0.9,center}{
\begin{tikzcd}
    &\color{red}g_1^{g_1,2h_2}\arrow[dl,"2"]&&&&\color{red}g_{q-1}^{g_{q-1},2h_q}\arrow[dl,"2"]&&\color{red}g_q^{g_q,2h_{q+1}}\arrow[dl,"2"]\\
    \color{green}m_1^{h_2,2g_1}\arrow[rr]&&\color{green}h_2^{h_3,2g_2}\arrow[r]\arrow[ul,"2"]&\cdots\arrow[r]&\color{green}h_{q-1}^{h_q,2g_{q-1}}\arrow[rr]&&\color{green}h_{q}^{h_{q+1},2g_q}\arrow[rr]\arrow[ul,"2"]&&\color{green}h_{q+1}^{r_{1,2}}\arrow[ul,"2"]
\end{tikzcd}}\]

In terms of triangulation, all pending arcs are now connected to the first interior puncture (like in the case where $p=2,q>1$).

At the vertices $h_{q+1},h_q,...,h_2,m_1$, and adjacent vertices, we have:

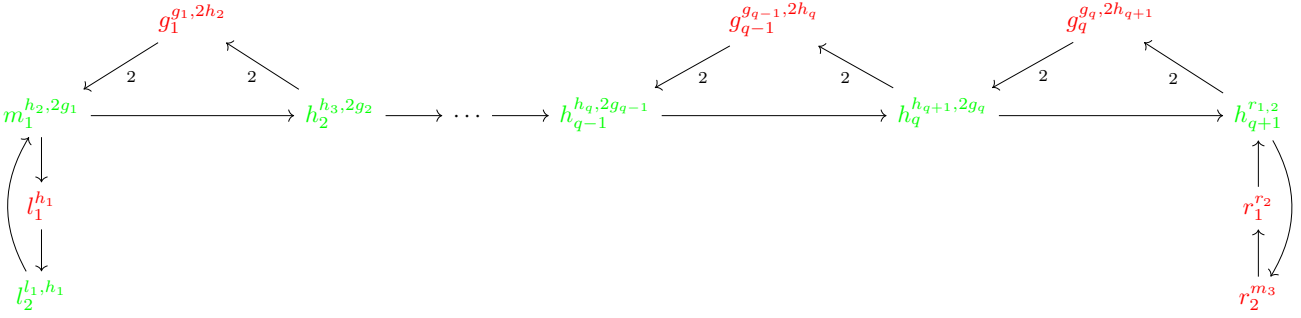
\[\adjustbox{scale=0.9,center}{
\begin{tikzcd}
    &\color{red}g_1^{g_1,2h_2}\arrow[dl,"2"]&&&&\color{red}g_{q-1}^{g_{q-1},2h_q}\arrow[dl,"2"]&&\color{red}g_q^{g_q,2h_{q+1}}\arrow[dl,"2"]\\
    \color{green}m_1^{h_2,2g_1}\arrow[rr]\arrow[d]&&\color{green}h_2^{h_3,2g_2}\arrow[r]\arrow[ul,"2"]&\cdots\arrow[r]&\color{green}h_{q-1}^{h_q,2g_{q-1}}\arrow[rr]&&\color{green}h_{q}^{h_{q+1},2g_q}\arrow[rr]\arrow[ul,"2"]&&\color{green}h_{q+1}^{r_{1,2}}\arrow[ul,"2"]\arrow[dd,bend left]
    \\\color{red}l_1^{h_1}\arrow[d]&&&&&&&&\color{red}r_1^{r_2}\arrow[u]
    \\\color{green}l_2^{l_1,h_1}\arrow[uu,bend left]&&&&&&&&\color{red}r_2^{m_3}\arrow[u]
\end{tikzcd}}\]

For convenience, we shall show the result of step 4(a) in a lemma:

\begin{lemma}\label{p>4}
    Mutating the following diagram (the blue colour of the vertex $r_2$ means that the greenness of $r_2$ does not matter):
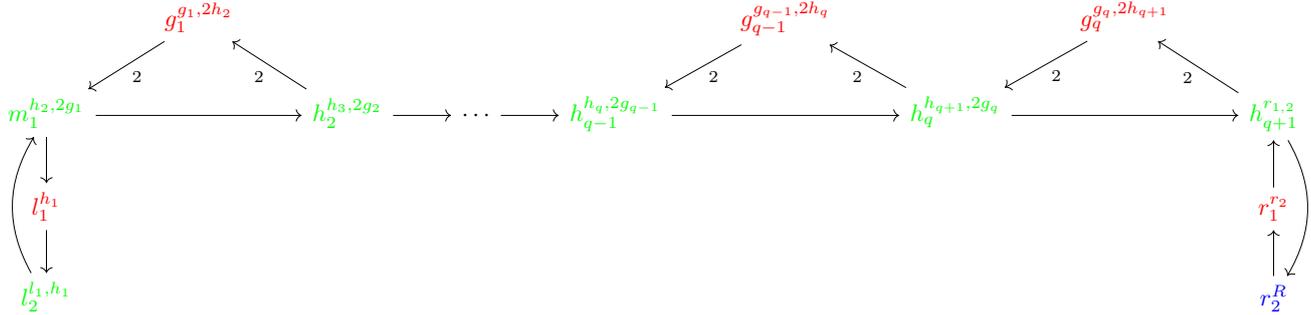
\[\adjustbox{scale=0.9,center}{
\begin{tikzcd}
    &\color{red}g_1^{g_1,2h_2}\arrow[dl,"2"]&&&&\color{red}g_{q-1}^{g_{q-1},2h_q}\arrow[dl,"2"]&&\color{red}g_q^{g_q,2h_{q+1}}\arrow[dl,"2"]\\
    \color{green}m_1^{h_2,2g_1}\arrow[rr]\arrow[d]&&\color{green}h_2^{h_3,2g_2}\arrow[r]\arrow[ul,"2"]&\cdots\arrow[r]&\color{green}h_{q-1}^{h_q,2g_{q-1}}\arrow[rr]&&\color{green}h_{q}^{h_{q+1},2g_q}\arrow[rr]\arrow[ul,"2"]&&\color{green}h_{q+1}^{r_{1,2}}\arrow[ul,"2"]\arrow[dd,bend left]
    \\\color{red}l_1^{h_1}\arrow[d]&&&&&&&&\color{red}r_1^{r_2}\arrow[u]
    \\\color{green}l_2^{l_1,h_1}\arrow[uu,bend left]&&&&&&&&\color{blue}r_2^{R}\arrow[u]
\end{tikzcd}}\]
    , where $R$ is a set of (frozen) vertices, with the mutation sequence $(h_{q+1},h_q,...,h_2,m_1)$ gives the following diagram (after relocating some vertices):
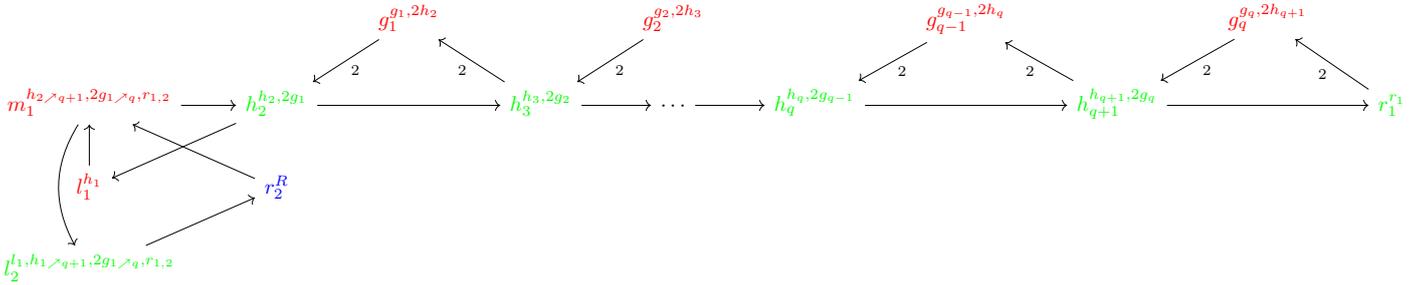
\[\adjustbox{scale=0.8,center}{
\begin{tikzcd}
    &&\color{red}g_1^{g_1,2h_2}\arrow[dl,"2"]&&\color{red}g_2^{g_2,2h_3}\arrow[dl,"2"]&&\color{red}g_{q-1}^{g_{q-1},2h_q}\arrow[dl,"2"]&&\color{red}g_q^{g_q,2h_{q+1}}\arrow[dl,"2"]\\
    \color{red}m_1^{h_{2\nearrow q+1},2g_{1\nearrow q},r_{1,2}}\arrow[dd,bend right]\arrow[r]&\color{green}h_{2}^{h_{2},2g_{1}}\arrow[dl]\arrow[rr]&&\color{green}h_3^{h_3,2g_2}\arrow[ul,"2"]\arrow[r]&\cdots\arrow[r]&\color{green}h_{q}^{h_q,2g_{q-1}}\arrow[rr]&&\color{green}h_{q+1}^{h_{q+1},2g_q}\arrow[ul,"2"]\arrow[rr]&&\color{green}r_1^{r_1}\arrow[ul,"2"]\\
    \color{red}l_1^{h_1}\arrow[u]&\color{blue}r_2^{R}\arrow[ul]\\
    \color{green}l_2^{l_1,h_{1\nearrow q+1},2g_{1\nearrow q},r_{1,2}}\arrow[ur]
\end{tikzcd}}\]
\end{lemma}

\begin{proof}
    Look at $h_{q+1}$, $h_{q}$, $h_{q-1}$, and adjacent vertices:

$\adjustbox{scale=0.9,center}
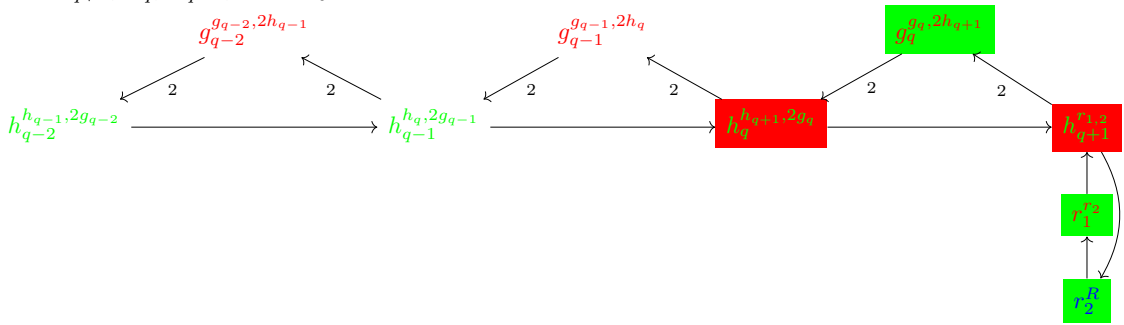
{
\begin{tikzcd}
    &\color{red}g_{q-2}^{g_{q-2},2h_{q-1}}\arrow[dl,"2"]&&\color{red}g_{q-1}^{g_{q-1},2h_q}\arrow[dl,"2"]&&|[color=red,fill=green]|g_q^{g_q,2h_{q+1}}\arrow[dl,"2"]\\
    \color{green}h_{q-2}^{h_{q-1},2g_{q-2}}\arrow[rr]&&\color{green}h_{q-1}^{h_q,2g_{q-1}}\arrow[rr]\arrow[ul,"2"]&&|[color=green,fill=red]|h_{q}^{h_{q+1},2g_q}\arrow[rr]\arrow[ul,"2"]&&|[color=green,fill=red]|h_{q+1}^{r_{1,2}}\arrow[ul,"2"]\arrow[dd,bend left]\\
    &&&&&&|[color=red,fill=green]|r_1^{r_2}\arrow[u]\\
    &&&&&&|[color=blue,fill=green]|r_2^{R}\arrow[u]
\end{tikzcd}}$

Mutate at $h_{q+1}$:

$\adjustbox{scale=0.9,center}
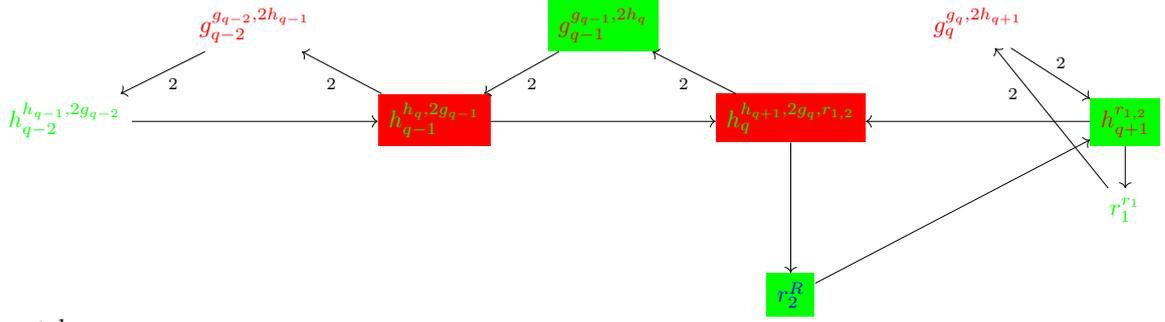
{
\begin{tikzcd}
    &\color{red}g_{q-2}^{g_{q-2},2h_{q-1}}\arrow[dl,"2"]&&|[color=red,fill=green]|g_{q-1}^{g_{q-1},2h_q}\arrow[dl,"2"]&&\color{red}g_q^{g_q,2h_{q+1}}\arrow[dr,"2"]\\
    \color{green}h_{q-2}^{h_{q-1},2g_{q-2}}\arrow[rr]&&|[color=green,fill=red]|h_{q-1}^{h_q,2g_{q-1}}\arrow[rr]\arrow[ul,"2"]&&|[color=green,fill=red]|h_{q}^{h_{q+1},2g_q,r_{1,2}}\arrow[dd]\arrow[ul,"2"]&&|[color=red,fill=green]|h_{q+1}^{r_{1,2}}\arrow[ll]\arrow[d]\\
    &&&&&&\color{green}r_1^{r_1}\arrow[uul,"2",near end]\\
    &&&&|[color=blue,fill=green]|r_2^{R}\arrow[uurr]
\end{tikzcd}}$

Mutate at $h_{q}$:

$\adjustbox{scale=0.9,center}
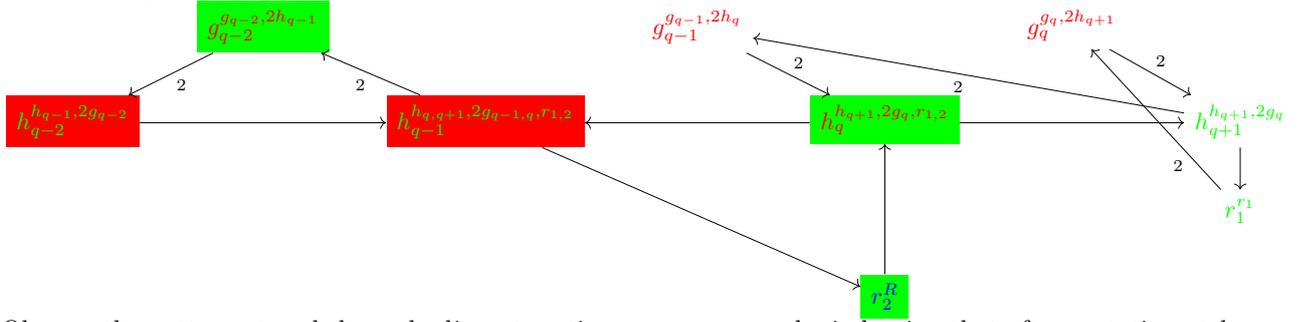
{
\begin{tikzcd}
    &|[color=red,fill=green]|g_{q-2}^{g_{q-2},2h_{q-1}}\arrow[dl,"2"]&&\color{red}g_{q-1}^{g_{q-1},2h_q}\arrow[dr,"2"]&&\color{red}g_q^{g_q,2h_{q+1}}\arrow[dr,"2"]\\
    |[color=green,fill=red]|h_{q-2}^{h_{q-1},2g_{q-2}}\arrow[rr]&&|[color=green,fill=red]|h_{q-1}^{h_{q,q+1},2g_{q-1,q},r_{1,2}}\arrow[ddrr]\arrow[ul,"2"]&&|[color=red,fill=green]|h_{q}^{h_{q+1},2g_q,r_{1,2}}\arrow[ll]\arrow[rr]&&\color{green}h_{q+1}^{h_{q+1},2g_q}\arrow[ulll,"2"]\arrow[d]\\
    &&&&&&\color{green}r_1^{r_1}\arrow[uul,"2",near start]\\
    &&&&|[color=blue,fill=green]|r_2^{R}\arrow[uu]
\end{tikzcd}}$

Observe the pattern at each $h_k$ and adjacent vertices, we can argue by induction that after mutating at $h_2$, we have:

$\adjustbox{scale=0.8,center}
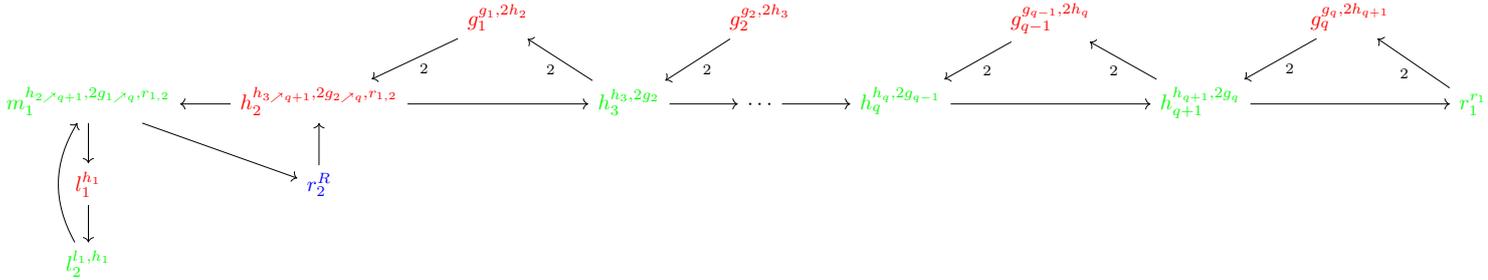
{
\begin{tikzcd}
    &&\color{red}g_1^{g_1,2h_2}\arrow[dl,"2"]&&\color{red}g_2^{g_2,2h_3}\arrow[dl,"2"]&&\color{red}g_{q-1}^{g_{q-1},2h_q}\arrow[dl,"2"]&&\color{red}g_q^{g_q,2h_{q+1}}\arrow[dl,"2"]\\
    \color{green}m_1^{h_{2\nearrow q+1},2g_{1\nearrow q},r_{1,2}}\arrow[d]\arrow[dr]&\color{red}h_{2}^{h_{3\nearrow q+1},2g_{2\nearrow q},r_{1,2}}\arrow[l]\arrow[rr]&&\color{green}h_3^{h_3,2g_2}\arrow[ul,"2"]\arrow[r]&\cdots\arrow[r]&\color{green}h_{q}^{h_q,2g_{q-1}}\arrow[rr]&&\color{green}h_{q+1}^{h_{q+1},2g_q}\arrow[ul,"2"]\arrow[rr]&&\color{green}r_1^{r_1}\arrow[ul,"2"]\\
    \color{red}l_1^{h_1}\arrow[d]&\color{blue}r_2^{R}\arrow[u]\\
    \color{green}l_2^{l_1,h_1}\arrow[uu,bend left]
\end{tikzcd}}$

Mutate at $m_1$:

$\adjustbox{scale=0.8,center}
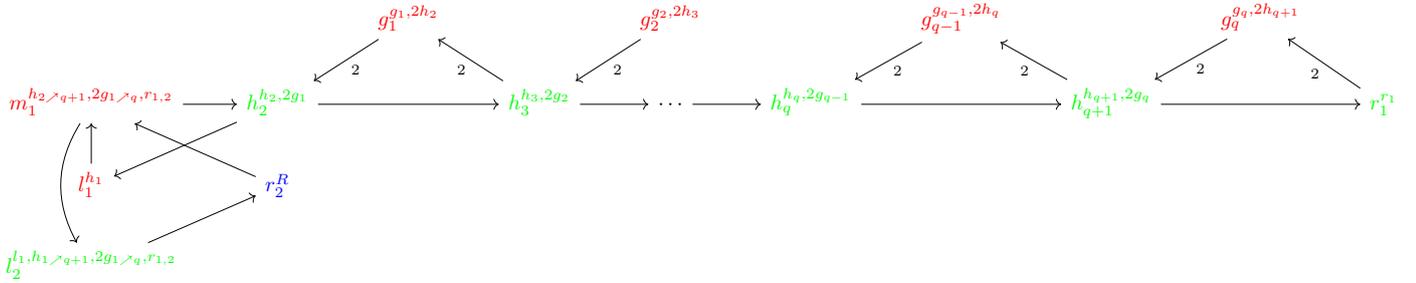
{
\begin{tikzcd}
    &&\color{red}g_1^{g_1,2h_2}\arrow[dl,"2"]&&\color{red}g_2^{g_2,2h_3}\arrow[dl,"2"]&&\color{red}g_{q-1}^{g_{q-1},2h_q}\arrow[dl,"2"]&&\color{red}g_q^{g_q,2h_{q+1}}\arrow[dl,"2"]\\
    \color{red}m_1^{h_{2\nearrow q+1},2g_{1\nearrow q},r_{1,2}}\arrow[dd,bend right]\arrow[r]&\color{green}h_{2}^{h_{2},2g_{1}}\arrow[dl]\arrow[rr]&&\color{green}h_3^{h_3,2g_2}\arrow[ul,"2"]\arrow[r]&\cdots\arrow[r]&\color{green}h_{q}^{h_q,2g_{q-1}}\arrow[rr]&&\color{green}h_{q+1}^{h_{q+1},2g_q}\arrow[ul,"2"]\arrow[rr]&&\color{green}r_1^{r_1}\arrow[ul,"2"]\\
    \color{red}l_1^{h_1}\arrow[u]&\color{blue}r_2^{R}\arrow[ul]\\
    \color{green}l_2^{l_1,h_{1\nearrow q+1},2g_{1\nearrow q},r_{1,2}}\arrow[ur]
\end{tikzcd}}$
\end{proof}

In terms of triangulation, we have wrapped the arcs $g_1,g_2,...,g_q,h_2,h_3,...,h_q,l_1,r_1$ inside $m_1$ (and thus removed all core arcs away from the corner).

By applying the above lemma, after step 4(a), we have:

$\adjustbox{scale=0.8,center}
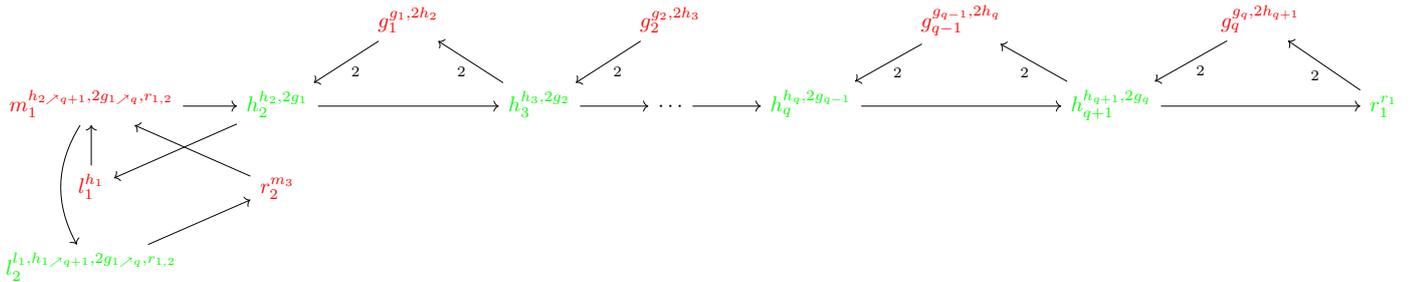
{
\begin{tikzcd}
    &&\color{red}g_1^{g_1,2h_2}\arrow[dl,"2"]&&\color{red}g_2^{g_2,2h_3}\arrow[dl,"2"]&&\color{red}g_{q-1}^{g_{q-1},2h_q}\arrow[dl,"2"]&&\color{red}g_q^{g_q,2h_{q+1}}\arrow[dl,"2"]\\
    \color{red}m_1^{h_{2\nearrow q+1},2g_{1\nearrow q},r_{1,2}}\arrow[dd,bend right]\arrow[r]&\color{green}h_{2}^{h_{2},2g_{1}}\arrow[dl]\arrow[rr]&&\color{green}h_3^{h_3,2g_2}\arrow[ul,"2"]\arrow[r]&\cdots\arrow[r]&\color{green}h_{q}^{h_q,2g_{q-1}}\arrow[rr]&&\color{green}h_{q+1}^{h_{q+1},2g_q}\arrow[ul,"2"]\arrow[rr]&&\color{green}r_1^{r_1}\arrow[ul,"2"]\\
    \color{red}l_1^{h_1}\arrow[u]&\color{red}r_2^{m_3}\arrow[ul]\\
    \color{green}l_2^{l_1,h_{1\nearrow q+1},2g_{1\nearrow q},r_{1,2}}\arrow[ur]
\end{tikzcd}}$

We shall merge steps 4(c), 5, and 6(a) so that we can apply Lemma \ref{corner}

At $l_{p-1},r_{p-1},f_1,f_2,m_{p-2},l_{p-2},r_{p-2}$ and adjacent vertices, we have:

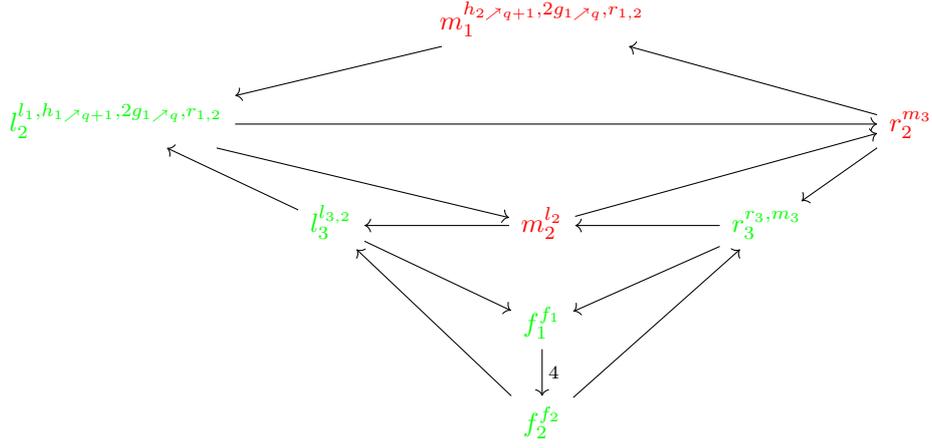
\[\begin{tikzcd}
    &&\color{red}m_1^{h_{2\nearrow q+1},2g_{1\nearrow q},r_{1,2}}\arrow[dll]&&\\
    \color{green}l_{2}^{l_1,h_{1\nearrow q+1},2g_{1\nearrow q},r_{1,2}}\arrow[rrrr]\arrow[drr]&&&&\color{red}r_{2}^{m_{3}}\arrow[dl]\arrow[ull]\\
    &\color{green}l_{3}^{l_{3,2}}\arrow[ul]\arrow[dr]&\color{red}m_{2}^{l_{2}}\arrow[l]\arrow[urr]&\color{green}r_{3}^{r_{3},m_{3}}\arrow[l]\arrow[dl]\\
    &&\color{green}f_1^{f_1}\arrow[d,"4"]&\\
    &&\color{green}f_2^{f_2}\arrow[uul]\arrow[uur]&
\end{tikzcd}\]

By applying Lemma \ref{corner}, after mutating at $l_{3},f_1,r_{3},f_2,l_{3},f_1,m_{2},l_{2},r_{2},m_{2},f_1,l_{3},f_2,r_{3},f_1,l_{3}$, we have:

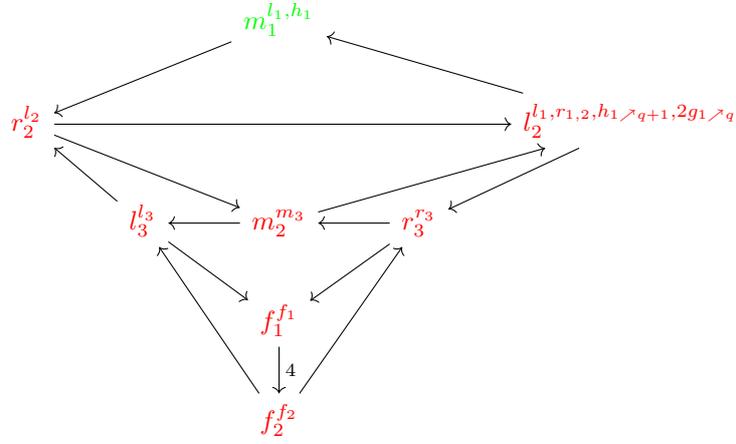
\[\begin{tikzcd}
    &&\color{green}m_1^{l_1,h_1}\arrow[dll]&&\\
    \color{red}r_{2}^{l_{2}}\arrow[rrrr]\arrow[drr]&&&&\color{red}l_{2}^{l_{1},r_{1,2},h_{1\nearrow q+1},2g_{1\nearrow q}}\arrow[dl]\arrow[ull]\\
    &\color{red}l_{3}^{l_{3}}\arrow[ul]\arrow[dr]&\color{red}m_{2}^{m_{3}}\arrow[l]\arrow[urr]&\color{red}r_{3}^{r_{3}}\arrow[l]\arrow[dl]\\
    &&\color{red}f_1^{f_1}\arrow[d,"4"]&\\
    &&\color{red}f_2^{f_2}\arrow[uul]\arrow[uur]&
    \vspace*{-0.1\linewidth}
\end{tikzcd}\]

Step 6(c) consists of $m_1,h_2,h_3,...,h_{q+1},g_1,g_2,...,g_q$. We shall look at them and adjacent vertices:

$\adjustbox{scale=0.8,center}
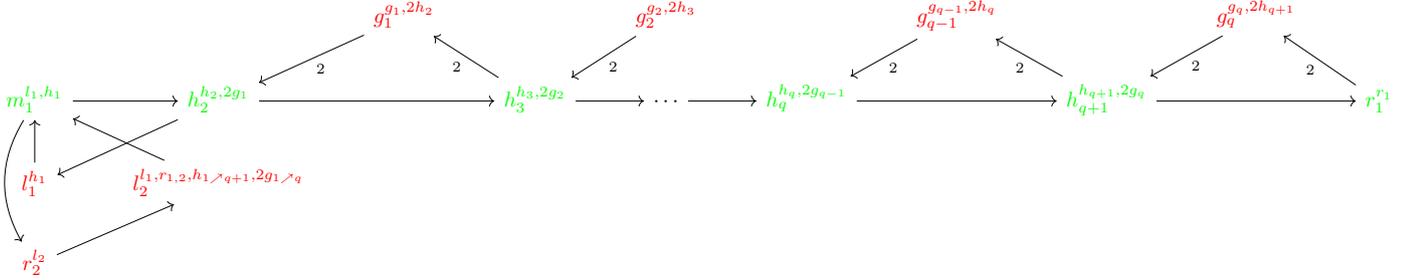
{
\begin{tikzcd}
    &&\color{red}g_1^{g_1,2h_2}\arrow[dl,"2"]&&\color{red}g_2^{g_2,2h_3}\arrow[dl,"2"]&&\color{red}g_{q-1}^{g_{q-1},2h_q}\arrow[dl,"2"]&&\color{red}g_q^{g_q,2h_{q+1}}\arrow[dl,"2"]\\
    \color{green}m_1^{l_1,h_1}\arrow[dd,bend right]\arrow[r]&\color{green}h_{2}^{h_{2},2g_{1}}\arrow[dl]\arrow[rr]&&\color{green}h_3^{h_3,2g_2}\arrow[ul,"2"]\arrow[r]&\cdots\arrow[r]&\color{green}h_{q}^{h_q,2g_{q-1}}\arrow[rr]&&\color{green}h_{q+1}^{h_{q+1},2g_q}\arrow[ul,"2"]\arrow[rr]&&\color{green}r_1^{r_1}\arrow[ul,"2"]\\
    \color{red}l_1^{h_1}\arrow[u]&\color{red}l_{2}^{l_{1},r_{1,2},h_{1\nearrow q+1},2g_{1\nearrow q}}\arrow[ul]\\
    \color{red}r_{2}^{l_{2}}\arrow[ur]
\end{tikzcd}}$

Mutate at $m_1$:

\[\adjustbox{scale=0.8,center}
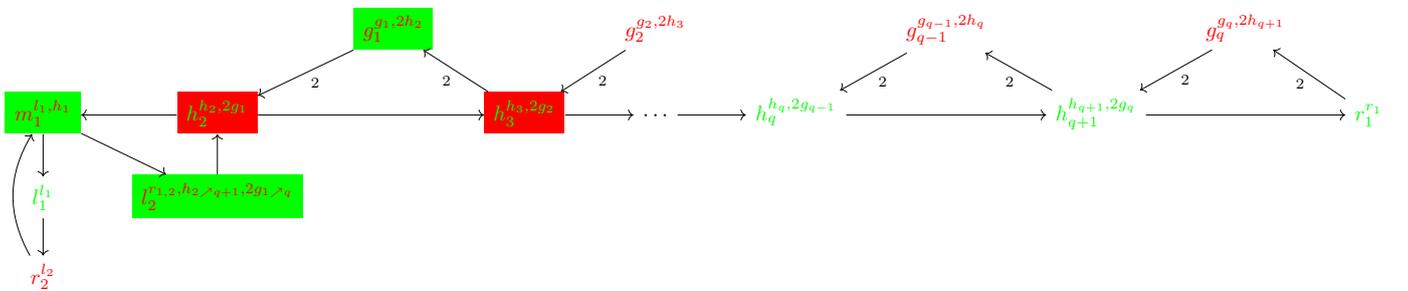
{
\begin{tikzcd}
    &&|[color=red,fill=green]|g_1^{g_1,2h_2}\arrow[dl,"2"]&&\color{red}g_2^{g_2,2h_3}\arrow[dl,"2"]&&\color{red}g_{q-1}^{g_{q-1},2h_q}\arrow[dl,"2"]&&\color{red}g_q^{g_q,2h_{q+1}}\arrow[dl,"2"]\\
    |[color=red,fill=green]|m_1^{l_1,h_1}\arrow[dr]\arrow[d]&|[color=green,fill=red]|h_{2}^{h_{2},2g_{1}}\arrow[l]\arrow[rr]&&|[color=green,fill=red]|h_3^{h_3,2g_2}\arrow[ul,"2"]\arrow[r]&\cdots\arrow[r]&\color{green}h_{q}^{h_q,2g_{q-1}}\arrow[rr]&&\color{green}h_{q+1}^{h_{q+1},2g_q}\arrow[ul,"2"]\arrow[rr]&&\color{green}r_1^{r_1}\arrow[ul,"2"]\\
    \color{green}l_1^{l_1}\arrow[d]&|[color=red,fill=green]|l_2^{r_{1,2},h_{2\nearrow q+1},2g_{1\nearrow q}}\arrow[u]\\
    \color{red}r_2^{l_2}\arrow[uu,bend left]
\end{tikzcd}}\]

Mutate at $h_2$:

\[\adjustbox{scale=0.8,center}
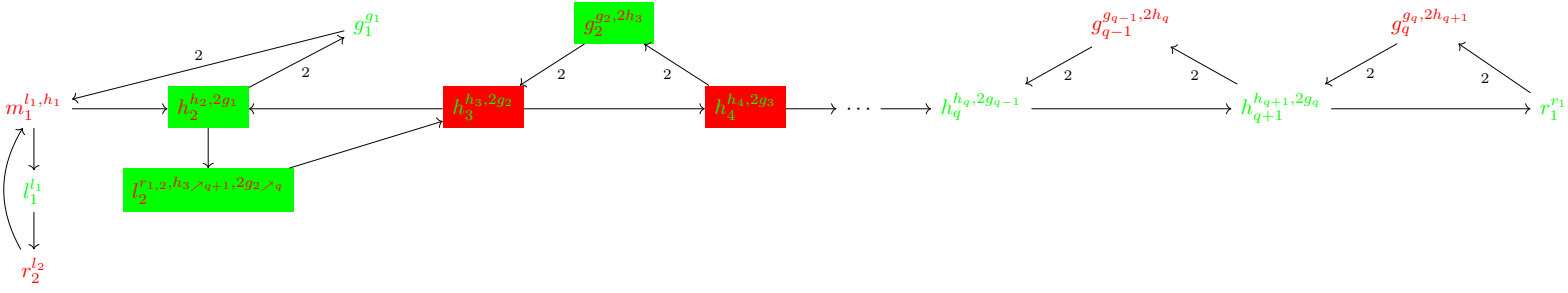
{
\begin{tikzcd}
    &&\color{green}g_1^{g_1}\arrow[dll,"2",swap]&&|[color=red,fill=green]|g_2^{g_2,2h_3}\arrow[dl,"2"]&&&&\color{red}g_{q-1}^{g_{q-1},2h_q}\arrow[dl,"2"]&&\color{red}g_q^{g_q,2h_{q+1}}\arrow[dl,"2"]\\
    \color{red}m_1^{l_1,h_1}\arrow[r]\arrow[d]&|[color=red,fill=green]|h_{2}^{h_{2},2g_{1}}\arrow[d]\arrow[ur,"2",swap]&&|[color=green,fill=red]|h_3^{h_3,2g_2}\arrow[ll]\arrow[rr]&&|[color=green,fill=red]|h_4^{h_4,2g_3}\arrow[ul,"2"]\arrow[r]&\cdots\arrow[r]&\color{green}h_{q}^{h_q,2g_{q-1}}\arrow[rr]&&\color{green}h_{q+1}^{h_{q+1},2g_q}\arrow[ul,"2"]\arrow[rr]&&\color{green}r_1^{r_1}\arrow[ul,"2"]\\
    \color{green}l_1^{l_1}\arrow[d]&|[color=red,fill=green]|l_2^{r_{1,2},h_{3\nearrow q+1},2g_{2\nearrow q}}\arrow[urr]\\
    \color{red}r_2^{l_2}\arrow[uu,bend left]
\end{tikzcd}}\]

Observe that the subdiagrams are isomorphic, we may conclude that after $h_{q+1}$, the diagram becomes:

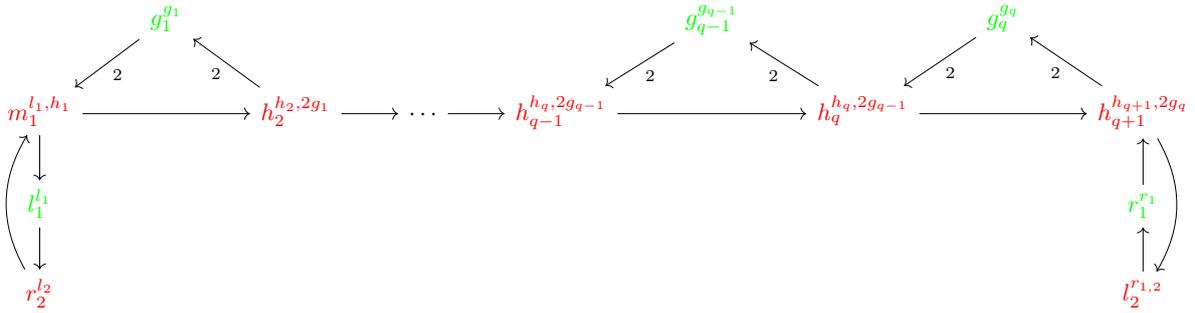
\[\adjustbox{scale=0.9,center}{
\begin{tikzcd}
    &\color{green}g_1^{g_1}\arrow[dl,"2"]&&&&\color{green}g_{q-1}^{g_{q-1}}\arrow[dl,"2"]&&\color{green}g_q^{g_q}\arrow[dl,"2"]\\
    \color{red}m_1^{l_1,h_1}\arrow[rr]\arrow[d]&&\color{red}h_2^{h_2,2g_1}\arrow[r]\arrow[ul,"2"]&\cdots\arrow[r]&\color{red}h_{q-1}^{h_q,2g_{q-1}}\arrow[rr]&&\color{red}h_{q}^{h_{q},2g_{q-1}}\arrow[rr]\arrow[ul,"2"]&&\color{red}h_{q+1}^{h_{q+1},2g_q}\arrow[ul,"2"]\arrow[dd,bend left]\\
    \color{green}l_1^{l_1}\arrow[d]&&&&&&&&\color{green}r_1^{r_1}\arrow[u]\\
    \color{red}r_2^{l_2}\arrow[uu,bend left]&&&&&&&&\color{red}l_2^{r_{1,2}}\arrow[u]
\end{tikzcd}}\]

We have unwrapped $m_1,h_2,h_3,...,h_{q+1}$.

Mutate at $g_1,g_2,...,g_q$:

\[\adjustbox{scale=0.9,center}{
\begin{tikzcd}
    &\color{red}g_1^{g_1}\arrow[dr,"2"]&&&&\color{red}g_{q-1}^{g_{q-1}}\arrow[dr,"2"]&&\color{red}g_q^{g_q}\arrow[dr,"2"]\\
    \color{red}m_1^{l_1,h_1}\arrow[d]\arrow[ur,"2"]&&\color{red}h_2^{h_2}\arrow[ll]&\cdots\arrow[l]&\color{red}h_{q-1}^{h_q}\arrow[l]\arrow[ur,"2"]&&\color{red}h_{q}^{h_{q}}\arrow[ll]\arrow[ur,"2"]&&\color{red}h_{q+1}^{h_{q+1}}\arrow[ll]\arrow[dd,bend left]\\
    \color{green}l_1^{l_1}\arrow[d]&&&&&&&&\color{green}r_1^{r_1}\arrow[u]\\
    \color{red}r_2^{l_2}\arrow[uu,bend left]&&&&&&&&\color{red}l_2^{r_{1,2}}\arrow[u]
\end{tikzcd}}\]

Now all pending arcs are connected to the corner again.

Step 6(d) consists of $l_1,r_1$.

At $l_1,r_1$ and adjacent vertices, we have:

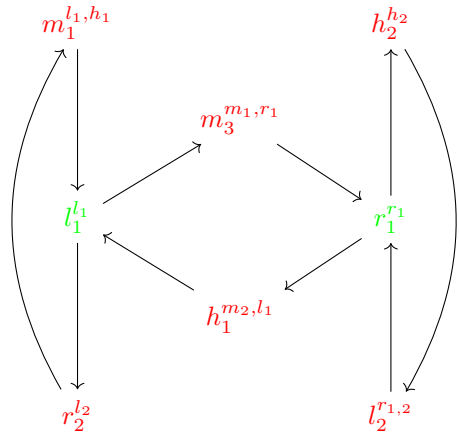
\[\begin{tikzcd}
    \color{red}m_1^{l_1,h_1}\arrow[dd]&&\color{red}h_2^{h_2}\arrow[dddd,bend left]\\
    &\color{red}m_3^{m_1,r_1}\arrow[dr]&\\
    \color{green}l_{1}^{l_1}\arrow[dd]\arrow[ur]&&\color{green}r_{1}^{r_1}\arrow[uu]\arrow[dl]\\
    &\color{red}h_1^{m_2,l_1}\arrow[ul]&\\
    \color{red}\color{red}r_2^{l_2}\arrow[uuuu,bend left]&&\color{red}l_2^{r_{1,2}}\arrow[uu]
\end{tikzcd}\]

After mutating at $l_1,r_1$, we have:

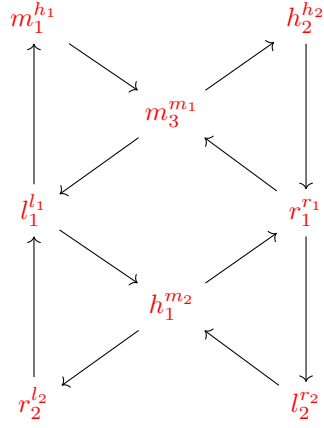
\[\begin{tikzcd}
    \color{red}m_1^{h_1}\arrow[dr]&&\color{red}h_2^{h_2}\arrow[dd]\\
    &\color{red}m_3^{m_1}\arrow[ur]\arrow[dl]&\\
    \color{red}l_{1}^{l_1}\arrow[uu]\arrow[dr]&&\color{red}r_{1}^{r_1}\arrow[dd]\arrow[ul]\\
    &\color{red}h_1^{m_2}\arrow[ur]\arrow[dl]&\\
    \color{red}\color{red}r_2^{l_2}\arrow[uu]&&\color{red}l_2^{r_2}\arrow[ul]
\end{tikzcd}\]

Now $l_1$ and $r_1$ return to their original positions, as promised, and all vertices are red again. Therefore $\Delta_1$ is a maximal green sequence. $\blacksquare$

\subsubsection{The case where $p>4$ and is even}\label{even}

\begin{proposition}
    $\Delta_1$ is a maximal green sequence when $p>4$ and even.
\end{proposition}

\textbf{Proof} When $p>4$ and is even, the mutation sequence $\Delta_1$ translates to 

\begin{enumerate}
    \item Tagging alternate punctures notched. $(\alpha_{p-2},\alpha_{p-4},...,\alpha_2)$
    \item Tagging the first interior puncture notched. $(h_{q+1},h_q,...,h_1,m_{1},h_2,h_3,...,h_{q+1})$
    \item Tagging the remaining interior punctures notched and moving half the inner arcs away from the corner. $(\beta_{p-3},\beta_{p-5},...,\beta_{1})$
    \item Moving arcs away from the corner.
    \begin{enumerate}
        \item Moving core arcs away from the corner. $(g_1,g_2,...,g_q,h_{q+1},h_{q},...,h_2,m_1)$
        \item Moving the remaining inner arcs away from the corner. $(l_{2},r_{2},l_{4},r_{4},...,l_{p-4},r_{p-4})$
        \item Moving outer arcs away from the corner. $(l_{p-1},f_1,r_{p-1},f_2,l_{p-1},f_1)$
    \end{enumerate}
    \item Tagging the corner notched. $(m_{p-2},l_{p-2},r_{p-2},m_{p-2})$
    \item Moving arcs back to the corner.
    \begin{enumerate}
        \item Moving outer arcs back to the corner. 
        $(f_1,l_{p-1},f_2,r_{p-1},f_1,l_{p-1})$
        \item Moving inner arcs back to the corner. 
        $(\delta_{p-4},\delta_{p-6},...,\delta_{2})$
        \item Moving core arcs back to the corner. $(m_1,h_2,h_3,...,h_{q+1},g_1,g_2,...,g_q)$
        \item Moving the remaining pair of inner arc back to the corner. $(l_1,r_1)$
    \end{enumerate}
\end{enumerate}

where

\begin{itemize}
    \item $\alpha_k=r_k,l_k,m_{k+1},m_{k},l_k,r_k$
    \item $\beta_k=\begin{cases}
        r_k,l_k,m_{k+2},m_{k-1}&\text{if }k>1,\\
        r_k,l_k,m_{k+2},h_1&\text{otherwise}
    \end{cases}$
    \item $\delta_k=r_k,r_{k+1},l_k,l_{k+1}$
\end{itemize}

The triangulation $T_{1,p,q}$ looks like this:

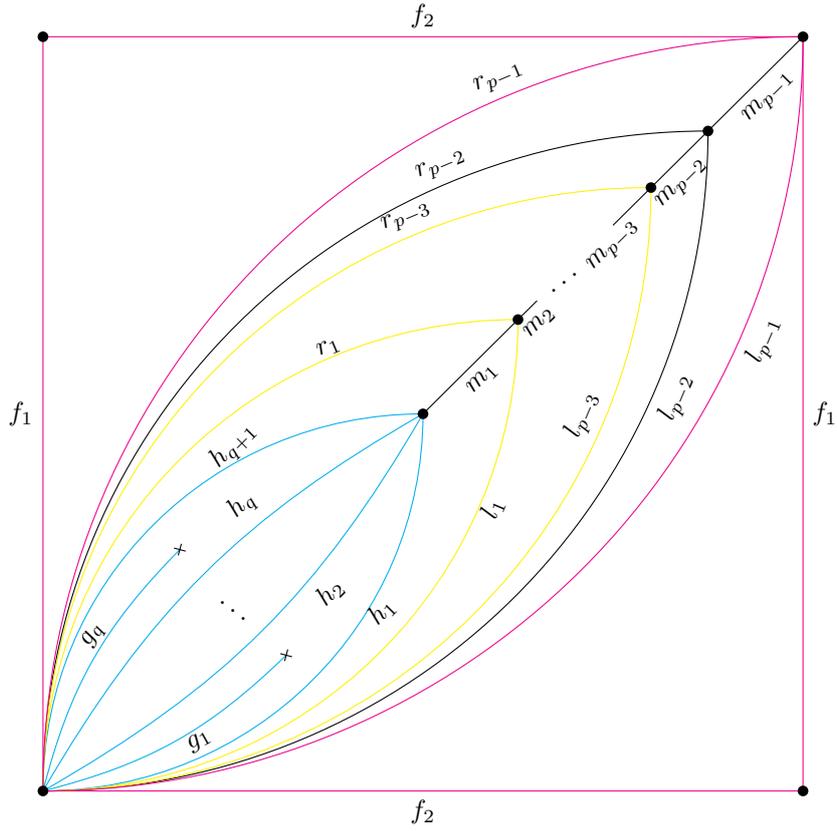
\[\begin{tikzpicture}[every edge quotes/.style={auto=right}]
    \begin{scope}[every node/.style={sloped,allow upside down}][every edge quotes/.style={auto=right}]

\node at (6.85,6.85) {$\iddots$};
        \node at ($(0,5)!.5!(5,0)$) {$\ddots$};
        \draw[cyan] (0,0) to [bend right=15,black] node [below=.15,right] {$g_1$} (3.2,1.8);
        \draw[cyan] (0,0) to [bend left=15,black] node [above=.15,right] {$g_q$} (1.8,3.2);
        \draw[magenta] (0,0) -- node[rotate=-90,left,black] {$f_1$} (0,10);
        \draw[magenta] (0,10) -- node[above,black] {$f_2$} (10,10);
        \draw[magenta] (10,0) -- node[rotate=-90,right,black] {$f_1$} (10,10);
        \draw[magenta] (0,0) -- node[below,black] {$f_2$} (10,0);
        \draw[cyan] (0,0) to [bend right=45,black] node[below=.15,right=1] {$h_1$} (5,5);
        \draw[cyan] (0,0) to [bend left=45,black] node[above=.15,right=1] {$h_{q+1}$} (5,5);
        \draw[cyan] (0,0) to [bend right=15,black] node[right=1,below] {$h_2$} (5,5);
        \draw[cyan] (0,0) to [bend left=15,black] node[right=1,above] {$h_q$} (5,5);
        \draw[yellow] (0,0) to [bend right=45,black] node[near end,right,below=-.1] {$l_1$} (6.25,6.25);
        \draw[yellow] (0,0) to [bend left=45] node[near end,right,above=-.1,black] {$r_1$} (6.25,6.25);
        \draw (0,0) to [bend right=45,black] node[near end,below] {$l_{p-2}$} (8.75,8.75);
        \draw (0,0) to [bend left=45,black] node[near end,above] {$r_{p-2}$} (8.75,8.75);
        \draw[yellow] (0,0) to [bend right=45,black] node[near end,above] {$l_{p-3}$} (8,8);
        \draw[yellow] (0,0) to [bend left=45,black] node[near end,above] {$r_{p-3}$} (8,8);
        \draw[magenta] (0,0) to [bend right=45,black] node[near end,below,black] {$l_{p-1}$} (10,10);
        \draw[magenta] (0,0) to [bend left=45,black] node[near end,above] {$r_{p-1}$} (10,10);
        \draw (5,5) to node[below] {$m_1$} (6.25,6.25);
        \draw (8.75,8.75) to node[below] {$m_{p-1}$} (10,10);
        \draw (6.25,6.25) to node[below] {$m_2$} (6.5,6.5);
        \draw (8,8) to node[left=.2,below] {$m_{p-2}$} (8.75,8.75);
        \draw (7.5,7.5) to node[below=.2,left] {$m_{p-3}$} (8,8);
        \fill (0,10) circle (2pt);
        \fill (10,0) circle (2pt);
        \fill (0,0) circle (2pt);
        \fill (10,10) circle (2pt);
        \fill (5,5) circle (2pt);
        \fill (3.2,1.8) node[cross=2pt,rotate=30] {};
        \fill (1.8,3.2) node[cross=2pt,rotate=30] {};
        \fill (6.25,6.25) circle (2pt);
        \fill (8.75,8.75) circle (2pt);
        \fill (8,8) circle (2pt);
    \end{scope}
\end{tikzpicture}\]

Step 1 consists of $(\alpha_{p-2},\alpha_{p-4},...,\alpha_2)$

We first look at the vertices involved in $\alpha_k$ and adjacent vertices.

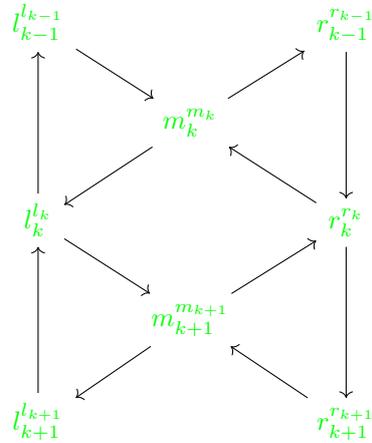
\[\begin{tikzcd}
    \color{green}l_{k-1}^{l_{k-1}}\arrow[dr]&&\color{green}r_{k-1}^{r_{k-1}}\arrow[dd]\\
    &\color{green}m_k^{m_k}\arrow[ur]\arrow[dl]&\\
    \color{green}l_{k}^{l_k}\arrow[uu]\arrow[dr]&&\color{green}r_{k}^{r_k}\arrow[ul]\arrow[dd]\\
    &\color{green}m_{k+1}^{m_{k+1}}\arrow[ur]\arrow[dl]&\\
    \color{green}l_{k+1}^{l_{k+1}}\arrow[uu]&&\color{green}r_{k+1}^{r_{k+1}}\arrow[ul]
\end{tikzcd}\]

By computation, mutating at $\alpha_k$ itself changes the subdiargam like this:

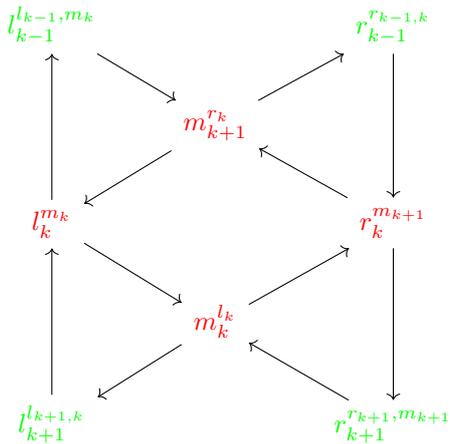
\[\begin{tikzcd}
    \color{green}l_{k-1}^{l_{k-1},m_k}\arrow[dr]&&\color{green}r_{k-1}^{r_{k-1,k}}\arrow[dd]\\
    &\color{red}m_{k+1}^{r_k}\arrow[ur]\arrow[dl]&\\
    \color{red}l_{k}^{m_k}\arrow[uu]\arrow[dr]&&\color{red}r_{k}^{m_{k+1}}\arrow[ul]\arrow[dd]\\
    &\color{red}m_k^{l_k}\arrow[ur]\arrow[dl]&\\
    \color{green}l_{k+1}^{l_{k+1,k}}\arrow[uu]&&\color{green}r_{k+1}^{r_{k+1},m_{k+1}}\arrow[ul]
\end{tikzcd}\]

As $l_{k+1},r_{k+1}$ are connected to additional frozen vertices because of the mutations in $\alpha_{k+2}$ (for $k<p-2$), and $l_{k-1},r_{k-1}$ are connected to additional frozen vertices because of the mutations in $\alpha_{k-2}$ (for $k>2$), after step 1, the above subdiagram looks like this for $2<k<p-2$ (for $k=2$ and $k=p-2$, ignore the change on $l_{k-1},r_{k-1}$ and $l_{k+1},r_{k+1}$ respectively):

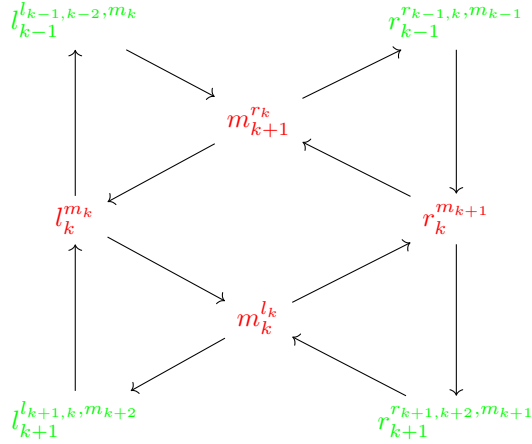
\[\begin{tikzcd}
    \color{green}l_{k-1}^{l_{k-1,k-2},m_k}\arrow[dr]&&\color{green}r_{k-1}^{r_{k-1,k},m_{k-1}}\arrow[dd]\\
    &\color{red}m_{k+1}^{r_k}\arrow[ur]\arrow[dl]&\\
    \color{red}l_{k}^{m_k}\arrow[uu]\arrow[dr]&&\color{red}r_{k}^{m_{k+1}}\arrow[ul]\arrow[dd]\\
    &\color{red}m_k^{l_k}\arrow[ur]\arrow[dl]&\\
    \color{green}l_{k+1}^{l_{k+1,k},m_{k+2}}\arrow[uu]&&\color{green}r_{k+1}^{r_{k+1,k+2},m_{k+1}}\arrow[ul]
\end{tikzcd}\]

In terms of triangulation, the flips that correspond to mutations in $\alpha_k$ tag the $k+1$-th interior puncture notched, just like in the case where $p=4$.

After step 1, the $p-1$-th, $p-3$-th,..., $3$-rd interior puncture are tagged notched.

Step 2 consists of $(h_{q+1},h_q,...,h_1,m_{1},h_2,h_3,...,h_{q+1})$. Look at them and adjacent vertices:

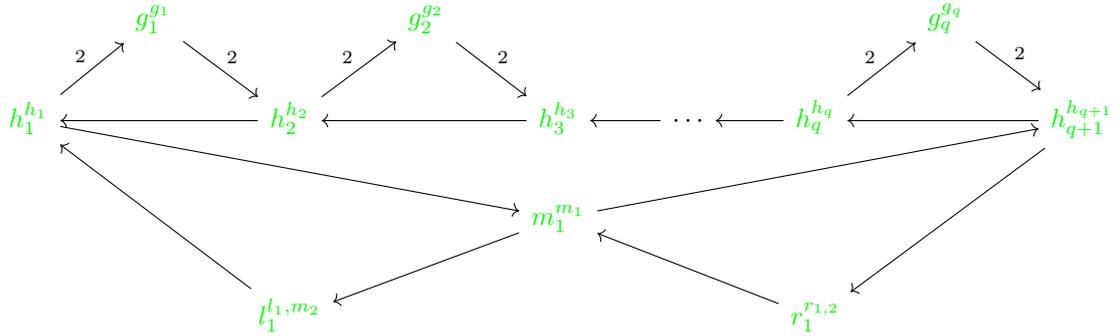
\[\begin{tikzcd}
    &|[color=green]|g_{1}^{g_1}\arrow[dr,"2"]&&|[color=green]|g_{2}^{g_2}\arrow[dr,"2"]&&&&|[color=green]|g_{q}^{g_q}\arrow[dr,"2"]\\
    |[color=green]|h_{1}^{h_1}\arrow[ur,"2"]\arrow[drrrr]&&|[color=green]|h_{2}^{h_2}\arrow[ur,"2"]\arrow[ll]&&|[color=green]|{h_{3}^{h_{3}}}\arrow[ll]&\cdots\arrow[l]&|[color=green]|h_{q}^{h_q}\arrow[ur,"2"]\arrow[l]&&|[color=green]|{h_{q+1}^{h_{q+1}}}\arrow[ll]\arrow[ddll]\\
    &&&&|[color=green]|m_1^{m_1}\arrow[urrrr]\arrow[dll]\\
    &&|[color=green]|l_1^{l_1,m_2}\arrow[uull]&&&&|[color=green]|r_1^{r_{1,2}}\arrow[ull]
\end{tikzcd}\]

By applying Lemma \ref{p>2}, we have:

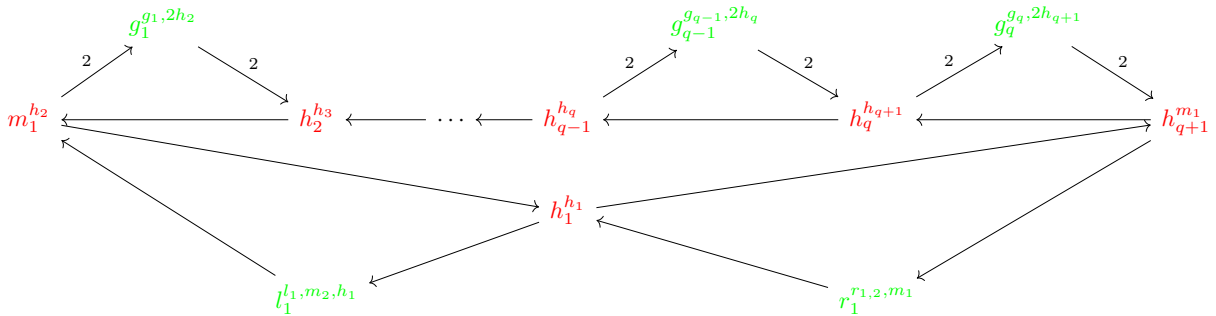
\[\adjustbox{scale=0.9,center}{
\begin{tikzcd}
    &\color{green}g_1^{g_1,2h_2}\arrow[dr,"2"]&&&&\color{green}g_{q-1}^{g_{q-1},2h_q}\arrow[dr,"2"]&&\color{green}g_q^{g_q,2h_{q+1}}\arrow[dr,"2"]\\
    \color{red}m_1^{h_2}\arrow[drrrr]\arrow[ur,"2"]&&\color{red}h_2^{h_3}\arrow[ll]&\cdots\arrow[l]&\color{red}h_{q-1}^{h_q}\arrow[ur,"2"]\arrow[l]&&\color{red}h_{q}^{h_{q+1}}\arrow[ll]\arrow[ur,"2"]&&\color{red}h_{q+1}^{m_1}\arrow[ll]\arrow[ddll]\\
    &&&&\color{red}h_1^{h_1}\arrow[urrrr]\arrow[dll]\\
    &&\color{green}l_1^{l_1,m_2,h_1}\arrow[uull]&&&&\color{green}r_1^{r_{1,2},m_1}\arrow[ull]
\end{tikzcd}}\]

In terms of triangulation, we have tagged the first interior puncture notched.

Step 3 consists of $\beta_{p-3},\beta_{p-5},...\beta_{1}$.

Now look at the vertices involved in $\beta_k$ for $k>1$ and adjacent vertices. Note that as a result of the $\alpha_k$'s, $m_{k-1}$ and $m_{k+2}$ (instead of $m_k$ and $m_{k+1}$) are adjacent to $l_k$ and $r_k$.

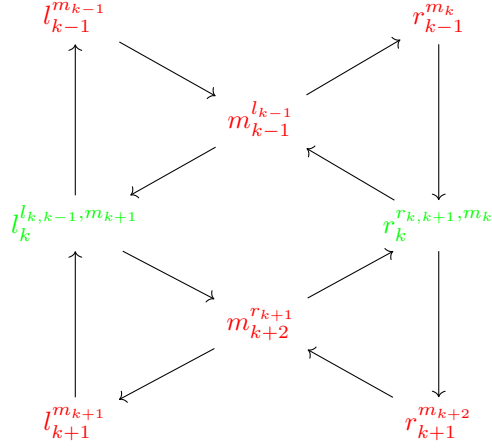
\[\begin{tikzcd}
    \color{red}l_{k-1}^{m_{k-1}}\arrow[dr]&&\color{red}r_{k-1}^{m_k}\arrow[dd]\\
    &\color{red}m_{k-1}^{l_{k-1}}\arrow[ur]\arrow[dl]&\\
    \color{green}l_{k}^{l_{k,k-1},m_{k+1}}\arrow[uu]\arrow[dr]&&\color{green}r_{k}^{r_{k,k+1},m_k}\arrow[ul]\arrow[dd]\\
    &\color{red}m_{k+2}^{r_{k+1}}\arrow[ur]\arrow[dl]&\\
    \color{red}l_{k+1}^{m_{k+1}}\arrow[uu]&&\color{red}r_{k+1}^{m_{k+2}}\arrow[ul]
\end{tikzcd}\]

By computation, mutating at $\beta_k$ itself gives:

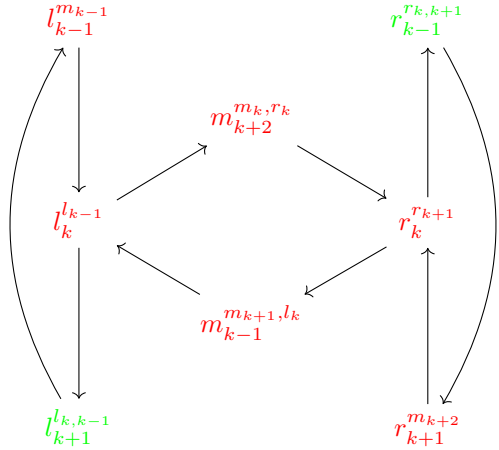
\[\begin{tikzcd}
    \color{red}l_{k-1}^{m_{k-1}}\arrow[dd]&&\color{green}r_{k-1}^{r_{k,k+1}}\arrow[dddd,bend left]\\
    &\color{red}m_{k+2}^{m_k,r_k}\arrow[dr]&\\
    \color{red}l_{k}^{l_{k-1}
    }\arrow[dd]\arrow[ur]&&\color{red}r_{k}^{r_{k+1}}\arrow[dl]\arrow[uu]\\
    &\color{red}m_{k-1}^{m_{k+1},l_k}\arrow[ul]&\\
    \color{green}l_{k+1}^{l_{k,k-1}}\arrow[uuuu,bend left]&&\color{red}r_{k+1}^{m_{k+2}}\arrow[uu]
\end{tikzcd}\]

Again, as $r_{k+1}$ is connected to additional frozen vertices because of the mutations in $\beta_{k+2}$ (for $k<p-3$), and $l_{k-1}$ is connected to additional frozen vertices because of the mutations in $\beta_{k-2}$, after step 1, the above subdiagram looks like this for $3<k<p-3$ (for $k=p-3$, ignore the change on $r_{k+1}$, for $k=3$, ignore the change on $l_{k-1}$):

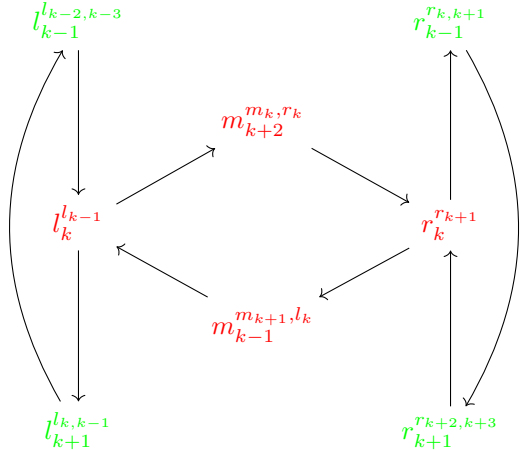
\[\begin{tikzcd}
    \color{green}l_{k-1}^{l_{k-2,k-3}}\arrow[dd]&&\color{green}r_{k-1}^{r_{k,k+1}}\arrow[dddd,bend left]\\
    &\color{red}m_{k+2}^{m_k,r_k}\arrow[dr]&\\
    \color{red}l_{k}^{l_{k-1}
    }\arrow[dd]\arrow[ur]&&\color{red}r_{k}^{r_{k+1}}\arrow[dl]\arrow[uu]\\
    &\color{red}m_{k-1}^{m_{k+1},l_k}\arrow[ul]&\\
    \color{green}l_{k+1}^{l_{k,k-1}}\arrow[uuuu,bend left]&&\color{green}r_{k+1}^{r_{k+2,k+3}}\arrow[uu]
\end{tikzcd}\]

In terms of triangulation, $\beta_k$ tags the $k+1$-th interior puncture notched and encloses this puncture with $l_k$ and $r_k$ because unlike in $\alpha_k$ there is no $l_k,r_k$ after $m_{k-1}$

After mutating at $\beta_3$, at the vertices involved in $\beta_1$ and adjacent vertices, we have:

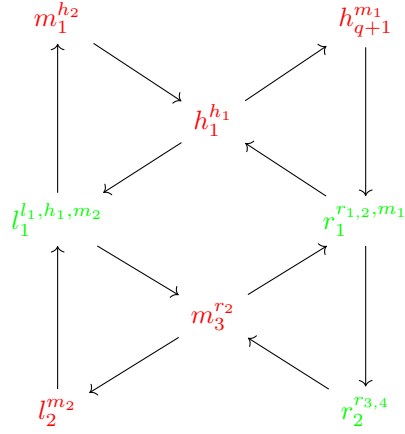
\[\begin{tikzcd}
    \color{red}m_1^{h_2}\arrow[dr]&&\color{red}h_{q+1}^{m_1}\arrow[dd]\\
    &\color{red}h_{1}^{h_1}\arrow[ur]\arrow[dl]&\\
    \color{green}l_{1}^{l_1,h_1,m_{2}}\arrow[uu]\arrow[dr]&&\color{green}r_{1}^{r_{1,2},m_1}\arrow[ul]\arrow[dd]\\
    &\color{red}m_{3}^{r_{2}}\arrow[ur]\arrow[dl]&\\
    \color{red}l_{2}^{m_{2}}\arrow[uu]&&\color{green}r_{2}^{r_{3,4}}\arrow[ul]
\end{tikzcd}\]

After mutating at $\beta_1$, we have:

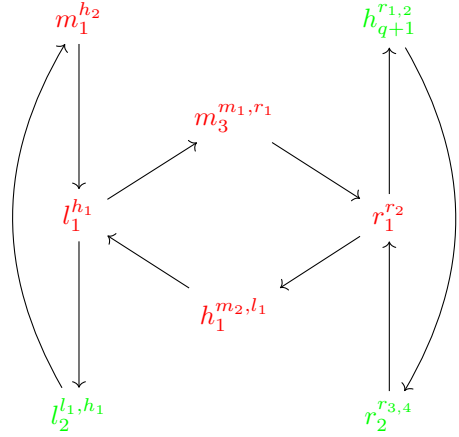
\[\begin{tikzcd}
    \color{red}m_1^{h_2}\arrow[dd]&&\color{green}h_{q+1}^{r_{1,2}}\arrow[dddd,bend left]\\
    &\color{red}m_{3}^{m_1,r_1}\arrow[dr]&\\
    \color{red}l_{1}^{h_1}\arrow[dd]\arrow[ur]&&\color{red}r_{1}^{r_2}\arrow[dl]\arrow[uu]\\
    &\color{red}h_1^{m_2,l_1}\arrow[ul]&\\
    \color{green}l_{2}^{l_1,h_1}\arrow[uuuu,bend left]&&\color{green}r_{2}^{r_{3,4}}\arrow[uu]
\end{tikzcd}\]

In terms of triangulation, this is same as other $\beta_k$'s with the $2$-nd interior puncture tagged notched.

After step 3, the $p-2$-th, $p-4$-th,..., $2$-nd interior puncture are tagged notched. Together with step 1 and step 2, the remaining puncture not yet tagged notched is the corner. Also the arcs $l_k$ and $r_k$ are moved away from the corner for $k=1,3,...,p-3$.

Step 4(a) consists of $(g_1,g_2,...,g_q,h_{q+1},h_{q},...,h_2,m_1)$.

Mutating at $g_1$ does the following:

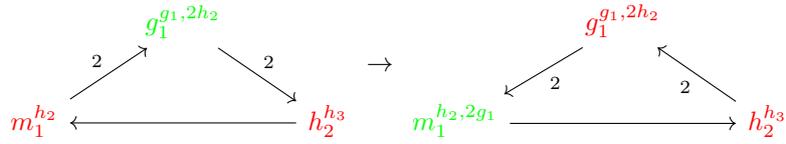
\[
\begin{tikzcd}
    &\color{green}g_{1}^{g_1,2h_2}\arrow[dr,"2"]&\\
    \color{red}m_1^{h_2}\arrow[ur,"2"]&&\color{red}h_2^{h_3}\arrow[ll]
\end{tikzcd}
\rightarrow
\begin{tikzcd}
    &\color{red}g_{1}^{g_1,2h_2}\arrow[dl,"2"]&\\
    \color{green}m_1^{h_2,2g_1}\arrow[rr]&&\color{red}h_2^{h_3}\arrow[ul,"2"]
\end{tikzcd}
\]

Observe that for all $g_k$, $k=1,2,...,q$, the full subdiagram containing $g_k$ and adjancent vertices are isomorphic (again except the frozen vertices attached), so after mutating at $g_1,g_2,...,g_q$, we have:

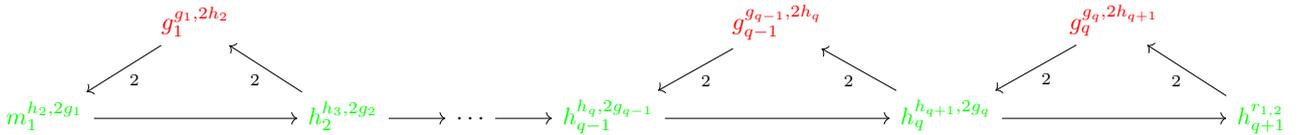
\[\adjustbox{scale=0.9,center}{
\begin{tikzcd}
    &\color{red}g_1^{g_1,2h_2}\arrow[dl,"2"]&&&&\color{red}g_{q-1}^{g_{q-1},2h_q}\arrow[dl,"2"]&&\color{red}g_q^{g_q,2h_{q+1}}\arrow[dl,"2"]\\
    \color{green}m_1^{h_2,2g_1}\arrow[rr]&&\color{green}h_2^{h_3,2g_2}\arrow[r]\arrow[ul,"2"]&\cdots\arrow[r]&\color{green}h_{q-1}^{h_q,2g_{q-1}}\arrow[rr]&&\color{green}h_{q}^{h_{q+1},2g_q}\arrow[rr]\arrow[ul,"2"]&&\color{green}h_{q+1}^{r_{1,2}}\arrow[ul,"2"]
\end{tikzcd}}\]

In terms of triangulation, all pending arcs are now connected to the first interior puncture (like in the case where $p=2,q>1$).

Now look at $h_{q+1}$, $h_{q}$, $h_{q-1}$, and adjacent vertices:

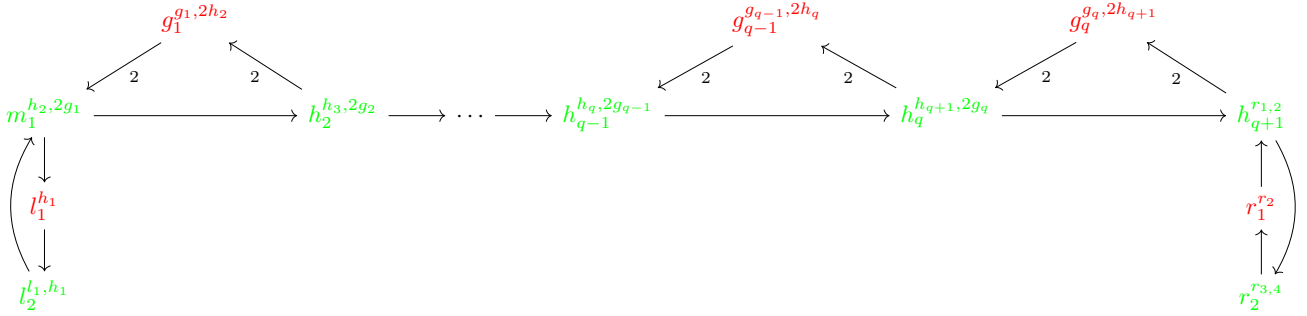
\[\adjustbox{scale=0.9,center}{
\begin{tikzcd}
    &\color{red}g_1^{g_1,2h_2}\arrow[dl,"2"]&&&&\color{red}g_{q-1}^{g_{q-1},2h_q}\arrow[dl,"2"]&&\color{red}g_q^{g_q,2h_{q+1}}\arrow[dl,"2"]\\
    \color{green}m_1^{h_2,2g_1}\arrow[rr]\arrow[d]&&\color{green}h_2^{h_3,2g_2}\arrow[r]\arrow[ul,"2"]&\cdots\arrow[r]&\color{green}h_{q-1}^{h_q,2g_{q-1}}\arrow[rr]&&\color{green}h_{q}^{h_{q+1},2g_q}\arrow[rr]\arrow[ul,"2"]&&\color{green}h_{q+1}^{r_{1,2}}\arrow[ul,"2"]\arrow[dd,bend left]
    \\\color{red}l_1^{h_1}\arrow[d]&&&&&&&&\color{red}r_1^{r_2}\arrow[u]
    \\\color{green}l_2^{l_1,h_1}\arrow[uu,bend left]&&&&&&&&\color{green}r_2^{r_{3,4}}\arrow[u]
\end{tikzcd}}\]

By applying Lemma \ref{p>4}, after step 4(a), we have:

$\adjustbox{scale=0.8,center}
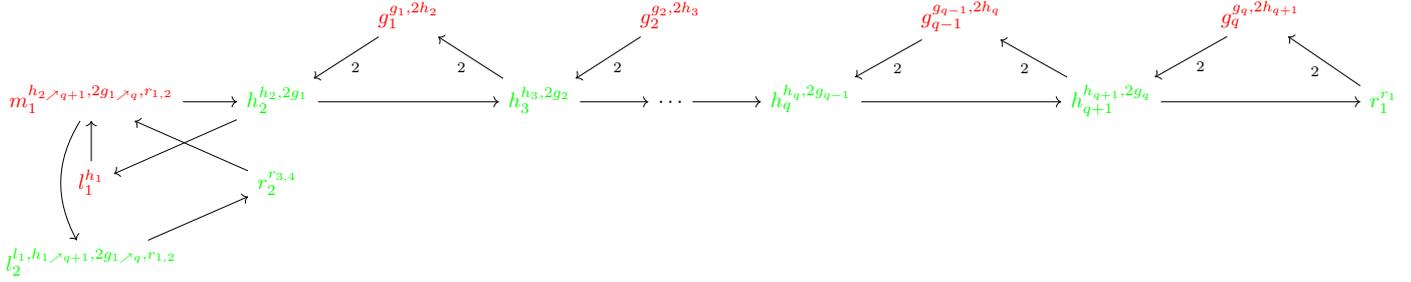
{
\begin{tikzcd}
    &&\color{red}g_1^{g_1,2h_2}\arrow[dl,"2"]&&\color{red}g_2^{g_2,2h_3}\arrow[dl,"2"]&&\color{red}g_{q-1}^{g_{q-1},2h_q}\arrow[dl,"2"]&&\color{red}g_q^{g_q,2h_{q+1}}\arrow[dl,"2"]\\
    \color{red}m_1^{h_{2\nearrow q+1},2g_{1\nearrow q},r_{1,2}}\arrow[dd,bend right]\arrow[r]&\color{green}h_{2}^{h_{2},2g_{1}}\arrow[dl]\arrow[rr]&&\color{green}h_3^{h_3,2g_2}\arrow[ul,"2"]\arrow[r]&\cdots\arrow[r]&\color{green}h_{q}^{h_q,2g_{q-1}}\arrow[rr]&&\color{green}h_{q+1}^{h_{q+1},2g_q}\arrow[ul,"2"]\arrow[rr]&&\color{green}r_1^{r_1}\arrow[ul,"2"]\\
    \color{red}l_1^{h_1}\arrow[u]&\color{green}r_2^{r_{3,4}}\arrow[ul]\\
    \color{green}l_2^{l_1,h_{1\nearrow q+1},2g_{1\nearrow q},r_{1,2}}\arrow[ur]
\end{tikzcd}}$

In terms of triangulation, we have wrapped the arcs $g_1,g_2,...,g_q,h_2,h_3,...,h_q,l_1,r_1$ inside $m_1$ (and thus removed all core arcs away from the corner). We will wrap this with additional arcs successively later.

At the vertices $l_2,r_2,l_4,r_4,...,l_{p-4},r_{p-4}$ and adjacent vertices, the diagram looks like this (for simplicity, write $h_{1\nearrow q+1},2g_{1\nearrow q}$ as $G$)

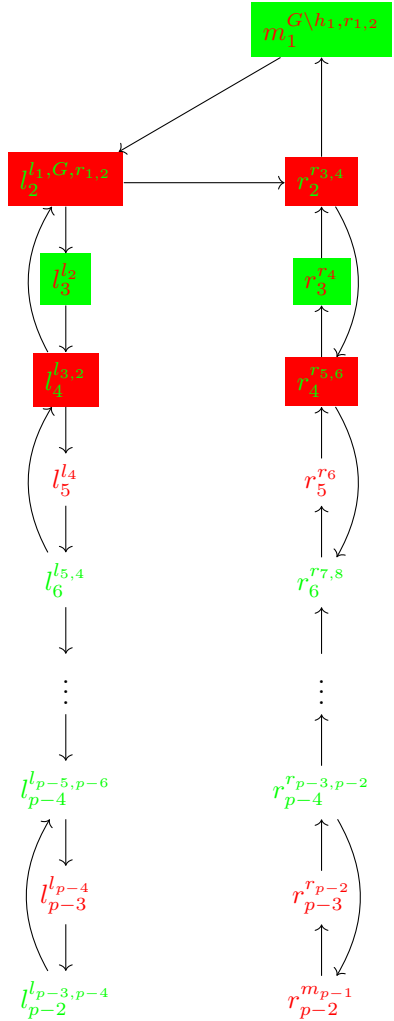
\[\begin{tikzcd}
    &&|[color=red,fill=green]|m_1^{G\setminus h_1,r_{1,2}}\arrow[ddll]\\\\
    |[color=green,fill=red]|l_2^{l_1,G,r_{1,2}}\arrow[d]\arrow[rr]&&|[color=green,fill=red]|r_2^{r_{3,4}}\arrow[uu]\arrow[dd,bend left]\\
    |[color=red,fill=green]|l_3^{l_2}\arrow[d]&&|[color=red,fill=green]|r_3^{r_4}\arrow[u]\\
    |[color=green,fill=red]|l_4^{l_{3,2}}\arrow[d]\arrow[uu,bend left]&&|[color=green,fill=red]|r_4^{r_{5,6}}\arrow[u]\arrow[dd,bend left]\\
    \color{red}l_5^{l_4}\arrow[d]&&\color{red}r_5^{r_6}\arrow[u]\\
    \color{green}l_6^{l_{5,4}}\arrow[uu,bend left]\arrow[d]&&\color{green}r_6^{r_{7,8}}\arrow[u]\\
    \vdots\arrow[d]&&\vdots\arrow[u]\\
    \color{green}l_{p-4}^{l_{p-5,p-6}}\arrow[d]&&\color{green}r_{p-4}^{r_{p-3,p-2}}\arrow[u]\arrow[dd,bend left]\\
    \color{red}l_{p-3}^{l_{p-4}}\arrow[d]&&\color{red}r_{p-3}^{r_{p-2}}\arrow[u]\\
    \color{green}l_{p-2}^{l_{p-3,p-4}}\arrow[uu,bend left]&&\color{red}r_{p-2}^{m_{p-1}}\arrow[u]
\end{tikzcd}\]

After mutating at $l_2,r_2$, we have:

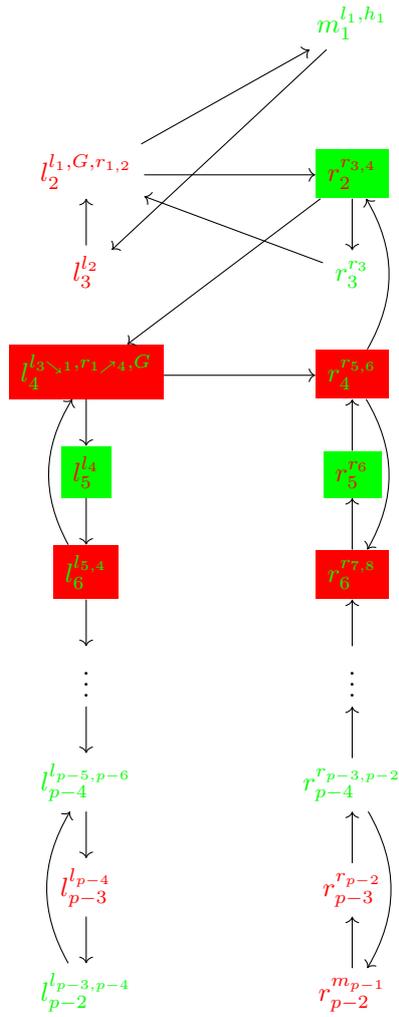
\[\begin{tikzcd}
    &&\color{green}m_1^{l_1,h_1}\arrow[dddll]\\\\
    \color{red}l_2^{l_1,G,r_{1,2}}\arrow[uurr]\arrow[rr]&&|[color=red,fill=green]|r_2^{r_{3,4}}\arrow[d]\arrow[ddll]\\
    \color{red}l_3^{l_2}\arrow[u]&&\color{green}r_3^{r_3}\arrow[ull]\\
    |[color=green,fill=red]|l_4^{l_{3\searrow1},r_{1\nearrow4},G}\arrow[d]\arrow[rr]&&|[color=green,fill=red]|r_4^{r_{5,6}}\arrow[uu,bend right]\arrow[dd,bend left]\\
    |[color=red,fill=green]|l_5^{l_4}\arrow[d]&&|[color=red,fill=green]|r_5^{r_6}\arrow[u]\\
    |[color=green,fill=red]|l_6^{l_{5,4}}\arrow[uu,bend left]\arrow[d]&&|[color=green,fill=red]|r_6^{r_{7,8}}\arrow[u]\\
    \vdots\arrow[d]&&\vdots\arrow[u]\\
    \color{green}l_{p-4}^{l_{p-5,p-6}}\arrow[d]&&\color{green}r_{p-4}^{r_{p-3,p-2}}\arrow[u]\arrow[dd,bend left]\\
    \color{red}l_{p-3}^{l_{p-4}}\arrow[d]&&\color{red}r_{p-3}^{r_{p-2}}\arrow[u]\\
    \color{green}l_{p-2}^{l_{p-3,p-4}}\arrow[uu,bend left]&&\color{red}r_{p-2}^{m_{p-1}}\arrow[u]
\end{tikzcd}\]

After mutating at $l_4,r_4$, we have:

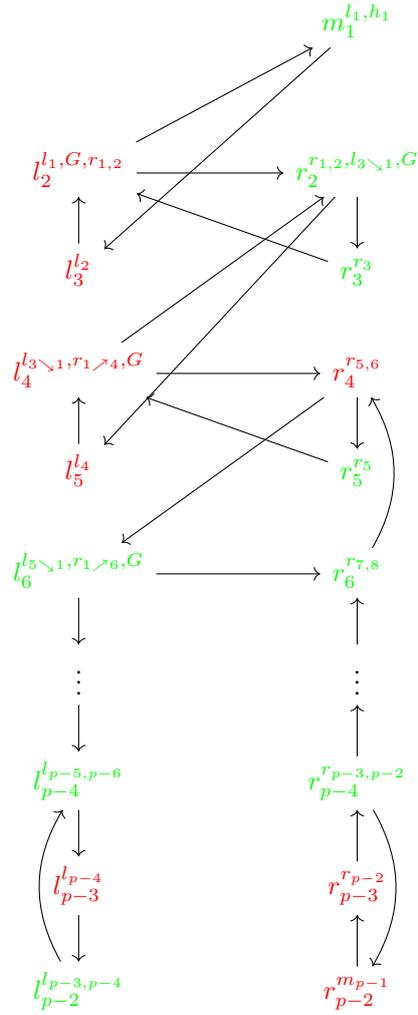
\[\begin{tikzcd}
    &&\color{green}m_1^{l_1,h_1}\arrow[dddll]\\\\
    \color{red}l_2^{l_1,G,r_{1,2}}\arrow[uurr]\arrow[rr]&&\color{green}r_2^{r_{1,2},l_{3\searrow1},G}\arrow[d]\arrow[dddll]\\
    \color{red}l_3^{l_2}\arrow[u]&&\color{green}r_3^{r_3}\arrow[ull]\\
    \color{red}l_4^{l_{3\searrow1},r_{1\nearrow4},G}\arrow[uurr]\arrow[rr]&&\color{red}r_4^{r_{5,6}}\arrow[d]\arrow[ddll]\\
    \color{red}l_5^{l_4}\arrow[u]&&\color{green}r_5^{r_5}\arrow[ull]\\
    \color{green}l_6^{l_{5\searrow1},r_{1\nearrow6},G}\arrow[rr]\arrow[d]&&\color{green}r_6^{r_{7,8}}\arrow[uu,bend right]\\
    \vdots\arrow[d]&&\vdots\arrow[u]\\
    \color{green}l_{p-4}^{l_{p-5,p-6}}\arrow[d]&&\color{green}r_{p-4}^{r_{p-3,p-2}}\arrow[u]\arrow[dd,bend left]\\
    \color{red}l_{p-3}^{l_{p-4}}\arrow[d]&&\color{red}r_{p-3}^{r_{p-2}}\arrow[u]\\
    \color{green}l_{p-2}^{l_{p-3,p-4}}\arrow[uu,bend left]&&\color{red}r_{p-2}^{m_{p-1}}\arrow[u]
\end{tikzcd}\]

Observe the pattern at the highlighted vertices (and also the frozen vertices connected to $r_{k-2}$ after $l_k,r_k$), we can argue by induction that after mutating at $l_{p-4},r_{p-4}$, we have:

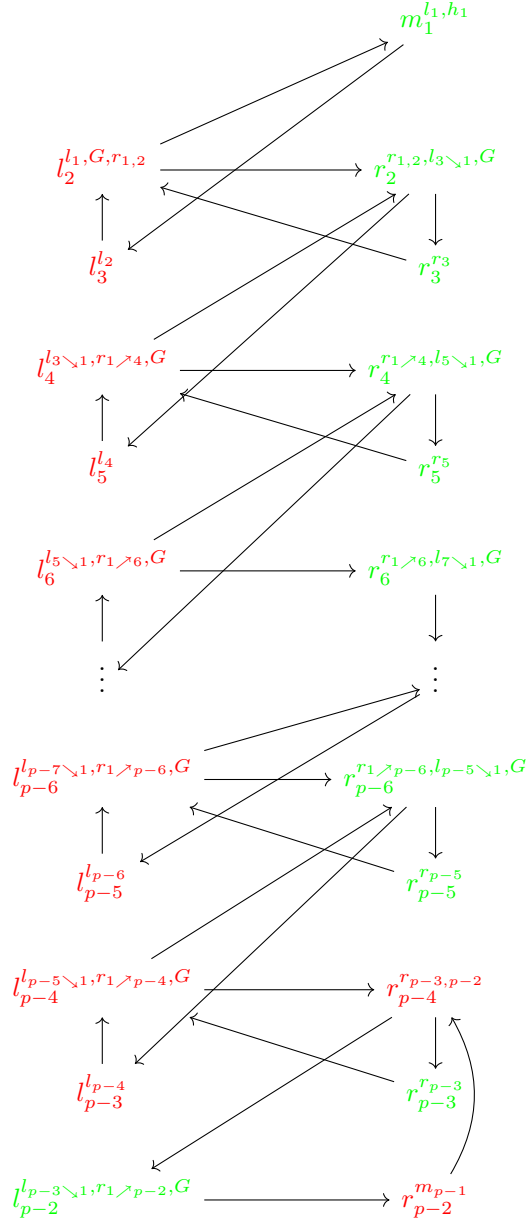
\[\begin{tikzcd}
    &&\color{green}m_1^{l_1,h_1}\arrow[dddll]\\\\
    \color{red}l_2^{l_1,G,r_{1,2}}\arrow[uurr]\arrow[rr]&&\color{green}r_2^{r_{1,2},l_{3\searrow1},G}\arrow[d]\arrow[dddll]\\
    \color{red}l_3^{l_2}\arrow[u]&&\color{green}r_3^{r_3}\arrow[ull]\\
    \color{red}l_4^{l_{3\searrow1},r_{1\nearrow4},G}\arrow[uurr]\arrow[rr]&&\color{green}r_4^{r_{1\nearrow4},l_{5\searrow1},G}\arrow[d]\arrow[dddll]\\
    \color{red}l_5^{l_4}\arrow[u]&&\color{green}r_5^{r_5}\arrow[ull]\\
    \color{red}l_6^{l_{5\searrow1},r_{1\nearrow6},G}\arrow[uurr]\arrow[rr]&&\color{green}r_6^{r_{1\nearrow6},l_{7\searrow1},G}\arrow[d]\\
    \vdots\arrow[u]&&\vdots\arrow[ddll]\\
    \color{red}l_{p-6}^{l_{p-7\searrow1},r_{1\nearrow p-6},G}\arrow[urr]\arrow[rr]&&\color{green}r_{p-6}^{r_{1\nearrow p-6},l_{p-5\searrow1},G}\arrow[d]\arrow[dddll]\\
    \color{red}l_{p-5}^{l_{p-6}}\arrow[u]&&\color{green}r_{p-5}^{r_{p-5}}\arrow[ull]\\
    \color{red}l_{p-4}^{l_{p-5\searrow1},r_{1\nearrow p-4},G}\arrow[rr]\arrow[uurr]&&\color{red}r_{p-4}^{r_{p-3,p-2}}\arrow[d]\arrow[ddll]\\
    \color{red}l_{p-3}^{l_{p-4}}\arrow[u]&&\color{green}r_{p-3}^{r_{p-3}}\arrow[ull]\\
    \color{green}l_{p-2}^{l_{p-3\searrow1},r_{1\nearrow p-2},G}\arrow[rr]&&\color{red}r_{p-2}^{m_{p-1}}\arrow[uu,bend right]
\end{tikzcd}\]

In terms of triangulation, the arcs shown except $l_{p-2}$ and $r_{p-2}$ are wrapped in $r_{p-4}$.

We shall merge steps 4(c), 5, and 6(a) so that we can apply Lemma \ref{corner}.

At $l_{p-1},r_{p-1},f_1,f_2,m_{p-2},l_{p-2},r_{p-2}$ and adjacent vertices, we have:

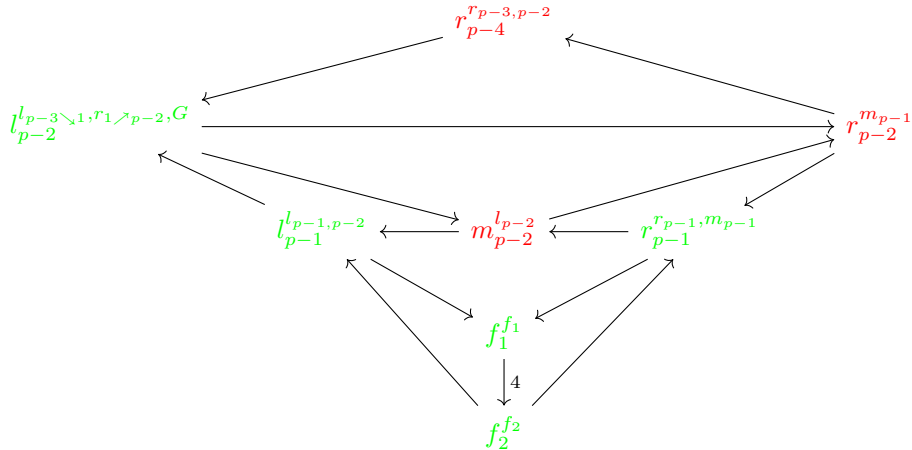
\[\begin{tikzcd}
    &&\color{red}r_{p-4}^{r_{p-3,p-2}}\arrow[dll]&&\\
    \color{green}l_{p-2}^{l_{p-3\searrow1},r_{1\nearrow p-2},G}\arrow[rrrr]\arrow[drr]&&&&\color{red}r_{p-2}^{m_{p-1}}\arrow[dl]\arrow[ull]\\
    &\color{green}l_{p-1}^{l_{p-1,p-2}}\arrow[ul]\arrow[dr]&\color{red}m_{p-2}^{l_{p-2}}\arrow[l]\arrow[urr]&\color{green}r_{p-1}^{r_{p-1},m_{p-1}}\arrow[l]\arrow[dl]\\
    &&\color{green}f_1^{f_1}\arrow[d,"4"]&\\
    &&\color{green}f_2^{f_2}\arrow[uul]\arrow[uur]&
\end{tikzcd}\]

By applying Lemma \ref{corner}, after mutating at $l_{p-1},f_1,r_{p-1},f_2,l_{p-1},f_1,m_{p-2},l_{p-2},r_{p-2},m_{p-2},f_1,l_{p-1},f_2,r_{p-1},f_1,l_{p-1}$, we have:

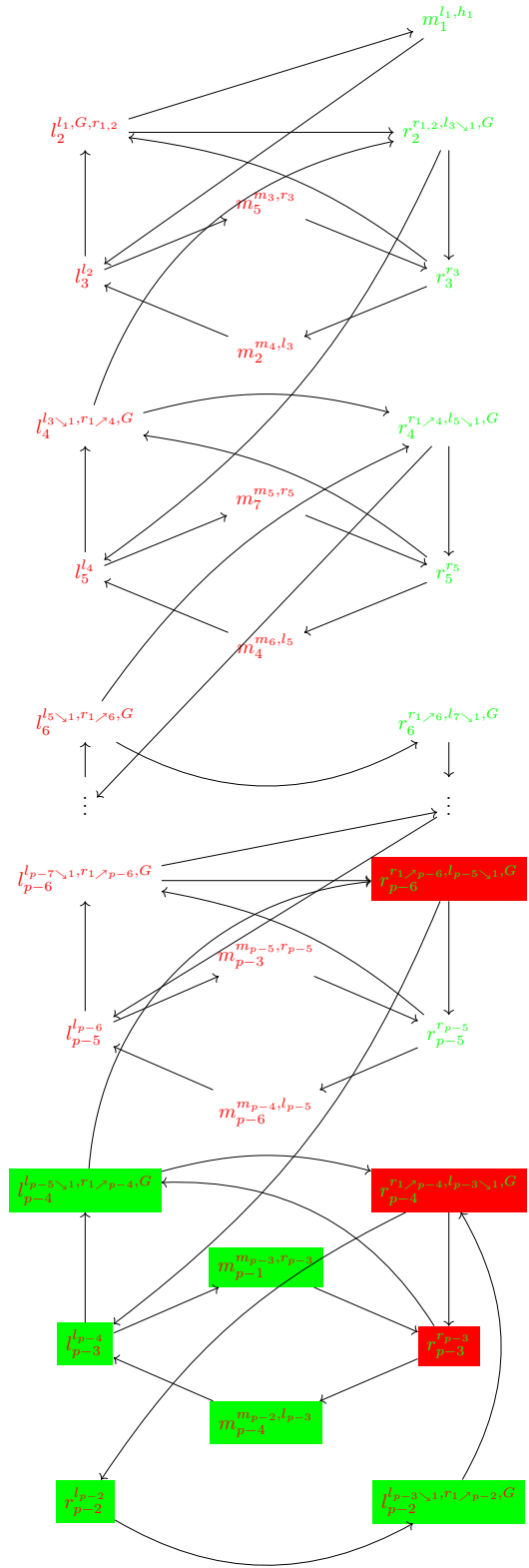
\[\begin{tikzcd}
    &&\color{green}r_{p-4}^{r_{1\nearrow p-4},l_{p-3\searrow1},G}\arrow[dll]&&\\
    \color{red}r_{p-2}^{l_{p-2}}\arrow[rrrr]\arrow[drr]&&&&\color{red}l_{p-2}^{l_{p-3\searrow1},r_{1\nearrow p-2},G}\arrow[dl]\arrow[ull]\\
    &\color{red}l_{p-1}^{l_{p-1}}\arrow[ul]\arrow[dr]&\color{red}m_{p-2}^{m_{p-1}}\arrow[l]\arrow[urr]&\color{red}r_{p-1}^{r_{p-1}}\arrow[l]\arrow[dl]\\
    &&\color{red}f_1^{f_1}\arrow[d,"4"]&\\
    &&\color{red}f_2^{f_2}\arrow[uul]\arrow[uur]&
    \vspace*{-0.1\linewidth}
\end{tikzcd}\]

Step 6(b) consists of $(\delta_{p-4},\delta_{p-6},...,\delta_{2})$.

Now look at vertices involved in $\delta_k$ for $k=2,4,...,p-4$ and adjacent vertices:

\[\adjustbox{center,scale=0.75}{

\begin{tikzcd}
    &&\color{green}m_1^{l_1,h_1}\arrow[ddddll]\\\\
    \color{red}l_2^{l_1,G,r_{1,2}}\arrow[uurr]\arrow[rr]&&\color{green}r_2^{r_{1,2},l_{3\searrow1},G}\arrow[dd]\arrow[ddddddll,bend left=15]\\
    &\color{red}m_{5}^{m_3,r_{3}}\arrow[dr]&\\
    \color{red}l_3^{l_2}\arrow[uu]\arrow[ur]&&\color{green}r_3^{r_3}\arrow[uull,bend right=15]\arrow[dl]\\
    &\color{red}m_{2}^{m_4,l_{3}}\arrow[ul]&\\
    \color{red}l_4^{l_{3\searrow1},r_{1\nearrow4},G}\arrow[uuuurr,bend left]\arrow[rr,bend left=15]&&\color{green}r_4^{r_{1\nearrow4},l_{5\searrow1},G}\arrow[dd]\arrow[dddddll]\\
    &\color{red}m_{7}^{m_5,r_{5}}\arrow[dr]&\\
    \color{red}l_5^{l_4}\arrow[uu]\arrow[ur]&&\color{green}r_5^{r_5}\arrow[uull,bend right=15]\arrow[dl]\\
    &\color{red}m_{4}^{m_6,l_{5}}\arrow[ul]&\\
    \color{red}l_6^{l_{5\searrow1},r_{1\nearrow6},G}\arrow[uuuurr,bend left=15]\arrow[rr,bend right]&&\color{green}r_6^{r_{1\nearrow6},l_{7\searrow1},G}\arrow[d]\\
    \vdots\arrow[u]&&\vdots\arrow[dddll]\\
    \color{red}l_{p-6}^{l_{p-7\searrow1},r_{1\nearrow p-6},G}\arrow[urr]\arrow[rr]&&|[color=green,fill=red]|r_{p-6}^{r_{1\nearrow p-6},l_{p-5\searrow1},G}\arrow[dd]\arrow[ddddddll,bend left=15]\\
    &\color{red}m_{p-3}^{m_{p-5},r_{p-5}}\arrow[dr]&\\
    \color{red}l_{p-5}^{l_{p-6}}\arrow[uu]\arrow[ur]&&\color{green}r_{p-5}^{r_{p-5}}\arrow[uull,bend right=15]\arrow[dl]\\
    &\color{red}m_{p-6}^{m_{p-4},l_{p-5}}\arrow[ul]&\\
    |[color=red,fill=green]|l_{p-4}^{l_{p-5\searrow1},r_{1\nearrow p-4},G}\arrow[rr,bend left=15]\arrow[uuuurr,bend left=40]&&|[color=green,fill=red]|r_{p-4}^{r_{1\nearrow p-4},l_{p-3\searrow1},G}\arrow[dd]\arrow[ddddll,bend right=15]\\
    &|[color=red,fill=green]|m_{p-1}^{m_{p-3},r_{p-3}}\arrow[dr]&\\
    |[color=red,fill=green]|l_{p-3}^{l_{p-4}}\arrow[uu]\arrow[ur]&&|[color=green,fill=red]|r_{p-3}^{r_{p-3}}\arrow[uull,bend right]\arrow[dl]\\
    &|[color=red,fill=green]|m_{p-4}^{m_{p-2},l_{p-3}}\arrow[ul]&\\
     |[color=red,fill=green]|r_{p-2}^{l_{p-2}}\arrow[rr,bend right]&&|[color=red,fill=green]|l_{p-2}^{l_{p-3\searrow1},r_{1\nearrow p-2},G}\arrow[uuuu,bend right]
\end{tikzcd}}\]

After mutating at $\delta_{p-4}$, we have:

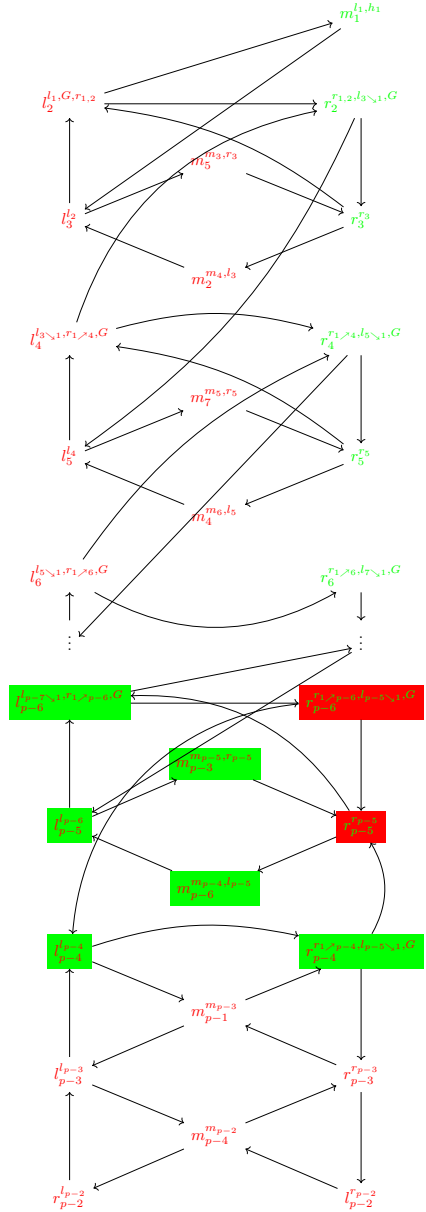
\[\adjustbox{center,scale=0.6}{\begin{tikzcd}
    &&\color{green}m_1^{l_1,h_1}\arrow[ddddll]\\\\
    \color{red}l_2^{l_1,G,r_{1,2}}\arrow[uurr]\arrow[rr]&&\color{green}r_2^{r_{1,2},l_{3\searrow1},G}\arrow[dd]\arrow[ddddddll,bend left=15]\\
    &\color{red}m_{5}^{m_3,r_{3}}\arrow[dr]&\\
    \color{red}l_3^{l_2}\arrow[uu]\arrow[ur]&&\color{green}r_3^{r_3}\arrow[uull,bend right=15]\arrow[dl]\\
    &\color{red}m_{2}^{m_4,l_{3}}\arrow[ul]&\\
    \color{red}l_4^{l_{3\searrow1},r_{1\nearrow4},G}\arrow[uuuurr,bend left]\arrow[rr,bend left=15]&&\color{green}r_4^{r_{1\nearrow4},l_{5\searrow1},G}\arrow[dd]\arrow[dddddll]\\
    &\color{red}m_{7}^{m_5,r_{5}}\arrow[dr]&\\
    \color{red}l_5^{l_4}\arrow[uu]\arrow[ur]&&\color{green}r_5^{r_5}\arrow[uull,bend right=15]\arrow[dl]\\
    &\color{red}m_{4}^{m_6,l_{5}}\arrow[ul]&\\
    \color{red}l_6^{l_{5\searrow1},r_{1\nearrow6},G}\arrow[uuuurr,bend left=15]\arrow[rr,bend right]&&\color{green}r_6^{r_{1\nearrow6},l_{7\searrow1},G}\arrow[d]\\
    \vdots\arrow[u]&&\vdots\arrow[dddll]\\
    |[color=red,fill=green]|l_{p-6}^{l_{p-7\searrow1},r_{1\nearrow p-6},G}\arrow[urr]\arrow[rr]&&|[color=green,fill=red]|r_{p-6}^{r_{1\nearrow p-6},l_{p-5\searrow1},G}\arrow[dd]\arrow[ddddll,bend right=40]\\
    &|[color=red,fill=green]|m_{p-3}^{m_{p-5},r_{p-5}}\arrow[dr]&\\
    |[color=red,fill=green]|l_{p-5}^{l_{p-6}}\arrow[uu]\arrow[ur]&&|[color=green,fill=red]|r_{p-5}^{r_{p-5}}\arrow[uull,bend right]\arrow[dl]\\
    &|[color=red,fill=green]|m_{p-6}^{m_{p-4},l_{p-5}}\arrow[ul]&\\
    |[color=red,fill=green]|l_{p-4}^{l_{p-4}}\arrow[dr]\arrow[rr,bend left=15]&&|[color=red,fill=green]|r_{p-4}^{r_{1\nearrow p-4},l_{p-5\searrow1},G}\arrow[dd]\arrow[uu,bend right]\\
    &\color{red}m_{p-1}^{m_{p-3}}\arrow[dl]\arrow[ur]&\\
    \color{red}l_{p-3}^{l_{p-3}}\arrow[uu]\arrow[dr]&&\color{red}r_{p-3}^{r_{p-3}}\arrow[dd]\arrow[ul]\\
    &\color{red}m_{p-4}^{m_{p-2}}\arrow[dl]\arrow[ur]&\\
     \color{red}r_{p-2}^{l_{p-2}}\arrow[uu]&&\color{red}l_{p-2}^{r_{p-2}}\arrow[ul]
\end{tikzcd}}\]

Mutating at $r_{p-4}$ "unwraps" the arc and mutating at $r_{p-3}$ puts it back to its original position which $\beta_{p-3}$ did not do because it was impossible. Same goes for $l_{p-4}$ and $l_{p-3}$.

Observe that the full subdiagram with vertices involved in $\delta_{p-6}$ and adjacent vertices is isomorphic to that of $\delta_{p-4}$ before mutating, we may conclude by induction that after mutating at $\delta_4$, the diagram is:

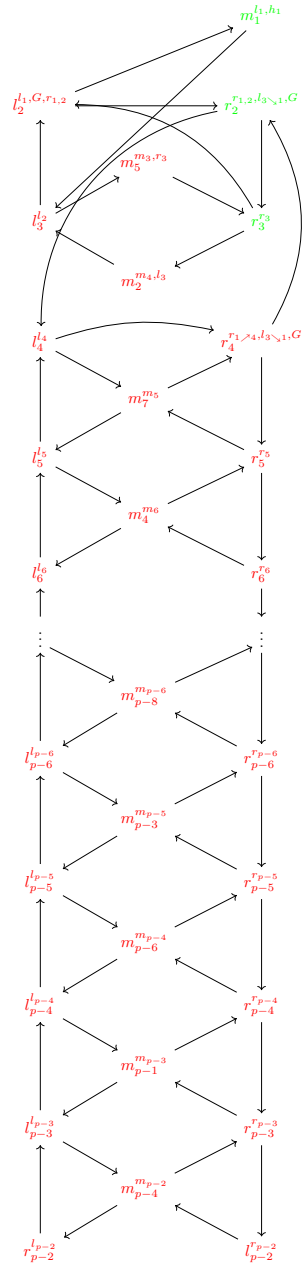
\[\adjustbox{center,scale=0.6
}{\begin{tikzcd}
    &&\color{green}m_1^{l_1,h_1}\arrow[ddddll]\\\\
    \color{red}l_2^{l_1,G,r_{1,2}}\arrow[uurr]\arrow[rr]&&\color{green}r_2^{r_{1,2},l_{3\searrow1},G}\arrow[dd]\arrow[ddddll,bend right=40]\\
    &\color{red}m_{5}^{m_3,r_{3}}\arrow[dr]&\\
    \color{red}l_3^{l_2}\arrow[uu]\arrow[ur]&&\color{green}r_3^{r_3}\arrow[uull,bend right]\arrow[dl]\\
    &\color{red}m_{2}^{m_4,l_3}\arrow[ul]&\\
    \color{red}l_4^{l_4}\arrow[rr,bend left=15]\arrow[dr]&&\color{red}r_4^{r_{1\nearrow4},l_{3\searrow1},G}\arrow[dd]\arrow[uuuu,bend right]\\
    &\color{red}m_{7}^{m_5}\arrow[ur]\arrow[dl]&\\
    \color{red}l_5^{l_5}\arrow[uu]\arrow[dr]&&\color{red}r_5^{r_5}\arrow[ul]\arrow[dd]\\
    &\color{red}m_{4}^{m_6}\arrow[dl]\arrow[ur]&\\
    \color{red}l_6^{l_6}\arrow[uu]&&\color{red}r_6^{r_6}\arrow[d]\arrow[ul]\\
    \vdots\arrow[u]\arrow[dr]&&\vdots\arrow[dd]\\
    &\color{red}m_{p-8}^{m_{p-6}}\arrow[dl]\arrow[ur]&\\
    \color{red}l_{p-6}^{l_{p-6}}\arrow[dr]\arrow[uu]&&\color{red}r_{p-6}^{r_{p-6}}\arrow[dd]\arrow[ul]\\
    &\color{red}m_{p-3}^{m_{p-5}}\arrow[dl]\arrow[ur]&\\
    \color{red}l_{p-5}^{l_{p-5}}\arrow[uu]\arrow[dr]&&\color{red}r_{p-5}^{r_{p-5}}\arrow[ul]\arrow[dd]\\
    &\color{red}m_{p-6}^{m_{p-4}}\arrow[dl]\arrow[ur]&\\
    \color{red}l_{p-4}^{l_{p-4}}\arrow[dr]\arrow[uu]&&\color{red}r_{p-4}^{r_{p-4}}\arrow[dd]\arrow[ul]\\
    &\color{red}m_{p-1}^{m_{p-3}}\arrow[dl]\arrow[ur]&\\
    \color{red}l_{p-3}^{l_{p-3}}\arrow[uu]\arrow[dr]&&\color{red}r_{p-3}^{r_{p-3}}\arrow[dd]\arrow[ul]\\
    &\color{red}m_{p-4}^{m_{p-2}}\arrow[dl]\arrow[ur]&\\
     \color{red}r_{p-2}^{l_{p-2}}\arrow[uu]&&\color{red}l_{p-2}^{r_{p-2}}\arrow[ul]
\end{tikzcd}}\]

By computation, after mutating at $\delta_2$, we have:

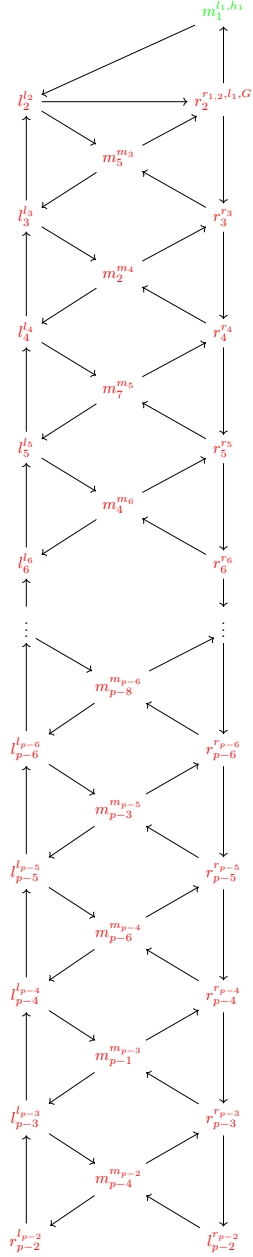
\[\adjustbox{center,scale=0.6
}{\begin{tikzcd}
    &&\color{green}m_1^{l_1,h_1}\arrow[ddll]\\\\
    \color{red}l_2^{l_2}\arrow[rr]\arrow[dr]&&\color{red}r_2^{r_{1,2},l_1,G}\arrow[dd]\arrow[uu]\\
    &\color{red}m_{5}^{m_3}\arrow[dl]\arrow[ur]&\\
    \color{red}l_3^{l_3}\arrow[uu]\arrow[dr]&&\color{red}r_3^{r_3}\arrow[ul]\arrow[dd]\\
    &\color{red}m_{2}^{m_4}\arrow[dl]\arrow[ur]&\\
    \color{red}l_4^{l_4}\arrow[uu]\arrow[dr]&&\color{red}r_4^{r_4}\arrow[dd]\arrow[ul]\\
    &\color{red}m_{7}^{m_5}\arrow[ur]\arrow[dl]&\\
    \color{red}l_5^{l_5}\arrow[uu]\arrow[dr]&&\color{red}r_5^{r_5}\arrow[ul]\arrow[dd]\\
    &\color{red}m_{4}^{m_6}\arrow[dl]\arrow[ur]&\\
    \color{red}l_6^{l_6}\arrow[uu]&&\color{red}r_6^{r_6}\arrow[d]\arrow[ul]\\
    \vdots\arrow[u]\arrow[dr]&&\vdots\arrow[dd]\\
    &\color{red}m_{p-8}^{m_{p-6}}\arrow[dl]\arrow[ur]&\\
    \color{red}l_{p-6}^{l_{p-6}}\arrow[dr]\arrow[uu]&&\color{red}r_{p-6}^{r_{p-6}}\arrow[dd]\arrow[ul]\\
    &\color{red}m_{p-3}^{m_{p-5}}\arrow[dl]\arrow[ur]&\\
    \color{red}l_{p-5}^{l_{p-5}}\arrow[uu]\arrow[dr]&&\color{red}r_{p-5}^{r_{p-5}}\arrow[ul]\arrow[dd]\\
    &\color{red}m_{p-6}^{m_{p-4}}\arrow[dl]\arrow[ur]&\\
    \color{red}l_{p-4}^{l_{p-4}}\arrow[dr]\arrow[uu]&&\color{red}r_{p-4}^{r_{p-4}}\arrow[dd]\arrow[ul]\\
    &\color{red}m_{p-1}^{m_{p-3}}\arrow[dl]\arrow[ur]&\\
    \color{red}l_{p-3}^{l_{p-3}}\arrow[uu]\arrow[dr]&&\color{red}r_{p-3}^{r_{p-3}}\arrow[dd]\arrow[ul]\\
    &\color{red}m_{p-4}^{m_{p-2}}\arrow[dl]\arrow[ur]&\\
     \color{red}r_{p-2}^{l_{p-2}}\arrow[uu]&&\color{red}l_{p-2}^{r_{p-2}}\arrow[ul]
\end{tikzcd}}\]

Now all the $l_k$ and $r_k$'s have returned to their original positions with the exception of $l_{p-2}$ and $r_{p-2}$ which swapped and $l_1$ and $r_1$ which we will later deal with.

Step 6(c) consists of $(m_1,h_2,h_3,...,h_{q+1},g_1,g_2,...,g_q)$. We shall look at them and adjacent vertices:

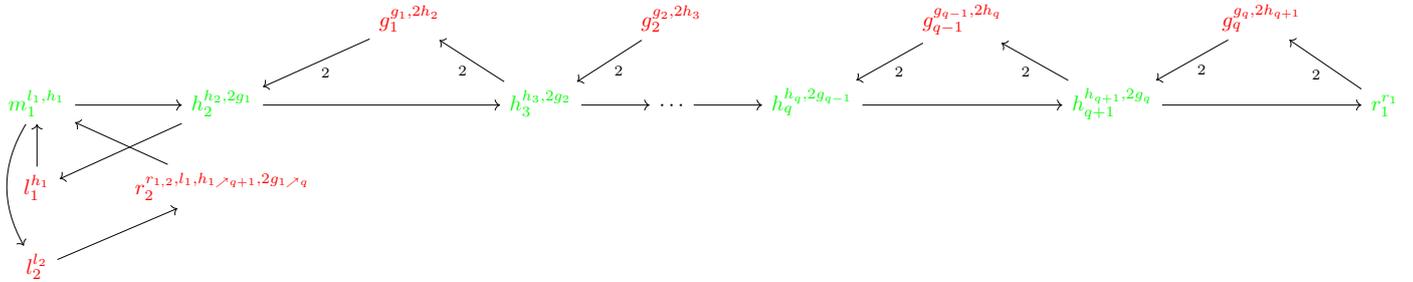
\[\adjustbox{scale=0.8,center}{
\begin{tikzcd}
    &&\color{red}g_1^{g_1,2h_2}\arrow[dl,"2"]&&\color{red}g_2^{g_2,2h_3}\arrow[dl,"2"]&&\color{red}g_{q-1}^{g_{q-1},2h_q}\arrow[dl,"2"]&&\color{red}g_q^{g_q,2h_{q+1}}\arrow[dl,"2"]\\
    \color{green}m_1^{l_1,h_1}\arrow[dd,bend right]\arrow[r]&\color{green}h_{2}^{h_{2},2g_{1}}\arrow[dl]\arrow[rr]&&\color{green}h_3^{h_3,2g_2}\arrow[ul,"2"]\arrow[r]&\cdots\arrow[r]&\color{green}h_{q}^{h_q,2g_{q-1}}\arrow[rr]&&\color{green}h_{q+1}^{h_{q+1},2g_q}\arrow[ul,"2"]\arrow[rr]&&\color{green}r_1^{r_1}\arrow[ul,"2"]\\
    \color{red}l_1^{h_1}\arrow[u]&\color{red}r_2^{r_{1,2},l_1,h_{1\nearrow q+1},2g_{1\nearrow q}}\arrow[ul]\\
    \color{red}l_2^{l_2}\arrow[ur]
\end{tikzcd}}\]

Observe that the subtriangulation is same as the case where $p=4$, after step 6(c), we have:

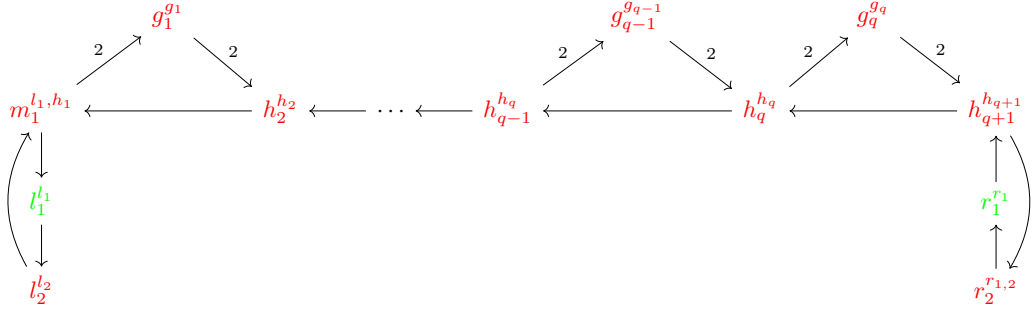
\[\adjustbox{scale=0.9,center}{
\begin{tikzcd}
    &\color{red}g_1^{g_1}\arrow[dr,"2"]&&&&\color{red}g_{q-1}^{g_{q-1}}\arrow[dr,"2"]&&\color{red}g_q^{g_q}\arrow[dr,"2"]\\
    \color{red}m_1^{l_1,h_1}\arrow[d]\arrow[ur,"2"]&&\color{red}h_2^{h_2}\arrow[ll]&\cdots\arrow[l]&\color{red}h_{q-1}^{h_q}\arrow[l]\arrow[ur,"2"]&&\color{red}h_{q}^{h_{q}}\arrow[ll]\arrow[ur,"2"]&&\color{red}h_{q+1}^{h_{q+1}}\arrow[ll]\arrow[dd,bend left]\\
    \color{green}l_1^{l_1}\arrow[d]&&&&&&&&\color{green}r_1^{r_1}\arrow[u]\\
    \color{red}l_2^{l_2}\arrow[uu,bend left]&&&&&&&&\color{red}r_2^{r_{1,2}}\arrow[u]
\end{tikzcd}}\]

Now we have unwrapped $m_1,h_2,h_3,...,h_{q+1}$ and all pending arcs are connected to the corner again.

Step 6(d) consists of $l_1,r_1$.

At $l_1,r_1$ and adjacent vertices, we have:

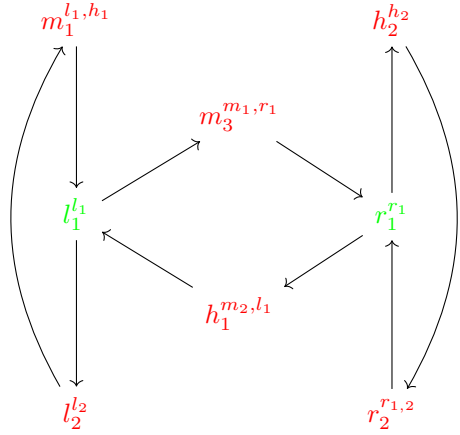
\[\begin{tikzcd}
    \color{red}m_1^{l_1,h_1}\arrow[dd]&&\color{red}h_2^{h_2}\arrow[dddd,bend left]\\
    &\color{red}m_3^{m_1,r_1}\arrow[dr]&\\
    \color{green}l_{1}^{l_1}\arrow[dd]\arrow[ur]&&\color{green}r_{1}^{r_1}\arrow[uu]\arrow[dl]\\
    &\color{red}h_1^{m_2,l_1}\arrow[ul]&\\
    \color{red}\color{red}l_2^{l_2}\arrow[uuuu,bend left]&&\color{red}r_2^{r_{1,2}}\arrow[uu]
\end{tikzcd}\]

After mutating at $l_1,r_1$, we have:

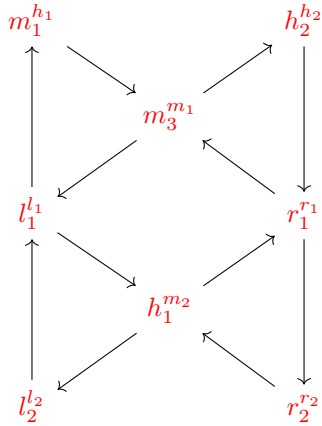
\[\begin{tikzcd}
    \color{red}m_1^{h_1}\arrow[dr]&&\color{red}h_2^{h_2}\arrow[dd]\\
    &\color{red}m_3^{m_1}\arrow[ur]\arrow[dl]&\\
    \color{red}l_{1}^{l_1}\arrow[uu]\arrow[dr]&&\color{red}r_{1}^{r_1}\arrow[dd]\arrow[ul]\\
    &\color{red}h_1^{m_2}\arrow[ur]\arrow[dl]&\\
    \color{red}\color{red}l_2^{l_2}\arrow[uu]&&\color{red}r_2^{r_2}\arrow[ul]
\end{tikzcd}\]

Now $l_1$ and $r_1$ return to their original positions, as promised, and all vertices are red again. Therefore $\Delta_1$ is a maximal green sequence. $\blacksquare$

\subsubsection{The case where $p=3$}\label{p=3}

\begin{proposition}
    $\Delta_1$ is a maximal green sequence when $p=3$.
\end{proposition}

\textbf{Proof} When $p=3$, the mutation sequence $\Delta_1$ translates to

\begin{enumerate}
    \item Tagging alternate punctures notched. $(r_1,l_1,m_2,m_1,l_1,r_1)$
    \item Tagging the first interior puncture notched. $(h_{q+1},h_q,...,h_1,m_{2},h_2,h_3,...,h_{q})$
    \item Tagging the remaining interior punctures notched and moving half the inner arcs away from the corner. $(\phi)$
    \item Moving arcs away from the corner.
    \begin{enumerate}
        \item Moving core arcs away from the corner. $(g_1,g_2,...,g_q,h_{q+1},h_{q},...,h_2,m_2)$
        \item Moving the remaining inner arcs away from the corner. $(\phi)$
        \item Moving outer arcs away from the corner. $(l_{2},f_1,r_{2},f_2,l_{2},f_1)$
    \end{enumerate}
    \item Tagging the corner notched. $(m_{1},l_{1},r_{1},m_{1})$
    \item Moving arcs back to the corner.
    \begin{enumerate}
        \item Moving outer arcs back to the corner. 
        $(f_1,l_{2},f_2,r_{2},f_1,l_{2})$
        \item Moving inner arcs back to the corner. 
        $(\phi)$
        \item Moving core arcs back to the corner. $(m_2,h_2,h_3,...,h_q,g_1,g_2,...,g_{q-1},g_q,h_{q+1},g_q)$
    \end{enumerate}
\end{enumerate}

In terms of triangulation, this is mostly similar to the case where $p=4$, so we will not mention the triangulations as much.

Step 1 consists of $r_1,l_1,m_2,m_1,l_1,r_1$.

At vertices involved in $r_1,l_1,m_2,m_1,l_1,r_1$ and adjacent vertices, we have:

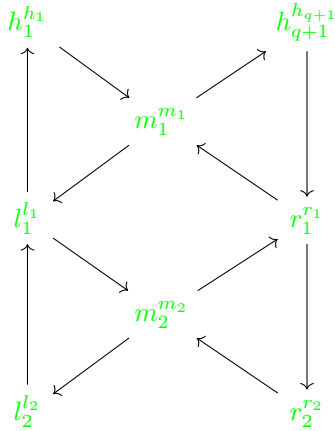
\[\begin{tikzcd}
    \color{green}h_1^{h_1}\arrow[dr]&&\color{green}h_{q+1}^{h_{q+1}}\arrow[dd]\\
    &\color{green}m_1^{m_1}\arrow[ur]\arrow[dl]&\\
    \color{green}l_{1}^{l_1}\arrow[uu]\arrow[dr]&&\color{green}r_{1}^{r_1}\arrow[ul]\arrow[dd]\\
    &\color{green}m_{2}^{m_2}\arrow[ur]\arrow[dl]&\\
    \color{green}l_{2}^{l_2}\arrow[uu]&&\color{green}r_{2}^{r_{2}}\arrow[ul]
\end{tikzcd}\]

After $r_1,l_1,m_2,m_1,l_1,r_1$, we have:

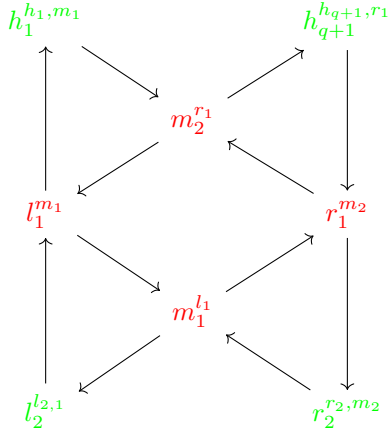
\[\begin{tikzcd}
    \color{green}h_1^{h_1,m_1}\arrow[dr]&&\color{green}h_{q+1}^{h_{q+1},r_1}\arrow[dd]\\
    &\color{red}m_2^{r_1}\arrow[ur]\arrow[dl]&\\
    \color{red}l_{1}^{m_1}\arrow[uu]\arrow[dr]&&\color{red}r_{1}^{m_2}\arrow[ul]\arrow[dd]\\
    &\color{red}m_{1}^{l_1}\arrow[ur]\arrow[dl]&\\
    \color{green}l_{2}^{l_{2,1}}\arrow[uu]&&\color{green}r_{2}^{r_{2},m_2}\arrow[ul]
\end{tikzcd}\]

Step 2 consists of $(h_{q+1},h_q,...,h_1,m_{2},h_2,h_3,...,h_{q})$. Look at these vertices and adjacent vertices:

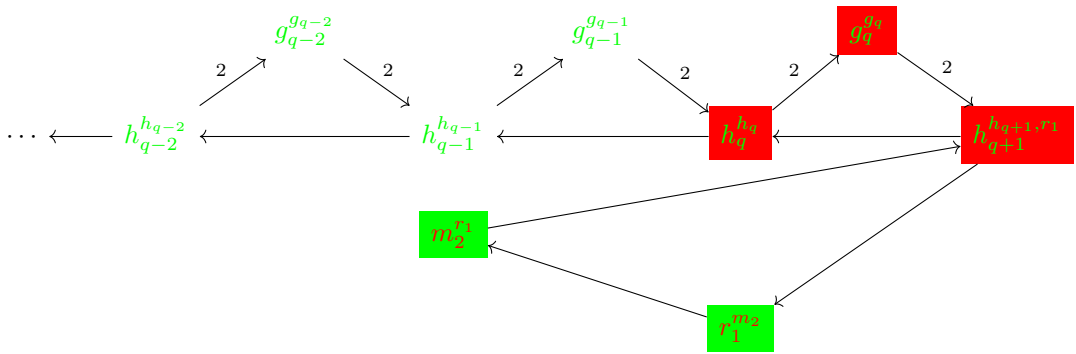
\[\begin{tikzcd}
    &&\color{green}g_{q-2}^{g_{q-2}}\arrow[dr,"2"]&&\color{green}g_{q-1}^{g_{q-1}}\arrow[dr,"2"]&&|[color=green,fill=red]|g_{q}^{g_q}\arrow[dr,"2"]\\
    \cdots&\color{green}h_{q-2}^{h_{q-2}}\arrow[ur,"2"]\arrow[l]&&\color{green}h_{q-1}^{h_{q-1}}\arrow[ur,"2"]\arrow[ll]&&|[color=green,fill=red]|h_{q}^{h_q}\arrow[ur,"2"]\arrow[ll]&&|[color=green,fill=red]|h_{q+1}^{h_{q+1},r_1}\arrow[ll]\arrow[ddll]\\
    &&&|[color=red,fill=green]|m_2^{r_1}\arrow[urrrr]\\
    &&&&&|[color=red,fill=green]|r_1^{m_2}\arrow[ull]
\end{tikzcd}\]

Mutate at $h_{q+1}$:

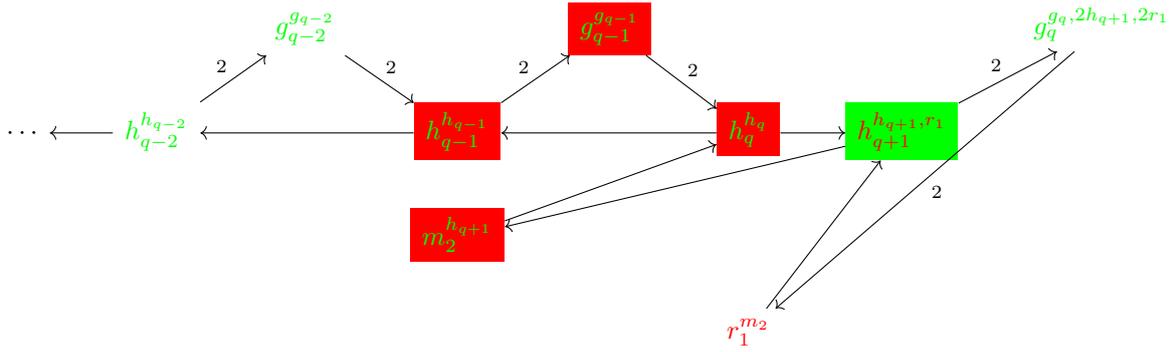
\begin{tikzcd}
    &&\color{green}g_{q-2}^{g_{q-2}}\arrow[dr,"2"]&&|[color=green,fill=red]|g_{q-1}^{g_{q-1}}\arrow[dr,"2"]&&&\color{green}g_{q}^{g_q,2h_{q+1},2r_1}\arrow[dddll,"2"]\\
    \cdots&\color{green}h_{q-2}^{h_{q-2}}\arrow[ur,"2"]\arrow[l]&&|[color=green,fill=red]|h_{q-1}^{h_{q-1}}\arrow[ur,"2"]\arrow[ll]&&|[color=green,fill=red]|h_{q}^{h_q}\arrow[r]\arrow[ll]&|[color=red,fill=green]|h_{q+1}^{h_{q+1},r_1}\arrow[ur,"2"]\arrow[dlll]\\
    &&&|[color=green,fill=red]|m_2^{h_{q+1}}\arrow[urr]\\
    &&&&&\color{red}r_1^{m_2}\arrow[uur]
\end{tikzcd}

Again we see the pattern at the highlighted subdiagrams and therefore after $h_2$ we have:

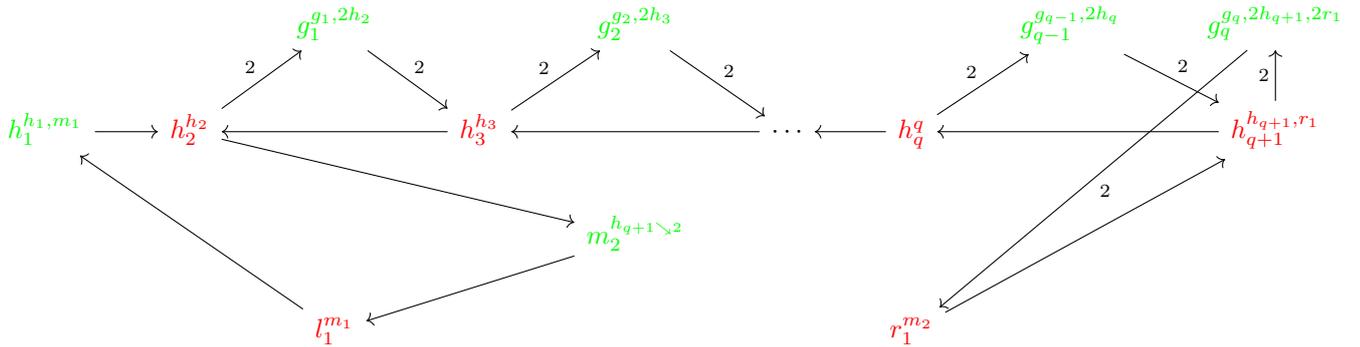
\[\begin{tikzcd}
 &&\color{green}g_1^{g_1,2h_2}\arrow[dr,"2"] &&\color{green}g_2^{g_2,2h_3}\arrow[dr,"2"]&&&\color{green}g_{q-1}^{g_{q-1},2h_{q}}\arrow[dr,"2"]&\color{green}g_q^{g_q,2h_{q+1},2r_1}\arrow[dddll,"2"]\\ \color{green}h_1^{h_1,m_1}\arrow[r]&\color{red}h_2^{h_2}\arrow[ur,"2"]\arrow[drrr]&&\color{red}h_3^{h_3}\arrow[ur,"2"]\arrow[ll]&&\cdots\arrow[ll]&\color{red}h_q^q\arrow[l]\arrow[ur,"2"]&&\color{red}h_{q+1}^{h_{q+1},r_1}\arrow[ll]\arrow[u,"2"]\\
 &&&&\color{green}m_2^{h_{q+1\searrow2}}\arrow[dll]&&&&\\
 &&\color{red}l_1^{m_1}\arrow[uull]&&&&\color{red}r_1^{m_2}\arrow[uurr]&&
\end{tikzcd}\]

Mutate at $h_1,m_2$:

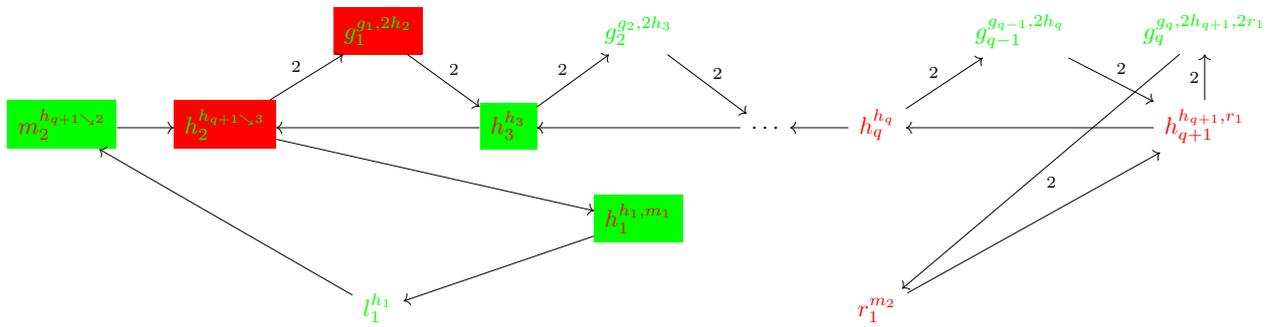
\[\adjustbox{scale=0.9,center}{\begin{tikzcd}
 &&|[color=green,fill=red]|g_1^{g_1,2h_2}\arrow[dr,"2"] &&\color{green}g_2^{g_2,2h_3}\arrow[dr,"2"]&&&\color{green}g_{q-1}^{g_{q-1},2h_{q}}\arrow[dr,"2"]&\color{green}g_q^{g_q,2h_{q+1},2r_1}\arrow[dddll,"2"]\\|[color=red,fill=green]|m_2^{h_{q+1\searrow2}}\arrow[r]&|[color=green,fill=red]|h_2^{h_{q+1\searrow3}}\arrow[ur,"2"]\arrow[drrr]&&|[color=red,fill=green]|h_3^{h_3}\arrow[ur,"2"]\arrow[ll]&&\cdots\arrow[ll]&\color{red}h_q^{h_q}\arrow[l]\arrow[ur,"2"]&&\color{red}h_{q+1}^{h_{q+1},r_1}\arrow[ll]\arrow[u,"2"]\\
 &&&&|[color=red,fill=green]|h_1^{h_1,m_1}\arrow[dll]&&&&\\
 &&\color{green}l_1^{h_1}\arrow[uull]&&&&\color{red}r_1^{m_2}\arrow[uurr]&&
\end{tikzcd}}\]

Mutate at $h_2$:

\[\adjustbox{scale=0.9,center}
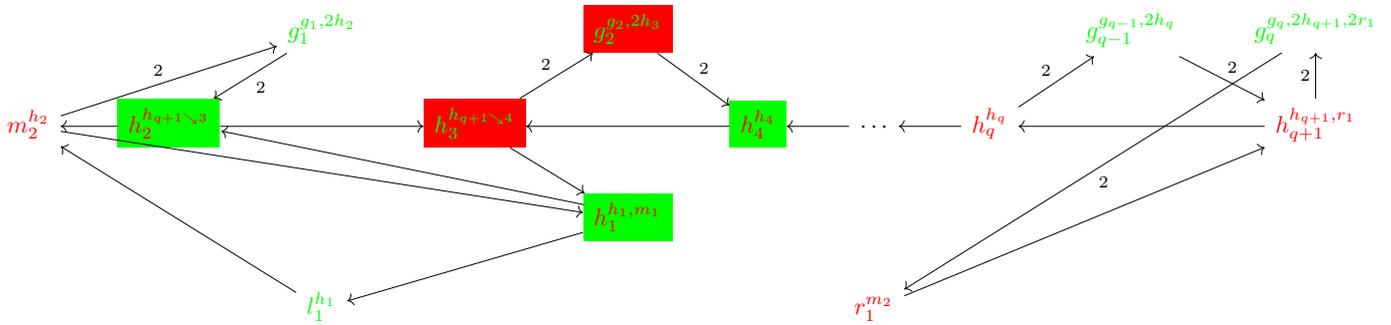
{\begin{tikzcd}
 &&\color{green}g_1^{g_1,2h_2}\arrow[dl,"2"] &&|[color=green,fill=red]|g_2^{g_2,2h_3}\arrow[dr,"2"]&&&&\color{green}g_{q-1}^{g_{q-1},2h_{q}}\arrow[dr,"2"]&\color{green}g_q^{g_q,2h_{q+1},2r_1}\arrow[dddlll,"2"]\\ \color{red}m_2^{h_{2}}\arrow[drrrr]\arrow[urr,"2"]&|[color=red,fill=green]|h_2^{h_{q+1\searrow3}}\arrow[l]\arrow[rr]&&|[color=green,fill=red]|h_3^{h_{q+1\searrow4}}\arrow[ur,"2"]\arrow[dr]&&|[color=red,fill=green]|h_4^{h_4}\arrow[ll]&\cdots\arrow[l]&\color{red}h_q^{h_q}\arrow[l]\arrow[ur,"2"]&&\color{red}h_{q+1}^{h_{q+1},r_1}\arrow[ll]\arrow[u,"2"]\\
 &&&&|[color=red,fill=green]|h_1^{h_1,m_1}\arrow[dll]\arrow[ulll,shift right]&&&&\\
 &&\color{green}l_1^{h_1}\arrow[uull]&&&&\color{red}r_1^{m_2}\arrow[uurrr]&&
\end{tikzcd}}\]

Observe the pattern at the highlighted subdiagrams, we see that after mutating at $h_3,...,h_{q}$, we have:

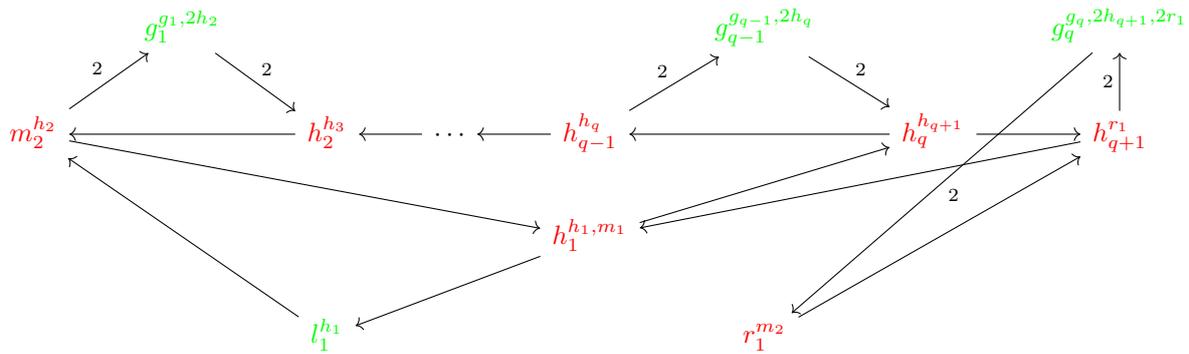
\begin{tikzcd}
    &\color{green}g_1^{g_1,2h_2}\arrow[dr,"2"]&&&&\color{green}g_{q-1}^{g_{q-1},2h_q}\arrow[dr,"2"]&&\color{green}g_q^{g_q,2h_{q+1},2r_1}\arrow[dddll,"2"]\\
    \color{red}m_2^{h_2}\arrow[drrrr]\arrow[ur,"2"]&&\color{red}h_2^{h_3}\arrow[ll]&\cdots\arrow[l]&\color{red}h_{q-1}^{h_q}\arrow[ur,"2"]\arrow[l]&&\color{red}h_{q}^{h_{q+1}}\arrow[ll]\arrow[r]&\color{red}h_{q+1}^{r_1}\arrow[u,"2"]\arrow[dlll]\\
    &&&&\color{red}h_1^{h_1,m_1}\arrow[urr]\arrow[dll]\\
    &&\color{green}l_1^{h_1}\arrow[uull]&&&\color{red}r_1^{m_2}\arrow[uurr]
\end{tikzcd}

Unlike the cases where $p$ is even, $h_{q+1}$ is still connected to the second interior puncture:

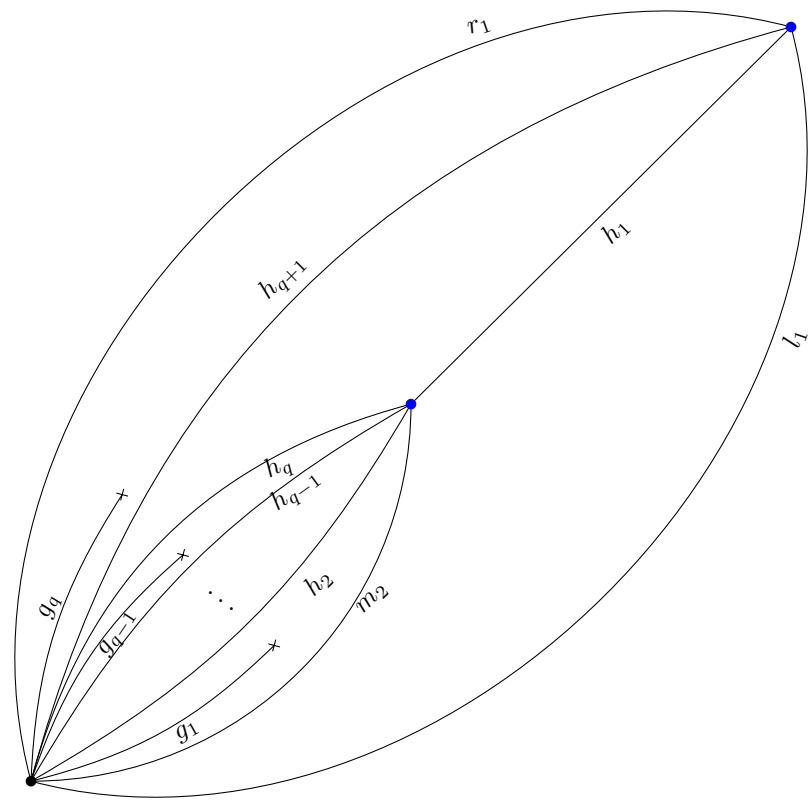
\[\begin{tikzpicture}[every edge quotes/.style={auto=right}]
    \begin{scope}[every node/.style={sloped,allow upside down}][every edge quotes/.style={auto=right}]
        \node at ($(0,5)!.5!(5,0)$) {$\ddots$};
        \draw (0,0) to [bend right=15] node [below=.15,right,black] {$g_1$} (3.2,1.8);
        \draw (0,0) to [bend left=15] node [above=.15,right,black] {$g_q$} (1.2,3.8);
        \draw (0,0) to [bend left=15] node [below=.15,right,black] {$g_{q-1}$} (2,3);
        \draw (0,0) to [bend right=45] node[below=.15,right=1] {$m_2$} (5,5);
        \draw (0,0) to [bend left=30] node[above] {$h_{q+1}$} (10,10);
        \draw (0,0) to [bend right=15] node[right=1,below,black] {$h_2$} (5,5);
        \draw (0,0) to [bend left=15] node[right=1,below=-.15,black,near end] {$h_{q-1}$} (5,5);
        \draw (0,0) to [bend left=30] node[right=1,below=-.15,black,near end] {$h_q$} (5,5);
        \draw (0,0) to [bend right=60] node[near end,below,black] {$l_{1}$} (10,10);
        \draw (0,0) to [bend left=60] node[near end,above,black] {$r_{1}$} (10,10);
        \draw (5,5) to node[below] {$h_1$} (10,10);
        \fill (0,0) circle (2pt);
        \fill[blue] (10,10) circle (2pt);
        \fill[blue] (5,5) circle (2pt);
        \fill (3.2,1.8) node[cross=2pt,rotate=30] {};
        \fill (1.2,3.8) node[cross=2pt,rotate=30] {};
        \fill (2,3) node[cross=2pt,rotate=30] {};
    \end{scope}
\end{tikzpicture}\]

This will affect the final part of the maximal green sequence.

Step 4(a) consists of $(g_1,g_2,...,g_q,h_{q+1},h_{q},...,h_2,m_2)$.

Now look at $g_1,g_2,...,g_q,h_q,h_{q-1},...,h_2,m_2$, and adjacent vertices:

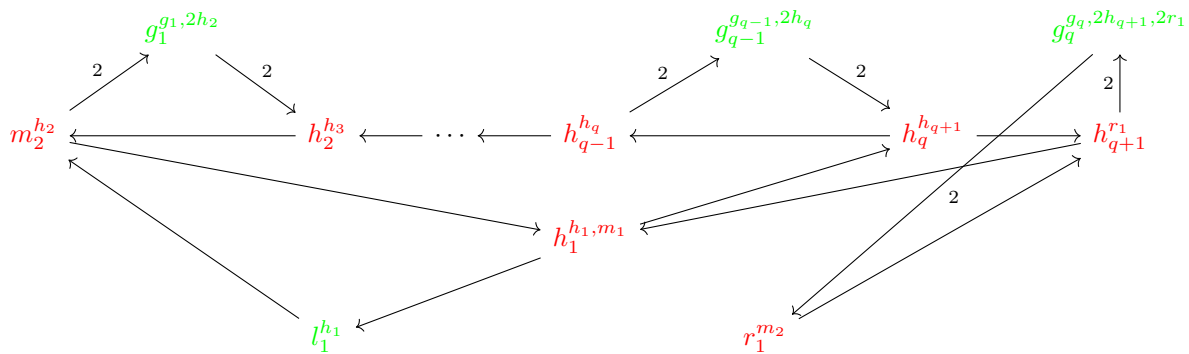
\[\begin{tikzcd}
    &\color{green}g_1^{g_1,2h_2}\arrow[dr,"2"]&&&&\color{green}g_{q-1}^{g_{q-1},2h_q}\arrow[dr,"2"]&&\color{green}g_q^{g_q,2h_{q+1},2r_1}\arrow[dddll,"2"]\\
    \color{red}m_2^{h_2}\arrow[drrrr]\arrow[ur,"2"]&&\color{red}h_2^{h_3}\arrow[ll]&\cdots\arrow[l]&\color{red}h_{q-1}^{h_q}\arrow[ur,"2"]\arrow[l]&&\color{red}h_{q}^{h_{q+1}}\arrow[ll]\arrow[r]&\color{red}h_{q+1}^{r_1}\arrow[u,"2"]\arrow[dlll]\\
    &&&&\color{red}h_1^{h_1,m_1}\arrow[urr]\arrow[dll]\\
    &&\color{green}l_1^{h_1}\arrow[uull]&&&\color{red}r_1^{m_2}\arrow[uurr]
\end{tikzcd}\]

For convenience, we shall show the result of step 4(a) in a lemma:

\begin{lemma}\label{p>3}
    Mutating the following diagram (the blue colour of the vertex $r_1$ means that the greenness of $r_1$ does not matter):
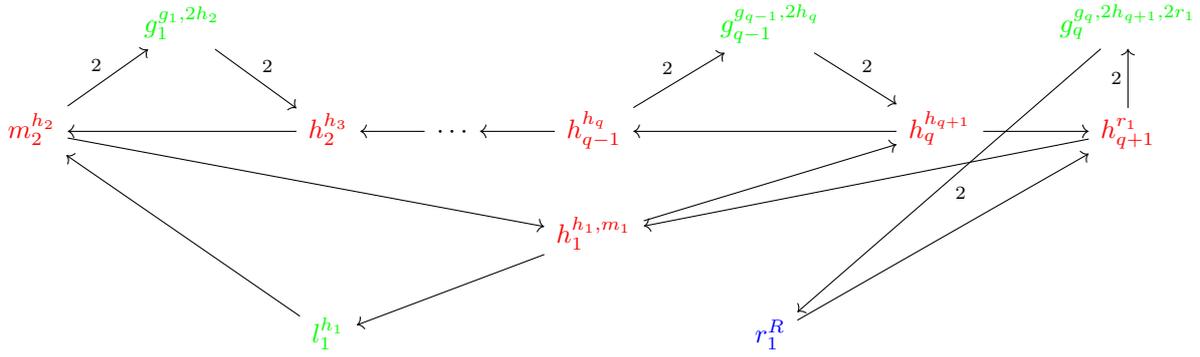
\[\begin{tikzcd}
    &\color{green}g_1^{g_1,2h_2}\arrow[dr,"2"]&&&&\color{green}g_{q-1}^{g_{q-1},2h_q}\arrow[dr,"2"]&&\color{green}g_q^{g_q,2h_{q+1},2r_1}\arrow[dddll,"2"]\\
    \color{red}m_2^{h_2}\arrow[drrrr]\arrow[ur,"2"]&&\color{red}h_2^{h_3}\arrow[ll]&\cdots\arrow[l]&\color{red}h_{q-1}^{h_q}\arrow[ur,"2"]\arrow[l]&&\color{red}h_{q}^{h_{q+1}}\arrow[ll]\arrow[r]&\color{red}h_{q+1}^{r_1}\arrow[u,"2"]\arrow[dlll]\\
    &&&&\color{red}h_1^{h_1,m_1}\arrow[urr]\arrow[dll]\\
    &&\color{green}l_1^{h_1}\arrow[uull]&&&\color{blue}r_1^{R}\arrow[uurr]
\end{tikzcd}\]
    , where $R$ is a set of (frozen) vertices, with the mutation sequence $(g_1,g_2,...,g_q,h_q,h_{q-1},...,h_2,m_2)$ gives the following diagram (after relocating some vertices):
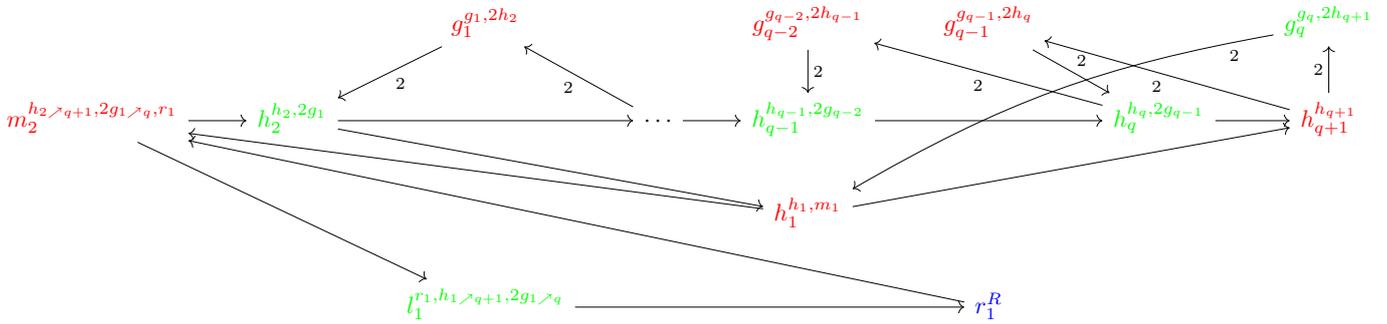
\[\adjustbox{scale=0.9,center}{\begin{tikzcd}
    &&\color{red}g_1^{g_1,2h_2}\arrow[dl,"2"]&&\color{red}g_{q-2}^{g_{q-2},2h_{q-1}}\arrow[d,"2"]&\color{red}g_{q-1}^{g_{q-1},2h_q}\arrow[dr,"2"]&&\color{green}g_q^{g_q,2h_{q+1}}\arrow[ddlll,"2",pos=0.1,bend right=10]\\
    \color{red}m_2^{h_{2\nearrow q+1},2g_{1\nearrow q},r_1}\arrow[r]\arrow[ddrr]&\color{green}h_2^{h_{2},2g_{1}}\arrow[rr]\arrow[drrr]&&\cdots\arrow[r]\arrow[ul,"2"]&\color{green}h_{q-1}^{h_{q-1},2g_{q-2}}\arrow[rr]&&\color{green}h_{q}^{h_{q},2g_{q-1}}\arrow[r]\arrow[ull,"2"]&\color{red}h_{q+1}^{h_{q+1}}\arrow[ull,"2"]\arrow[u,"2"]\\
    &&&&\color{red}h_1^{h_1,m_1}\arrow[urrr]\arrow[ullll]\\
    &&\color{green}l_1^{r_1,h_{1\nearrow q+1},2g_{1\nearrow q}}\arrow[rrr]&&&\color{blue}r_1^{R}\arrow[uulllll]
\end{tikzcd}}\]
\end{lemma}

\begin{proof}
    After $g_1,g_2,...,g_q$, we have:

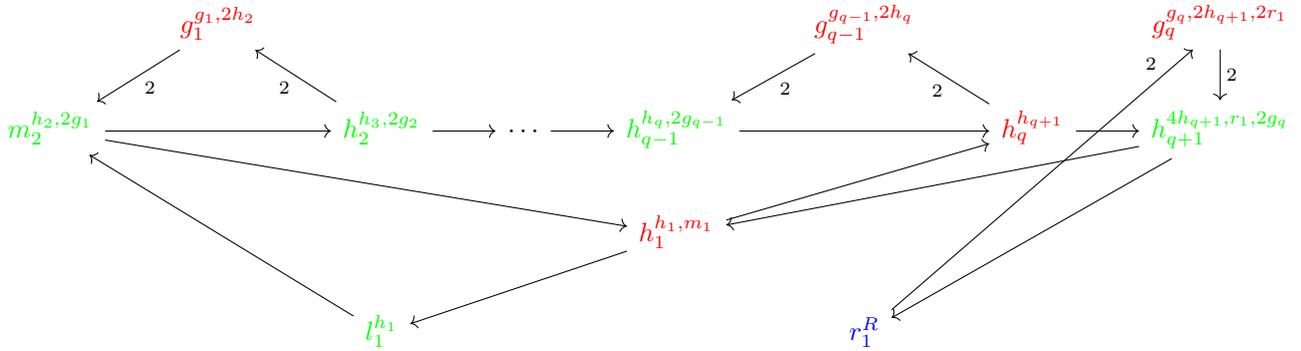
\[\begin{tikzcd}
    &\color{red}g_1^{g_1,2h_2}\arrow[dl,"2"]&&&&\color{red}g_{q-1}^{g_{q-1},2h_q}\arrow[dl,"2"]&&\color{red}g_q^{g_q,2h_{q+1},2r_1}\arrow[d,"2"]\\
    \color{green}m_2^{h_2,2g_1}\arrow[drrrr]\arrow[rr]&&\color{green}h_2^{h_3,2g_2}\arrow[r]\arrow[ul,"2"]&\cdots\arrow[r]&\color{green}h_{q-1}^{h_q,2g_{q-1}}\arrow[rr]&&\color{red}h_{q}^{h_{q+1}}\arrow[ul,"2"]\arrow[r]&\color{green}h_{q+1}^{4h_{q+1},r_1,2g_q}\arrow[dlll]\arrow[ddll]\\
    &&&&\color{red}h_1^{h_1,m_1}\arrow[urr]\arrow[dll]\\
    &&\color{green}l_1^{h_1}\arrow[uull]&&&\color{blue}r_1^{R}\arrow[uuurr,"2",pos=0.9]
\end{tikzcd}\]

Mutate at $h_{q+1}$:

$\adjustbox{scale=0.9,center}
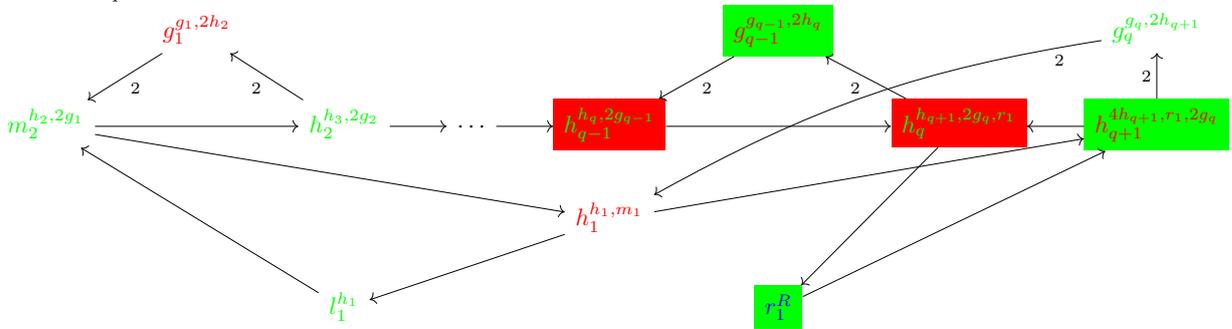
{\begin{tikzcd}
    &\color{red}g_1^{g_1,2h_2}\arrow[dl,"2"]&&&&|[color=red,fill=green]|g_{q-1}^{g_{q-1},2h_q}\arrow[dl,"2"]&&\color{green}g_q^{g_q,2h_{q+1}}\arrow[ddlll,"2",pos=0.1,bend right=10]\\
    \color{green}m_2^{h_2,2g_1}\arrow[drrrr]\arrow[rr]&&\color{green}h_2^{h_3,2g_2}\arrow[r]\arrow[ul,"2"]&\cdots\arrow[r]&|[color=green,fill=red]|h_{q-1}^{h_q,2g_{q-1}}\arrow[rr]&&|[color=green,fill=red]|h_{q}^{h_{q+1},2g_q,r_1}\arrow[ul,"2"]\arrow[ddl]&|[color=red,fill=green]|h_{q+1}^{4h_{q+1},r_1,2g_q}\arrow[l]\arrow[u,"2"]\\
    &&&&\color{red}h_1^{h_1,m_1}\arrow[urrr]\arrow[dll]\\
    &&\color{green}l_1^{h_1}\arrow[uull]&&&|[color=blue,fill=green]|r_1^{R}\arrow[uurr]
\end{tikzcd}}$

Mutate at $h_{q}$:

\[\adjustbox{scale=0.8,center}
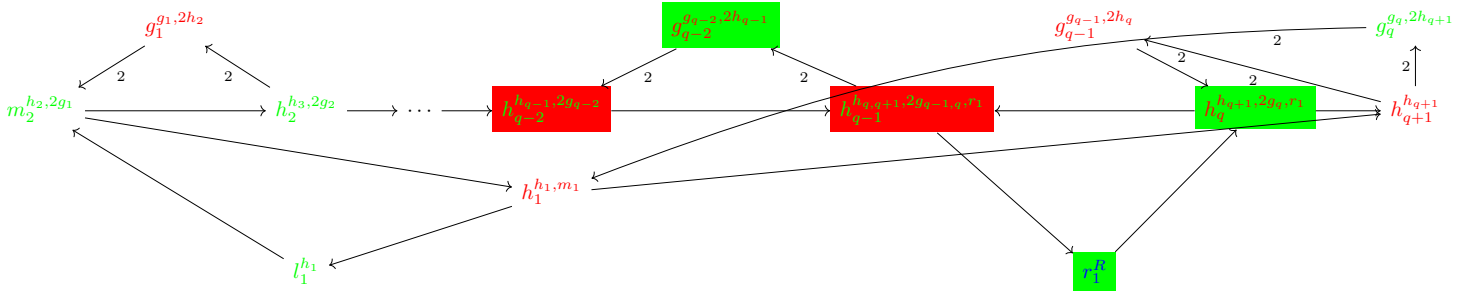
{\begin{tikzcd}
    &\color{red}g_1^{g_1,2h_2}\arrow[dl,"2"]&&&&|[color=red,fill=green]|g_{q-2}^{g_{q-2},2h_{q-1}}\arrow[dl,"2"]&&\color{red}g_{q-1}^{g_{q-1},2h_q}\arrow[dr,"2"]&&\color{green}g_q^{g_q,2h_{q+1}}\arrow[ddlllll,"2",pos=0.1,bend right=10]\\
    \color{green}m_2^{h_2,2g_1}\arrow[drrrr]\arrow[rr]&&\color{green}h_2^{h_3,2g_2}\arrow[r]\arrow[ul,"2"]&\cdots\arrow[r]&|[color=green,fill=red]|h_{q-2}^{h_{q-1},2g_{q-2}}\arrow[rr]&&|[color=green,fill=red]|h_{q-1}^{h_{q,q+1},2g_{q-1,q},r_1}\arrow[ddr]\arrow[ul,"2"]&&|[color=red,fill=green]|h_{q}^{h_{q+1},2g_q,r_1}\arrow[ll]\arrow[r]&\color{red}h_{q+1}^{h_{q+1}}\arrow[ull,"2"]\arrow[u,"2"]\\
    &&&&\color{red}h_1^{h_1,m_1}\arrow[urrrrr]\arrow[dll]\\
    &&\color{green}l_1^{h_1}\arrow[uull]&&&&&|[color=blue,fill=green]|r_1^{R}\arrow[uur]
\end{tikzcd}}\]

Notice the pattern at the highlighted subdiagrams, after $h_2$ we have:

$\adjustbox{scale=0.9,center}
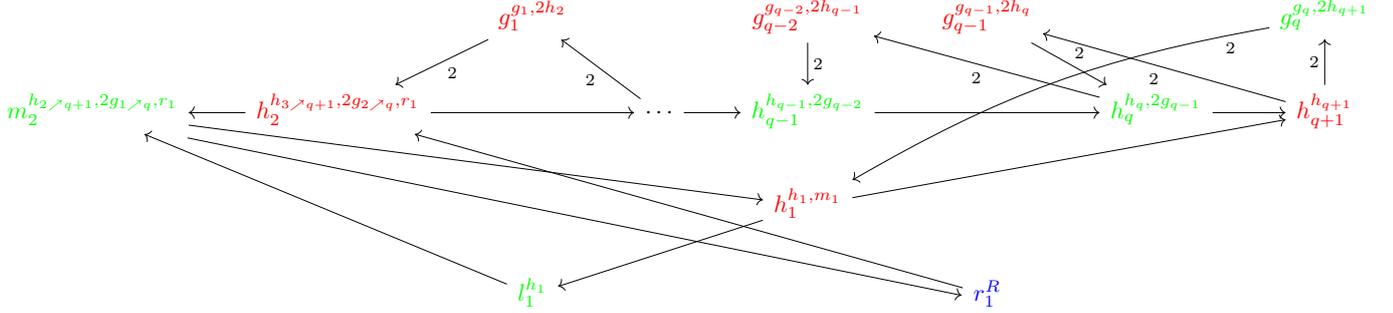
{\begin{tikzcd}
    &&\color{red}g_1^{g_1,2h_2}\arrow[dl,"2"]&&\color{red}g_{q-2}^{g_{q-2},2h_{q-1}}\arrow[d,"2"]&\color{red}g_{q-1}^{g_{q-1},2h_q}\arrow[dr,"2"]&&\color{green}g_q^{g_q,2h_{q+1}}\arrow[ddlll,"2",pos=0.1,bend right=10]\\
    \color{green}m_2^{h_{2\nearrow q+1},2g_{1\nearrow q},r_1}\arrow[drrrr]\arrow[ddrrrrr,shift right]&\color{red}h_2^{h_{3\nearrow q+1},2g_{2\nearrow q},r_1}\arrow[rr]\arrow[l]&&\cdots\arrow[ul,"2"]\arrow[r]&\color{green}h_{q-1}^{h_{q-1},2g_{q-2}}\arrow[rr]&&\color{green}h_{q}^{h_{q},2g_{q-1}}\arrow[r]\arrow[ull,"2"]&\color{red}h_{q+1}^{h_{q+1}}\arrow[ull,"2"]\arrow[u,"2"]\\
    &&&&\color{red}h_1^{h_1,m_1}\arrow[urrr]\arrow[dll]\\
    &&\color{green}l_1^{h_1}\arrow[uull]&&&\color{blue}r_1^{R}\arrow[uullll]
\end{tikzcd}}$

After $m_2$, we have:

$\adjustbox{scale=0.9,center}
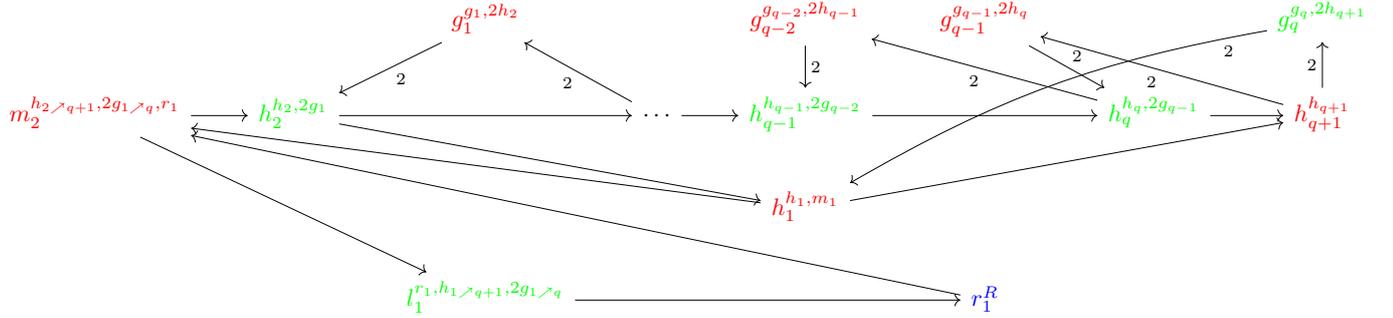
{\begin{tikzcd}
    &&\color{red}g_1^{g_1,2h_2}\arrow[dl,"2"]&&\color{red}g_{q-2}^{g_{q-2},2h_{q-1}}\arrow[d,"2"]&\color{red}g_{q-1}^{g_{q-1},2h_q}\arrow[dr,"2"]&&\color{green}g_q^{g_q,2h_{q+1}}\arrow[ddlll,"2",pos=0.1,bend right=10]\\
    \color{red}m_2^{h_{2\nearrow q+1},2g_{1\nearrow q},r_1}\arrow[r]\arrow[ddrr]&\color{green}h_2^{h_{2},2g_{1}}\arrow[rr]\arrow[drrr]&&\cdots\arrow[r]\arrow[ul,"2"]&\color{green}h_{q-1}^{h_{q-1},2g_{q-2}}\arrow[rr]&&\color{green}h_{q}^{h_{q},2g_{q-1}}\arrow[r]\arrow[ull,"2"]&\color{red}h_{q+1}^{h_{q+1}}\arrow[ull,"2"]\arrow[u,"2"]\\
    &&&&\color{red}h_1^{h_1,m_1}\arrow[urrr]\arrow[ullll]\\
    &&\color{green}l_1^{r_1,h_{1\nearrow q+1},2g_{1\nearrow q}}\arrow[rrr]&&&\color{blue}r_1^{R}\arrow[uulllll]
\end{tikzcd}}$
\end{proof}

By applying the above lemma, after step 4(a), we have:

$\adjustbox{scale=0.9,center}
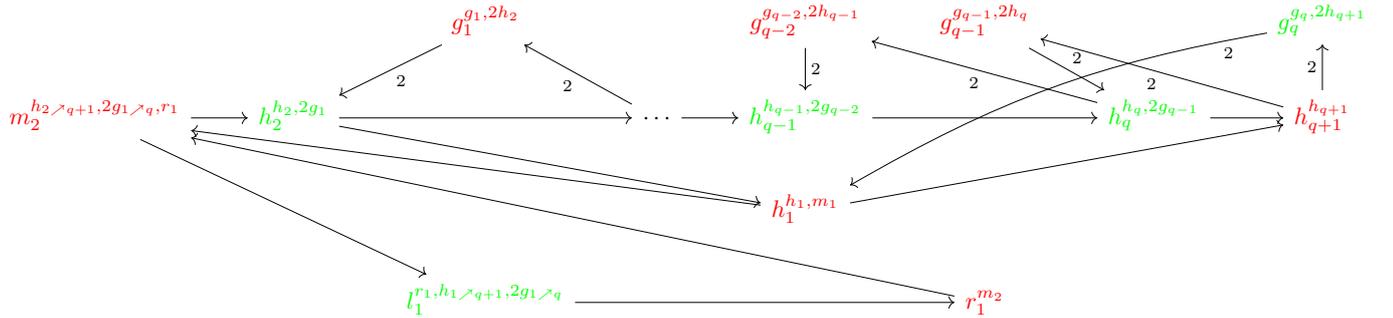
{\begin{tikzcd}
    &&\color{red}g_1^{g_1,2h_2}\arrow[dl,"2"]&&\color{red}g_{q-2}^{g_{q-2},2h_{q-1}}\arrow[d,"2"]&\color{red}g_{q-1}^{g_{q-1},2h_q}\arrow[dr,"2"]&&\color{green}g_q^{g_q,2h_{q+1}}\arrow[ddlll,"2",pos=0.1,bend right=10]\\
    \color{red}m_2^{h_{2\nearrow q+1},2g_{1\nearrow q},r_1}\arrow[r]\arrow[ddrr]&\color{green}h_2^{h_{2},2g_{1}}\arrow[rr]\arrow[drrr]&&\cdots\arrow[r]\arrow[ul,"2"]&\color{green}h_{q-1}^{h_{q-1},2g_{q-2}}\arrow[rr]&&\color{green}h_{q}^{h_{q},2g_{q-1}}\arrow[r]\arrow[ull,"2"]&\color{red}h_{q+1}^{h_{q+1}}\arrow[ull,"2"]\arrow[u,"2"]\\
    &&&&\color{red}h_1^{h_1,m_1}\arrow[urrr]\arrow[ullll]\\
    &&\color{green}l_1^{r_1,h_{1\nearrow q+1},2g_{1\nearrow q}}\arrow[rrr]&&&\color{red}r_1^{m_2}\arrow[uulllll]
\end{tikzcd}}$

Just like in the previous cases, we shall merge step 4(c), 5, and 6(a).

At $l_{2},r_{2},f_1,f_2,m_{1},l_{1},r_{1}$ and adjacent vertices, we have:

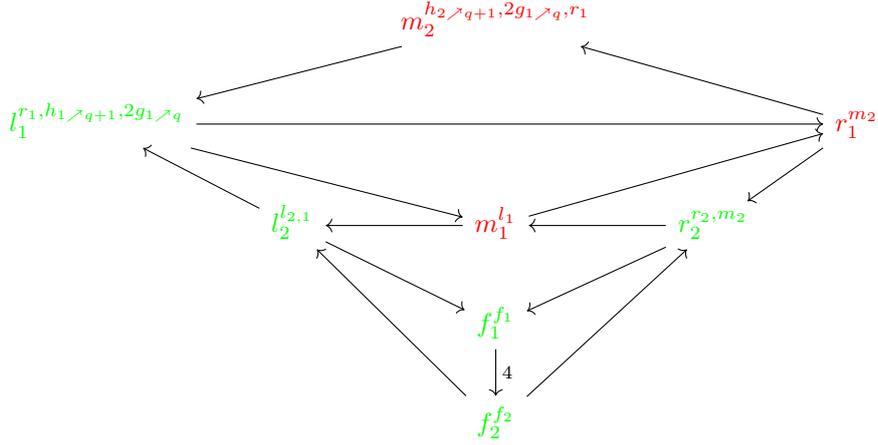
\[\begin{tikzcd}
    &&\color{red}m_2^{h_{2\nearrow q+1},2g_{1\nearrow q},r_1}\arrow[dll]&&\\
    \color{green}l_1^{r_1,h_{1\nearrow q+1},2g_{1\nearrow q}}\arrow[rrrr]\arrow[drr]&&&&\color{red}r_1^{m_2}\arrow[dl]\arrow[ull]\\
    &\color{green}l_{2}^{l_{2,1}}\arrow[ul]\arrow[dr]&\color{red}m_{1}^{l_{1}}\arrow[l]\arrow[urr]&\color{green}r_{2}^{r_{2},m_{2}}\arrow[l]\arrow[dl]\\
    &&\color{green}f_1^{f_1}\arrow[d,"4"]&\\
    &&\color{green}f_2^{f_2}\arrow[uul]\arrow[uur]&
\end{tikzcd}\]

By applying Lemma \ref{corner}, after mutating at $l_{2},f_1,r_{2},f_2,l_{2},f_1,m_{1},l_{1},r_{1},m_{1},f_1,l_{2},f_2,r_{2},f_1,l_{2}$, we have:

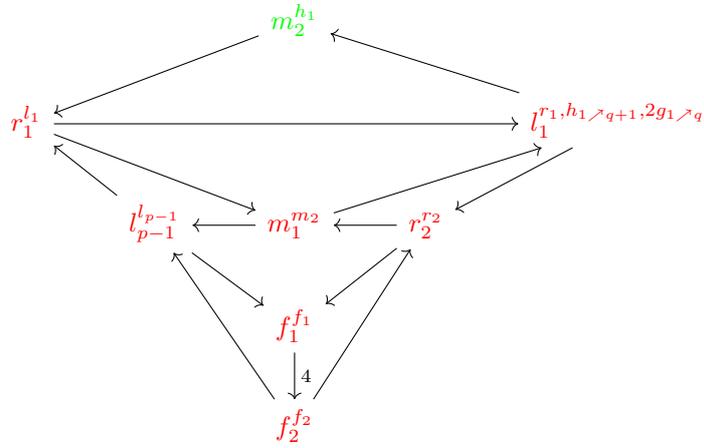
\[\begin{tikzcd}
    &&\color{green}m_2^{h_1}\arrow[dll]&&\\
    \color{red}r_{1}^{l_{1}}\arrow[rrrr]\arrow[drr]&&&&\color{red}l_{1}^{r_{1},h_{1\nearrow q+1},2g_{1\nearrow q}}\arrow[dl]\arrow[ull]\\
    &\color{red}l_{p-1}^{l_{p-1}}\arrow[ul]\arrow[dr]&\color{red}m_{1}^{m_{2}}\arrow[l]\arrow[urr]&\color{red}r_{2}^{r_{2}}\arrow[l]\arrow[dl]\\
    &&\color{red}f_1^{f_1}\arrow[d,"4"]&\\
    &&\color{red}f_2^{f_2}\arrow[uul]\arrow[uur]&
\end{tikzcd}\]

Step 6(c) consists of $(m_2,h_2,h_3,...,h_q,g_1,g_2,...,g_{q-1},g_q,h_{q+1},g_q)$.

At $m_2,h_2,h_3,...,h_{q+1},g_1,g_2,...,g_{q-1},g_q$ and adjacent vertices, the diagram looks like:

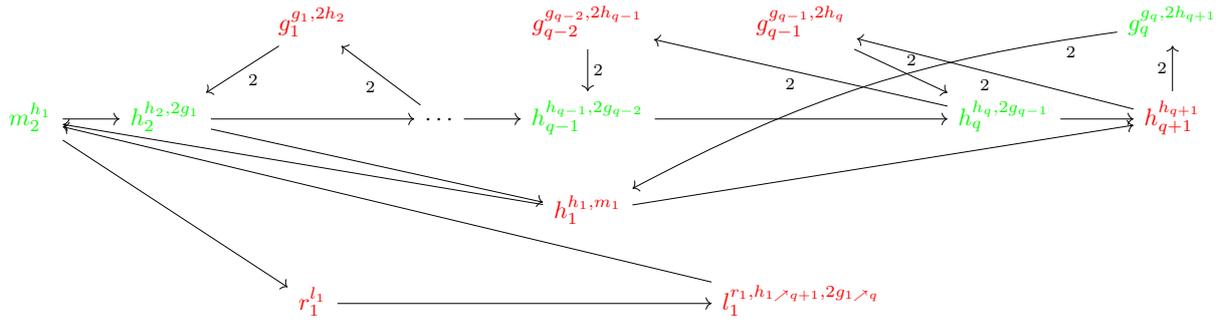
\[\adjustbox{scale=0.9,center}{\begin{tikzcd}
    &&\color{red}g_1^{g_1,2h_2}\arrow[dl,"2"]&&\color{red}g_{q-2}^{g_{q-2},2h_{q-1}}\arrow[d,"2"]&\color{red}g_{q-1}^{g_{q-1},2h_q}\arrow[dr,"2"]&&\color{green}g_q^{g_q,2h_{q+1}}\arrow[ddlll,"2",pos=0.1,bend right=10]\\
    \color{green}m_2^{h_1}\arrow[r]\arrow[ddrr]&\color{green}h_2^{h_{2},2g_{1}}\arrow[rr]\arrow[drrr]&&\cdots\arrow[r]\arrow[ul,"2"]&\color{green}h_{q-1}^{h_{q-1},2g_{q-2}}\arrow[rr]&&\color{green}h_{q}^{h_{q},2g_{q-1}}\arrow[r]\arrow[ull,"2"]&\color{red}h_{q+1}^{h_{q+1}}\arrow[ull,"2"]\arrow[u,"2"]\\
    &&&&\color{red}h_1^{h_1,m_1}\arrow[urrr]\arrow[ullll]\\
    &&\color{red}r_{1}^{l_{1}}\arrow[rrr]&&&\color{red}l_{1}^{r_{1},h_{1\nearrow q+1},2g_{1\nearrow q}}\arrow[uulllll]
\end{tikzcd}}\]

Mutate at $m_2$:

\[\adjustbox{scale=0.9,center}
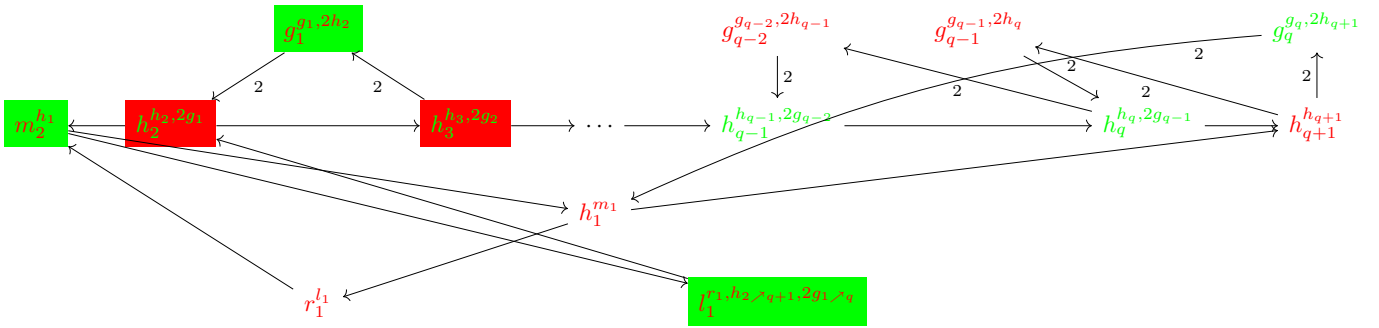
{\begin{tikzcd}
    &&|[color=red,fill=green]|g_1^{g_1,2h_2}\arrow[dl,"2"]&&&\color{red}g_{q-2}^{g_{q-2},2h_{q-1}}\arrow[d,"2"]&\color{red}g_{q-1}^{g_{q-1},2h_q}\arrow[dr,"2"]&&\color{green}g_q^{g_q,2h_{q+1}}\arrow[ddllll,"2",pos=0.1,bend right=10]\\
    |[color=red,fill=green]|m_2^{h_1}\arrow[drrrr]\arrow[ddrrrrr]&|[color=green,fill=red]|h_2^{h_{2},2g_{1}}\arrow[rr]\arrow[l]&&|[color=green,fill=red]|h_3^{h_3,2g_2}\arrow[r]\arrow[ul,"2"]&\cdots\arrow[r]&\color{green}h_{q-1}^{h_{q-1},2g_{q-2}}\arrow[rr]&&\color{green}h_{q}^{h_{q},2g_{q-1}}\arrow[r]\arrow[ull,"2"]&\color{red}h_{q+1}^{h_{q+1}}\arrow[ull,"2"]\arrow[u,"2"]\\
    &&&&\color{red}h_1^{m_1}\arrow[urrrr]\arrow[dll]\\
    &&\color{red}r_1^{l_{1}}\arrow[uull]&&&|[color=red,fill=green]|l_1^{r_1,h_{2\nearrow q+1},2g_{1\nearrow q}}\arrow[uullll]
\end{tikzcd}}\]

Mutate at $h_2$:

\[\adjustbox{scale=0.8,center}
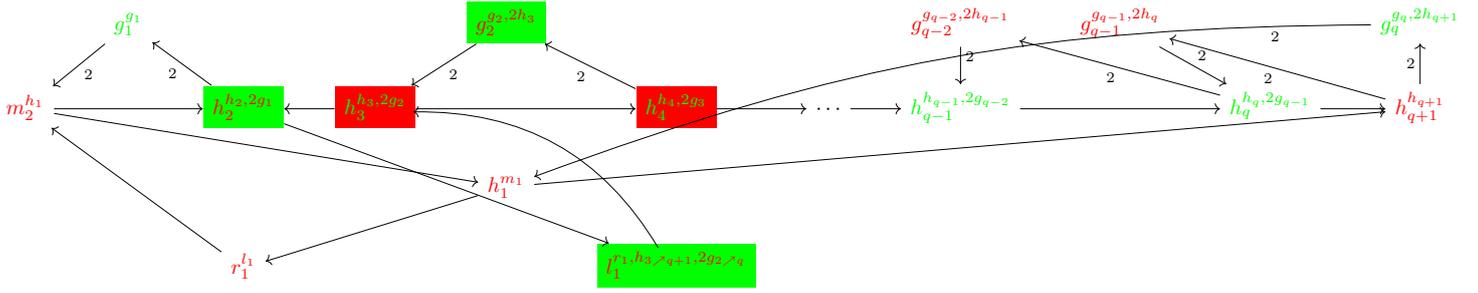
{\begin{tikzcd}
    &\color{green}g_1^{g_1}\arrow[dl,"2"]&&&|[color=red,fill=green]|g_2^{g_2,2h_3}\arrow[dl,"2"]&&&\color{red}g_{q-2}^{g_{q-2},2h_{q-1}}\arrow[d,"2",near start]&\color{red}g_{q-1}^{g_{q-1},2h_q}\arrow[dr,"2"]&&\color{green}g_q^{g_q,2h_{q+1}}\arrow[ddllllll,"2",pos=0.1,bend right=10]\\
    \color{red}m_2^{h_1}\arrow[drrrr]\arrow[rr]&&|[color=red,fill=green]|h_2^{h_{2},2g_{1}}\arrow[ul,"2"]\arrow[ddrrr]&|[color=green,fill=red]|h_3^{h_3,2g_2}\arrow[l]\arrow[rr]&&|[color=green,fill=red]|h_4^{h_4,2g_3}\arrow[r]\arrow[ul,"2"]&\cdots\arrow[r]&\color{green}h_{q-1}^{h_{q-1},2g_{q-2}}\arrow[rr]&&\color{green}h_{q}^{h_{q},2g_{q-1}}\arrow[r]\arrow[ull,"2"]&\color{red}h_{q+1}^{h_{q+1}}\arrow[ull,"2"]\arrow[u,"2"]\\
    &&&&\color{red}h_1^{m_1}\arrow[urrrrrr]\arrow[dll]\\
    &&\color{red}r_1^{l_{1}}\arrow[uull]&&&|[color=red,fill=green]|l_1^{r_1,h_{3\nearrow q+1},2g_{2\nearrow q}}\arrow[uull,bend right,shift left]
\end{tikzcd}}\]

Again observe the pattern at the highlighted subdiagrams and we see that after $h_q$ we have:

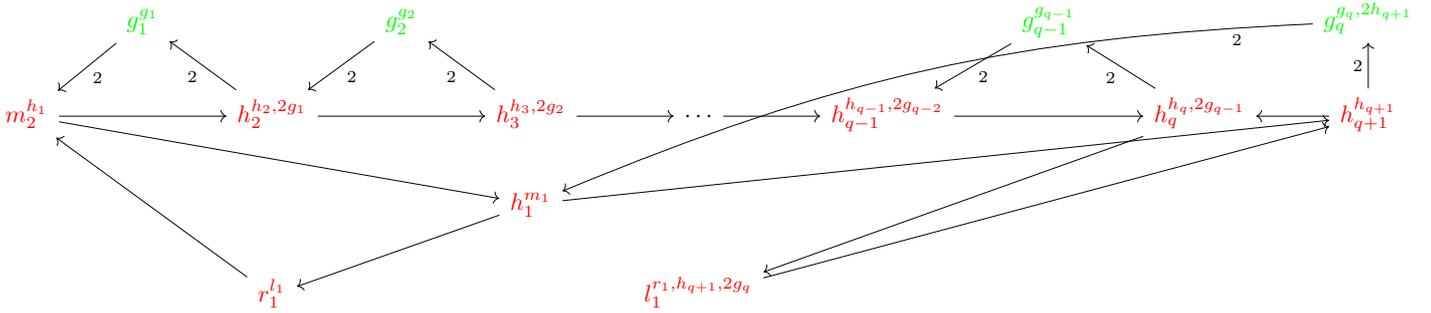
\[\adjustbox{scale=0.9,center}{\begin{tikzcd}
    &\color{green}g_1^{g_1}\arrow[dl,"2"]&&\color{green}g_2^{g_2}\arrow[dl,"2"]&&&&\color{green}g_{q-1}^{g_{q-1}}\arrow[dl,"2"]&&\color{green}g_q^{g_q,2h_{q+1}}\arrow[ddlllll,"2",pos=0.1,bend right=10]\\
    \color{red}m_2^{h_1}\arrow[drrrr]\arrow[rr]&&\color{red}h_2^{h_{2},2g_{1}}\arrow[ul,"2"]\arrow[rr]&&\color{red}h_3^{h_3,2g_2}\arrow[r]\arrow[ul,"2"]&\cdots\arrow[r]&\color{red}h_{q-1}^{h_{q-1},2g_{q-2}}\arrow[rr]&&\color{red}h_{q}^{h_{q},2g_{q-1}}\arrow[ul,"2"]\arrow[ddlll]&\color{red}h_{q+1}^{h_{q+1}}\arrow[u,"2"]\arrow[l]\\
    &&&&\color{red}h_1^{m_1}\arrow[urrrrr]\arrow[dll]\\
    &&\color{red}r_1^{l_{1}}\arrow[uull]&&&\color{red}l_1^{r_1,h_{q+1},2g_{q}}\arrow[uurrrr]
\end{tikzcd}}\]

Mutate at $g_1,g_2,...,g_{q-1}$:

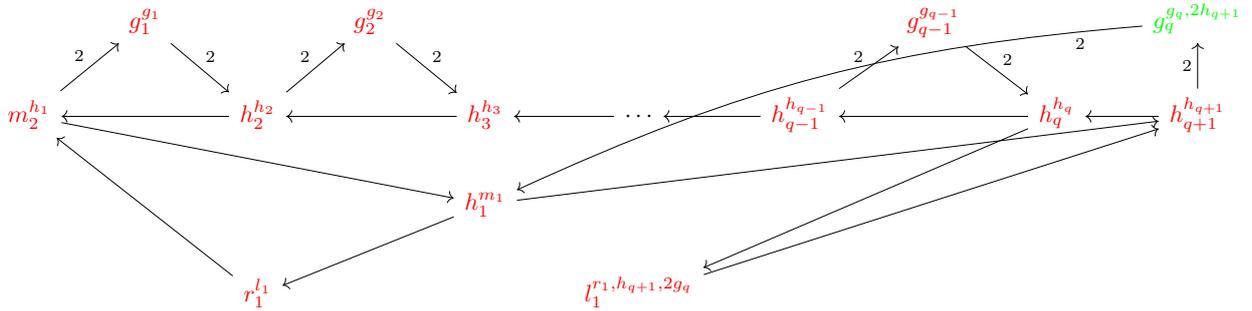
\[\adjustbox{scale=0.9,center}{\begin{tikzcd}
    &\color{red}g_1^{g_1}\arrow[dr,"2"]&&\color{red}g_2^{g_2}\arrow[dr,"2"]&&&&\color{red}g_{q-1}^{g_{q-1}}\arrow[dr,"2"]&&\color{green}g_q^{g_q,2h_{q+1}}\arrow[ddlllll,"2",pos=0.1,bend right=10]\\
    \color{red}m_2^{h_1}\arrow[drrrr]\arrow[ur,"2"]&&\color{red}h_2^{h_{2}}\arrow[ur,"2"]\arrow[ll]&&\color{red}h_3^{h_3}\arrow[ll]&\cdots\arrow[l]&\color{red}h_{q-1}^{h_{q-1}}\arrow[l]\arrow[ur,"2"]&&\color{red}h_{q}^{h_{q}}\arrow[ll]\arrow[ddlll]&\color{red}h_{q+1}^{h_{q+1}}\arrow[u,"2"]\arrow[l]\\
    &&&&\color{red}h_1^{m_1}\arrow[urrrrr]\arrow[dll]\\
    &&\color{red}r_1^{l_{1}}\arrow[uull]&&&\color{red}l_1^{r_1,h_{q+1},2g_{q}}\arrow[uurrrr]
\end{tikzcd}}\]

Now, due to the difference between this case and previous cases mentioned at the end of step 2, we need $3$ more mutations to bring $g_q$ back to its original position (unlike when $p$ is even we only needed $1$):

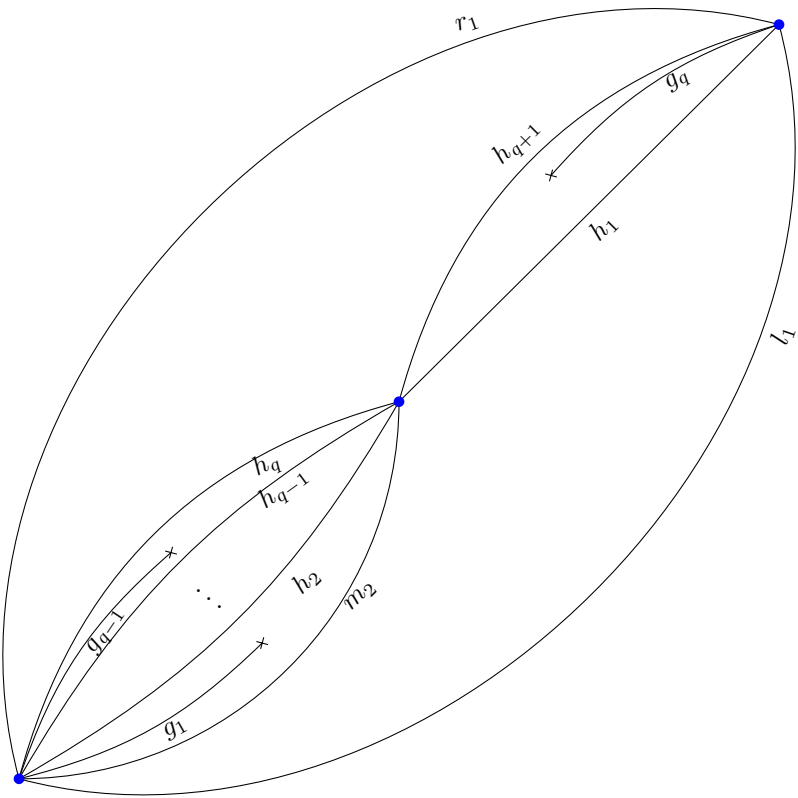
\[\begin{tikzpicture}[every edge quotes/.style={auto=right}]
    \begin{scope}[every node/.style={sloped,allow upside down}][every edge quotes/.style={auto=right}]
        \node at ($(0,5)!.5!(5,0)$) {$\ddots$};
        \draw (0,0) to [bend right=15] node [below=.15,right,black] {$g_1$} (3.2,1.8);
        \draw (7,8) to [bend left=15] node [below=.15,right,black] {$g_q$} (10,10);
        \draw (0,0) to [bend left=15] node [below=.15,right,black] {$g_{q-1}$} (2,3);
        \draw (0,0) to [bend right=45] node[below=.15,right=1] {$m_2$} (5,5);
        \draw (5,5) to [bend left=30] node[above] {$h_{q+1}$} (10,10);
        \draw (0,0) to [bend right=15] node[right=1,below,black] {$h_2$} (5,5);
        \draw (0,0) to [bend left=15] node[right=1,below=-.15,black,near end] {$h_{q-1}$} (5,5);
        \draw (0,0) to [bend left=30] node[right=1,below=-.15,black,near end] {$h_q$} (5,5);
        \draw (0,0) to [bend right=60] node[near end,below,black] {$l_{1}$} (10,10);
        \draw (0,0) to [bend left=60] node[near end,above,black] {$r_{1}$} (10,10);
        \draw (5,5) to node[below] {$h_1$} (10,10);
        \fill[blue] (0,0) circle (2pt);
        \fill[blue] (10,10) circle (2pt);
        \fill[blue] (5,5) circle (2pt);
        \fill (3.2,1.8) node[cross=2pt,rotate=30] {};
        \fill (7,8) node[cross=2pt,rotate=30] {};
        \fill (2,3) node[cross=2pt,rotate=30] {};
    \end{scope}
\end{tikzpicture}\]

Mutate at $g_q$:

\[\adjustbox{scale=0.9,center}
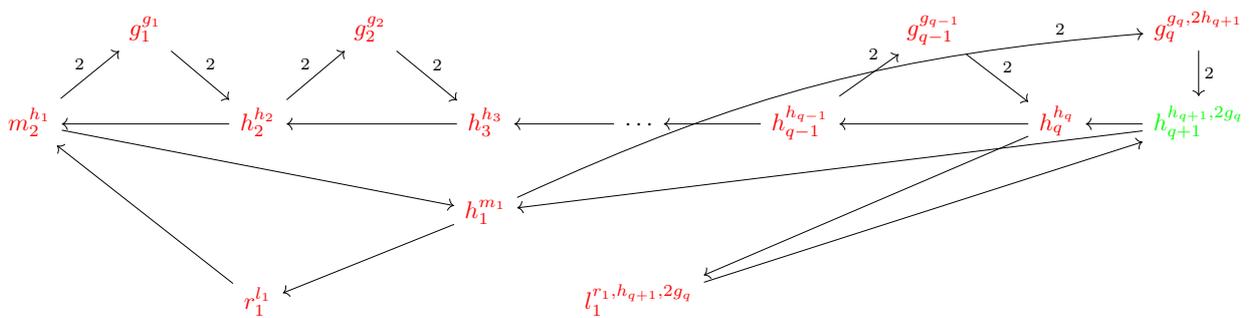
{\begin{tikzcd}
    &\color{red}g_1^{g_1}\arrow[dr,"2"]&&\color{red}g_2^{g_2}\arrow[dr,"2"]&&&&\color{red}g_{q-1}^{g_{q-1}}\arrow[dr,"2"]&&\color{red}g_q^{g_q,2h_{q+1}}\arrow[d,"2"]\\
    \color{red}m_2^{h_1}\arrow[drrrr]\arrow[ur,"2"]&&\color{red}h_2^{h_{2}}\arrow[ur,"2"]\arrow[ll]&&\color{red}h_3^{h_3}\arrow[ll]&\cdots\arrow[l]&\color{red}h_{q-1}^{h_{q-1}}\arrow[l]\arrow[ur,"2",near end]&&\color{red}h_{q}^{h_{q}}\arrow[ll]\arrow[ddlll]&\color{green}h_{q+1}^{h_{q+1},2g_q}\arrow[dlllll]\arrow[l]\\
    &&&&\color{red}h_1^{m_1}\arrow[uurrrrr,"2",pos=0.9,bend left=10]\arrow[dll]\\
    &&\color{red}r_1^{l_{1}}\arrow[uull]&&&\color{red}l_1^{r_1,h_{q+1},2g_{q}}\arrow[uurrrr]
\end{tikzcd}}\]

Mutate at $h_{q+1}$:

\[\adjustbox{scale=0.9,center}
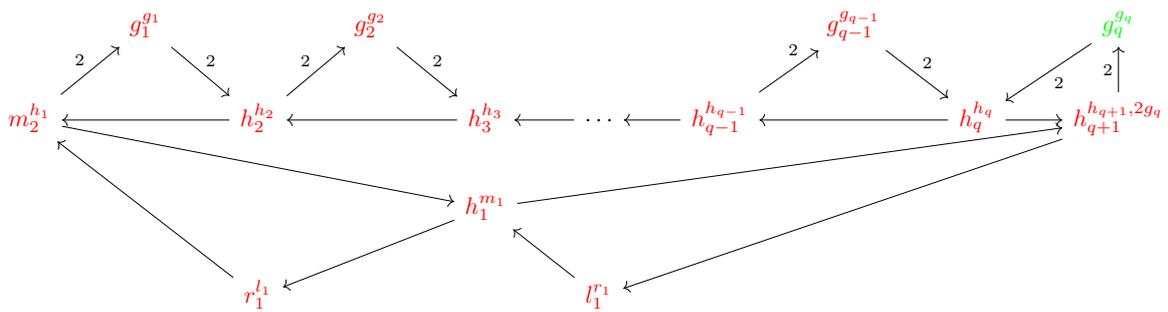
{\begin{tikzcd}
    &\color{red}g_1^{g_1}\arrow[dr,"2"]&&\color{red}g_2^{g_2}\arrow[dr,"2"]&&&&\color{red}g_{q-1}^{g_{q-1}}\arrow[dr,"2"]&&\color{green}g_q^{g_q}\arrow[dl,"2"]\\
    \color{red}m_2^{h_1}\arrow[drrrr]\arrow[ur,"2"]&&\color{red}h_2^{h_{2}}\arrow[ur,"2"]\arrow[ll]&&\color{red}h_3^{h_3}\arrow[ll]&\cdots\arrow[l]&\color{red}h_{q-1}^{h_{q-1}}\arrow[l]\arrow[ur,"2",near end]&&\color{red}h_{q}^{h_{q}}\arrow[ll]\arrow[r]&\color{red}h_{q+1}^{h_{q+1},2g_q}\arrow[ddllll]\arrow[u,"2"]\\
    &&&&\color{red}h_1^{m_1}\arrow[dll]\arrow[urrrrr]\\
    &&\color{red}r_1^{l_{1}}\arrow[uull]&&&\color{red}l_1^{r_1}\arrow[ul]
\end{tikzcd}}\]

Mutate at $g_q$:

\[\adjustbox{scale=0.9,center}
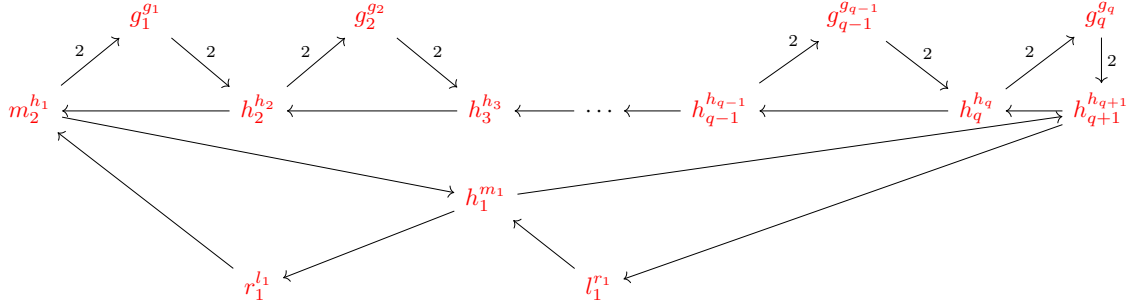
{\begin{tikzcd}
    &\color{red}g_1^{g_1}\arrow[dr,"2"]&&\color{red}g_2^{g_2}\arrow[dr,"2"]&&&&\color{red}g_{q-1}^{g_{q-1}}\arrow[dr,"2"]&&\color{red}g_q^{g_q}\arrow[d,"2"]\\
    \color{red}m_2^{h_1}\arrow[drrrr]\arrow[ur,"2"]&&\color{red}h_2^{h_{2}}\arrow[ur,"2"]\arrow[ll]&&\color{red}h_3^{h_3}\arrow[ll]&\cdots\arrow[l]&\color{red}h_{q-1}^{h_{q-1}}\arrow[l]\arrow[ur,"2",near end]&&\color{red}h_{q}^{h_{q}}\arrow[ll]\arrow[ur,"2"]&\color{red}h_{q+1}^{h_{q+1}}\arrow[ddllll]\arrow[l]\\
    &&&&\color{red}h_1^{m_1}\arrow[dll]\arrow[urrrrr]\\
    &&\color{red}r_1^{l_{1}}\arrow[uull]&&&\color{red}l_1^{r_1}\arrow[ul]
\end{tikzcd}}\]

Since all vertices are red, $\Delta_1$ is a maximal green sequence. $\blacksquare$

\subsubsection{The case where $p>4$ and is odd}\label{odd}

\begin{proposition}
    $\Delta_1$ is a maximal green sequence when $p>4$ and is odd.
\end{proposition}

\textbf{Proof} When $p>4$ and is odd, the mutation sequence $\Delta_1$ translates to

\begin{enumerate}
    \item Tagging alternate punctures notched. $(\alpha_{p-2},\alpha_{p-4},...,\alpha_{1})$
    \item Tagging the first interior puncture notched. $(h_{q+1},h_q,...,h_1,m_{2},h_2,h_3,...,h_{q})$
    \item Tagging the remaining interior punctures notched and moving half the inner arcs away from the corner. $(\beta_{p-3},\beta_{p-5},...,\beta_{2})$
    \item Moving arcs away from the corner.
    \begin{enumerate}
        \item Moving core arcs away from the corner. $(g_1,g_2,...,g_q,h_{q+1},h_{q},...,h_2,m_2)$
        \item Moving the remaining inner arcs away from the corner. $(l_{1},r_{1},l_{3},r_{3},...,l_{p-4},r_{p-4})$
        \item Moving outer arcs away from the corner. $(l_{p-1},f_1,r_{p-1},f_2,l_{p-1},f_1)$
    \end{enumerate}
    \item Tagging the corner notched. $(m_{p-2},l_{p-2},r_{p-2},m_{p-2})$
    \item Moving arcs back to the corner.
    \begin{enumerate}
        \item Moving outer arcs back to the corner. 
        $(f_1,l_{p-1},f_2,r_{p-1},f_1,l_{p-1})$
        \item Moving inner arcs back to the corner. 
        $(\delta_{p-4},\delta_{p-6},...,\delta_{1})$
        \item Moving core arcs back to the corner. $(m_2,h_2,h_3,...,h_q,g_1,g_2,...,g_{q-1},g_q,h_{q+1},g_q)$
    \end{enumerate}
\end{enumerate}

where

\begin{itemize}
    \item $\alpha_k=r_k,l_k,m_{k+1},m_{k},l_k,r_k$
    \item $\beta_k=r_k,l_k,m_{k+2},m_{k-1}$
    \item $\delta_k=r_k,r_{k+1},l_k,l_{k+1}$
\end{itemize}

In terms of triangulation, this is mostly similar to the case where $p>4$ and is even, so we will not mention the triangulations as much.

Step 1 consists of $(\alpha_{p-2},\alpha_{p-4},...,\alpha_{1})$.

For $\alpha_{p-2},\alpha_{p-4},...,\alpha_{3}$, the subdiagrams and subtriangulations are the same as the case where $p>2$ and is even, so we shall jump to the conclusion: after $\alpha_3$, the full subdiagram with vertices involved in $\alpha_k$ and adjacent vertices looks like this for $1<k\leq p-2$ (except that for $k=p-2$, $l_{k+1}$ and $r_{k+1}$ are only attached to $l_{k+1,k}$ and $r_{k+1},m_{k+1}$ respectively with weight $1$ and for $k=3$, $l_{k-1}$ and $r_{k-1}$ are not attached to $l_{k-2}$ and $m_{k-1}$ respectively yet):

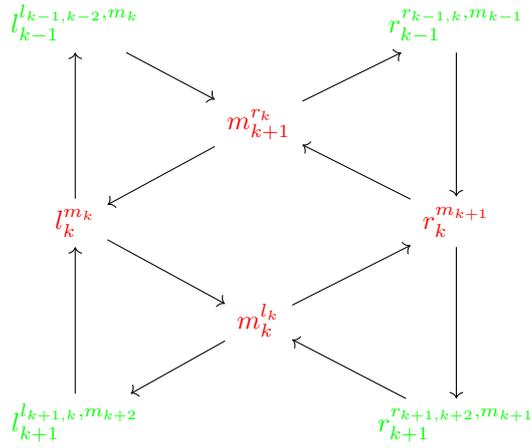
\[\begin{tikzcd}
    \color{green}l_{k-1}^{l_{k-1,k-2},m_k}\arrow[dr]&&\color{green}r_{k-1}^{r_{k-1,k},m_{k-1}}\arrow[dd]\\
    &\color{red}m_{k+1}^{r_k}\arrow[ur]\arrow[dl]&\\
    \color{red}l_{k}^{m_k}\arrow[uu]\arrow[dr]&&\color{red}r_{k}^{m_{k+1}}\arrow[ul]\arrow[dd]\\
    &\color{red}m_k^{l_k}\arrow[ur]\arrow[dl]&\\
    \color{green}l_{k+1}^{l_{k+1,k},m_{k+2}}\arrow[uu]&&\color{green}r_{k+1}^{r_{k+1,k+2},m_{k+1}}\arrow[ul]
\end{tikzcd}\]

After $\alpha_3$, at vertices involved in $\alpha_1$ and adjacent vertices, we have:

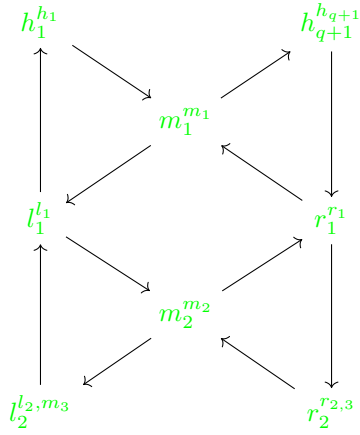
\[\begin{tikzcd}
    \color{green}h_1^{h_1}\arrow[dr]&&\color{green}h_{q+1}^{h_{q+1}}\arrow[dd]\\
    &\color{green}m_1^{m_1}\arrow[ur]\arrow[dl]&\\
    \color{green}l_{1}^{l_1}\arrow[uu]\arrow[dr]&&\color{green}r_{1}^{r_1}\arrow[ul]\arrow[dd]\\
    &\color{green}m_{2}^{m_2}\arrow[ur]\arrow[dl]&\\
    \color{green}l_{2}^{l_2,m_3}\arrow[uu]&&\color{green}r_{2}^{r_{2,3}}\arrow[ul]
\end{tikzcd}\]

After $\alpha_1$, we have:

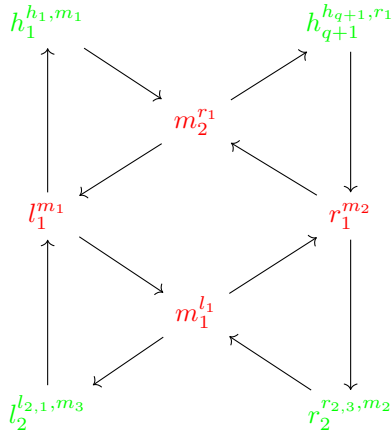
\[\begin{tikzcd}
    \color{green}h_1^{h_1,m_1}\arrow[dr]&&\color{green}h_{q+1}^{h_{q+1},r_1}\arrow[dd]\\
    &\color{red}m_2^{r_1}\arrow[ur]\arrow[dl]&\\
    \color{red}l_{1}^{m_1}\arrow[uu]\arrow[dr]&&\color{red}r_{1}^{m_2}\arrow[ul]\arrow[dd]\\
    &\color{red}m_{1}^{l_1}\arrow[ur]\arrow[dl]&\\
    \color{green}l_{2}^{l_{2,1},m_3}\arrow[uu]&&\color{green}r_{2}^{r_{2,3},m_2}\arrow[ul]
\end{tikzcd}\]

Step 2 consists of $(h_{q+1},h_q,...,h_1,m_{2},h_2,h_3,...,h_{q})$. Look at these vertices and adjacent vertices:

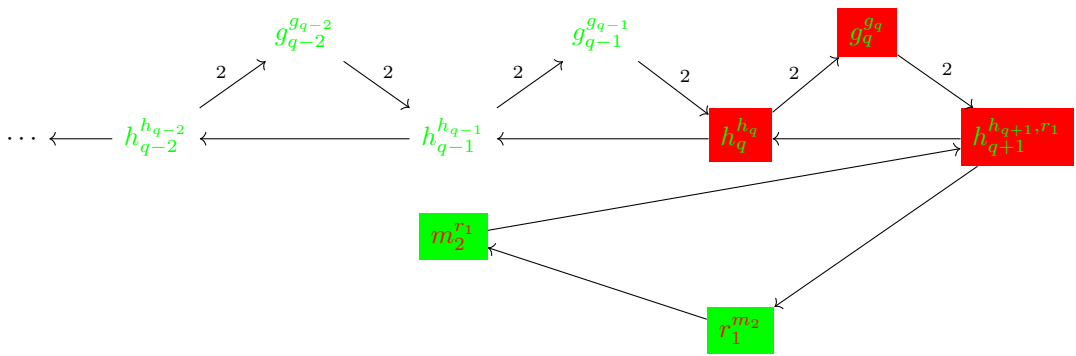
\[\begin{tikzcd}
    &&\color{green}g_{q-2}^{g_{q-2}}\arrow[dr,"2"]&&\color{green}g_{q-1}^{g_{q-1}}\arrow[dr,"2"]&&|[color=green,fill=red]|g_{q}^{g_q}\arrow[dr,"2"]\\
    \cdots&\color{green}h_{q-2}^{h_{q-2}}\arrow[ur,"2"]\arrow[l]&&\color{green}h_{q-1}^{h_{q-1}}\arrow[ur,"2"]\arrow[ll]&&|[color=green,fill=red]|h_{q}^{h_q}\arrow[ur,"2"]\arrow[ll]&&|[color=green,fill=red]|h_{q+1}^{h_{q+1},r_1}\arrow[ll]\arrow[ddll]\\
    &&&|[color=red,fill=green]|m_2^{r_1}\arrow[urrrr]\\
    &&&&&|[color=red,fill=green]|r_1^{m_2}\arrow[ull]
\end{tikzcd}\]

Observe that the subdiagram is same as the case where $p=3$, therefore after step 2, we have:

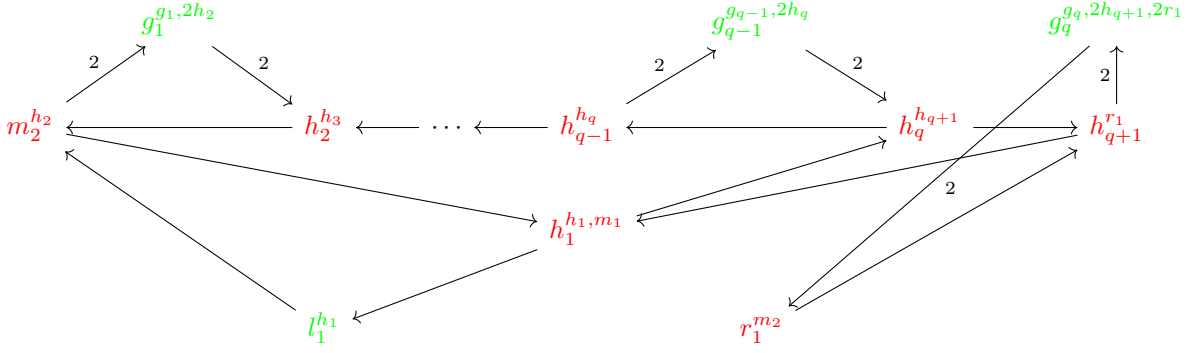
\begin{tikzcd}
    &\color{green}g_1^{g_1,2h_2}\arrow[dr,"2"]&&&&\color{green}g_{q-1}^{g_{q-1},2h_q}\arrow[dr,"2"]&&\color{green}g_q^{g_q,2h_{q+1},2r_1}\arrow[dddll,"2"]\\
    \color{red}m_2^{h_2}\arrow[drrrr]\arrow[ur,"2"]&&\color{red}h_2^{h_3}\arrow[ll]&\cdots\arrow[l]&\color{red}h_{q-1}^{h_q}\arrow[ur,"2"]\arrow[l]&&\color{red}h_{q}^{h_{q+1}}\arrow[ll]\arrow[r]&\color{red}h_{q+1}^{r_1}\arrow[u,"2"]\arrow[dlll]\\
    &&&&\color{red}h_1^{h_1,m_1}\arrow[urr]\arrow[dll]\\
    &&\color{green}l_1^{h_1}\arrow[uull]&&&\color{red}r_1^{m_2}\arrow[uurr]
\end{tikzcd}

Step 3 consists of $(\beta_{p-3},\beta_{p-5},...,\beta_{2})$.

Again this is similar to the case where $p>4$ and is even, so we shall jump to the conclusion: after $\beta_4$ the full subdiagram with vertices involved in $\beta_k$ and adjacent vertices looks like this for $2<k\leq p-3$ (except that for $k=p-3$, $r_{k+1}$ is red and is attached to $m_{k+2}$ with weight $1$ instead and for $k=4$, $l_{k-1}$ has not become the current state yet):

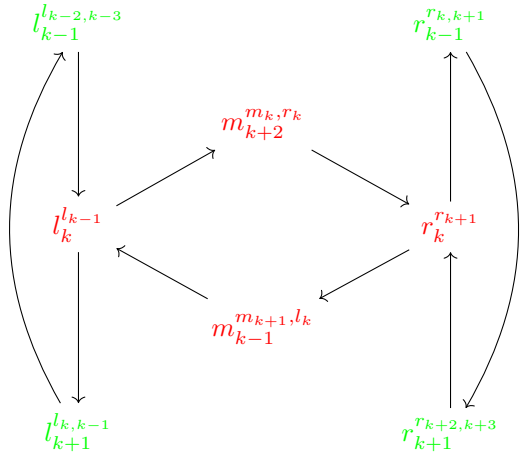
\[\begin{tikzcd}
    \color{green}l_{k-1}^{l_{k-2,k-3}}\arrow[dd]&&\color{green}r_{k-1}^{r_{k,k+1}}\arrow[dddd,bend left]\\
    &\color{red}m_{k+2}^{m_k,r_k}\arrow[dr]&\\
    \color{red}l_{k}^{l_{k-1}
    }\arrow[dd]\arrow[ur]&&\color{red}r_{k}^{r_{k+1}}\arrow[dl]\arrow[uu]\\
    &\color{red}m_{k-1}^{m_{k+1},l_k}\arrow[ul]&\\
    \color{green}l_{k+1}^{l_{k,k-1}}\arrow[uuuu,bend left]&&\color{green}r_{k+1}^{r_{k+2,k+3}}\arrow[uu]
\end{tikzcd}\]

After $\beta_4$, at the vertices involved in $\beta_2$ and adjacent vertices, we have:

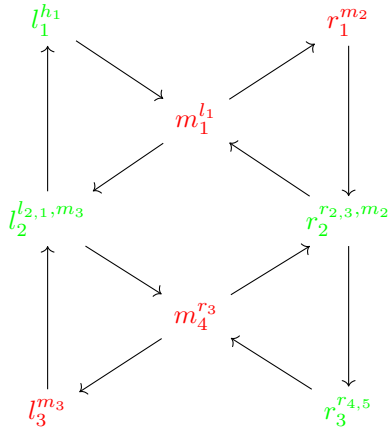
\[\begin{tikzcd}
    \color{green}l_1^{h_1}\arrow[dr]&&\color{red}r_1^{m_2}\arrow[dd]\\
    &\color{red}m_1^{l_1}\arrow[ur]\arrow[dl]&\\
    \color{green}l_{2}^{l_{2,1},m_{3}}\arrow[uu]\arrow[dr]&&\color{green}r_{2}^{r_{2,3},m_2}\arrow[ul]\arrow[dd]\\
    &\color{red}m_{4}^{r_{3}}\arrow[ur]\arrow[dl]&\\
    \color{red}l_{3}^{m_{3}}\arrow[uu]&&\color{green}r_{3}^{r_{4,5}}\arrow[ul]
\end{tikzcd}\]

After $\beta_2$, we have:

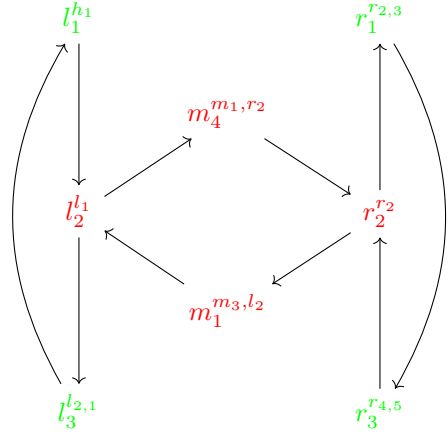
\[\begin{tikzcd}
    \color{green}l_1^{h_1}\arrow[dd]&&\color{green}r_1^{r_{2,3}}\arrow[dddd,bend left]\\
    &\color{red}m_{4}^{m_1,r_2}\arrow[dr]&\\
    \color{red}l_{2}^{l_1}\arrow[dd]\arrow[ur]&&\color{red}r_{2}^{r_2}\arrow[dl]\arrow[uu]\\
    &\color{red}m_1^{m_3,l_2}\arrow[ul]&\\
    \color{green}l_{3}^{l_{2,1}}\arrow[uuuu,bend left]&&\color{green}r_{3}^{r_{4,5}}\arrow[uu]
\end{tikzcd}\]

Step 4(a) consists of $(g_1,g_2,...,g_q,h_{q+1},h_{q},...,h_2,m_2)$.

Now look at $g_1,g_2,...,g_q,h_q,h_{q-1},...,h_2,m_2$, and adjacent vertices:

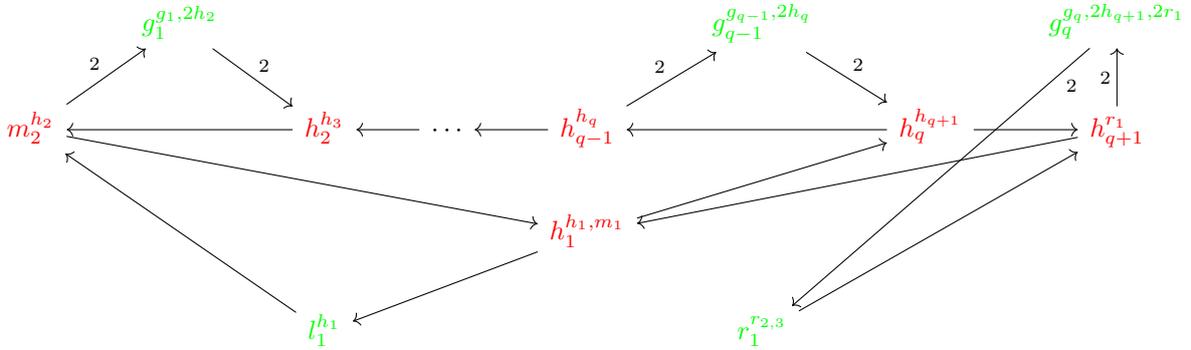
\[\begin{tikzcd}
    &\color{green}g_1^{g_1,2h_2}\arrow[dr,"2"]&&&&\color{green}g_{q-1}^{g_{q-1},2h_q}\arrow[dr,"2"]&&\color{green}g_q^{g_q,2h_{q+1},2r_1}\arrow[dddll,"2",pos=0.1]\\
    \color{red}m_2^{h_2}\arrow[drrrr]\arrow[ur,"2"]&&\color{red}h_2^{h_3}\arrow[ll]&\cdots\arrow[l]&\color{red}h_{q-1}^{h_q}\arrow[ur,"2"]\arrow[l]&&\color{red}h_{q}^{h_{q+1}}\arrow[ll]\arrow[r]&\color{red}h_{q+1}^{r_1}\arrow[u,"2"]\arrow[dlll]\\
    &&&&\color{red}h_1^{h_1,m_1}\arrow[urr]\arrow[dll]\\
    &&\color{green}l_1^{h_1}\arrow[uull]&&&\color{green}r_1^{r_{2,3}}\arrow[uurr]
\end{tikzcd}\]

By applying Lemma \ref{p>3}, after step 4(a), we have:

$\adjustbox{scale=0.9,center}
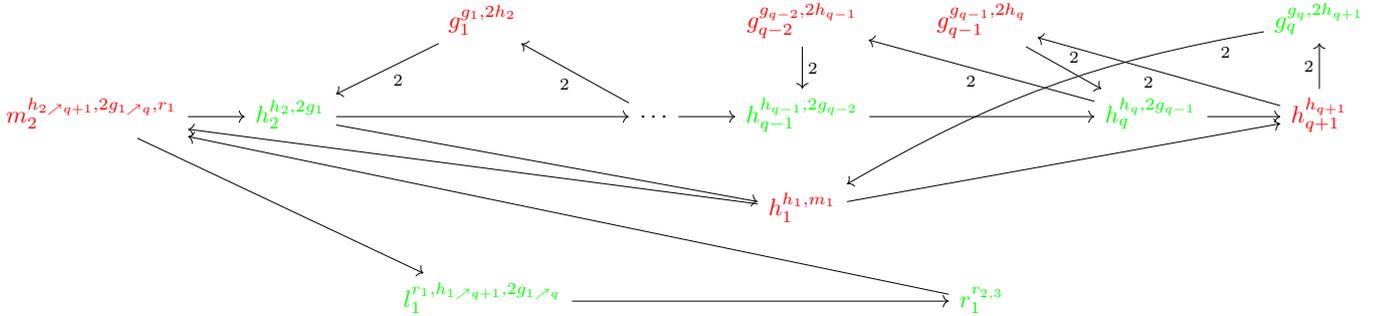
{\begin{tikzcd}
    &&\color{red}g_1^{g_1,2h_2}\arrow[dl,"2"]&&\color{red}g_{q-2}^{g_{q-2},2h_{q-1}}\arrow[d,"2"]&\color{red}g_{q-1}^{g_{q-1},2h_q}\arrow[dr,"2"]&&\color{green}g_q^{g_q,2h_{q+1}}\arrow[ddlll,"2",pos=0.1,bend right=10]\\
    \color{red}m_2^{h_{2\nearrow q+1},2g_{1\nearrow q},r_1}\arrow[r]\arrow[ddrr]&\color{green}h_2^{h_{2},2g_{1}}\arrow[rr]\arrow[drrr]&&\cdots\arrow[r]\arrow[ul,"2"]&\color{green}h_{q-1}^{h_{q-1},2g_{q-2}}\arrow[rr]&&\color{green}h_{q}^{h_{q},2g_{q-1}}\arrow[r]\arrow[ull,"2"]&\color{red}h_{q+1}^{h_{q+1}}\arrow[ull,"2"]\arrow[u,"2"]\\
    &&&&\color{red}h_1^{h_1,m_1}\arrow[urrr]\arrow[ullll]\\
    &&\color{green}l_1^{r_1,h_{1\nearrow q+1},2g_{1\nearrow q}}\arrow[rrr]&&&\color{green}r_1^{r_{2,3}}\arrow[uulllll]
\end{tikzcd}}$

For simplicity we shall write $G=h_{1\nearrow q+1},2g_{1\nearrow q}$.

Step 4(b) consists of $l_1,r_1,l_3,r_3,...,l_{p-4},r_{p-4}$. At these and adjacent vertices, the diagram looks like this:

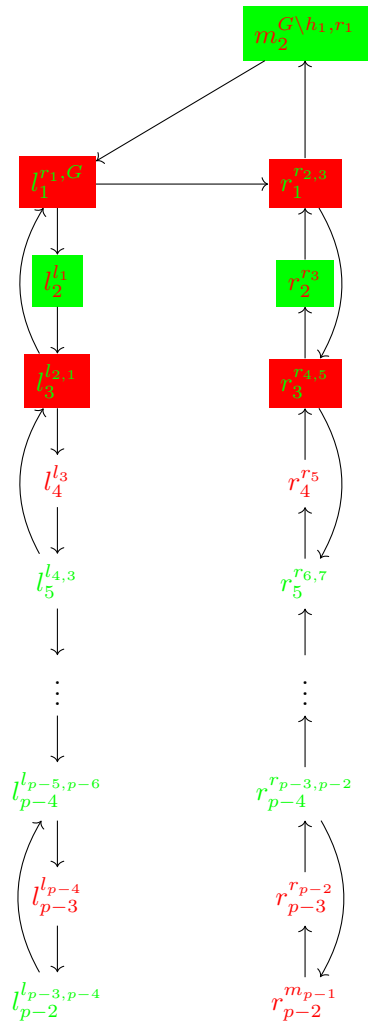
\[\begin{tikzcd}
    &&|[color=red,fill=green]|m_2^{G\setminus h_1,r_1}\arrow[ddll]\\\\
    |[color=green,fill=red]|l_1^{r_1,G}\arrow[d]\arrow[rr]&&|[color=green,fill=red]|r_1^{r_{2,3}}\arrow[uu]\arrow[dd,bend left]\\
    |[color=red,fill=green]|l_2^{l_1}\arrow[d]&&|[color=red,fill=green]|r_2^{r_3}\arrow[u]\\
    |[color=green,fill=red]|l_3^{l_{2,1}}\arrow[d]\arrow[uu,bend left]&&|[color=green,fill=red]|r_3^{r_{4,5}}\arrow[u]\arrow[dd,bend left]\\
    \color{red}l_4^{l_3}\arrow[d]&&\color{red}r_4^{r_5}\arrow[u]\\
    \color{green}l_5^{l_{4,3}}\arrow[uu,bend left]\arrow[d]&&\color{green}r_5^{r_{6,7}}\arrow[u]\\
    \vdots\arrow[d]&&\vdots\arrow[u]\\
    \color{green}l_{p-4}^{l_{p-5,p-6}}\arrow[d]&&\color{green}r_{p-4}^{r_{p-3,p-2}}\arrow[u]\arrow[dd,bend left]\\
    \color{red}l_{p-3}^{l_{p-4}}\arrow[d]&&\color{red}r_{p-3}^{r_{p-2}}\arrow[u]\\
    \color{green}l_{p-2}^{l_{p-3,p-4}}\arrow[uu,bend left]&&\color{red}r_{p-2}^{m_{p-1}}\arrow[u]
\end{tikzcd}\]

Mutate at $l_1,r_1$:

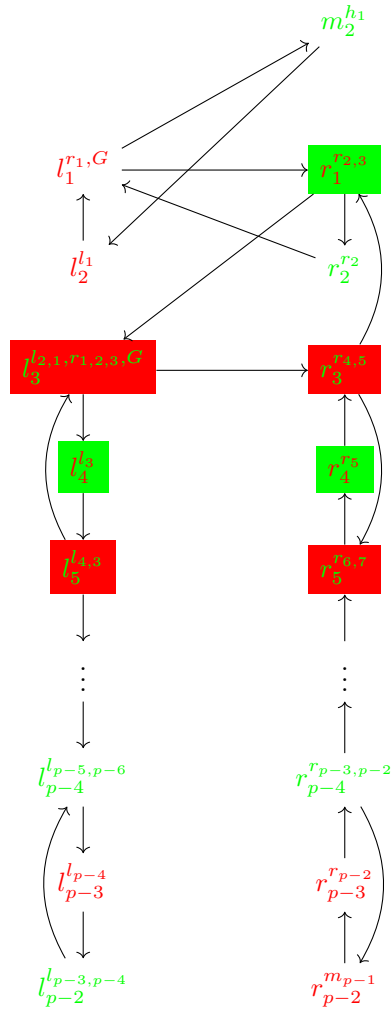
\[\begin{tikzcd}
    &&\color{green}m_2^{h_1}\arrow[dddll]\\\\
    \color{red}l_1^{r_1,G}\arrow[rr]\arrow[uurr]&&|[color=red,fill=green]|r_1^{r_{2,3}}\arrow[d]\arrow[ddll]\\
    \color{red}l_2^{l_1}\arrow[u]&&\color{green}r_2^{r_2}\arrow[ull]\\
    |[color=green,fill=red]|l_3^{l_{2,1},r_{1,2,3},G}\arrow[d]\arrow[rr]&&|[color=green,fill=red]|r_3^{r_{4,5}}\arrow[uu,bend right]\arrow[dd,bend left]\\
    |[color=red,fill=green]|l_4^{l_3}\arrow[d]&&|[color=red,fill=green]|r_4^{r_5}\arrow[u]\\
    |[color=green,fill=red]|l_5^{l_{4,3}}\arrow[uu,bend left]\arrow[d]&&|[color=green,fill=red]|r_5^{r_{6,7}}\arrow[u]\\
    \vdots\arrow[d]&&\vdots\arrow[u]\\
    \color{green}l_{p-4}^{l_{p-5,p-6}}\arrow[d]&&\color{green}r_{p-4}^{r_{p-3,p-2}}\arrow[u]\arrow[dd,bend left]\\
    \color{red}l_{p-3}^{l_{p-4}}\arrow[d]&&\color{red}r_{p-3}^{r_{p-2}}\arrow[u]\\
    \color{green}l_{p-2}^{l_{p-3,p-4}}\arrow[uu,bend left]&&\color{red}r_{p-2}^{m_{p-1}}\arrow[u]
\end{tikzcd}\]

Similar to the case where $p>4$ and is even, after $l_1,r_1,l_3,r_3,...,l_{p-4},r_{p-4}$, we have:

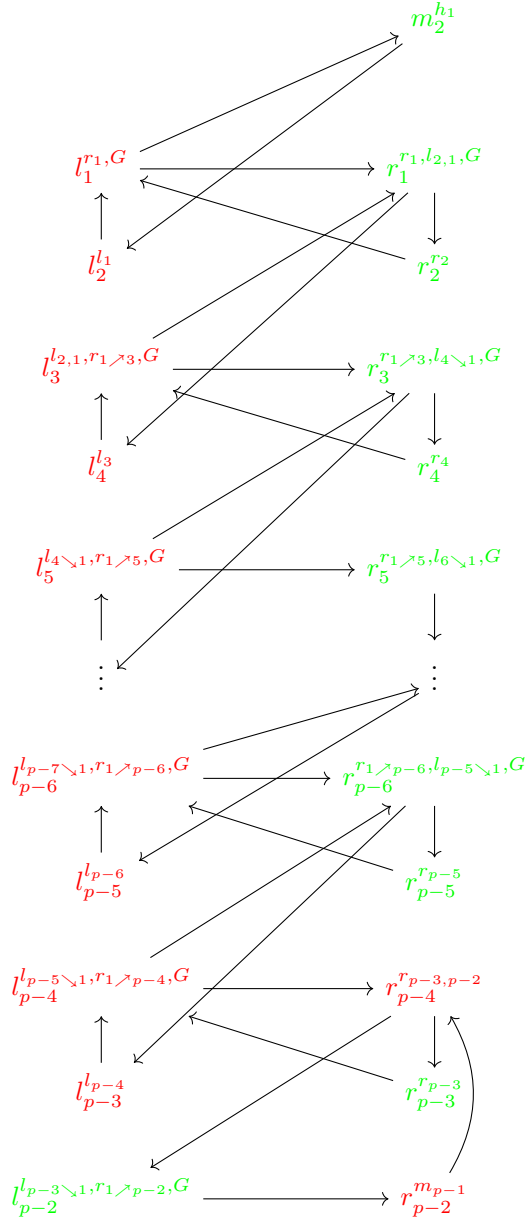
\[\begin{tikzcd}
    &&\color{green}m_2^{h_1}\arrow[dddll]\\\\
    \color{red}l_1^{r_1,G}\arrow[uurr]\arrow[rr]&&\color{green}r_1^{r_1,l_{2,1},G}\arrow[d]\arrow[dddll]\\
    \color{red}l_2^{l_1}\arrow[u]&&\color{green}r_2^{r_2}\arrow[ull]\\
    \color{red}l_3^{l_{2,1},r_{1\nearrow3},G}\arrow[uurr]\arrow[rr]&&\color{green}r_3^{r_{1\nearrow3},l_{4\searrow1},G}\arrow[d]\arrow[dddll]\\
    \color{red}l_4^{l_{3}}\arrow[u]&&\color{green}r_4^{r_{4}}\arrow[ull]\\
    \color{red}l_5^{l_{4\searrow1},r_{1\nearrow5},G}\arrow[uurr]\arrow[rr]&&\color{green}r_5^{r_{1\nearrow5},l_{6\searrow1},G}\arrow[d]\\
    \vdots\arrow[u]&&\vdots\arrow[ddll]\\
    \color{red}l_{p-6}^{l_{p-7\searrow1},r_{1\nearrow p-6},G}\arrow[urr]\arrow[rr]&&\color{green}r_{p-6}^{r_{1\nearrow p-6},l_{p-5\searrow1},G}\arrow[d]\arrow[dddll]\\
    \color{red}l_{p-5}^{l_{p-6}}\arrow[u]&&\color{green}r_{p-5}^{r_{p-5}}\arrow[ull]\\
    \color{red}l_{p-4}^{l_{p-5\searrow1},r_{1\nearrow p-4},G}\arrow[rr]\arrow[uurr]&&\color{red}r_{p-4}^{r_{p-3,p-2}}\arrow[d]\arrow[ddll]\\
    \color{red}l_{p-3}^{l_{p-4}}\arrow[u]&&\color{green}r_{p-3}^{r_{p-3}}\arrow[ull]\\
    \color{green}l_{p-2}^{l_{p-3\searrow1},r_{1\nearrow p-2},G}\arrow[rr]&&\color{red}r_{p-2}^{m_{p-1}}\arrow[uu,bend right]
\end{tikzcd}\]

Just like in the previous cases, we shall merge step 4(c), 5, and 6(a).

At $l_{p-1},r_{p-1},f_1,f_2,m_{p-2},l_{p-2},r_{p-2}$ and adjacent vertices, we have:

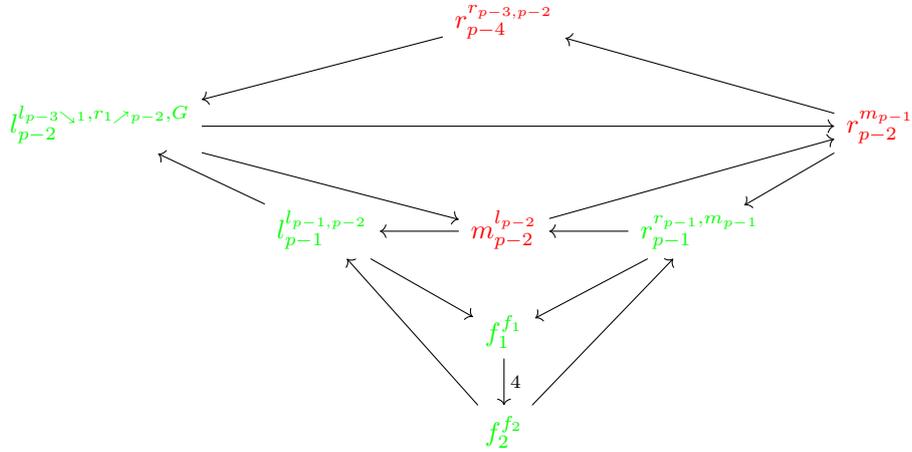
\[\begin{tikzcd}
    &&\color{red}r_{p-4}^{r_{p-3,p-2}}\arrow[dll]&&\\
    \color{green}l_{p-2}^{l_{p-3\searrow1},r_{1\nearrow p-2},G}\arrow[rrrr]\arrow[drr]&&&&\color{red}r_{p-2}^{m_{p-1}}\arrow[dl]\arrow[ull]\\
    &\color{green}l_{p-1}^{l_{p-1,p-2}}\arrow[ul]\arrow[dr]&\color{red}m_{p-2}^{l_{p-2}}\arrow[l]\arrow[urr]&\color{green}r_{p-1}^{r_{p-1},m_{p-1}}\arrow[l]\arrow[dl]\\
    &&\color{green}f_1^{f_1}\arrow[d,"4"]&\\
    &&\color{green}f_2^{f_2}\arrow[uul]\arrow[uur]&
\end{tikzcd}\]

By applying Lemma \ref{corner}, after mutating at $l_{p-1},f_1,r_{p-1},f_2,l_{p-1},f_1,m_{p-2},l_{p-2},r_{p-2},m_{p-2},f_1,l_{p-1},f_2,r_{p-1},f_1,l_{p-1}$, we have:

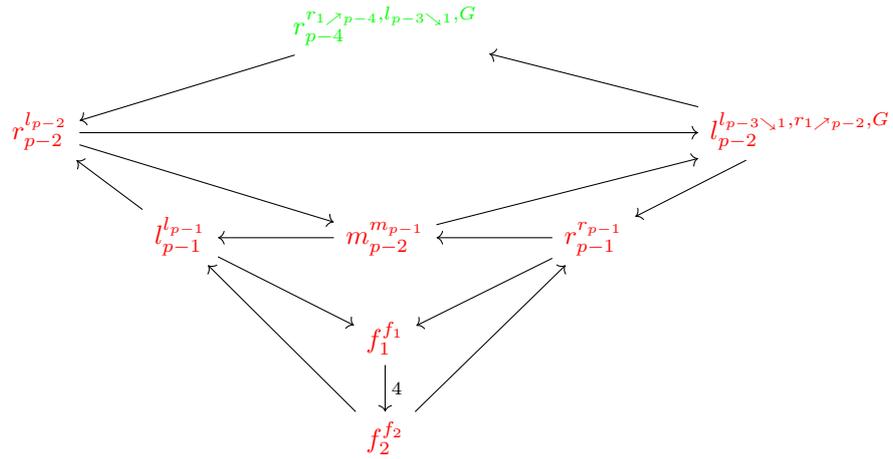
\[\begin{tikzcd}
    &&\color{green}r_{p-4}^{r_{1\nearrow p-4},l_{p-3\searrow1},G}\arrow[dll]&&\\
    \color{red}r_{p-2}^{l_{p-2}}\arrow[rrrr]\arrow[drr]&&&&\color{red}l_{p-2}^{l_{p-3\searrow1},r_{1\nearrow p-2},G}\arrow[dl]\arrow[ull]\\
    &\color{red}l_{p-1}^{l_{p-1}}\arrow[ul]\arrow[dr]&\color{red}m_{p-2}^{m_{p-1}}\arrow[l]\arrow[urr]&\color{red}r_{p-1}^{r_{p-1}}\arrow[l]\arrow[dl]\\
    &&\color{red}f_1^{f_1}\arrow[d,"4"]&\\
    &&\color{red}f_2^{f_2}\arrow[uul]\arrow[uur]&
\end{tikzcd}\]

Step 6(b) consists of $(\delta_{p-4},\delta_{p-6},...,\delta_{1})$.

Now look at the vertices involved in $\delta_k$ for $k=3,5,...,p-4$ and adjacent vertices:

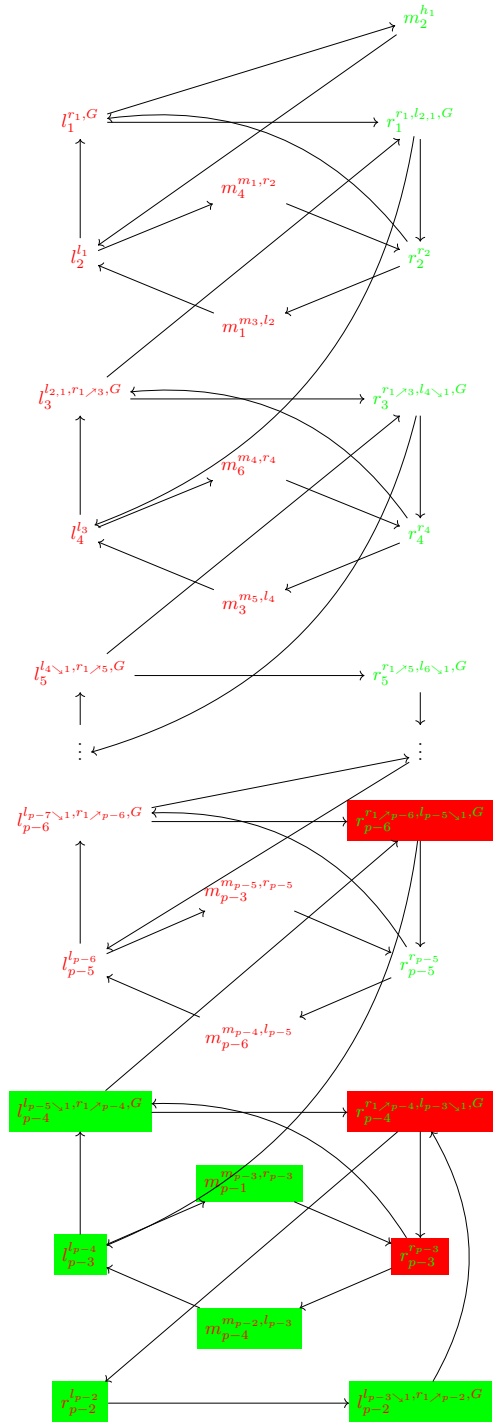
\[\adjustbox{scale=0.7,center}{\begin{tikzcd}
    &&\color{green}m_2^{h_1}\arrow[ddddll]\\\\
    \color{red}l_1^{r_1,G}\arrow[uurr]\arrow[rr]&&\color{green}r_1^{r_1,l_{2,1},G}\arrow[dd]\arrow[ddddddll,bend left,shift right]\\
    &\color{red}m_4^{m_1,r_2}\arrow[dr]&
    \\
    \color{red}l_2^{l_1}\arrow[uu]\arrow[ur]&&\color{green}r_2^{r_2}\arrow[uull,bend right]\arrow[dl]\\
    &\color{red}m_1^{m_3,l_2}\arrow[ul]&\\
    \color{red}l_3^{l_{2,1},r_{1\nearrow3},G}\arrow[uuuurr]\arrow[rr]&&\color{green}r_3^{r_{1\nearrow3},l_{4\searrow1},G}\arrow[dd]\arrow[dddddll,bend left]\\
    &\color{red}m_{6}^{m_4,r_4}\arrow[dr]&\\
    \color{red}l_4^{l_3}\arrow[uu]\arrow[ur]&&\color{green}r_4^{r_4}\arrow[uull,bend right]\arrow[dl]\\
    &\color{red}m_3^{m_5,l_4}\arrow[ul]&\\
    \color{red}l_5^{l_{4\searrow1},r_{1\nearrow5},G}\arrow[uuuurr]\arrow[rr]&&\color{green}r_5^{r_{1\nearrow5},l_{6\searrow1},G}\arrow[d]\\
    \vdots\arrow[u]&&\vdots\arrow[dddll]\\
    \color{red}l_{p-6}^{l_{p-7\searrow1},r_{1\nearrow p-6},G}\arrow[urr]\arrow[rr]&&|[color=green,fill=red]|r_{p-6}^{r_{1\nearrow p-6},l_{p-5\searrow1},G}\arrow[dd]\arrow[ddddddll,bend left]\\
    &\color{red}m_{p-3}^{m_{p-5},r_{p-5}}\arrow[dr]&\\
    \color{red}l_{p-5}^{l_{p-6}}\arrow[uu]\arrow[ur]&&\color{green}r_{p-5}^{r_{p-5}}\arrow[uull,bend right]\arrow[dl]\\
    &\color{red}m_{p-6}^{m_{p-4},l_{p-5}}\arrow[ul]&\\
    |[color=red,fill=green]|l_{p-4}^{l_{p-5\searrow1},r_{1\nearrow p-4},G}\arrow[rr]\arrow[uuuurr]&&|[color=green,fill=red]|r_{p-4}^{r_{1\nearrow p-4},l_{p-3\searrow1},G}\arrow[dd]\arrow[ddddll]\\
    &|[color=red,fill=green]|m_{p-1}^{m_{p-3},r_{p-3}}\arrow[dr]&\\
    |[color=red,fill=green]|l_{p-3}^{l_{p-4}}\arrow[uu]\arrow[ur]&&|[color=green,fill=red]|r_{p-3}^{r_{p-3}}\arrow[uull,bend right]\arrow[dl]\\
    &|[color=red,fill=green]|m_{p-4}^{m_{p-2},l_{p-3}}\arrow[ul]&\\
    |[color=red,fill=green]|r_{p-2}^{l_{p-2}}\arrow[rr]&&|[color=red,fill=green]|l_{p-2}^{l_{p-3\searrow1},r_{1\nearrow p-2},G}\arrow[uuuu,bend right]
\end{tikzcd}}\]

Mutating at $r_{p-4},r_{p-3},l_{p-4},l_{p-3}$ gives:

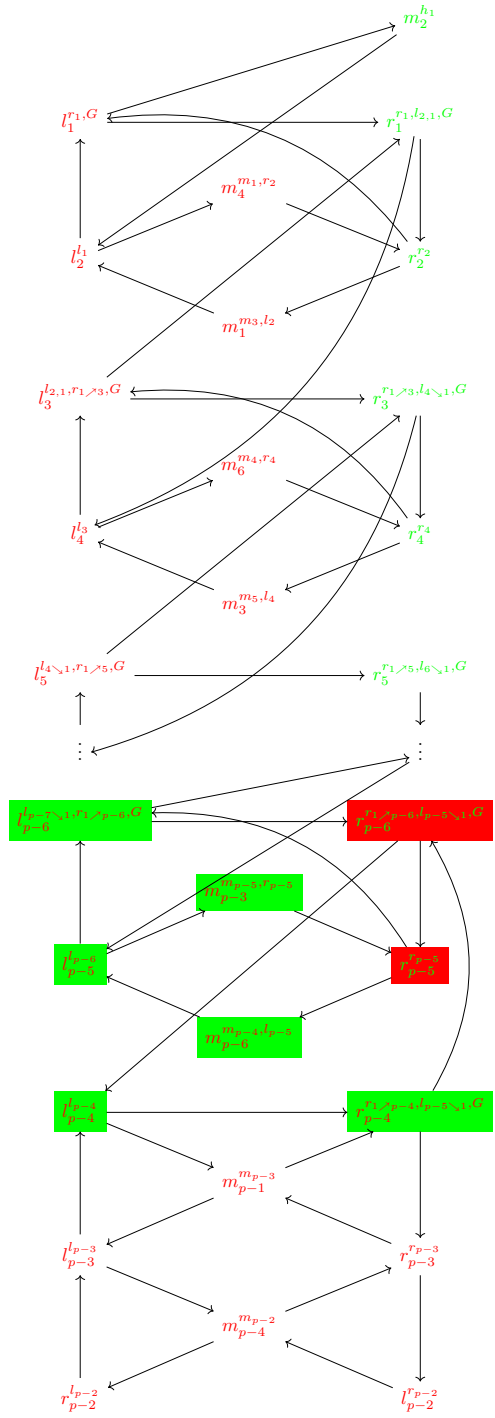
\[\adjustbox{scale=0.7,center}{\begin{tikzcd}
    &&\color{green}m_2^{h_1}\arrow[ddddll]\\\\
    \color{red}l_1^{r_1,G}\arrow[uurr]\arrow[rr]&&\color{green}r_1^{r_1,l_{2,1},G}\arrow[dd]\arrow[ddddddll,bend left,shift right]\\
    &\color{red}m_4^{m_1,r_2}\arrow[dr]&
    \\
    \color{red}l_2^{l_1}\arrow[uu]\arrow[ur]&&\color{green}r_2^{r_2}\arrow[uull,bend right]\arrow[dl]\\
    &\color{red}m_1^{m_3,l_2}\arrow[ul]&\\
    \color{red}l_3^{l_{2,1},r_{1\nearrow3},G}\arrow[uuuurr]\arrow[rr]&&\color{green}r_3^{r_{1\nearrow3},l_{4\searrow1},G}\arrow[dd]\arrow[dddddll,bend left]\\
    &\color{red}m_{6}^{m_4,r_4}\arrow[dr]&\\
    \color{red}l_4^{l_3}\arrow[uu]\arrow[ur]&&\color{green}r_4^{r_4}\arrow[uull,bend right]\arrow[dl]\\
    &\color{red}m_3^{m_5,l_4}\arrow[ul]&\\
    \color{red}l_5^{l_{4\searrow1},r_{1\nearrow5},G}\arrow[uuuurr]\arrow[rr]&&\color{green}r_5^{r_{1\nearrow5},l_{6\searrow1},G}\arrow[d]\\
    \vdots\arrow[u]&&\vdots\arrow[dddll]\\
    |[color=red,fill=green]|l_{p-6}^{l_{p-7\searrow1},r_{1\nearrow p-6},G}\arrow[urr]\arrow[rr]&&|[color=green,fill=red]|r_{p-6}^{r_{1\nearrow p-6},l_{p-5\searrow1},G}\arrow[dd]\arrow[ddddll]\\
    &|[color=red,fill=green]|m_{p-3}^{m_{p-5},r_{p-5}}\arrow[dr]&\\
    |[color=red,fill=green]|l_{p-5}^{l_{p-6}}\arrow[uu]\arrow[ur]&&|[color=green,fill=red]|r_{p-5}^{r_{p-5}}\arrow[uull,bend right]\arrow[dl]\\
    &|[color=red,fill=green]|m_{p-6}^{m_{p-4},l_{p-5}}\arrow[ul]&\\
    |[color=red,fill=green]|l_{p-4}^{l_{p-4}}\arrow[dr]\arrow[rr]&&|[color=red,fill=green]|r_{p-4}^{r_{1\nearrow p-4},l_{p-5\searrow1},G}\arrow[dd]\arrow[uuuu,bend right]\\
    &\color{red}m_{p-1}^{m_{p-3}}\arrow[dl]\arrow[ur]&\\
    \color{red}l_{p-3}^{l_{p-3}}\arrow[dr]\arrow[uu]&&\color{red}r_{p-3}^{r_{p-3}}\arrow[ul]\arrow[dd]\\
    &\color{red}m_{p-4}^{m_{p-2}}\arrow[dl]\arrow[ur]&\\
    \color{red}r_{p-2}^{l_{p-2}}\arrow[uu]&&\color{red}l_{p-2}^{r_{p-2}}\arrow[ul]
\end{tikzcd}}\]

By induction, after $\delta_3$, the diagram is:

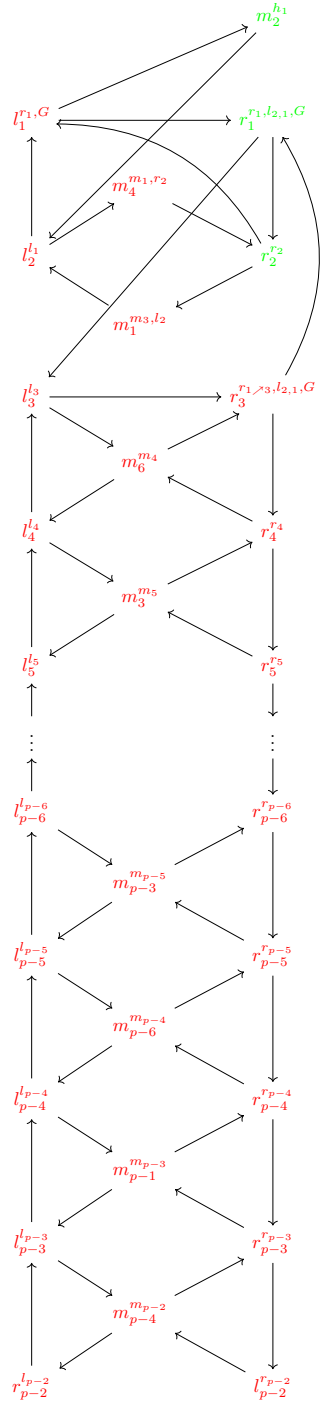
\[\adjustbox{scale=0.7,center}{\begin{tikzcd}
    &&\color{green}m_2^{h_1}\arrow[ddddll]\\\\
    \color{red}l_1^{r_1,G}\arrow[uurr]\arrow[rr]&&\color{green}r_1^{r_1,l_{2,1},G}\arrow[dd]\arrow[ddddll]\\
    &\color{red}m_4^{m_1,r_2}\arrow[dr]&
    \\
    \color{red}l_2^{l_1}\arrow[uu]\arrow[ur]&&\color{green}r_2^{r_2}\arrow[uull,bend right,shift left]\arrow[dl]\\
    &\color{red}m_1^{m_3,l_2}\arrow[ul]&\\
    \color{red}l_3^{l_3}\arrow[rr]\arrow[dr]&&\color{red}r_3^{r_{1\nearrow3},l_{2,1},G}\arrow[dd]\arrow[uuuu,bend right]\\
    &\color{red}m_{6}^{m_4}\arrow[ur]\arrow[dl]&\\
    \color{red}l_4^{l_4}\arrow[uu]\arrow[dr]&&\color{red}r_4^{r_4}\arrow[dd]\arrow[ul]\\
    &\color{red}m_3^{m_5}\arrow[dl]\arrow[ur]&\\
    \color{red}l_5^{l_5}\arrow[uu]&&\color{red}r_5^{r_5}\arrow[d]\arrow[ul]\\
    \vdots\arrow[u]&&\vdots\arrow[d]\\
    \color{red}l_{p-6}^{l_{p-6}}\arrow[u]\arrow[dr]&&\color{red}r_{p-6}^{r_{p-6}}\arrow[dd]\\
    &\color{red}m_{p-3}^{m_{p-5}}\arrow[ur]\arrow[dl]&\\
    \color{red}l_{p-5}^{l_{p-5}}\arrow[uu]\arrow[dr]&&\color{red}r_{p-5}^{r_{p-5}}\arrow[dd]\arrow[ul]\\
    &\color{red}m_{p-6}^{m_{p-4}}\arrow[ur]\arrow[dl]&\\
    \color{red}l_{p-4}^{l_{p-4}}\arrow[dr]\arrow[uu]&&\color{red}r_{p-4}^{r_{p-4}}\arrow[dd]\arrow[ul]\\
    &\color{red}m_{p-1}^{m_{p-3}}\arrow[dl]\arrow[ur]&\\
    \color{red}l_{p-3}^{l_{p-3}}\arrow[dr]\arrow[uu]&&\color{red}r_{p-3}^{r_{p-3}}\arrow[ul]\arrow[dd]\\
    &\color{red}m_{p-4}^{m_{p-2}}\arrow[dl]\arrow[ur]&\\
    \color{red}r_{p-2}^{l_{p-2}}\arrow[uu]&&\color{red}l_{p-2}^{r_{p-2}}\arrow[ul]
\end{tikzcd}}\]

After $\delta_1$:

\[\adjustbox{scale=0.7,center}
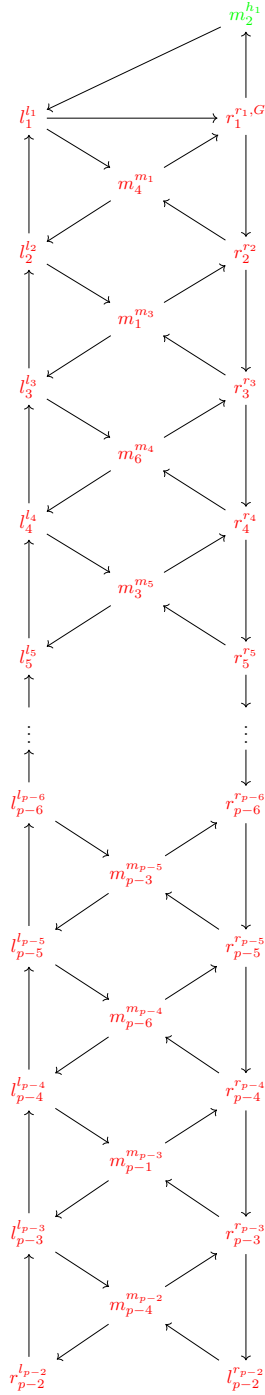
{\begin{tikzcd}
    &&\color{green}m_2^{h_1}\arrow[ddll]\\\\
    \color{red}l_1^{l_{1}}\arrow[rr]\arrow[dr]&&\color{red}r_1^{r_1,G}\arrow[dd]\arrow[uu]\\
    &\color{red}m_4^{m_1}\arrow[dl]\arrow[ur]&
    \\
    \color{red}l_2^{l_2}\arrow[uu]\arrow[dr]&&\color{red}r_2^{r_2}\arrow[ul]\arrow[dd]\\
    &\color{red}m_1^{m_3}\arrow[dl]\arrow[ur]&\\
    \color{red}l_3^{l_3}\arrow[uu]\arrow[dr]&&\color{red}r_3^{r_{3}}\arrow[dd]\arrow[ul]\\
    &\color{red}m_{6}^{m_4}\arrow[ur]\arrow[dl]&\\
    \color{red}l_4^{l_4}\arrow[uu]\arrow[dr]&&\color{red}r_4^{r_4}\arrow[dd]\arrow[ul]\\
    &\color{red}m_3^{m_5}\arrow[dl]\arrow[ur]&\\
    \color{red}l_5^{l_5}\arrow[uu]&&\color{red}r_5^{r_5}\arrow[d]\arrow[ul]\\
    \vdots\arrow[u]&&\vdots\arrow[d]\\
    \color{red}l_{p-6}^{l_{p-6}}\arrow[u]\arrow[dr]&&\color{red}r_{p-6}^{r_{p-6}}\arrow[dd]\\
    &\color{red}m_{p-3}^{m_{p-5}}\arrow[ur]\arrow[dl]&\\
    \color{red}l_{p-5}^{l_{p-5}}\arrow[uu]\arrow[dr]&&\color{red}r_{p-5}^{r_{p-5}}\arrow[dd]\arrow[ul]\\
    &\color{red}m_{p-6}^{m_{p-4}}\arrow[ur]\arrow[dl]&\\
    \color{red}l_{p-4}^{l_{p-4}}\arrow[dr]\arrow[uu]&&\color{red}r_{p-4}^{r_{p-4}}\arrow[dd]\arrow[ul]\\
    &\color{red}m_{p-1}^{m_{p-3}}\arrow[dl]\arrow[ur]&\\
    \color{red}l_{p-3}^{l_{p-3}}\arrow[dr]\arrow[uu]&&\color{red}r_{p-3}^{r_{p-3}}\arrow[ul]\arrow[dd]\\
    &\color{red}m_{p-4}^{m_{p-2}}\arrow[dl]\arrow[ur]&\\
    \color{red}r_{p-2}^{l_{p-2}}\arrow[uu]&&\color{red}l_{p-2}^{r_{p-2}}\arrow[ul]
\end{tikzcd}}\]

Step 6(c) consists of $(m_2,h_2,h_3,...,h_q,g_1,g_2,...,g_{q-1},g_q,h_{q+1},g_q)$.

At $m_2,h_2,h_3,...,h_{q+1},g_1,g_2,...,g_{q-1},g_q$ and adjacent vertices, the diagram looks like:

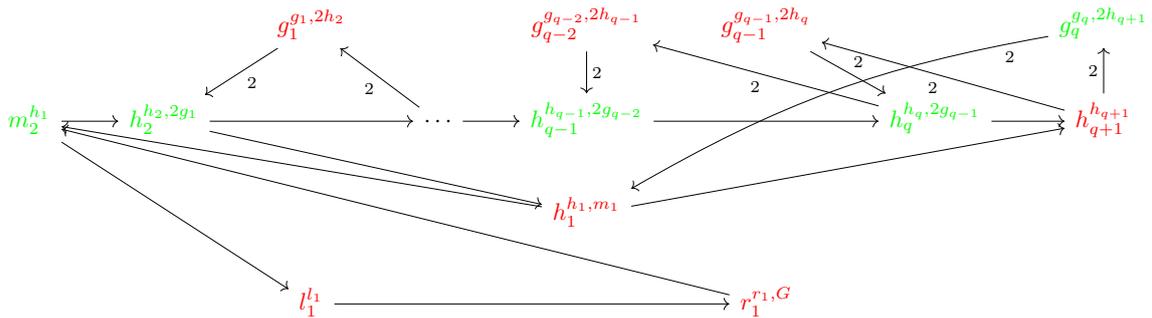
\[\adjustbox{scale=0.9,center}{\begin{tikzcd}
    &&\color{red}g_1^{g_1,2h_2}\arrow[dl,"2"]&&\color{red}g_{q-2}^{g_{q-2},2h_{q-1}}\arrow[d,"2"]&\color{red}g_{q-1}^{g_{q-1},2h_q}\arrow[dr,"2"]&&\color{green}g_q^{g_q,2h_{q+1}}\arrow[ddlll,"2",pos=0.1,bend right=10]\\
    \color{green}m_2^{h_1}\arrow[r]\arrow[ddrr]&\color{green}h_2^{h_{2},2g_{1}}\arrow[rr]\arrow[drrr]&&\cdots\arrow[r]\arrow[ul,"2"]&\color{green}h_{q-1}^{h_{q-1},2g_{q-2}}\arrow[rr]&&\color{green}h_{q}^{h_{q},2g_{q-1}}\arrow[r]\arrow[ull,"2"]&\color{red}h_{q+1}^{h_{q+1}}\arrow[ull,"2"]\arrow[u,"2"]\\
    &&&&\color{red}h_1^{h_1,m_1}\arrow[urrr]\arrow[ullll]\\
    &&\color{red}l_1^{l_{1}}\arrow[rrr]&&&\color{red}r_1^{r_1,G}\arrow[uulllll]
\end{tikzcd}}\]

Observe that the subtriangulation is same as the case where $p=3$, after step 6(c), we have:

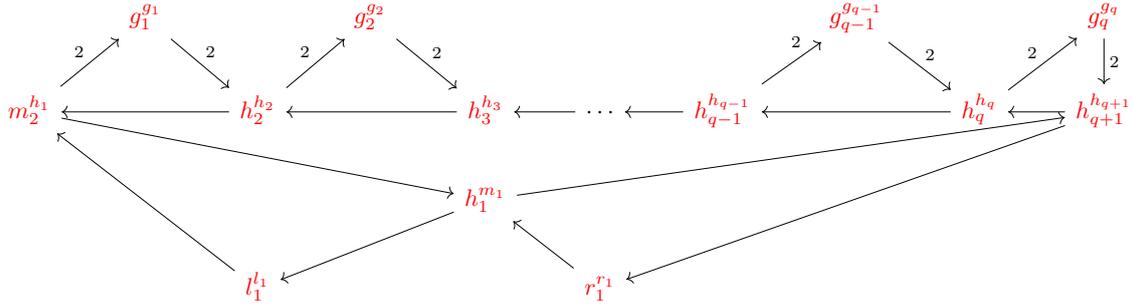
\[\adjustbox{scale=0.9,center}{\begin{tikzcd}
    &\color{red}g_1^{g_1}\arrow[dr,"2"]&&\color{red}g_2^{g_2}\arrow[dr,"2"]&&&&\color{red}g_{q-1}^{g_{q-1}}\arrow[dr,"2"]&&\color{red}g_q^{g_q}\arrow[d,"2"]\\
    \color{red}m_2^{h_1}\arrow[drrrr]\arrow[ur,"2"]&&\color{red}h_2^{h_{2}}\arrow[ur,"2"]\arrow[ll]&&\color{red}h_3^{h_3}\arrow[ll]&\cdots\arrow[l]&\color{red}h_{q-1}^{h_{q-1}}\arrow[l]\arrow[ur,"2",near end]&&\color{red}h_{q}^{h_{q}}\arrow[ll]\arrow[ur,"2"]&\color{red}h_{q+1}^{h_{q+1}}\arrow[ddllll]\arrow[l]\\
    &&&&\color{red}h_1^{m_1}\arrow[dll]\arrow[urrrrr]\\
    &&\color{red}l_1^{l_{1}}\arrow[uull]&&&\color{red}r_1^{r_1}\arrow[ul]
\end{tikzcd}}\]

Since all vertices are red, $\Delta_1$ is a maximal green sequence. $\blacksquare$

Since we have gone through all cases for $n=1$, Theorem \ref{main_1} is proved for $n=1$.

\section{The case where $n=0$}\label{n=0}

\subsection{Constructing the triangulation for $n=0$}\label{construction_0}

As we have shown in section \ref{puncture} that maximal green sequences do not exist when $p=1$, we shall omit the construction in this case. We shall divide the construction into the cases $p=2$ and $p>2$.

An orientable surface of genus $0$ with $2$ punctures and $q$ orbifold points, $q\geq2$, has triangulation $T_{0,2,q}$, as shown below:

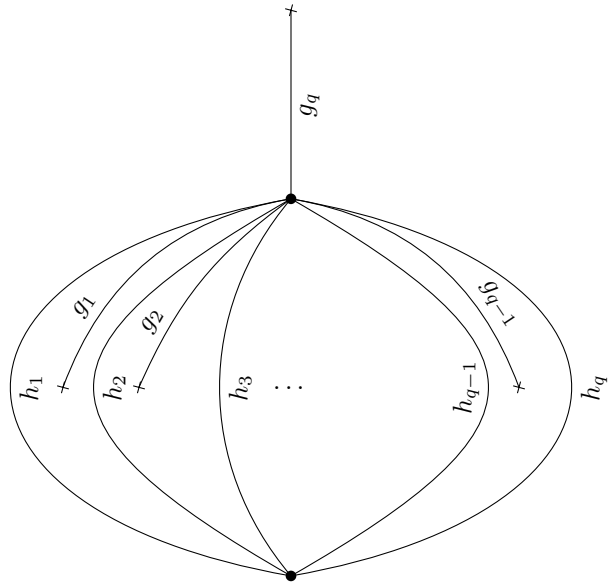
\[
    \begin{tikzpicture}[every edge quotes/.style={auto=right}]
        \begin{scope}[every node/.style={sloped,allow upside down}][every edge quotes/.style={auto=right}]
        \node at ($(0,0)!.5!(0,5)$) {$\dots$};
        \draw (0,0) to [bend left=80,min distance=50mm] node[below] {$h_1$} (0,5);
        \draw (0,0) to [bend left=60,min distance=40mm] node[below] {$h_2$} (0,5);
        \draw (0,0) to [bend left=40] node[below] {$h_3$} (0,5);
        \draw (0,0) to [bend right=60,min distance=40mm] node[above] {$h_{q-1}$} (0,5);
        \draw (0,0) to [bend right=80,min distance=50mm] node[below] {$h_q$} (0,5);
        \draw (0,5) to [bend right] node[above,rotate=180,near end] {$g_1$} (-3,2.5);
        \draw (0,5) to [bend right=15] node[above,rotate=180,near end] {$g_2$} (-2,2.5);
        \draw (0,5) to [bend left] node[above,near end] {$g_{q-1}$} (3,2.5);
        \draw (0,5) to node[below] {$g_q$} (0,7.5);
        \fill (-3,2.5) node[cross=2pt,rotate=30] {};
        \fill (-2,2.5) node[cross=2pt,rotate=30] {};
        \fill (3,2.5) node[cross=2pt,rotate=30] {};
        \fill (0,7.5) node[cross=2pt,rotate=30] {};
        \fill (0,5) circle (2pt);
        \fill (0,0) circle (2pt);
    \end{scope}
\end{tikzpicture}\]

The following diagram is associated to $T_{0,2,q}$

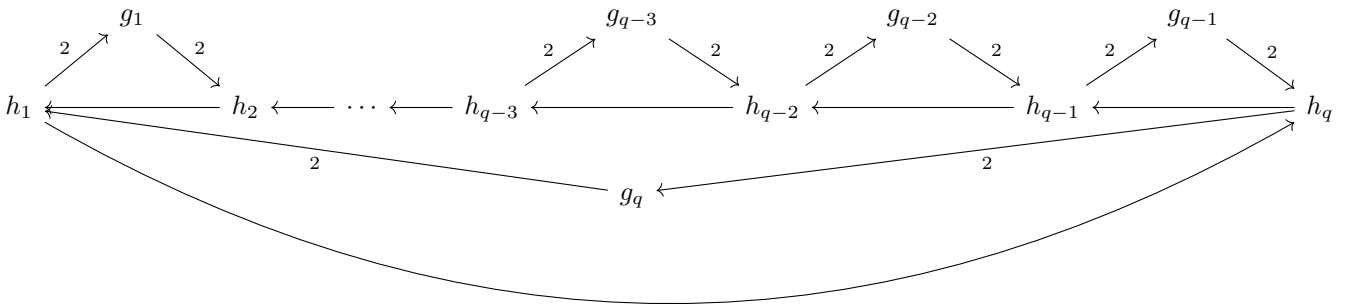
\[\begin{tikzcd}
    &g_1\arrow[dr,"2"]&&&&g_{q-3}\arrow[dr,"2"]&&g_{q-2}\arrow[dr,"2"]&&g_{q-1}\arrow[dr,"2"]\\
    h_1\arrow[ur,"2"]\arrow[rrrrrrrrrr,bend right]&&h_2\arrow[ll]&\cdots\arrow[l]&h_{q-3}\arrow[l]\arrow[ur,"2"]&&h_{q-2}\arrow[ur,"2"]\arrow[ll]&&h_{q-1}\arrow[ur,"2"]\arrow[ll]&&h_q\arrow[dlllll,"2"]\arrow[ll]\\&&&&&g_q\arrow[ulllll,"2"]
\end{tikzcd}\]

An orientable surface of genus $0$ with $p$ punctures and $q$ orbifold points, $p,q\in\mathbb{Z}^+$, $p>2$, has triangulation $T_{0,p,q}$, as shown below:

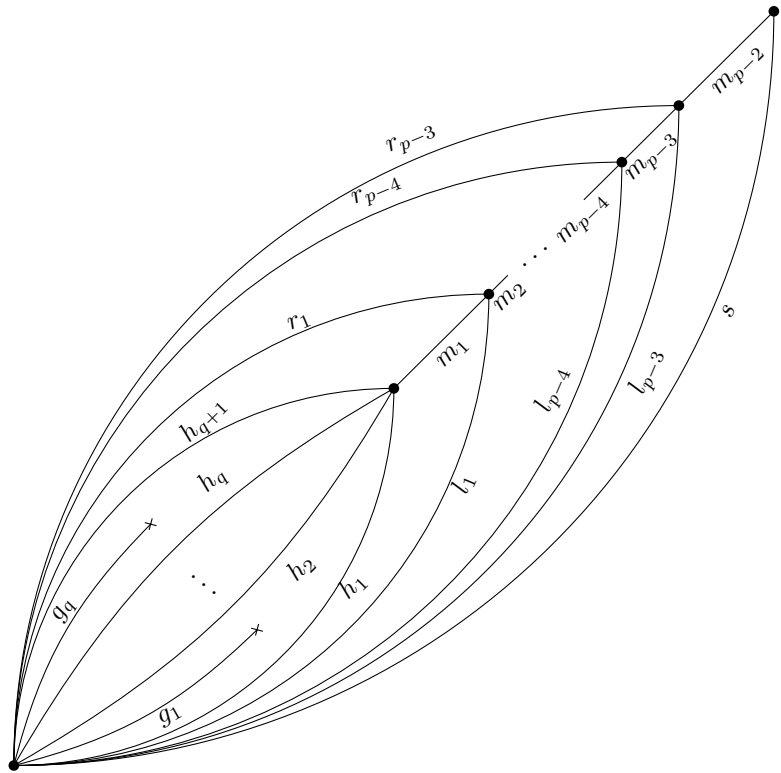
\[\begin{tikzpicture}[every edge quotes/.style={auto=right}]
    \begin{scope}[every node/.style={sloped,allow upside down}][every edge quotes/.style={auto=right}]

\node at (6.85,6.85) {$\iddots$};
        \node at ($(0,5)!.5!(5,0)$) {$\ddots$};
        \draw (0,0) to [bend right=15,black] node [below=.15,right] {$g_1$} (3.2,1.8);
        \draw (0,0) to [bend left=15,black] node [above=.15,right] {$g_q$} (1.8,3.2);
        \draw (0,0) to [bend right=45,black] node[below=.15,right=1] {$h_1$} (5,5);
        \draw (0,0) to [bend left=45,black] node[above=.15,right=1] {$h_{q+1}$} (5,5);
        \draw (0,0) to [bend right=15,black] node[right=1,below] {$h_2$} (5,5);
        \draw (0,0) to [bend left=15,black] node[right=1,above] {$h_q$} (5,5);
        \draw (0,0) to [bend right=45,black] node[near end,right,below=-.1] {$l_1$} (6.25,6.25);
        \draw (0,0) to [bend left=45] node[near end,right,above=-.1,black] {$r_1$} (6.25,6.25);
        \draw (0,0) to [bend right=45,black] node[near end,below] {$l_{p-3}$} (8.75,8.75);
        \draw (0,0) to [bend left=45,black] node[near end,above] {$r_{p-3}$} (8.75,8.75);
        \draw (0,0) to [bend right=45,black] node[near end,above] {$l_{p-4}$} (8,8);
        \draw (0,0) to [bend left=45,black] node[near end,above] {$r_{p-4}$} (8,8);
        \draw (0,0) to [bend right=45,black] node[near end,below,black] {$s$} (10,10);
        \draw (5,5) to node[below] {$m_1$} (6.25,6.25);
        \draw (8.75,8.75) to node[below] {$m_{p-2}$} (10,10);
        \draw (6.25,6.25) to node[below] {$m_2$} (6.5,6.5);
        \draw (8,8) to node[left=.2,below] {$m_{p-3}$} (8.75,8.75);
        \draw (7.5,7.5) to node[below=.2,left] {$m_{p-4}$} (8,8);
        \fill (0,0) circle (2pt);
        \fill (10,10) circle (2pt);
        \fill (5,5) circle (2pt);
        \fill (3.2,1.8) node[cross=2pt,rotate=30] {};
        \fill (1.8,3.2) node[cross=2pt,rotate=30] {};
        \fill (6.25,6.25) circle (2pt);
        \fill (8.75,8.75) circle (2pt);
        \fill (8,8) circle (2pt);
    \end{scope}
\end{tikzpicture}\]

\label{core}

If we remove arc $s$ from this triangulation, we see that the resulting subtriangulation is contained in $T_{1,p-1,q}$. This fact will be used in the proof.

As a result, the associated diagram of $T_{0,p,q}$ is also similar to that of $T_{1,p-1,q}$:

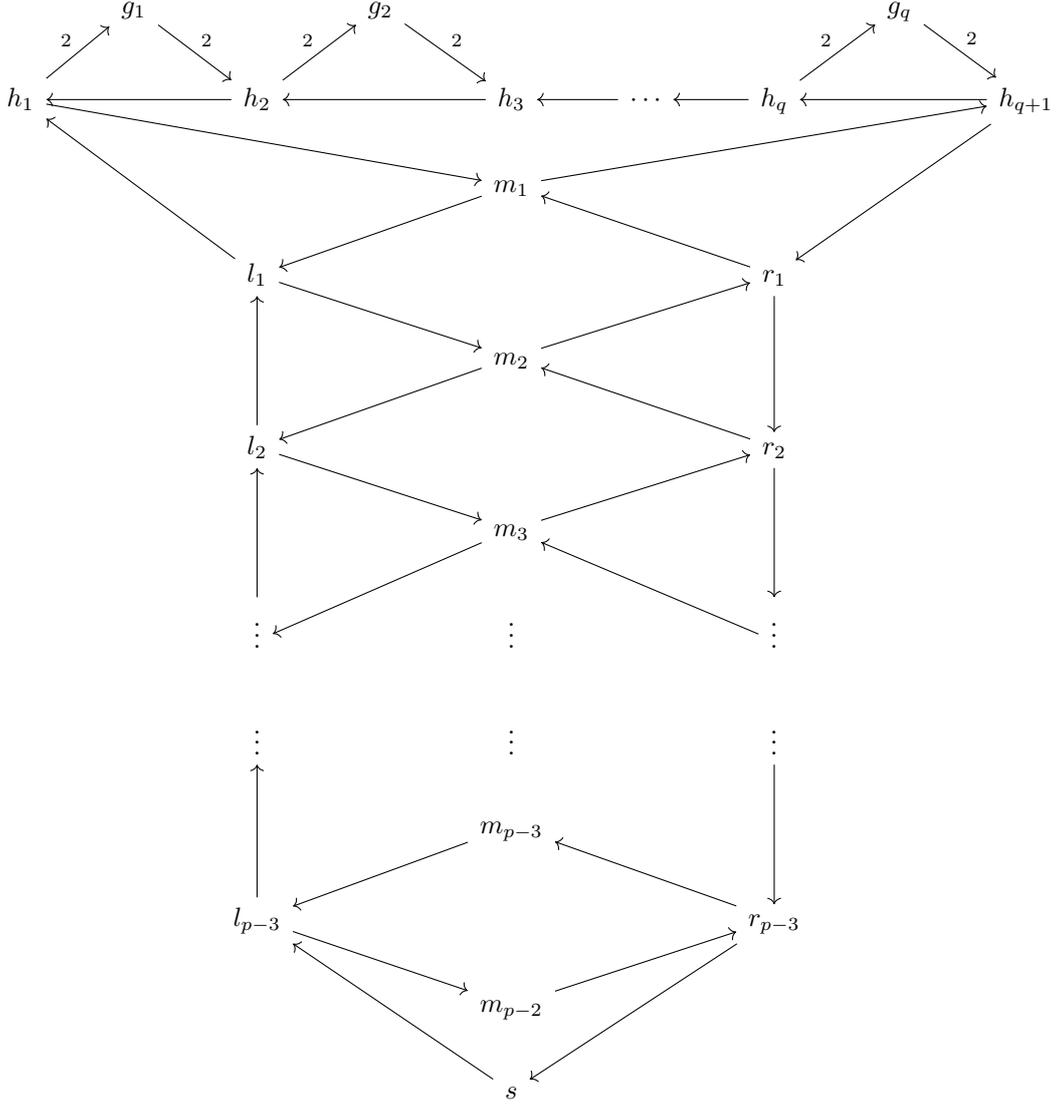
\[\begin{tikzcd}
 & g_1\arrow[dr,"2"] & & g_2\arrow[dr,"2"]&&&&g_q\arrow[dr,"2"]\\
 h_1\arrow[drrrr]\arrow[ur,"2"] && h_2\arrow[ur,"2"]\arrow[ll] && h_3\arrow[ll]&\cdots\arrow[l]&h_q\arrow[l]\arrow[ur,"2"]&&h_{q+1}\arrow[ll]\arrow[ddll]\\
 &&&&m_1\arrow[dll]\arrow[urrrr]&&&&\\
 &&l_1\arrow[uull]\arrow[drr]&&&&r_1\arrow[ull]\arrow[dd]&&\\
 &&&&m_2\arrow[dll]\arrow[urr]&&&&\\
 &&l_2\arrow[uu]\arrow[drr]&&&&r_2\arrow[ull]\arrow[dd]&&\\
 &&&&m_3\arrow[dll]\arrow[urr]&&&&\\
 &&\vdots\arrow[uu]&&\vdots&&\vdots\arrow[ull]\\
 &&\vdots&&\vdots&&\vdots\arrow[dd]\\
 &&&&m_{p-3}\arrow[dll]&&&&\\
 &&l_{p-3}\arrow[uu]\arrow[drr]&&&&r_{p-3}\arrow[ull]\arrow[ddll]\\
 &&&&m_{p-2}\arrow[urr]&&&&\\
 &&&&s\arrow[uull]&&
\end{tikzcd}\]

\subsection{Proof of the main result for $n=0$}\label{proof_0}

\begin{theorem}\label{main_0}
    For an orbifold $\mathcal{O}$ of genus $0$ with $p$ punctures and $q$ orbifold points such that $p+q\geq4$, the diagram $D(T_{0,p,q})$ associated to the triangulation $T_{0,p,q}$ has a maximal green sequence if $p\geq 2$. Moreover, if $p=1$ then $D(T)$ does not admit a maximal green sequence for any triangulation $T$ of $\mathcal{O}$.
\end{theorem}

The proof of the "moreover" part is mentioned in section \ref{puncture}, and we shall prove the rest by providing maximal green sequences.

As the triangulations when $p=2$ and $p>2$ are different, we shall divide the proof into two cases.

\subsubsection{The case where $p=2$}\label{0,2}

\begin{proposition}
    The mutation sequence $\Delta_0$ formed by concatenating the following steps is a maximal green sequence of $D(T_{0,2,q})$:
    \begin{enumerate}
    \item Tagging the lower puncture notched. $(h_q,h_{q-1},...,h_1,h_3,h_4,...,h_q)$
    \item Moving all pending arcs away from the lower puncture. $(g_1,g_2,...,g_q)$
    \item Tagging the upper puncture notched. $(h_2,h_1,h_3,h_4,...,h_q,h_{q-2},h_{q-3},...,h_3,h_1,h_2)$
    \item Moving all pending arcs back to the lower puncture. $(g_1,g_2,...,g_q)$
\end{enumerate}
\end{proposition}

In all steps, all vertices are either mutated at or adjacent to vertices mutated at, so we shall show the entire diagram in all steps.

Step 1 consists of $(h_q,h_{q-1},...,h_1,h_3,h_4,...,h_q)$.

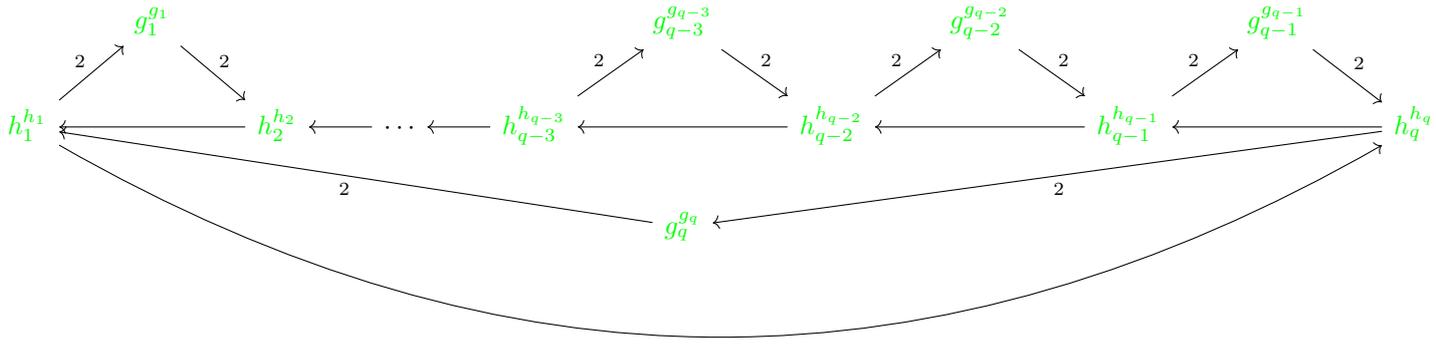
\[\begin{tikzcd}
    &\color{green}g_1^{g_1}\arrow[dr,"2"]&&&&\color{green}g_{q-3}^{g_{q-3}}\arrow[dr,"2"]&&\color{green}g_{q-2}^{g_{q-2}}\arrow[dr,"2"]&&\color{green}g_{q-1}^{g_{q-1}}\arrow[dr,"2"]\\
    \color{green}h_1^{h_1}\arrow[ur,"2"]\arrow[rrrrrrrrrr,bend right]&&\color{green}h_2^{h_2}\arrow[ll]&\cdots\arrow[l]&\color{green}h_{q-3}^{h_{q-3}}\arrow[l]\arrow[ur,"2"]&&\color{green}h_{q-2}^{h_{q-2}}\arrow[ur,"2"]\arrow[ll]&&\color{green}h_{q-1}^{h_{q-1}}\arrow[ur,"2"]\arrow[ll]&&\color{green}h_q^{h_q}\arrow[dlllll,"2"]\arrow[ll]\\&&&&&\color{green}g_q^{g_q}\arrow[ulllll,"2"]
\end{tikzcd}\]

Mutate at $h_q$:

\[\adjustbox{scale=0.8}
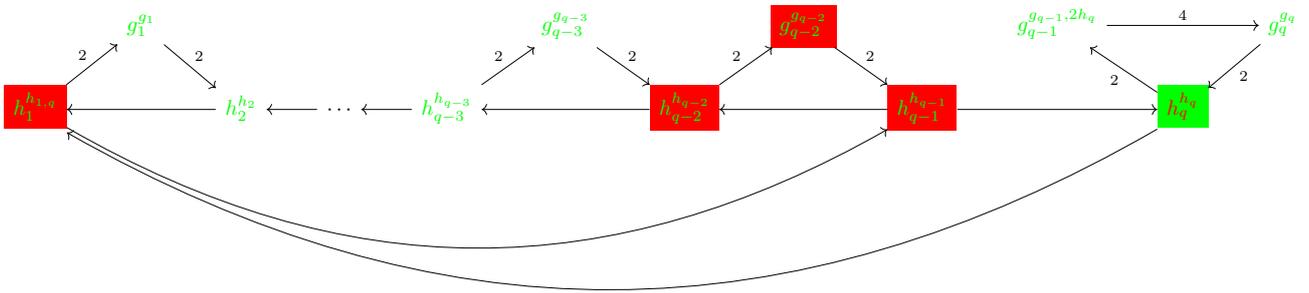
{\begin{tikzcd}
    &\color{green}g_1^{g_1}\arrow[dr,"2"]&&&&\color{green}g_{q-3}^{g_{q-3}}\arrow[dr,"2"]&&|[color=green,fill=red]|g_{q-2}^{g_{q-2}}\arrow[dr,"2"]&&\color{green}g_{q-1}^{g_{q-1},2h_q}\arrow[rr,"4"]&&\color{green}g_q^{g_q}\arrow[dl,"2"]\\
    |[color=green,fill=red]|h_1^{h_{1,q}}\arrow[ur,"2"]\arrow[rrrrrrrr,bend right]&&\color{green}h_2^{h_2}\arrow[ll]&\cdots\arrow[l]&\color{green}h_{q-3}^{h_{q-3}}\arrow[l]\arrow[ur,"2"]&&|[color=green,fill=red]|h_{q-2}^{h_{q-2}}\arrow[ur,"2"]\arrow[ll]&&|[color=green,fill=red]|h_{q-1}^{h_{q-1}}\arrow[rr]\arrow[ll]&&|[color=red,fill=green]|h_q^{h_q}\arrow[llllllllll,bend left,shift left]\arrow[ul,"2"]
\end{tikzcd}}\]

Mutate at $h_{q-1}$:

\[\adjustbox{scale=0.8}
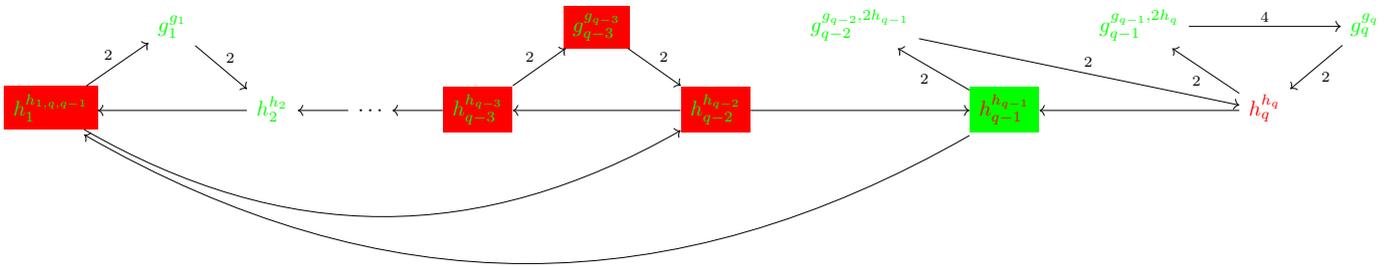
{\begin{tikzcd}
    &\color{green}g_1^{g_1}\arrow[dr,"2"]&&&&|[color=green,fill=red]|g_{q-3}^{g_{q-3}}\arrow[dr,"2"]&&\color{green}g_{q-2}^{g_{q-2},2h_{q-1}}\arrow[drrr,"2"]&&\color{green}g_{q-1}^{g_{q-1},2h_q}\arrow[rr,"4"]&&\color{green}g_q^{g_q}\arrow[dl,"2"]\\
    |[color=green,fill=red]|h_1^{h_{1,q,q-1}}\arrow[ur,"2"]\arrow[rrrrrr,bend right]&&\color{green}h_2^{h_2}\arrow[ll]&\cdots\arrow[l]&|[color=green,fill=red]|h_{q-3}^{h_{q-3}}\arrow[l]\arrow[ur,"2"]&&|[color=green,fill=red]|h_{q-2}^{h_{q-2}}\arrow[rr]\arrow[ll]&&|[color=red,fill=green]|h_{q-1}^{h_{q-1}}\arrow[ul,"2"]\arrow[llllllll,bend left,shift left]&&\color{red}h_q^{h_q}\arrow[ll]\arrow[ul,"2"]
\end{tikzcd}}\]

Observe the pattern at the highlighted vertices, after mutating at $h_3$, we have:

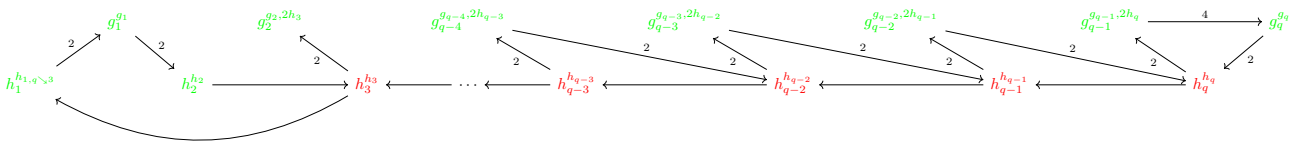
\[\adjustbox{scale=0.6}{\begin{tikzcd}
    &\color{green}g_1^{g_1}\arrow[dr,"2"]&&\color{green}g_2^{g_2,2h_3}&&\color{green}g_{q-4}^{g_{q-4},2h_{q-3}}\arrow[drrr,"2"]&&|[color=green]|g_{q-3}^{g_{q-3},2h_{q-2}}\arrow[drrr,"2"]&&\color{green}g_{q-2}^{g_{q-2},2h_{q-1}}\arrow[drrr,"2"]&&\color{green}g_{q-1}^{g_{q-1},2h_q}\arrow[rr,"4"]&&\color{green}g_q^{g_q}\arrow[dl,"2"]\\
    |[color=green]|h_1^{h_{1,q\searrow3}}\arrow[ur,"2"]&&\color{green}h_2^{h_2}\arrow[rr]&&\color{red}h_3^{h_3}\arrow[ul,"2"]\arrow[llll,bend left]&\cdots\arrow[l]&|[color=red]|h_{q-3}^{h_{q-3}}\arrow[l]\arrow[ul,"2"]&&|[color=red]|h_{q-2}^{h_{q-2}}\arrow[ul,"2"]\arrow[ll]&&|[color=red]|h_{q-1}^{h_{q-1}}\arrow[ul,"2"]\arrow[ll]&&\color{red}h_q^{h_q}\arrow[ll]\arrow[ul,"2"]
\end{tikzcd}}\]

In terms of triangulation, we have wrapped most arcs inside $h_3$ like in the case where $n=1,p=2,q>1$ (and we shall reuse the picture with slight modification):

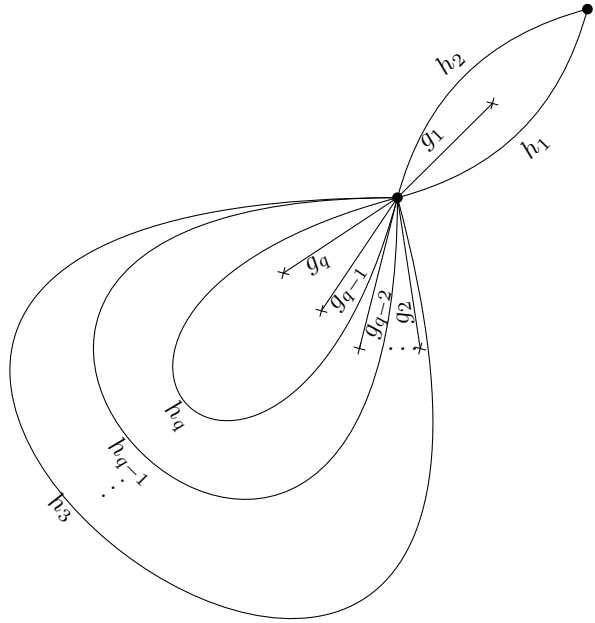
\begin{flushleft}
\hspace*{-0.25\linewidth}
\begin{tikzpicture}[every node/.style={draw}][every edge quotes/.style={auto=right}]
    \begin{scope}[every node/.style={sloped,allow upside down}][every edge quotes/.style={auto=right}]
        \fill (1,1.5) node[cross=2pt,rotate=30] {};
        \fill (1.5,1) node[cross=2pt,rotate=30] {};
        \fill (2,0.5) node[cross=2pt,rotate=30] {};
        \fill (2.8,0.5) node[cross=2pt,rotate=30] {};
        \node[font=\small] at (-1.25,-1.25) {$\iddots$};
        \node at (2.6,0.5) {$\dots$};
        \draw (1,1.5) to node[near start,below=-.1] {$g_q$} (2.5,2.5);
        \draw (1.5,1) to node[near start,below=-.1] {$g_{q-1}$} (2.5,2.5);
        \draw (2,0.5) to node[near start,below=-.1] {$g_{q-2}$} (2.5,2.5);
        \draw (2.8,0.5) to node[near start,above=-.1] {$g_{2}$} (2.5,2.5);
        \draw (2.5,2.5) to [loop,in=255,out=195,min distance=60mm] node[below=-.1] {$h_q$} (2.5,2.5);
        \draw (2.5,2.5) to [loop,in=270,out=180,min distance=90mm] node[below=-.1] {$h_{q-1}$} (2.5,2.5);
        \draw (2.5,2.5) to [loop,in=285,out=180,min distance=130mm] node[below=0.1,left=.7] {$h_{3}$} (2.5,2.5);
        \draw (2.5,2.5) to [bend right] node [below] {$h_1$} (5,5);
        \draw (2.5,2.5) to [bend left] node [above] {$h_2$} (5,5);
        \draw (2.5,2.5) to node [above] {$g_1$} (3.75,3.75);
        \fill (2.5,2.5) circle (2pt);
        \fill (5,5) circle (2pt);
        \fill (3.75,3.75) node[cross=2pt,rotate=30] {};
    \end{scope}
\end{tikzpicture}
\end{flushleft}

\vspace*{-0.4\linewidth}

Observe that the arcs $h_3,g_1,h_{2},h_1$ form a triangulated once-punctured digon. Mutate at $h_2,h_1$ (and switch their positions):

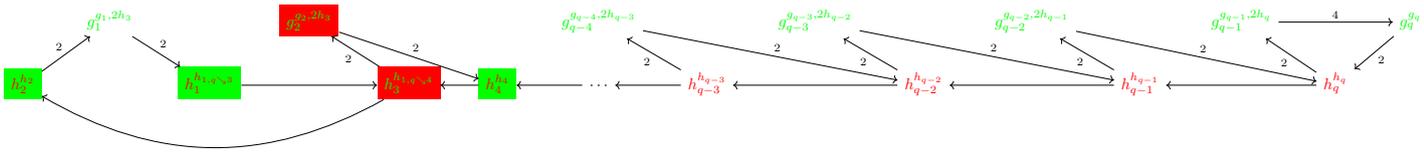
\[\adjustbox{scale=0.6}{\begin{tikzcd}
    &\color{green}g_1^{g_1,2h_3}\arrow[dr,"2"]&&|[color=green,fill=red]|g_2^{g_2,2h_3}\arrow[drr,"2"]&&&\color{green}g_{q-4}^{g_{q-4},2h_{q-3}}\arrow[drrr,"2"]&&|[color=green]|g_{q-3}^{g_{q-3},2h_{q-2}}\arrow[drrr,"2"]&&\color{green}g_{q-2}^{g_{q-2},2h_{q-1}}\arrow[drrr,"2"]&&\color{green}g_{q-1}^{g_{q-1},2h_q}\arrow[rr,"4"]&&\color{green}g_q^{g_q}\arrow[dl,"2"]\\|[color=red,fill=green]|h_2^{h_2}\arrow[ur,"2"]&&|[color=red,fill=green]|h_1^{h_{1,q\searrow3}}\arrow[rr]&&|[color=green,fill=red]|h_3^{h_{1,q\searrow4}}\arrow[ul,"2"]\arrow[llll,bend left]&|[color=red,fill=green]|h_4^{h_4}\arrow[l]&\cdots\arrow[l]&|[color=red]|h_{q-3}^{h_{q-3}}\arrow[l]\arrow[ul,"2"]&&|[color=red]|h_{q-2}^{h_{q-2}}\arrow[ul,"2"]\arrow[ll]&&|[color=red]|h_{q-1}^{h_{q-1}}\arrow[ul,"2"]\arrow[ll]&&\color{red}h_q^{h_q}\arrow[ll]\arrow[ul,"2"]
\end{tikzcd}}\]

Now the upper puncture (lower puncture in the original picture) is tagged notched.

After mutating at $h_3,h_4,...,h_{q-1}$, again by observing the pattern at the highlighted vertices, we have:

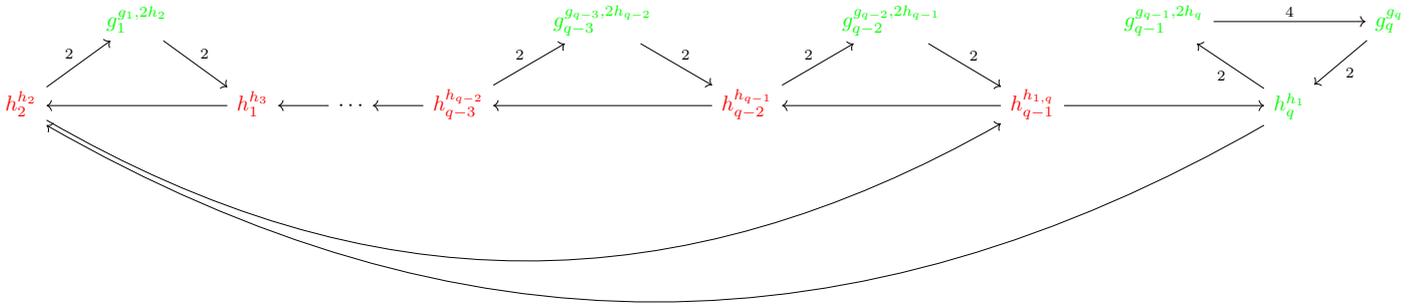
\[\adjustbox{scale=0.8}{\begin{tikzcd}
    &\color{green}g_1^{g_1,2h_2}\arrow[dr,"2"]&&&&\color{green}g_{q-3}^{g_{q-3},2h_{q-2}}\arrow[dr,"2"]&&|[color=green]|g_{q-2}^{g_{q-2},2h_{q-1}}\arrow[dr,"2"]&&\color{green}g_{q-1}^{g_{q-1},2h_q}\arrow[rr,"4"]&&\color{green}g_q^{g_q}\arrow[dl,"2"]\\
    |[color=red]|h_2^{h_{2}}\arrow[ur,"2"]\arrow[rrrrrrrr,bend right]&&\color{red}h_1^{h_3}\arrow[ll]&\cdots\arrow[l]&\color{red}h_{q-3}^{h_{q-2}}\arrow[l]\arrow[ur,"2"]&&|[color=red]|h_{q-2}^{h_{q-1}}\arrow[ur,"2"]\arrow[ll]&&|[color=red]|h_{q-1}^{h_{1,q}}\arrow[rr]\arrow[ll]&&|[color=green]|h_q^{h_1}\arrow[llllllllll,bend left,shift left]\arrow[ul,"2"]
\end{tikzcd}}\]

Mutate at $h_q$:

\[\adjustbox{scale=0.8}
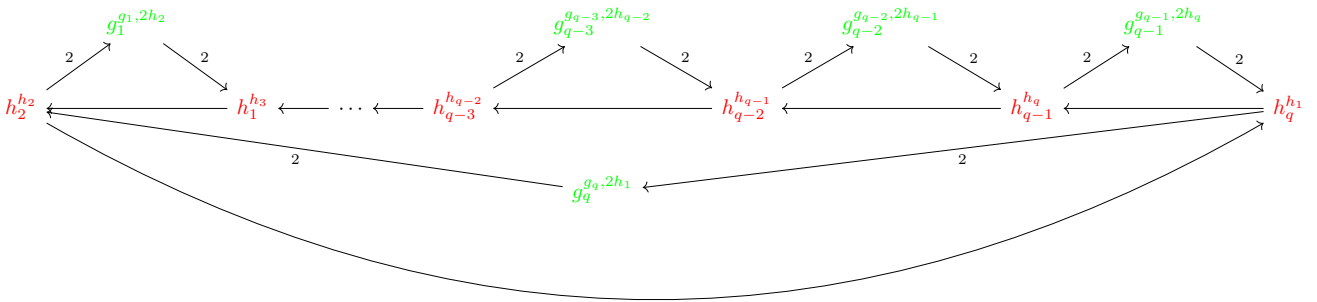
{\begin{tikzcd}
    &\color{green}g_1^{g_1,2h_2}\arrow[dr,"2"]&&&&\color{green}g_{q-3}^{g_{q-3},2h_{q-2}}\arrow[dr,"2"]&&|[color=green]|g_{q-2}^{g_{q-2},2h_{q-1}}\arrow[dr,"2"]&&\color{green}g_{q-1}^{g_{q-1},2h_q}\arrow[dr,"2"]\\
    |[color=red]|h_2^{h_{2}}\arrow[ur,"2"]\arrow[rrrrrrrrrr,bend right]&&\color{red}h_1^{h_3}\arrow[ll]&\cdots\arrow[l]&\color{red}h_{q-3}^{h_{q-2}}\arrow[l]\arrow[ur,"2"]&&|[color=red]|h_{q-2}^{h_{q-1}}\arrow[ur,"2"]\arrow[ll]&&|[color=red]|h_{q-1}^{h_{q}}\arrow[ur,"2"]\arrow[ll]&&|[color=red]|h_q^{h_1}\arrow[ll]\arrow[dlllll,"2"]\\&&&&&\color{green}g_q^{g_q,2h_1}\arrow[ulllll,"2"]
\end{tikzcd}}\]

We get back the original triangulation with the lower puncture tagged notched and $h_1,h_2$ swapped:

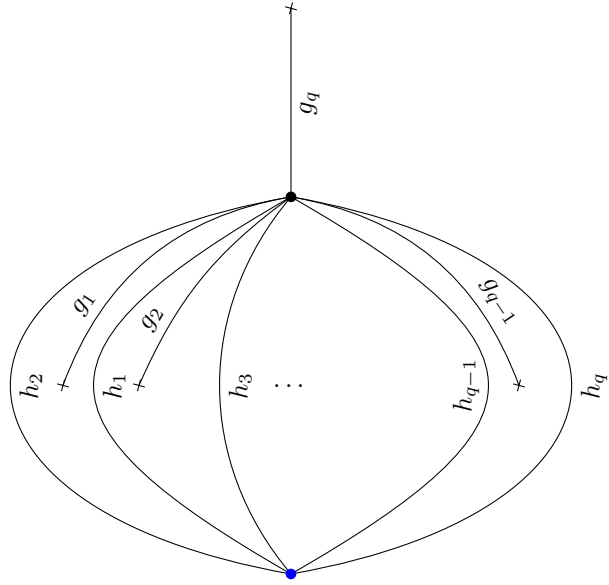
\[
    \begin{tikzpicture}[every edge quotes/.style={auto=right}]
        \begin{scope}[every node/.style={sloped,allow upside down}][every edge quotes/.style={auto=right}]
        \node at ($(0,0)!.5!(0,5)$) {$\dots$};
        \draw (0,0) to [bend left=80,min distance=50mm] node[below] {$h_2$} (0,5);
        \draw (0,0) to [bend left=60,min distance=40mm] node[below] {$h_1$} (0,5);
        \draw (0,0) to [bend left=40] node[below] {$h_3$} (0,5);
        \draw (0,0) to [bend right=60,min distance=40mm] node[above] {$h_{q-1}$} (0,5);
        \draw (0,0) to [bend right=80,min distance=50mm] node[below] {$h_q$} (0,5);
        \draw (0,5) to [bend right] node[above,rotate=180,near end] {$g_1$} (-3,2.5);
        \draw (0,5) to [bend right=15] node[above,rotate=180,near end] {$g_2$} (-2,2.5);
        \draw (0,5) to [bend left] node[above,near end] {$g_{q-1}$} (3,2.5);
        \draw (0,5) to node[below] {$g_q$} (0,7.5);
        \fill (-3,2.5) node[cross=2pt,rotate=30] {};
        \fill (-2,2.5) node[cross=2pt,rotate=30] {};
        \fill (3,2.5) node[cross=2pt,rotate=30] {};
        \fill (0,7.5) node[cross=2pt,rotate=30] {};
        \fill (0,5) circle (2pt);
        \fill[blue] (0,0) circle (2pt);
    \end{scope}
\end{tikzpicture}\]

Step 2 consists of $(g_1,g_2,...,g_q)$, mutate:

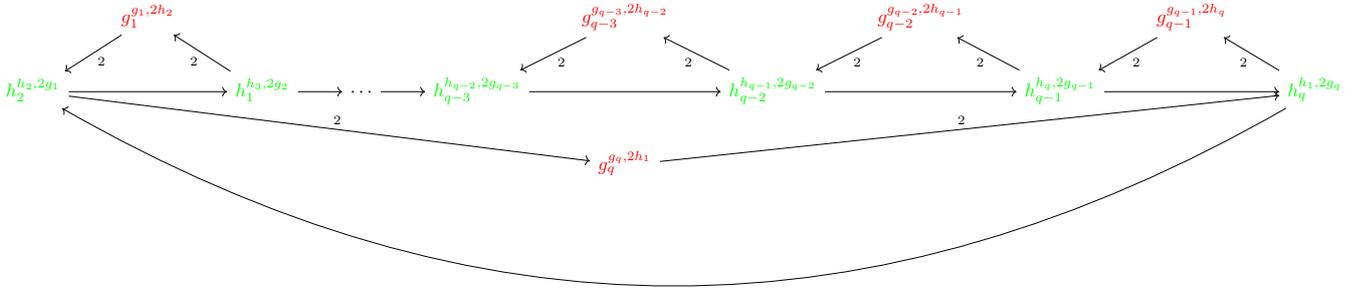
\[\adjustbox{scale=0.7}{\begin{tikzcd}
    &\color{red}g_1^{g_1,2h_2}\arrow[dl,"2"]&&&&\color{red}g_{q-3}^{g_{q-3},2h_{q-2}}\arrow[dl,"2"]&&|[color=red]|g_{q-2}^{g_{q-2},2h_{q-1}}\arrow[dl,"2"]&&\color{red}g_{q-1}^{g_{q-1},2h_q}\arrow[dl,"2"]\\
    |[color=green]|h_2^{h_{2},2g_1}\arrow[drrrrr,"2"]\arrow[rr]&&\color{green}h_1^{h_3,2g_2}\arrow[r]\arrow[ul,"2"]&\cdots\arrow[r]&\color{green}h_{q-3}^{h_{q-2},2g_{q-3}}\arrow[rr]&&|[color=green]|h_{q-2}^{h_{q-1},2g_{q-2}}\arrow[ul,"2"]\arrow[rr]&&|[color=green]|h_{q-1}^{h_{q},2g_{q-1}}\arrow[ul,"2"]\arrow[rr]&&|[color=green]|h_q^{h_1,2g_q}\arrow[llllllllll,bend left]\arrow[ul,"2"]\\&&&&&\color{red}g_q^{g_q,2h_1}\arrow[urrrrr,"2"]
\end{tikzcd}}\]

We can relocate the vertices to make the diagram isomorphic to the starting diagram (except the frozen vertices):

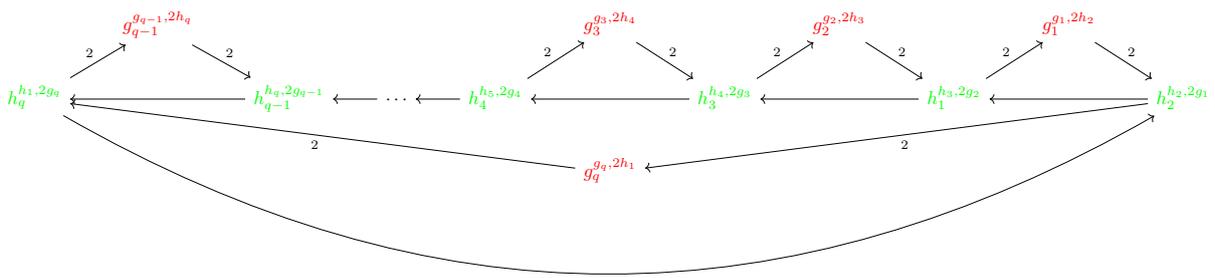
\[\adjustbox{scale=0.7}{\begin{tikzcd}
    &\color{red}g_{q-1}^{g_{q-1},2h_q}\arrow[dr,"2"]&&&&\color{red}g_{3}^{g_{3},2h_4}\arrow[dr,"2"]&&\color{red}g_{2}^{g_{2},2h_3}\arrow[dr,"2"]&&\color{red}g_{1}^{g_{1},2h_2}\arrow[dr,"2"]\\
    \color{green}h_q^{h_1,2g_q}\arrow[ur,"2"]\arrow[rrrrrrrrrr,bend right]&&\color{green}h_{q-1}^{h_{q},2g_{q-1}}\arrow[ll]&\cdots\arrow[l]&\color{green}h_{4}^{h_{5},2g_4}\arrow[l]\arrow[ur,"2"]&&\color{green}h_{3}^{h_{4},2g_3}\arrow[ur,"2"]\arrow[ll]&&\color{green}h_{1}^{h_{3},2g_2}\arrow[ur,"2"]\arrow[ll]&&\color{green}h_2^{h_2,2g_1}\arrow[dlllll,"2"]\arrow[ll]\\&&&&&\color{red}g_q^{g_q,2h_1}\arrow[ulllll,"2"]
\end{tikzcd}}\]

Step 3 consists of $(h_2,h_1,h_3,h_4,...,h_q,h_{q-2},h_{q-3},...,h_3,h_1,h_2)$. Since the diagram is isomorphic to the starting diagram, we see that after mutating at $h_2,h_1,h_3,h_4,...,h_q$, we have:

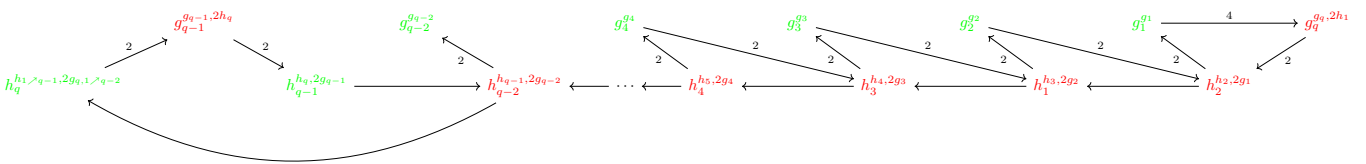
\[\adjustbox{scale=0.6}{\begin{tikzcd}
    &\color{red}g_{q-1}^{g_{q-1},2h_q}\arrow[dr,"2"]&&\color{green}g_{q-2}^{g_{q-2}}&&\color{green}g_{4}^{g_4}\arrow[drrr,"2"]&&|[color=green]|g_{3}^{g_{3}}\arrow[drrr,"2"]&&\color{green}g_{2}^{g_2}\arrow[drrr,"2"]&&\color{green}g_{1}^{g_{1}}\arrow[rr,"4"]&&\color{red}g_q^{g_q,2h_1}\arrow[dl,"2"]\\
    |[color=green]|h_q^{h_{1\nearrow q-1},2g_{q,1\nearrow q-2}}\arrow[ur,"2"]&&\color{green}h_{q-1}^{h_q,2g_{q-1}}\arrow[rr]&&\color{red}h_{q-2}^{h_{q-1},2g_{q-2}}\arrow[ul,"2"]\arrow[llll,bend left]&\cdots\arrow[l]&|[color=red]|h_{4}^{h_{5},2g_4}\arrow[l]\arrow[ul,"2"]&&|[color=red]|h_{3}^{h_{4},2g_3}\arrow[ul,"2"]\arrow[ll]&&|[color=red]|h_{1}^{h_{3},2g_2}\arrow[ul,"2"]\arrow[ll]&&\color{red}h_2^{h_2,2g_1}\arrow[ll]\arrow[ul,"2"]
\end{tikzcd}}\]

Just like in the middle of step 1, $h_{q-2},g_{q-1},h_{q-1},h_q$ form a once-punctured digon. Mutate at $h_{q-1},h_q$ and switch their positions:

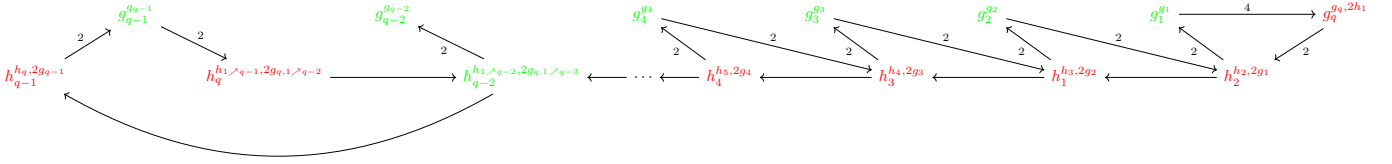
\[\adjustbox{scale=0.6}{\begin{tikzcd}
    &\color{green}g_{q-1}^{g_{q-1}}\arrow[dr,"2"]&&\color{green}g_{q-2}^{g_{q-2}}&&\color{green}g_{4}^{g_4}\arrow[drrr,"2"]&&|[color=green]|g_{3}^{g_{3}}\arrow[drrr,"2"]&&\color{green}g_{2}^{g_2}\arrow[drrr,"2"]&&\color{green}g_{1}^{g_{1}}\arrow[rr,"4"]&&\color{red}g_q^{g_q,2h_1}\arrow[dl,"2"]\\
    |[color=red]|h_{q-1}^{h_q,2g_{q-1}}\arrow[ur,"2"]&&\color{red}h_q^{h_{1\nearrow q-1},2g_{q,1\nearrow q-2}}\arrow[rr]&&\color{green}h_{q-2}^{h_{1\nearrow q-2},2g_{q,1\nearrow q-3}}\arrow[ul,"2"]\arrow[llll,bend left]&\cdots\arrow[l]&|[color=red]|h_{4}^{h_{5},2g_4}\arrow[l]\arrow[ul,"2"]&&|[color=red]|h_{3}^{h_{4},2g_3}\arrow[ul,"2"]\arrow[ll]&&|[color=red]|h_{1}^{h_{3},2g_2}\arrow[ul,"2"]\arrow[ll]&&\color{red}h_2^{h_2,2g_1}\arrow[ll]\arrow[ul,"2"]
\end{tikzcd}}\]

Now the upper puncture is tagged notched.

Mutate at $h_{q-2},h_{q-3},...,h_3,h_1$:

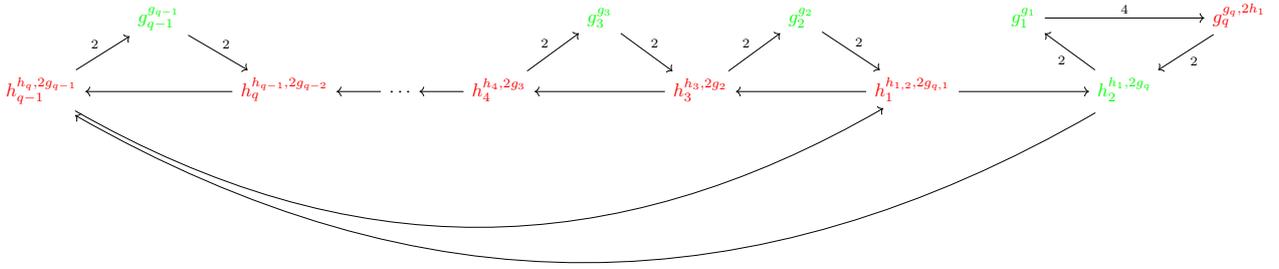
\[\adjustbox{scale=0.7}{\begin{tikzcd}
    &\color{green}g_{q-1}^{g_{q-1}}\arrow[dr,"2"]&&&&\color{green}g_{3}^{g_{3}}\arrow[dr,"2"]&&|[color=green]|g_{2}^{g_2}\arrow[dr,"2"]&&\color{green}g_{1}^{g_{1}}\arrow[rr,"4"]&&\color{red}g_q^{g_q,2h_{1}}\arrow[dl,"2"]\\
    |[color=red]|h_{q-1}^{h_q,2g_{q-1}}\arrow[ur,"2"]\arrow[rrrrrrrr,bend right]&&\color{red}h_q^{h_{q-1},2g_{q-2}}\arrow[ll]&\cdots\arrow[l]&\color{red}h_{4}^{h_{4},2g_3}\arrow[l]\arrow[ur,"2"]&&|[color=red]|h_{3}^{h_{3},2g_2}\arrow[ur,"2"]\arrow[ll]&&|[color=red]|h_{1}^{h_{1,2},2g_{q,1}}\arrow[rr]\arrow[ll]&&|[color=green]|h_2^{h_1,2g_{q}}\arrow[llllllllll,bend left,shift left]\arrow[ul,"2"]
\end{tikzcd}}\]

Mutate at $h_2$:

\[\adjustbox{scale=0.8}
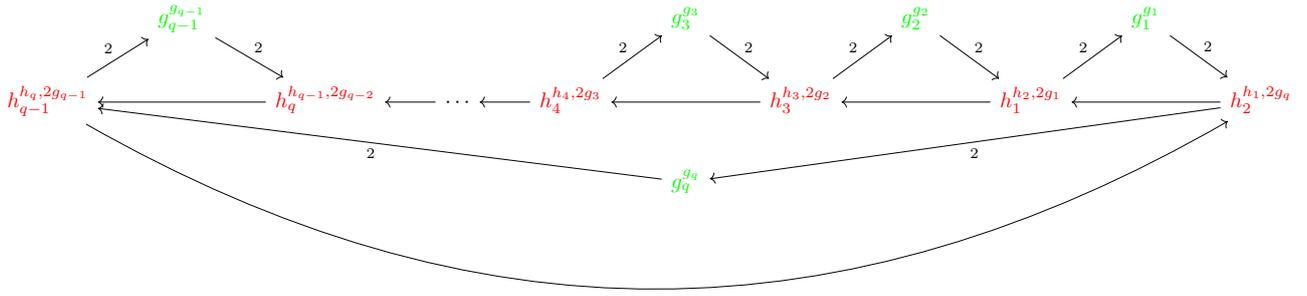
{\begin{tikzcd}
    &\color{green}g_{q-1}^{g_{q-1}}\arrow[dr,"2"]&&&&\color{green}g_{3}^{g_{3}}\arrow[dr,"2"]&&|[color=green]|g_{2}^{g_2}\arrow[dr,"2"]&&\color{green}g_{1}^{g_{1}}\arrow[dr,"2"]\\
    |[color=red]|h_{q-1}^{h_q,2g_{q-1}}\arrow[ur,"2"]\arrow[rrrrrrrrrr,bend right]&&\color{red}h_q^{h_{q-1},2g_{q-2}}\arrow[ll]&\cdots\arrow[l]&\color{red}h_{4}^{h_{4},2g_3}\arrow[l]\arrow[ur,"2"]&&|[color=red]|h_{3}^{h_{3},2g_2}\arrow[ur,"2"]\arrow[ll]&&|[color=red]|h_{1}^{h_{2},2g_1}\arrow[ur,"2"]\arrow[ll]&&|[color=red]|h_2^{h_1,2g_{q}}\arrow[ll]\arrow[dlllll,"2"]\\&&&&&\color{green}g_q^{g_q}\arrow[ulllll,"2"]
\end{tikzcd}}\]

Mutate at $g_1,g_2,...,g_q$:

\[\adjustbox{scale=0.8}{\begin{tikzcd}
    &\color{red}g_{q-1}^{g_{q-1}}\arrow[dl,"2"]&&&&\color{red}g_{3}^{g_{3}}\arrow[dl,"2"]&&|[color=red]|g_{2}^{g_2}\arrow[dl,"2"]&&\color{red}g_{1}^{g_{1}}\arrow[dl,"2"]\\
    |[color=red]|h_{q-1}^{h_q}\arrow[drrrrr,"2"]\arrow[rr]&&\color{red}h_q^{h_{q-1}}\arrow[r]\arrow[ul,"2"]&\cdots\arrow[r]&\color{red}h_{4}^{h_{4}}\arrow[rr]&&|[color=red]|h_{3}^{h_{3}}\arrow[ul,"2"]\arrow[rr]&&|[color=red]|h_{1}^{h_{2}}\arrow[ul,"2"]\arrow[rr]&&|[color=red]|h_2^{h_1,2g_{q}}\arrow[llllllllll,bend left]\arrow[ul,"2"]\\&&&&&\color{red}g_q^{g_q}\arrow[urrrrr,"2"]
\end{tikzcd}}\]

Since all vertices are red, $\Delta_0$ is a maximal green sequence. $\blacksquare$

\subsubsection{The case where $p>2$}\label{0,>2}

\begin{proposition}
    The mutation sequence $\Delta_0$ formed by concatenating the following steps is a maximal green sequence of $D(T_{0,p,q})$ for $p>2$:
    \begin{enumerate}
    \item Tagging alternate punctures notched. $(\alpha_{p-3},\alpha_{p-5},...,\alpha_{p-1-2\lfloor \frac{p-2}{2}\rfloor})$
    \item Tagging the first interior puncture notched. $(h_{q+1},h_q,...,h_1,m_{p-2\lfloor\frac{p-1}{2}\rfloor},h_2,h_3,...,h_{q+2-p+2\lfloor \frac{p-1}{2}\rfloor})$
    \item Tagging the remaining interior punctures notched and moving half the inner arcs away from the bottom left. $(\beta_{p-4},\beta_{p-6},...,\beta_{p-2\lfloor\frac{p-1}{2}\rfloor})$
    \item Moving arcs away from the bottom left.
    \begin{enumerate}
        \item Moving core arcs away from the bottom left. $(g_1,g_2,...,g_q,h_{q+\operatorname{sgn}(p-3)},h_{q+\operatorname{sgn}(p-3)-1},...,h_2,\mu)$
        \item Moving the remaining inner arcs away from the bottom left. $(l_{p-1-2\lfloor\frac{p-2}{2}\rfloor},r_{p-1-2\lfloor\frac{p-2}{2}\rfloor},l_{p+1-2\lfloor\frac{p-2}{2}\rfloor},r_{p+1-2\lfloor\frac{p-2}{2}\rfloor},\\...,l_{p-5},r_{p-5})$
    \end{enumerate}
    \item Tagging the bottom left and top right notched. $(\gamma)$
    \item Moving arcs back to the bottom left.
    \begin{enumerate}
        \item Moving inner arcs back to the bottom left. 
        $(\delta_{p-5},\delta_{p-7},...,\delta_{p-1-2\lfloor\frac{p-2}{2}\rfloor})$
        \item Moving core arcs back to the bottom left. $(\mu,\varepsilon,g_1,g_2,...,g_{q-1},g_q,\omega)$
        \item (If $p>3$ and is odd) Moving the remaining pair of inner arc back to the bottom left. $(l_1,r_1)$
    \end{enumerate}
\end{enumerate}

where

\begin{itemize}
    \item $\alpha_k=r_k,l_k,m_{k+1},m_{k},l_k,r_k$
    \item $\beta_k=\begin{cases}
        r_k,l_k,m_{k+2},m_{k-1}&\text{if }k>1,\\
        r_k,l_k,m_{k+2},h_1&\text{otherwise}
    \end{cases}$
    \item $\mu=\begin{cases}
        m_{p-2\lfloor\frac{p-1}{2}\rfloor}&\text{if }p>3,\\
        \phi & \text{if }p=3
    \end{cases}$
    \item $\gamma=\begin{cases}
        s,l_{p-3},r_{p-3},m_{p-3}&\text{if }p>3,\\
        s,m_1,h_{q+1},h_1 & \text{if }p=3
    \end{cases}$
    \item $\delta_k=r_k,r_{k+1},l_k,l_{k+1}$
    \item $\varepsilon=\begin{cases}
        h_2,h_3,...,h_{q+1}&\text{if }p\text{ is odd and}>3\\
        h_2,h_3,...,h_q &\text{otherwise}
    \end{cases}$
    \item $\omega=\begin{cases}
        h_{q+1},g_q&\text{if }p\text{ is even},\\
        \phi & \text{otherwise}
    \end{cases}$
\end{itemize}
\end{proposition}

Notice that all steps except step 5 are in the mutation sequence $\Delta_1$ with $p$ replaced by $p-1$. We will make use of this fact with the following lemma:

\begin{lemma}\label{one_puncture_summary}

Let $p>1$. Denote the full subdiagram of $D(T_{1,p,q})$ formed by removing $l_{p-1},r_{p-1},f_1,f_2$ by $\Tilde{D}(T_{1,p,q})$. Let $D$ be a diagram containing $\Tilde{D}(T_{1,p,q})$ as a full subdiagram. If the full subdiagram containing all vertices in $\Tilde{D}(T_{1,p,q})$ and adjacent vertices is one of the following:
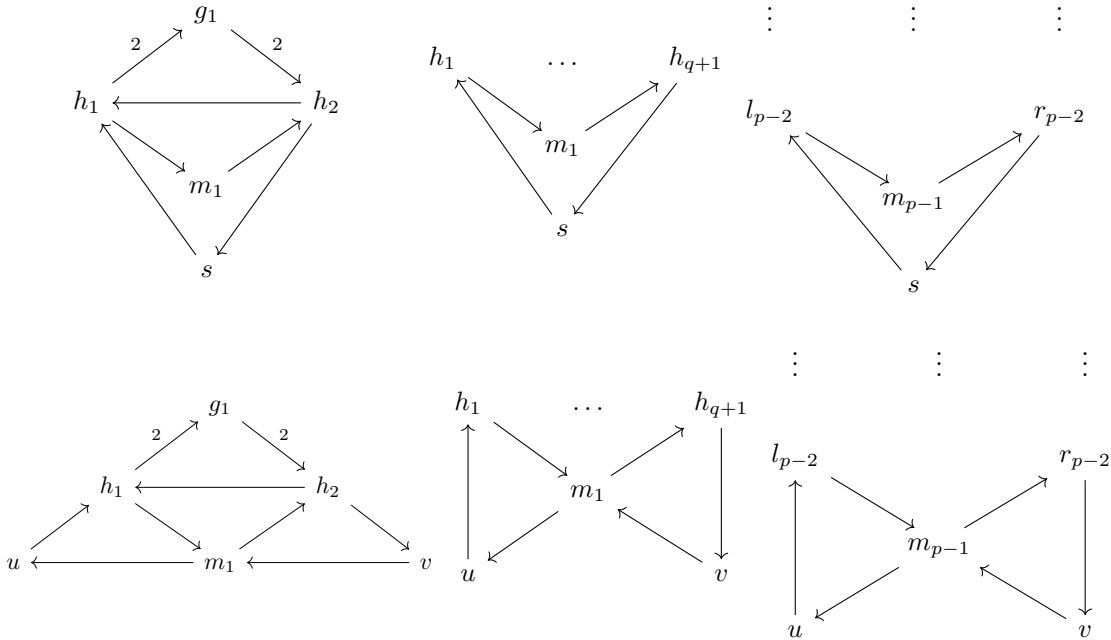
\[\adjustbox{scale=1}{\begin{tikzcd}
    &&g_1\arrow[dr,"2"]&&\\
    &h_1\arrow[ur,"2"]\arrow[dr]&&h_2\arrow[ll]\arrow[ddl]\\
    &&m_1\arrow[ur]\\&&
    s\arrow[uul]
\end{tikzcd}}\adjustbox{scale=1}{\begin{tikzcd}
 h_{1}\arrow[dr]&\dots&h_{q+1}\arrow[ddl]\\
 &m_{1}\arrow[ur]\\
 &s\arrow[uul]
\end{tikzcd}}\adjustbox{scale=1}{\begin{tikzcd}
 \vdots&\vdots&\vdots\\
 l_{p-2}\arrow[dr]&&r_{p-2}\arrow[ddl]\\
 &m_{p-1}\arrow[ur]&&&&\\
 &s\arrow[uul]&&
\end{tikzcd}}\]\[\adjustbox{scale=0.9}{\begin{tikzcd}
    &&g_1\arrow[dr,"2"]&&\\
    &h_1\arrow[ur,"2"]\arrow[dr]&&h_2\arrow[ll]\arrow[dr]&\\
    u\arrow[ur]&&m_1\arrow[ll]\arrow[ur]&&v\arrow[ll]
\end{tikzcd}}\begin{tikzcd}
 h_1\arrow[dr]&\dots&h_{q+1}\arrow[dd]\\
 &m_{1}\arrow[ur]\arrow[dl]\\
 u\arrow[uu]&&v\arrow[ul]
\end{tikzcd}\begin{tikzcd}
 \vdots&\vdots&\vdots\\
 l_{p-2}\arrow[dr]&&r_{p-2}\arrow[dd]\\
 &m_{p-1}\arrow[ur]\arrow[dl]&&&&\\
 u\arrow[uu]&&v\arrow[ul]&&
\end{tikzcd}\]

Then after mutating $\hat{D}$ following step 1 to 4(b) of the mutation sequence in section \ref{puncture}, we get (the vertices in $\Tilde{D}(T_{1,p,q})\setminus\{w,x,y,z\}$ are the same as that when $n=1$):

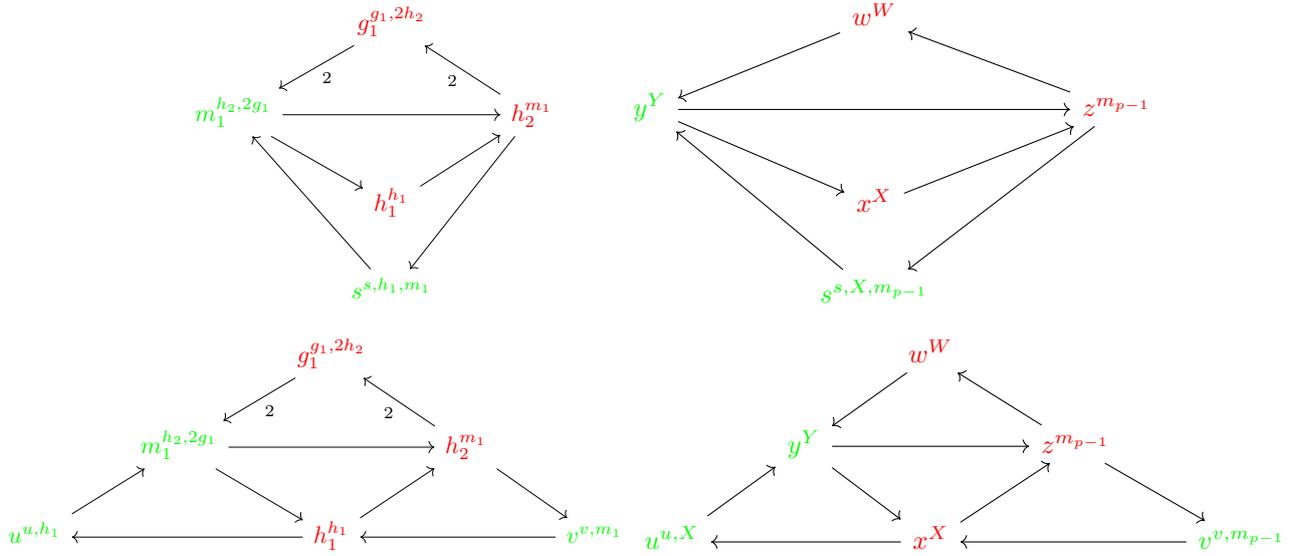
\[\adjustbox{scale=0.9}{\begin{tikzcd}
    &&\color{red}g_1^{g_1,2h_2}\arrow[dl,"2"]&&\\
    &\color{green}m_1^{h_2,2g_1}\arrow[rr]\arrow[dr]&&\color{red}h_2^{m_1}\arrow[ul,"2"]\arrow[ddl]&\\
    &&\color{red}h_1^{h_1}\arrow[ur]\\&&
    \color{green}s^{s,h_1,m_{1}}\arrow[uul]
\end{tikzcd}}\begin{tikzcd}
    &&\color{red}w^{W}\arrow[dll]\\
    \color{green}y^{Y}\arrow[rrrr]\arrow[drr]&&&&\color{red}z^{m_{p-1}}\arrow[ddll]\arrow[ull]\\
    &&\color{red}x^{X}\arrow[urr]\\&&
    \color{green}s^{s,X,m_{p-1}}\arrow[uull]
\end{tikzcd}\]
\[\adjustbox{scale=0.9}{\begin{tikzcd}
    &&\color{red}g_1^{g_1,2h_2}\arrow[dl,"2"]&&\\
    &\color{green}m_1^{h_2,2g_1}\arrow[rr]\arrow[dr]&&\color{red}h_2^{m_1}\arrow[ul,"2"]\arrow[dr]&\\
    \color{green}u^{u,h_1}\arrow[ur]&&\color{red}h_1^{h_1}\arrow[ll]\arrow[ur]&&\color{green}v^{v,m_1}\arrow[ll]
\end{tikzcd}}\begin{tikzcd}
    &&\color{red}w^{W}\arrow[dl]\\&
    \color{green}y^{Y}\arrow[rr]\arrow[dr]&&\color{red}z^{m_{p-1}}\arrow[dr]\arrow[ul]\\
    \color{green}u^{u,X}\arrow[ur]&&\color{red}x^{X}\arrow[ll]\arrow[ur]&&\color{green}v^{v,m_{p-1}}\arrow[ll]
\end{tikzcd}\]

where $w=\begin{cases}
    h_2&\text{if }p=2,q>1\\
    m_2&\text{if }p=3\\
    m_1&\text{if }p=4\\
    r_{p-4}&\text{otherwise}
\end{cases}, x=\begin{cases}
    h_1&\text{if }p=2,q>1\\
    m_{p-2}&\text{otherwise}
\end{cases},y=\begin{cases}
    m_1&\text{if }p=2,q>1\\
    l_{p-2}&\text{otherwise}
\end{cases},
z=\begin{cases}
    h_{q+1}&\text{if }p=2,q>1\\
    r_{p-2}&\text{otherwise}
\end{cases}$

$W,Y$ are sets of (frozen) vertices with $W\subset Y$, and $X=\begin{cases}
    h_1&\text{if }p=2,q>1\\
    l_{p-2}&\text{otherwise}
\end{cases}$.

after this, for $p=2,q=1$, if after a sequence $\zeta$ of mutation at vertices other than $g_1$, we get the diagrams on the left with vertices in $D\setminus \Tilde{D}(T_{1,p,q})$ not connected to $g_1$;

for $p\neq2$ or $q\neq1$, $\zeta$ does not contain vertices not in $\Tilde{D}(T_{1,p,q})\setminus\{x,y,z\}$, we get the diagrams on the right with vertices in $D\setminus \Tilde{D}(T_{1,p,q})$ not connected to $w$:

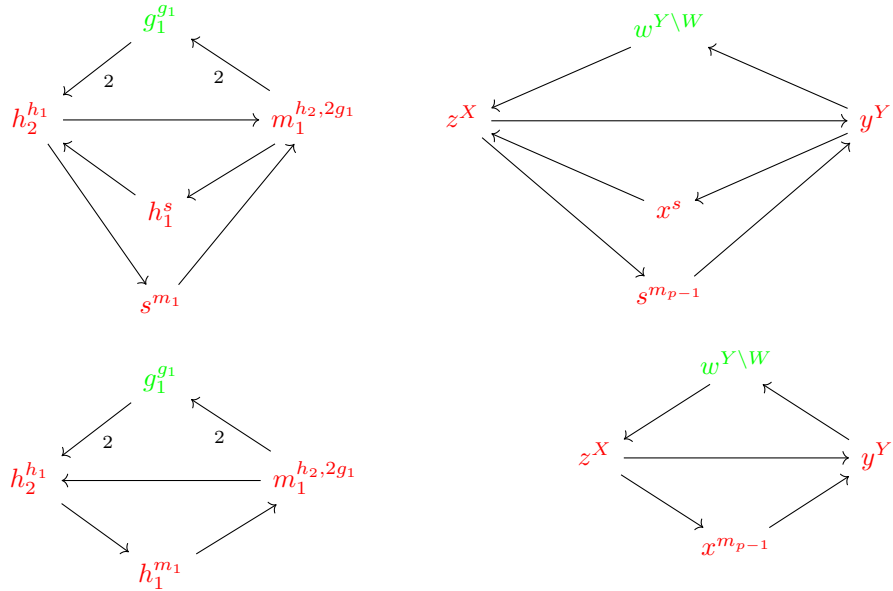
\[\adjustbox{scale=1}{\begin{tikzcd}
    &&\color{green}g_1^{g_1}\arrow[dl,"2"]&&\\
    &\color{red}h_2^{h_1}\arrow[ddr]\arrow[rr]&&\color{red}m_1^{h_2,2g_1}\arrow[ul,"2"]\arrow[dl]&\\
    &&\color{red}h_1^{s}\arrow[ul]\\&&
    \color{red}s^{m_{1}}\arrow[uur]
\end{tikzcd}}\begin{tikzcd}
    &&\color{green}w^{Y\setminus W}\arrow[dll]\\
    \color{red}z^{X}\arrow[ddrr]\arrow[rrrr]&&&&\color{red}y^{Y}\arrow[ull]\arrow[dll]\\
    &&\color{red}x^{s}\arrow[ull]\\&&
    \color{red}s^{m_{p-1}}\arrow[uurr]
\end{tikzcd}\]\[\adjustbox{scale=1}{\begin{tikzcd}
    &&&\color{green}g_1^{g_1}\arrow[dl,"2"]&&\\
    &&\color{red}h_2^{h_1}\arrow[dr]&&\color{red}m_1^{h_2,2g_1}\arrow[ul,"2"]\arrow[ll]&\\
    &&&\color{red}h_1^{m_1}\arrow[ur]&&
\end{tikzcd}}\begin{tikzcd}
    &&&\color{green}w^{Y\setminus W}\arrow[dl]\\&&
    \color{red}z^{X}\arrow[rr]\arrow[dr]&&\color{red}y^{Y}\arrow[ul]\\&
    &&\color{red}x^{m_{p-1}}\arrow[ur]&&\\
\end{tikzcd}\]

and vertices in $D\setminus\Tilde{D}(T_{1,p,q})$ all red, then after mutating $\hat{D}$ following step 6(b) to (d) of the mutation sequence in section \ref{puncture}, all vertices in $\hat{D}$ will be red.

\end{lemma}

\begin{proof}
    Since after step 1 to 4(b), the only vertices affected are vertices in $\Tilde{D}(T_{1,p,q})$ and adjacent vertices, to prove the first part, we only need to observe that $s$ (or $u,v$) is affected by the mutations in the way of $l_{p-1},r_{p-1}$ when $n=1$, and the only time $l_{p-1},r_{p-1}$ are adjacent to additional frozen vertices is during step 2 (if $p=2$) or step 1 (if $p>2$).

    Now observe that after $\zeta$, the part of diagram shown is same as that when $n=1$ after step 6(a) in all cases, and that $\zeta$ does not change $\Tilde{D}(T_{1,p,q})\setminus\{w,x,y,z\}$ (since vertices in $\Tilde{D}(T_{1,p,q})\setminus\{w,x,y,z\}$ are adjacent to only $w$, and $\zeta$ does not involve $w$ nor these vertices), we may conclude that after step 6(b) to (d), the remaining vertices will be red.
\end{proof}

With this lemma, it suffices to show that step 5 satisfies the conditions of $\zeta$ in the lemma.

As the case $p=3,q=1$ (which corresponds to the case where $p=2,q=1$ in the lemma) and the remaining cases have similar diagrams and mutation sequences, we shall give the proof for the remaining cases (that is, $p>3$ or $q>1$).

After step 4(b), by Lemma \ref{one_puncture_summary}, at the vertices involved in $\gamma$ and adjacent vertices, the diagram looks like this:

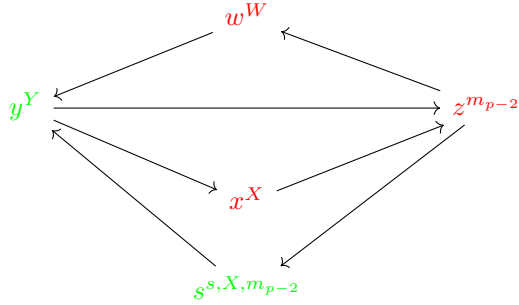
\[\begin{tikzcd}
    &&\color{red}w^{W}\arrow[dll]\\
    \color{green}y^{Y}\arrow[rrrr]\arrow[drr]&&&&\color{red}z^{m_{p-2}}\arrow[ddll]\arrow[ull]\\
    &&\color{red}x^{X}\arrow[urr]\\&&
    \color{green}s^{s,X,m_{p-2}}\arrow[uull]
\end{tikzcd}\]

where $w=\begin{cases}
    h_2&\text{if }p=3,q>1\\
    m_2&\text{if }p=4\\
    m_1&\text{if }p=5\\
    r_{p-5}&\text{otherwise}
\end{cases}, x=\begin{cases}
    h_1&\text{if }p=3\\
    m_{p-3}&\text{otherwise}
\end{cases},y=\begin{cases}
    m_1&\text{if }p=3\\
    l_{p-3}&\text{otherwise}
\end{cases},
z=\begin{cases}
    h_{q+1}&\text{if }p=3\\
    r_{p-3}&\text{otherwise}
\end{cases}$,

$W,Y$ are sets of (frozen) vertices with $W\subset Y$, and $X=\begin{cases}
    h_1&\text{if }p=3\\
    l_{p-3}&\text{otherwise}
\end{cases}$.

In terms of triangulation, after step 4(b), the arc $w$ wraps in the remaining arcs:

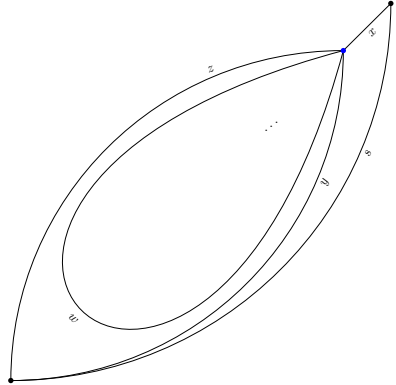
\[\adjustbox{scale=0.5}{\begin{tikzpicture}[every edge quotes/.style={auto=right}]
    \begin{scope}[every node/.style={sloped,allow upside down}][every edge quotes/.style={auto=right}] 
        \node at (6.85,6.85) {$\iddots$};
        \draw (0,0) to [bend right=45,black] node[near end,below] {$y$} (8.75,8.75);
        \draw (0,0) to [bend left=45,black] node[near end,above] {$z$} (8.75,8.75);
        \draw (0,0) to [bend right=45,black] node[near end,below,black] {$s$} (10,10);
        \draw (8.75,8.75) to node[below] {$x$} (10,10);
        \draw (8.75,8.75) to [loop,in=255,out=195,min distance=150mm] node[below=0.1] {$w$} (8.75,8.75);
        \fill (0,0) circle (2pt);
        \fill (10,10) circle (2pt);
        \fill[blue] (8.75,8.75) circle (2pt);
    \end{scope}
\end{tikzpicture}}\]

Notice the once-punctured digon formed by $y,z,s,x$. If we flip the arcs $s,x$, then the top right puncture is tagged notched. However, mutating at $x$ after $s$ is impossible as $x$ is still red:

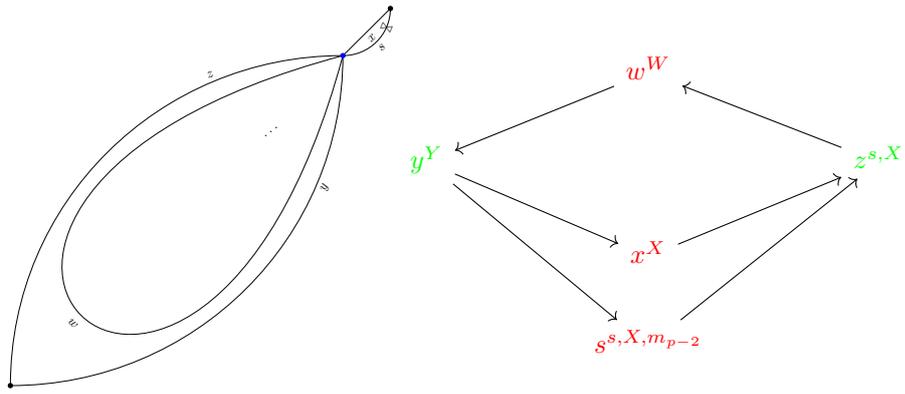
\adjustbox{scale=0.5}{\begin{tikzpicture}[every edge quotes/.style={auto=right}]
    \begin{scope}[every node/.style={sloped,allow upside down}][every edge quotes/.style={auto=right}]
        \node at (6.85,6.85) {$\iddots$};
        \draw (0,0) to [bend right=45,black] node[near end,below] {$y$} (8.75,8.75);
        \draw (0,0) to [bend left=45,black] node[near end,above] {$z$} (8.75,8.75);
        \draw (8.75,8.75) to [bend right=45,black] node[below,black] {$s$} node[near end,rotate=90] {$\bowtie$}  (10,10);
        \draw (8.75,8.75) to node[below] {$x$} (10,10);
        \draw (8.75,8.75) to [loop,in=255,out=195,min distance=150mm] node[below=0.1] {$w$} (8.75,8.75);
        \fill (0,0) circle (2pt);
        \fill (10,10) circle (2pt);
        \fill[blue] (8.75,8.75) circle (2pt);
    \end{scope}
\end{tikzpicture}}\begin{tikzcd}[baseline=-150pt]
    &&\color{red}w^{W}\arrow[dll]\\
    \color{green}y^{Y}\arrow[ddrr]\arrow[drr]&&&&\color{green}z^{s,X}\arrow[ull]\\
    &&\color{red}x^{X}\arrow[urr]\\&&
    \color{red}s^{s,X,m_{p-2}}\arrow[uurr]
\end{tikzcd}

Observe that by treating $s$ and $x$ as the same arc, there is a once-punctured digon formed by $w,s$ and $x$, $y,z$. Mutate at $y,z$:

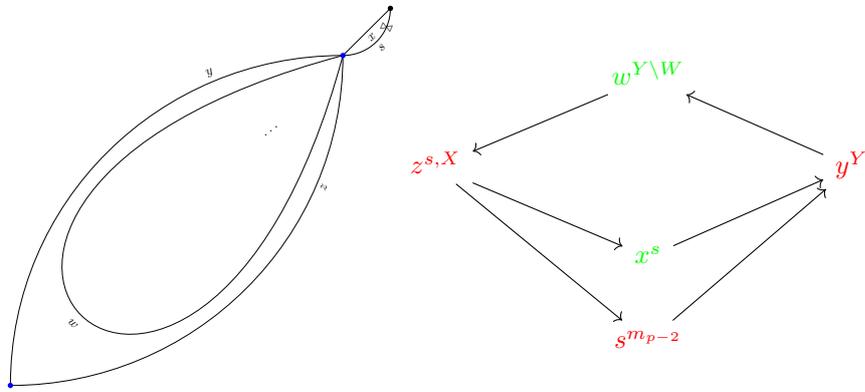
\adjustbox{scale=0.5}{\begin{tikzpicture}[every edge quotes/.style={auto=right}]
    \begin{scope}[every node/.style={sloped,allow upside down}][every edge quotes/.style={auto=right}]
        \node at (6.85,6.85) {$\iddots$};
        \draw (0,0) to [bend right=45,black] node[near end,below] {$z$} (8.75,8.75);
        \draw (0,0) to [bend left=45,black] node[near end,above] {$y$} (8.75,8.75);
        \draw (8.75,8.75) to [bend right=45,black] node[below,black] {$s$} node[near end,rotate=90] {$\bowtie$}  (10,10);
        \draw (8.75,8.75) to node[below] {$x$} (10,10);
        \draw (8.75,8.75) to [loop,in=255,out=195,min distance=150mm] node[below=0.1] {$w$} (8.75,8.75);
        \fill[blue] (0,0) circle (2pt);
        \fill (10,10) circle (2pt);
        \fill[blue] (8.75,8.75) circle (2pt);
    \end{scope}
\end{tikzpicture}}\begin{tikzcd}[baseline=-150pt]
    &&\color{green}w^{Y\setminus W}\arrow[dll]\\
    \color{red}z^{s,X}\arrow[ddrr]\arrow[drr]&&&&\color{red}y^{Y}\arrow[ull]\\
    &&\color{green}x^{s}\arrow[urr]\\&&
    \color{red}s^{m_{p-2}}\arrow[uurr]
\end{tikzcd}

Now we can mutate at $x$ and tag the top right puncture notched:

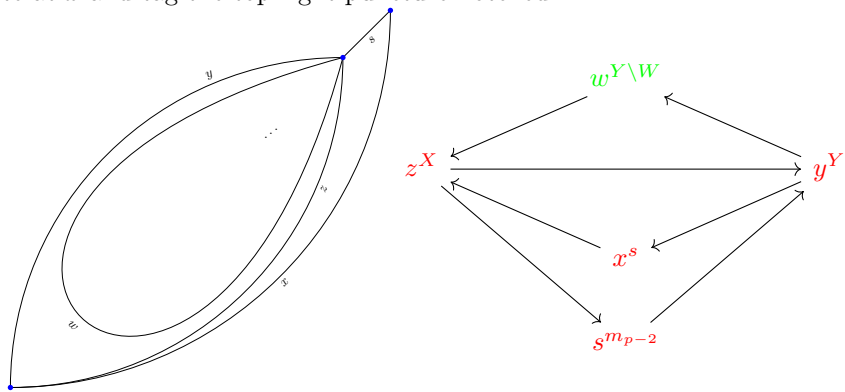
\adjustbox{scale=0.5}{\begin{tikzpicture}[every edge quotes/.style={auto=right}]
    \begin{scope}[every node/.style={sloped,allow upside down}][every edge quotes/.style={auto=right}]
        \node at (6.85,6.85) {$\iddots$};
        \draw (0,0) to [bend right=45,black] node[near end,below] {$z$} (8.75,8.75);
        \draw (0,0) to [bend left=45,black] node[near end,above] {$y$} (8.75,8.75);
        \draw (8.75,8.75) to node[below,black] {$s$}  (10,10);
        \draw (0,0) to [bend right=45,black] node[below] {$x$} (10,10);
        \draw (8.75,8.75) to [loop,in=255,out=195,min distance=150mm] node[below=0.1] {$w$} (8.75,8.75);
        \fill[blue] (0,0) circle (2pt);
        \fill[blue] (10,10) circle (2pt);
        \fill[blue] (8.75,8.75) circle (2pt);
    \end{scope}
\end{tikzpicture}}\begin{tikzcd}[baseline=-150pt]
    &&\color{green}w^{Y\setminus W}\arrow[dll]\\
    \color{red}z^{X}\arrow[ddrr]\arrow[rrrr]&&&&\color{red}y^{Y}\arrow[ull]\arrow[dll]\\
    &&\color{red}x^{s}\arrow[ull]\\&&
    \color{red}s^{m_{p-2}}\arrow[uurr]
\end{tikzcd}

and we obtain the desired diagram in Lemma \ref{one_puncture_summary}. Therefore, the mutation sequence $\Delta_0$ is a maximal green sequence. $\blacksquare$

\section{The case where $n>1$}\label{n>1}

\subsection{Constructing the triangulation for $n>1$}\label{construction_n}

As we have shown in section \ref{puncture} that maximal green sequences do not exist when $p=1$, we shall omit the construction in this case.

An orientable surface of genus $2$ with $p$ punctures and $q$ orbifold points, $p>1$ has triangulation $T_{2,p,q}$, as shown below:

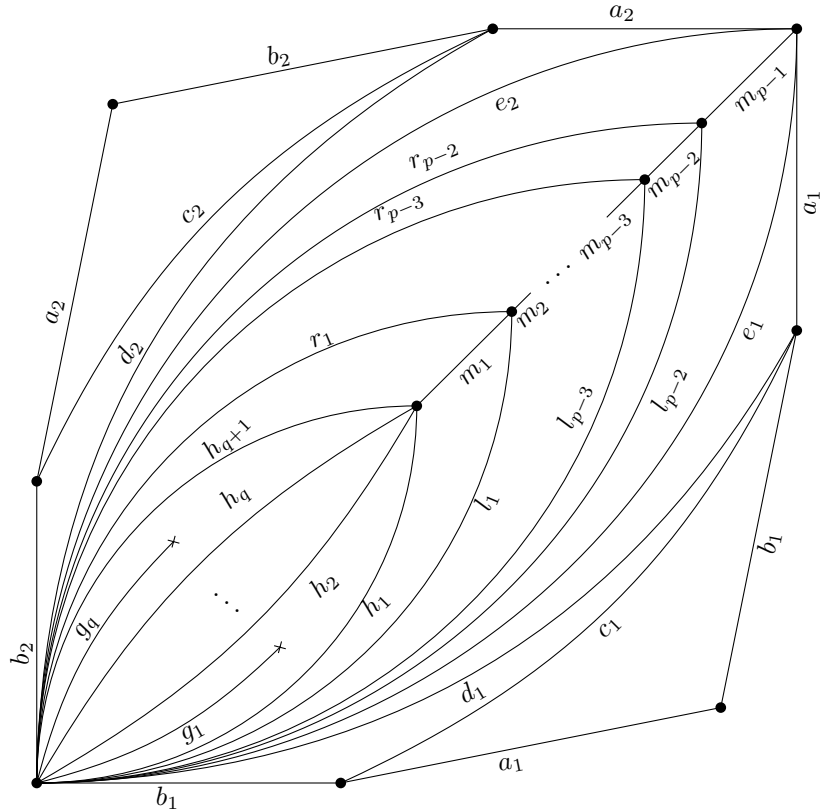
\[\begin{tikzpicture}[every edge quotes/.style={auto=right}]
    \begin{scope}[every node/.style={sloped,allow upside down}][every edge quotes/.style={auto=right}]

\node at (6.85,6.85) {$\iddots$};
        \node at ($(0,5)!.5!(5,0)$) {$\ddots$};
        \draw (0,0) to [bend right=15,black] node [below=.15,right] {$g_1$} (3.2,1.8);
        \draw (0,0) to [bend left=15,black] node [above=.15,right] {$g_q$} (1.8,3.2);
        \draw (0,0) to [bend right=45,black] node[below=.15,right=1] {$h_1$} (5,5);
        \draw (0,0) to [bend left=45,black] node[above=.15,right=1] {$h_{q+1}$} (5,5);
        \draw (0,0) to [bend right=15,black] node[right=1,below] {$h_2$} (5,5);
        \draw (0,0) to [bend left=15,black] node[right=1,above] {$h_q$} (5,5);
        \draw (0,0) to [bend right=45,black] node[near end,right,below=-.1] {$l_1$} (6.25,6.25);
        \draw (0,0) to [bend left=45] node[near end,right,above=-.1,black] {$r_1$} (6.25,6.25);
        \draw (0,0) to [bend right=45,black] node[near end,below] {$l_{p-2}$} (8.75,8.75);
        \draw (0,0) to [bend left=45,black] node[near end,above] {$r_{p-2}$} (8.75,8.75);
        \draw (0,0) to [bend right=45,black] node[near end,above] {$l_{p-3}$} (8,8);
        \draw (0,0) to [bend left=45,black] node[near end,above] {$r_{p-3}$} (8,8);
        \draw (0,0) to [bend right=45,black] node[near end,below,black] {$e_1$} (10,10);
        \draw (0,0) to [bend left=45,black] node[near end,below,black] {$e_2$} (10,10);
        \draw (5,5) to node[below] {$m_1$} (6.25,6.25);
        \draw (8.75,8.75) to node[below] {$m_{p-1}$} (10,10);
        \draw (6.25,6.25) to node[below] {$m_2$} (6.5,6.5);
        \draw (8,8) to node[left=.2,below] {$m_{p-2}$} (8.75,8.75);
        \draw (7.5,7.5) to node[below=.2,left] {$m_{p-3}$} (8,8);
        \draw (0,0) to node[below=.2,left] {$b_1$} (4,0);
        \draw (4,0) to node[below=.2,left] {$a_1$} (9,1);
        \draw (9,1) to node[below=.2,left] {$b_1$} (10,6);
        \draw (10,6) to node[below=.2,left] {$a_1$} (10,10);
        \draw (0,0) [bend right=30,black] to node[below=.2,left] {$d_1$} (10,6);
        \draw (4,0) [bend right=20,black] to node[below=.2,left] {$c_1$} (10,6);
        \draw (0,0) to node[above=.2,left] {$b_2$} (0,4);
        \draw (0,4) to node[above=.2,left] {$a_2$} (1,9);
        \draw (1,9) to node[above=.2,left] {$b_2$} (6,10);
        \draw (6,10) to node[above=.2,left] {$a_2$} (10,10);
        \draw (0,0) [bend left=30,black] to node[above=.2,left] {$d_2$} (6,10);
        \draw (0,4) [bend left=20,black] to node[above=.2,left] {$c_2$} (6,10);
        \fill (0,0) circle (2pt);
        \fill (10,10) circle (2pt);
        \fill (5,5) circle (2pt);
        \fill (3.2,1.8) node[cross=2pt,rotate=30] {};
        \fill (1.8,3.2) node[cross=2pt,rotate=30] {};
        \fill (6.25,6.25) circle (2pt);
        \fill (8.75,8.75) circle (2pt);
        \fill (8,8) circle (2pt);
        \fill (4,0) circle (2pt);
        \fill (10,6) circle (2pt);
        \fill (9,1) circle (2pt);
        \fill (0,4) circle (2pt);
        \fill (6,10) circle (2pt);
        \fill (1,9) circle (2pt);
    \end{scope}
\end{tikzpicture}\]

Notice that the triangulation consists of two parts: one of them being the subtriangulation mentioned in page \pageref{core}, the other being two copies of the following pentagon (here we have $k=1,2$):

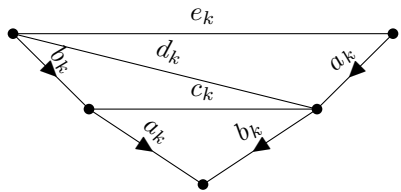
\[\begin{tikzpicture}[every edge quotes/.style={auto=right}]
    \begin{scope}[every node/.style={sloped,allow upside down}][every edge quotes/.style={auto=right}]
        \draw (0,2) to node[above] {$e_k$} (5,2);
        \draw (0,2) to node[above] {$d_k$} (4,1);
        \draw (1,1) to node[above] {$c_k$} (4,1);
        \draw (0,2) to node[above=-0.05] {$b_k$} node [rotate=-90,black] {$\blacktriangle$} (1,1);
        \draw (1,1) to node[above] {$a_k$} node [rotate=-90,black] {$\blacktriangle$} (2.5,0);
        \draw (2.5,0) to node[above] {$b_k$} node [rotate=90,black] {$\blacktriangle$} (4,1);
        \draw (4,1) to node[above] {$a_k$} node [rotate=90,black] {$\blacktriangle$} (5,2);
        \fill (0,2) circle (2pt);
        \fill (5,2) circle (2pt);
        \fill (1,1) circle (2pt);
        \fill (4,1) circle (2pt);
        \fill (2.5,0) circle (2pt);
    \end{scope}
\end{tikzpicture}\]

We shall use an "arc" labeled $\hat{e_k}$ as a shorthand of this pentagon. For example, $T_{2,p,q}$ can be rewritten as follows:

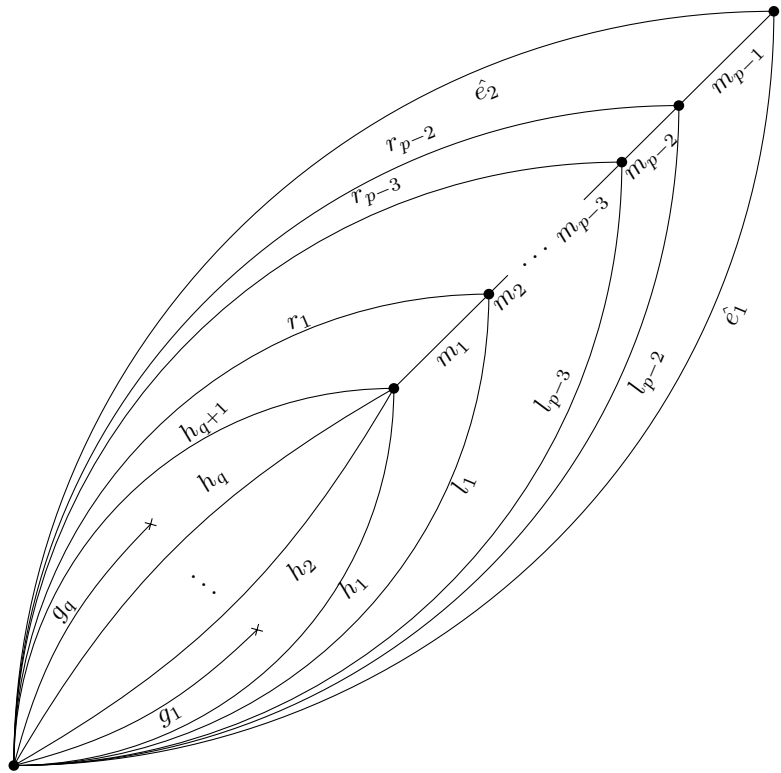
\[\begin{tikzpicture}[every edge quotes/.style={auto=right}]
    \begin{scope}[every node/.style={sloped,allow upside down}][every edge quotes/.style={auto=right}]

\node at (6.85,6.85) {$\iddots$};
        \node at ($(0,5)!.5!(5,0)$) {$\ddots$};
        \draw (0,0) to [bend right=15,black] node [below=.15,right] {$g_1$} (3.2,1.8);
        \draw (0,0) to [bend left=15,black] node [above=.15,right] {$g_q$} (1.8,3.2);
        \draw (0,0) to [bend right=45,black] node[below=.15,right=1] {$h_1$} (5,5);
        \draw (0,0) to [bend left=45,black] node[above=.15,right=1] {$h_{q+1}$} (5,5);
        \draw (0,0) to [bend right=15,black] node[right=1,below] {$h_2$} (5,5);
        \draw (0,0) to [bend left=15,black] node[right=1,above] {$h_q$} (5,5);
        \draw (0,0) to [bend right=45,black] node[near end,right,below=-.1] {$l_1$} (6.25,6.25);
        \draw (0,0) to [bend left=45] node[near end,right,above=-.1,black] {$r_1$} (6.25,6.25);
        \draw (0,0) to [bend right=45,black] node[near end,below] {$l_{p-2}$} (8.75,8.75);
        \draw (0,0) to [bend left=45,black] node[near end,above] {$r_{p-2}$} (8.75,8.75);
        \draw (0,0) to [bend right=45,black] node[near end,above] {$l_{p-3}$} (8,8);
        \draw (0,0) to [bend left=45,black] node[near end,above] {$r_{p-3}$} (8,8);
        \draw (0,0) to [bend right=45,black] node[near end,below,black] {$\hat{e_1}$} (10,10);
        \draw (0,0) to [bend left=45,black] node[near end,below,black] {$\hat{e_2}$} (10,10);
        \draw (5,5) to node[below] {$m_1$} (6.25,6.25);
        \draw (8.75,8.75) to node[below] {$m_{p-1}$} (10,10);
        \draw (6.25,6.25) to node[below] {$m_2$} (6.5,6.5);
        \draw (8,8) to node[left=.2,below] {$m_{p-2}$} (8.75,8.75);
        \draw (7.5,7.5) to node[below=.2,left] {$m_{p-3}$} (8,8);
        \fill (0,0) circle (2pt);
        \fill (10,10) circle (2pt);
        \fill (5,5) circle (2pt);
        \fill (3.2,1.8) node[cross=2pt,rotate=30] {};
        \fill (1.8,3.2) node[cross=2pt,rotate=30] {};
        \fill (6.25,6.25) circle (2pt);
        \fill (8.75,8.75) circle (2pt);
        \fill (8,8) circle (2pt);
    \end{scope}
\end{tikzpicture}\]

If we glue the arcs $e_1,e_2$ of the pentagons along the arcs $\hat{e_1},\hat{e_2}$ respectively, we get back the triangulation $T_{2,p,q}$.

The following diagram is associated to $T_{2,p,q}$:

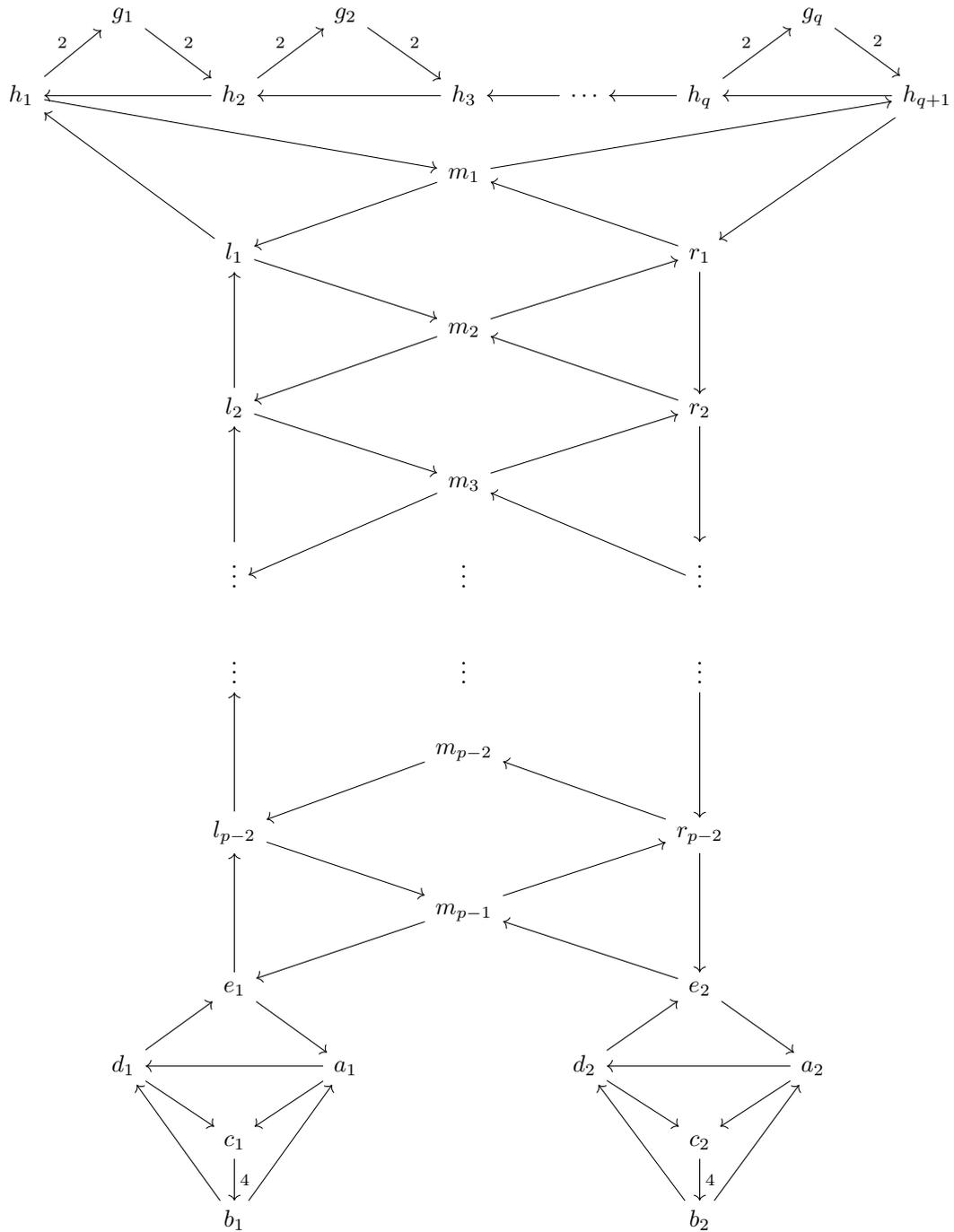
\[\begin{tikzcd}
 & g_1\arrow[dr,"2"] & & g_2\arrow[dr,"2"]&&&&g_q\arrow[dr,"2"]\\
 h_1\arrow[drrrr]\arrow[ur,"2"] && h_2\arrow[ur,"2"]\arrow[ll] && h_3\arrow[ll]&\cdots\arrow[l]&h_q\arrow[l]\arrow[ur,"2"]&&h_{q+1}\arrow[ll]\arrow[ddll]\\
 &&&&m_1\arrow[dll]\arrow[urrrr]&&&&\\
 &&l_1\arrow[uull]\arrow[drr]&&&&r_1\arrow[ull]\arrow[dd]&&\\
 &&&&m_2\arrow[dll]\arrow[urr]&&&&\\
 &&l_2\arrow[uu]\arrow[drr]&&&&r_2\arrow[ull]\arrow[dd]&&\\
 &&&&m_3\arrow[dll]\arrow[urr]&&&&\\
 &&\vdots\arrow[uu]&&\vdots&&\vdots\arrow[ull]\\
 &&\vdots&&\vdots&&\vdots\arrow[dd]\\
 &&&&m_{p-2}\arrow[dll]&&&&\\
 &&l_{p-2}\arrow[uu]\arrow[drr]&&&&r_{p-2}\arrow[ull]\arrow[dd]\\
 &&&&m_{p-1}\arrow[urr]\arrow[dll]&&&&\\
 &&e_1\arrow[uu]\arrow[dr]&&&&e_2\arrow[ull]\arrow[dr]\\
    &d_1\arrow[ur]\arrow[dr]&&a_1\arrow[ll]\arrow[dl]&&d_2\arrow[ur]\arrow[dr]&&a_2\arrow[ll]\arrow[dl]\\
    &&c_1\arrow[d,"4"]&&&&c_2\arrow[d,"4"]\\
    &&b_1\arrow[uur]\arrow[uul]&&&&b_2\arrow[uur]\arrow[uul]
\end{tikzcd}\]

An orientable surface of genus $n$ with $p$ punctures and $q$ orbifold points, $n,p,q\in\mathbb{Z}$, $n>2$, $p,q\geq1$, has triangulation $T_{n,p,q}$, as shown below:

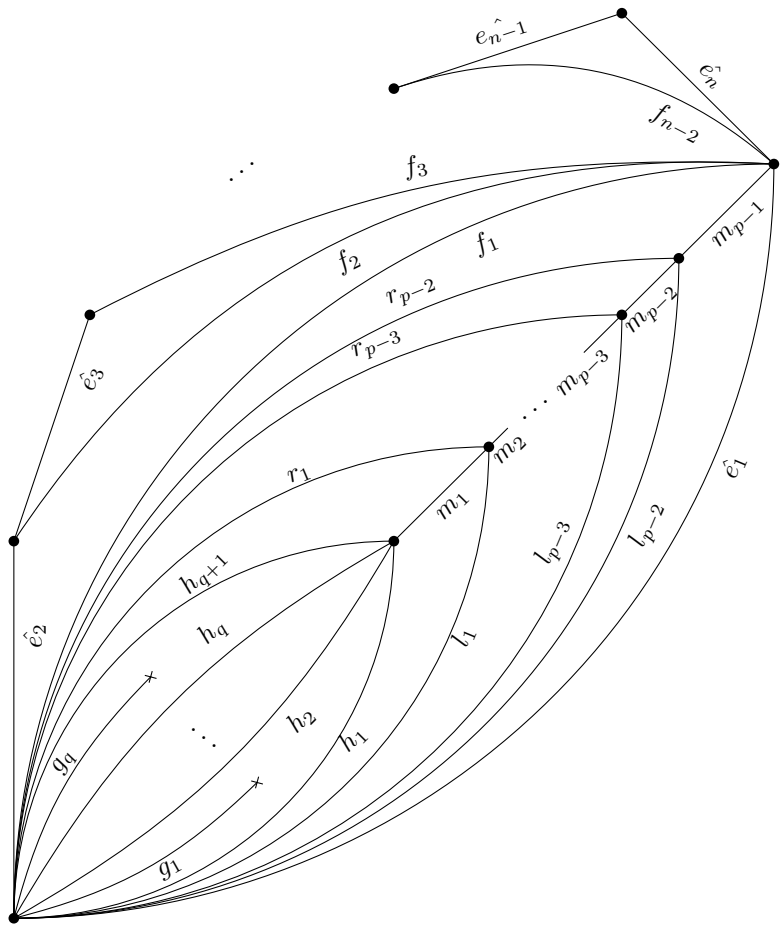
\[\begin{tikzpicture}[every edge quotes/.style={auto=right}]
    \begin{scope}[every node/.style={sloped,allow upside down}][every edge quotes/.style={auto=right}]

\node at (6.85,6.85) {$\iddots$};
        \node at ($(0,5)!.5!(5,0)$) {$\ddots$};
        \draw (0,0) to [bend right=15,black] node [below=.15,right] {$g_1$} (3.2,1.8);
        \draw (0,0) to [bend left=15,black] node [above=.15,right] {$g_q$} (1.8,3.2);
        \draw (0,0) to [bend right=45,black] node[below=.15,right=1] {$h_1$} (5,5);
        \draw (0,0) to [bend left=45,black] node[above=.15,right=1] {$h_{q+1}$} (5,5);
        \draw (0,0) to [bend right=15,black] node[right=1,below] {$h_2$} (5,5);
        \draw (0,0) to [bend left=15,black] node[right=1,above] {$h_q$} (5,5);
        \draw (0,0) to [bend right=45,black] node[near end,right,below=-.1] {$l_1$} (6.25,6.25);
        \draw (0,0) to [bend left=45] node[near end,right,above=-.1,black] {$r_1$} (6.25,6.25);
        \draw (0,0) to [bend right=45,black] node[near end,below] {$l_{p-2}$} (8.75,8.75);
        \draw (0,0) to [bend left=45,black] node[near end,above] {$r_{p-2}$} (8.75,8.75);
        \draw (0,0) to [bend right=45,black] node[near end,above] {$l_{p-3}$} (8,8);
        \draw (0,0) to [bend left=45,black] node[near end,above] {$r_{p-3}$} (8,8);
        \draw (0,0) to [bend right=45,black] node[near end,below,black] {$\hat{e_1}$} (10,10);
        \draw (0,0) to node[near end,below,black] {$\hat{e_2}$} (0,5);
        \draw (0,5) to node[near end,below,black] {$\hat{e_3}$} (1,8);
        \draw (8,12) to node[above,black] {$\hat{e_n}$} (10,10);
        \draw (5,11) to node[above,black] {$\hat{e_{n-1}}$} (8,12);
        \draw (0,0) to [bend left=45,black] node[near end,below] {$f_1$} (10,10);
        \draw (0,5) to [bend left=30,black] node[below] {$f_2$} (10,10);
        \draw (1,8) to [bend left=15] node [above] {$f_3$} (10,10);
        \draw (5,11) to [bend left=30] node[near end,below] {$f_{n-2}$} (10,10);
        \draw (5,5) to node[below] {$m_1$} (6.25,6.25);
        \draw (8.75,8.75) to node[below] {$m_{p-1}$} (10,10);
        \draw (6.25,6.25) to node[below] {$m_2$} (6.5,6.5);
        \draw (8,8) to node[left=.2,below] {$m_{p-2}$} (8.75,8.75);
        \draw (7.5,7.5) to node[below=.2,left] {$m_{p-3}$} (8,8);
        \node at (3,10) {\reflectbox{$\ddots$}};
        \fill (0,0) circle (2pt);
        \fill (10,10) circle (2pt);
        \fill (0,5) circle (2pt);
        \fill (1,8) circle (2pt);
        \fill (8,12) circle (2pt);
        \fill (5,5) circle (2pt);
        \fill (3.2,1.8) node[cross=2pt,rotate=30] {};
        \fill (1.8,3.2) node[cross=2pt,rotate=30] {};
        \fill (6.25,6.25) circle (2pt);
        \fill (8.75,8.75) circle (2pt);
        \fill (8,8) circle (2pt);
        \fill (5,11) circle (2pt);
    \end{scope}
\end{tikzpicture}\]

This triangulation is same as $T_{2,p,q}$ with the "arc" $\hat{e_2}$ replaced by the "$n$-gon" bounded by the arc $f_1$ and the "arcs" $\hat{e_2},\hat{e_3},\dots,\hat{e_n}$, triangulated by the arcs $f_2,f_3,\dots,f_{n-2}$.

The following diagram is associated to $T_{n,p,q}$:

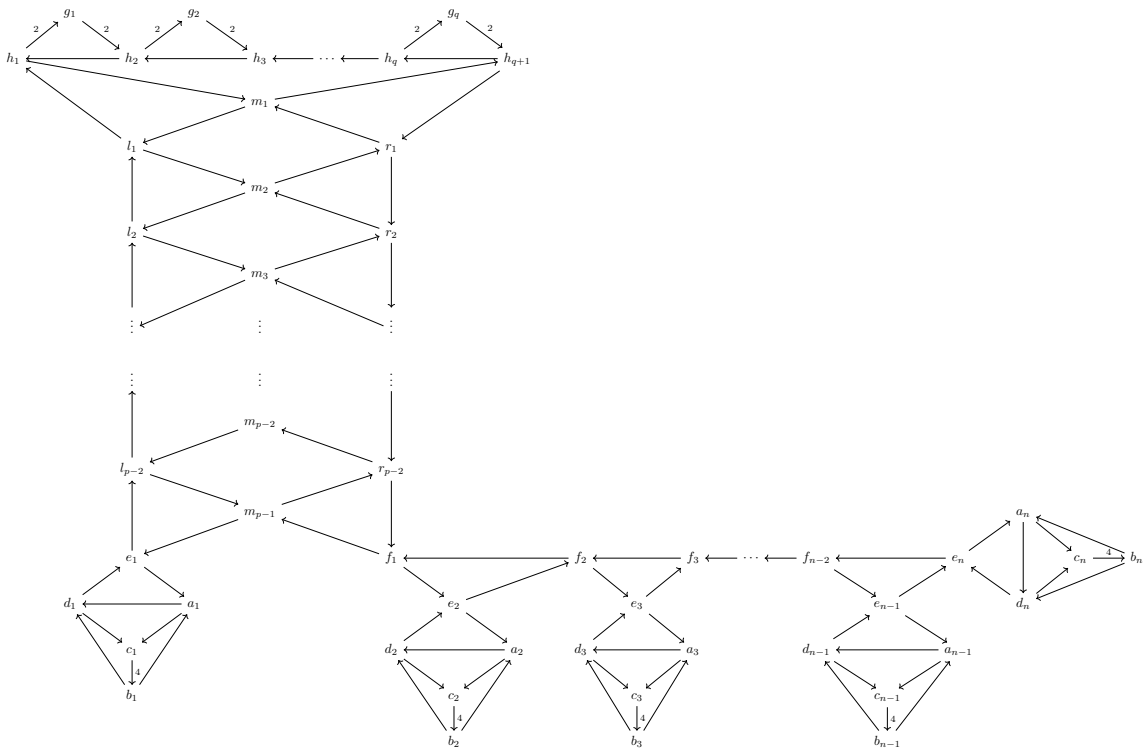
\[\adjustbox{scale=0.5,center}{\begin{tikzcd}
 & g_1\arrow[dr,"2"] & & g_2\arrow[dr,"2"]&&&&g_q\arrow[dr,"2"]\\
 h_1\arrow[drrrr]\arrow[ur,"2"] && h_2\arrow[ur,"2"]\arrow[ll] && h_3\arrow[ll]&\cdots\arrow[l]&h_q\arrow[l]\arrow[ur,"2"]&&h_{q+1}\arrow[ll]\arrow[ddll]\\
 &&&&m_1\arrow[dll]\arrow[urrrr]&&&&\\
 &&l_1\arrow[uull]\arrow[drr]&&&&r_1\arrow[ull]\arrow[dd]&&\\
 &&&&m_2\arrow[dll]\arrow[urr]&&&&\\
 &&l_2\arrow[uu]\arrow[drr]&&&&r_2\arrow[ull]\arrow[dd]&&\\
 &&&&m_3\arrow[dll]\arrow[urr]&&&&\\
 &&\vdots\arrow[uu]&&\vdots&&\vdots\arrow[ull]\\
 &&\vdots&&\vdots&&\vdots\arrow[dd]\\
 &&&&m_{p-2}\arrow[dll]&&&&\\
 &&l_{p-2}\arrow[uu]\arrow[drr]&&&&r_{p-2}\arrow[ull]\arrow[dd]\\
 &&&&m_{p-1}\arrow[urr]\arrow[dll]&&&&&&&&&&&&a_n\arrow[dd]\arrow[dr]\\
 &&e_1\arrow[uu]\arrow[dr]&&&&f_1\arrow[ull]\arrow[dr]&&&f_2\arrow[dr]\arrow[lll]&&f_3\arrow[ll]&\cdots\arrow[l]&f_{n-2}\arrow[l]\arrow[dr]&&e_n\arrow[ll]\arrow[ur]&&c_n\arrow[r,"4"]&b_n\arrow[ull]\arrow[dll]\\
    &d_1\arrow[ur]\arrow[dr]&&a_1\arrow[ll]\arrow[dl]&&&&e_2\arrow[dr]\arrow[urr]&&&e_3\arrow[dr]\arrow[ur]&&&&e_{n-1}\arrow[ur]\arrow[dr]&&d_n\arrow[ul]\arrow[ur]\\
    &&c_1\arrow[d,"4"]&&&&d_2\arrow[ur]\arrow[dr]&&a_2\arrow[ll]\arrow[dl]&d_3\arrow[ur]\arrow[dr]&&a_3\arrow[ll]\arrow[dl]&&d_{n-1}\arrow[ur]\arrow[dr]&&a_{n-1}\arrow[ll]\arrow[dl]\\
    &&b_1\arrow[uur]\arrow[uul]&&&&&c_2\arrow[d,"4"]&&&c_3\arrow[d,"4"]&&&&c_{n-1}\arrow[d,"4"]\\
    &&&&&&&b_2\arrow[uul]\arrow[uur]&&&b_3\arrow[uul]\arrow[uur]&&&&b_{n-1}\arrow[uul]\arrow[uur]
\end{tikzcd}}\]

\subsection{Proof of the main result for $n>1$}\label{proof_n}

\begin{theorem}\label{main_n}
    For an orbifold $\mathcal{O}$ of genus $n$ with $p$ punctures and $q$ orbifold points, where $n>1$, the diagram $D(T_{n,p,q})$ associated to the triangulation $T_{n,p,q}$ has a maximal green sequence if $p\geq 2$. Moreover, if $p=1$ then $D(T)$ does not admit a maximal green sequence for any triangulation $T$ of $\mathcal{O}$.
\end{theorem}

The proof of the "moreover" part is mentioned in section \ref{puncture}, and we shall prove the rest by providing maximal green sequences.

The mutation sequence $\Delta_n$ formed by concatenating the following steps is a maximal green sequence of $D(T_{n,p,q})$ for $n>1$:

\begin{enumerate}
    \item Tagging alternate punctures notched. $(\alpha_{p-2},\alpha_{p-4},...,\alpha_{p-2\lfloor \frac{p-1}{2}\rfloor})$
    \item Tagging the first interior puncture notched. $(h_{q+1},h_q,...,h_1,m_{p+1-2\lfloor\frac{p}{2}\rfloor},h_2,h_3,...,h_{q+1-p+2\lfloor \frac{p}{2}\rfloor})$
    \item Tagging the remaining interior punctures notched and moving half the inner arcs away from the corner. $(\beta_{p-3},\beta_{p-5},...,\beta_{p+1-2\lfloor\frac{p}{2}\rfloor})$
    \item Moving arcs away from the outside puncture.
    \begin{enumerate}
        \item Moving core arcs away from the outside puncture. $(g_1,g_2,...,g_q,h_{q+\operatorname{sgn}(p-2)},h_{q+\operatorname{sgn}(p-2)-1},...,h_2,\mu)$
        \item Moving the remaining inner arcs away from the outside puncture. $(l_{p-2\lfloor\frac{p-1}{2}\rfloor},r_{p-2\lfloor\frac{p-1}{2}\rfloor},l_{p+2-2\lfloor\frac{p-1}{2}\rfloor},r_{p+2-2\lfloor\frac{p-1}{2}\rfloor},\\...,l_{p-4},r_{p-4})$
        \item Separating the $e_i$'s. $(f_1,f_2...,f_{n-2})$
        \item Moving arcs in each pentagon away from the outside puncture. $(\iota_1,\iota_2,...,\iota_n)$
        \item Moving $f_i$'s away from the outside puncture. $(f_{n-2},f_{n-1},...,f_1)$
    \end{enumerate}
    \item Tagging the outside puncture notched. $(\gamma)$
    \item Moving arcs back to the outside puncture.
    \begin{enumerate}
        \item Moving $f_i$'s back to the outside puncture. $(f_1,f_2...,f_{n-2})$
        \item Moving arcs in each pentagon back to the outside puncture. $(\pi_1,\pi_2,...,\pi_n)$
        \item Moving $f_i$'s back to their original positions. $(f_{n-2},f_{n-1},...,f_1)$
        \item Moving inner arcs back to the outside puncture. 
        $(\delta_{p-4},\delta_{p-6},...,\delta_{p-2\lfloor\frac{p-1}{2}\rfloor})$
        \item Moving core arcs back to the outside puncture. $(\mu,\varepsilon,g_1,g_2,...,g_{q-1},g_q,\omega)$
        \item (If $p>2$ and is even) Moving the remaining pair of inner arc back to the outside puncture. $(l_1,r_1)$
    \end{enumerate}
\end{enumerate}

where

\begin{itemize}
    \item $\alpha_k=r_k,l_k,m_{k+1},m_{k},l_k,r_k$
    \item $\beta_k=\begin{cases}
        r_k,l_k,m_{k+2},m_{k-1}&\text{if }k>1,\\
        r_k,l_k,m_{k+2},h_1&\text{otherwise}
    \end{cases}$
    \item $\mu=\begin{cases}
        m_{p+1-2\lfloor\frac{p}{2}\rfloor}&\text{if }p>2,\\
        \phi & \text{if }p=2
    \end{cases}$
    \item $\iota_k=a_k,e_k,c_k,b_k,a_k,d_k,e_k,c_k,b_k$
    \item $\pi_k=b_k,c_k,e_k,d_k,b_k,a_k,c_k,e_k,b_k$
    \item $\gamma=\begin{cases}
        m_{p-2},l_{p-2},r_{p-2},m_{p-2}&\text{if }p>2,\\
        h_1,m_1,h_{q+1},h_1 & \text{if }p=2
    \end{cases}$
    \item $\delta_k=r_k,r_{k+1},l_k,l_{k+1}$
    \item $\varepsilon=\begin{cases}
        h_2,h_3,...,h_{q+1}&\text{if }p\text{ is even and}>2\\
        h_2,h_3,...,h_q &\text{otherwise}
    \end{cases}$
    \item $\omega=\begin{cases}
        h_{q+1},g_q&\text{if }p\text{ is odd},\\
        \phi & \text{otherwise}
    \end{cases}$
\end{itemize}

As the triangulations when $n=2$ and $n>2$ are different, we shall divide the proof into two cases.

Notice that all steps except step 4(c) to step 6(c) appear in the mutation sequence $\Delta_1$. Therefore in both cases, with Lemma \ref{one_puncture_summary}, it suffices to show that the steps 4(c) to 6(c) satisfy the conditions for $\zeta$ in the lemma. Also, as the case $p=2,q=1$ and the remaining cases have similar diagrams and mutation sequences, we shall give the proof for the remaining cases (that is, $p>2$ or $q>1$).

\subsubsection{The case where $n=2$}\label{n=2}

When $n=2$, the steps 4(c) to 6(b) translate to 

\begin{enumerate}[start=4]
    \item Moving arcs away from the outside puncture.
    \begin{enumerate}[start=3]
        \item Separating the $e_i$'s. $(\phi)$
        \item Moving arcs in each pentagon away from the outside puncture. $(\iota_1,\iota_2)$
        \item Moving $f_i$'s away from the outside puncture. $(\phi)$
    \end{enumerate}
    \item Tagging the outside puncture notched. $(\gamma)$
    \item Moving arcs back to the outside puncture.
    \begin{enumerate}
        \item Moving $f_i$'s back to the outside puncture. $(\phi)$
        \item Moving arcs in each pentagon back to the outside puncture. $(\pi_1,\pi_2)$
        \item Moving $f_i$'s back to their original positions. $(\phi)$
    \end{enumerate}
\end{enumerate}

where

\begin{itemize}
    \item $\iota_k=a_k,e_k,c_k,b_k,a_k,d_k,e_k,c_k,b_k$
    \item $\pi_k=b_k,c_k,e_k,d_k,b_k,a_k,c_k,e_k,b_k$
    \item $\gamma=\begin{cases}
        m_{p-2},l_{p-2},r_{p-2},m_{p-2}&\text{if }p>2,\\
        h_1,m_1,h_{q+1},h_1 & \text{if }p=2
    \end{cases}$
\end{itemize}

According to Lemma \ref{one_puncture_summary}, after step 1 to step 4(b), the full subdiagram containing vertices involved in step 4(c) to step 6(c) and adjacent vertices look like this:

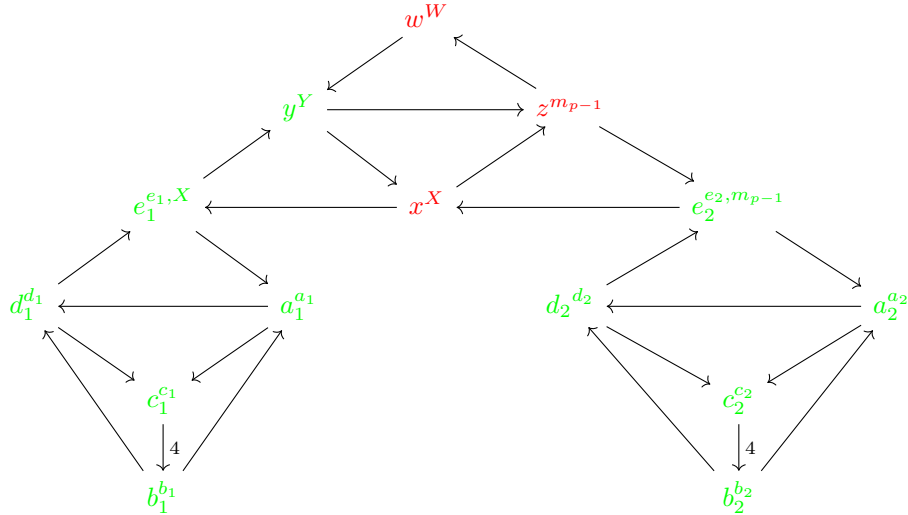
\[\begin{tikzcd}
    &&&\color{red}w^{W}\arrow[dl]\\&&
    \color{green}y^{Y}\arrow[rr]\arrow[dr]&&\color{red}z^{m_{p-1}}\arrow[dr]\arrow[ul]\\&
    \color{green}e_1^{e_1,X}\arrow[ur]\arrow[dr]&&\color{red}x^{X}\arrow[ll]\arrow[ur]&&\color{green}e_2^{e_2,m_{p-1}}\arrow[ll]\arrow[dr]\\
    \color{green}d_1^{d_1}\arrow[ur]\arrow[dr]&&\color{green}a_1^{a_1}\arrow[ll]\arrow[dl]&&\color{green}{d_2}^{d_2}\arrow[ur]\arrow[dr]&&\color{green}a_2^{a_2}\arrow[ll]\arrow[dl]\\
    &\color{green}c_1^{c_1}\arrow[d,"4"]&&&&\color{green}c_2^{c_2}\arrow[d,"4"]\\
    &\color{green}b_1^{b_1}\arrow[uur]\arrow[uul]&&&&\color{green}b_2^{b_2}\arrow[uur]\arrow[uul]
\end{tikzcd}\]

where $w=\begin{cases}
    h_2&\text{if }p=2,q>1\\
    m_2&\text{if }p=3\\
    m_1&\text{if }p=4\\
    r_{p-4}&\text{otherwise}
\end{cases}, x=\begin{cases}
    h_1&\text{if }p=2,q>1\\
    m_{p-2}&\text{otherwise}
\end{cases},y=\begin{cases}
    m_1&\text{if }p=2,q>1\\
    l_{p-2}&\text{otherwise}
\end{cases},
z=\begin{cases}
    h_{q+1}&\text{if }p=2,q>1\\
    r_{p-2}&\text{otherwise}
\end{cases}$

$W,Y$ are sets of (frozen) vertices with $W\subset Y$, and $X=\begin{cases}
    h_1&\text{if }p=2,q>1\\
    l_{p-2}&\text{otherwise}
\end{cases}$.

The corresponding subtriangulation looks like this:

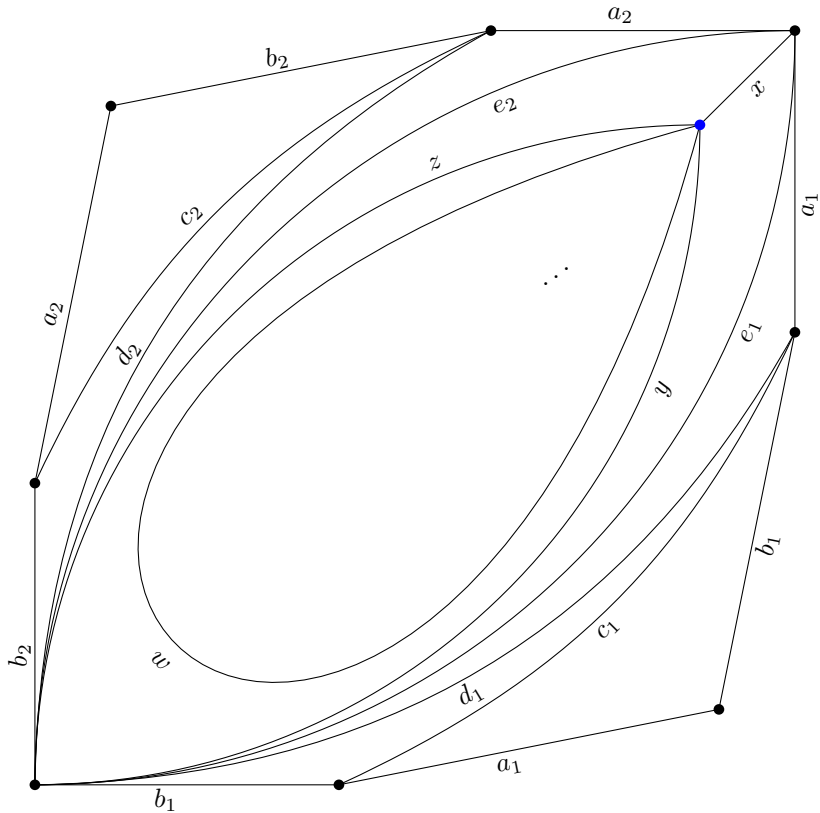
\[\hspace*{-0.25\linewidth}\begin{tikzpicture}[every edge quotes/.style={auto=right}]
    \begin{scope}[every node/.style={sloped,allow upside down}][every edge quotes/.style={auto=right}]
        \draw (0,0) to [bend right=45,black] node[near end,below] {$y$} (8.75,8.75);
        \draw (0,0) to [bend left=45,black] node[near end,above] {$z$} (8.75,8.75);
        \draw (0,0) to [bend right=45,black] node[near end,below,black] {$e_1$} (10,10);
        \draw (0,0) to [bend left=45,black] node[near end,below,black] {$e_2$} (10,10);
        \draw (8.75,8.75) to node[below] {$x$} (10,10);
        \draw (8.75,8.75) to [loop,in=255,out=195,min distance=150mm] node[below=0.1] {$w$} (8.75,8.75);
        \draw (0,0) to node[below=.2,left] {$b_1$} (4,0);
        \draw (4,0) to node[below=.2,left] {$a_1$} (9,1);
        \draw (9,1) to node[below=.2,left] {$b_1$} (10,6);
        \draw (10,6) to node[below=.2,left] {$a_1$} (10,10);
        \draw (0,0) [bend right=30,black] to node[below=.2,left] {$d_1$} (10,6);
        \draw (4,0) [bend right=20,black] to node[below=.2,left] {$c_1$} (10,6);
        \draw (0,0) to node[above=.2,left] {$b_2$} (0,4);
        \draw (0,4) to node[above=.2,left] {$a_2$} (1,9);
        \draw (1,9) to node[above=.2,left] {$b_2$} (6,10);
        \draw (6,10) to node[above=.2,left] {$a_2$} (10,10);
        \draw (0,0) [bend left=30,black] to node[above=.2,left] {$d_2$} (6,10);
        \draw (0,4) [bend left=20,black] to node[above=.2,left] {$c_2$} (6,10);
        \node at (6.85,6.85) {$\iddots$};
        \fill (0,0) circle (2pt);
        \fill (10,10) circle (2pt);
        \fill [blue] (8.75,8.75) circle (2pt);
        \fill (4,0) circle (2pt);
        \fill (10,6) circle (2pt);
        \fill (9,1) circle (2pt);
        \fill (0,4) circle (2pt);
        \fill (6,10) circle (2pt);
        \fill (1,9) circle (2pt);
    \end{scope}
\end{tikzpicture}\]

Step 4(d) consists of $(\iota_1,\iota_2)$. For convenience, we shall show the result of the mutation sequence $\iota_k$ in a lemma:

\begin{lemma}\label{lemma_iota}
    Mutating the following diagram (the blue colour of the vertex $t$ means that the greenness of $t$ does not matter):
    \[\adjustbox{scale=1}
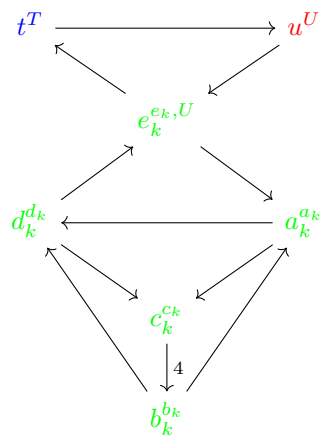
{\begin{tikzcd}
    \color{blue}t^{T}\arrow[rr]&&\color{red}u^{U}\arrow[dl]\\
    &\color{green}e_k^{e_k,U}\arrow[ul]\arrow[dr]&&\\
    \color{green}d_k^{d_k}\arrow[ur]\arrow[dr]&&\color{green}a_k^{a_k}\arrow[ll]\arrow[dl]\\
    &\color{green}c_k^{c_k}\arrow[d,"4"]\\
    &\color{green}b_k^{b_k}\arrow[uur]\arrow[uul]
\end{tikzcd}}\]
    with the mutation sequence $\iota_k$ gives the following diagram (after relocating some vertices):

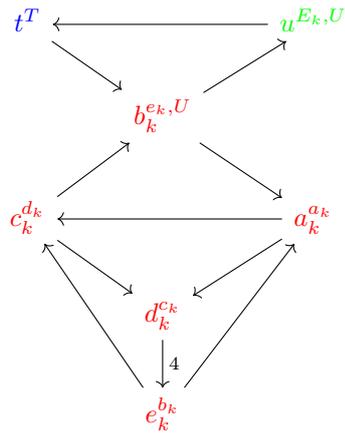
\[\adjustbox{scale=1}{\begin{tikzcd}
    \color{blue}t^{T}\arrow[dr]&&\color{green}u^{E_k,U}\arrow[ll]\\
    &\color{red}b_k^{e_k,U}\arrow[dr]\arrow[ur]&&\\
    \color{red}c_k^{d_k}\arrow[dr]\arrow[ur]&&\color{red}a_k^{a_k}\arrow[ll]\arrow[dl]\\
    &\color{red}d_k^{c_k}\arrow[d,"4"]\\
    &\color{red}e_k^{b_k}\arrow[uul]\arrow[uur]
\end{tikzcd}}\]

where $E_k=4e_k,4a_k,4c_k,4b_k,4d_k$.

\end{lemma}

\begin{proof}
    Direct computation.
\end{proof}

In terms of triangulation, at the start, the corresponding triangulation looks like this:

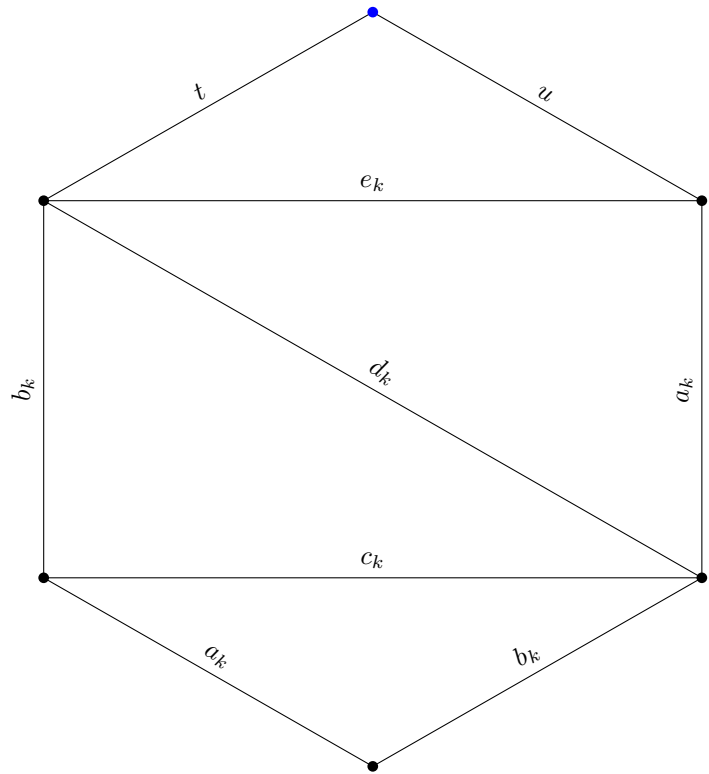
\[\begin{tikzpicture}[every edge quotes/.style={auto=right}]
    \begin{scope}[every node/.style={sloped,allow upside down}][every edge quotes/.style={auto=right}]
        \draw (-4.33,2.5) to node[above] {$t$} (0,5);
        \draw (0,5) to node[above] {$u$} (4.33,2.5);
        \draw (-4.33,2.5) to node[above] {$e_k$} (4.33,2.5);
        \draw (4.33,-2.5) to node[above] {$a_k$} (4.33,2.5);
        \draw (-4.33,-2.5) to node[above] {$b_k$} (-4.33,2.5);
        \draw (-4.33,2.5) to node[above] {$d_k$} (4.33,-2.5);
        \draw (-4.33,-2.5) to node[above] {$c_k$} (4.33,-2.5);
        \draw (0,-5) to node[above] {$b_k$} (4.33,-2.5);
        \draw (-4.33,-2.5) to node[above] {$a_k$} (0,-5);
        \fill (-4.33,2.5) circle (2pt);
        \fill (-4.33,-2.5) circle (2pt);
        \fill (4.33,-2.5) circle (2pt);
        \fill (4.33,2.5) circle (2pt);
        \fill (0,-5) circle (2pt);
        \fill [blue] (0,5) circle (2pt);
    \end{scope}
\end{tikzpicture}\]

After the mutation sequence $\iota_k$, the triangulation becomes this (the subscripts $k$ in $a_k$ to $e_k$ are removed for readibility):

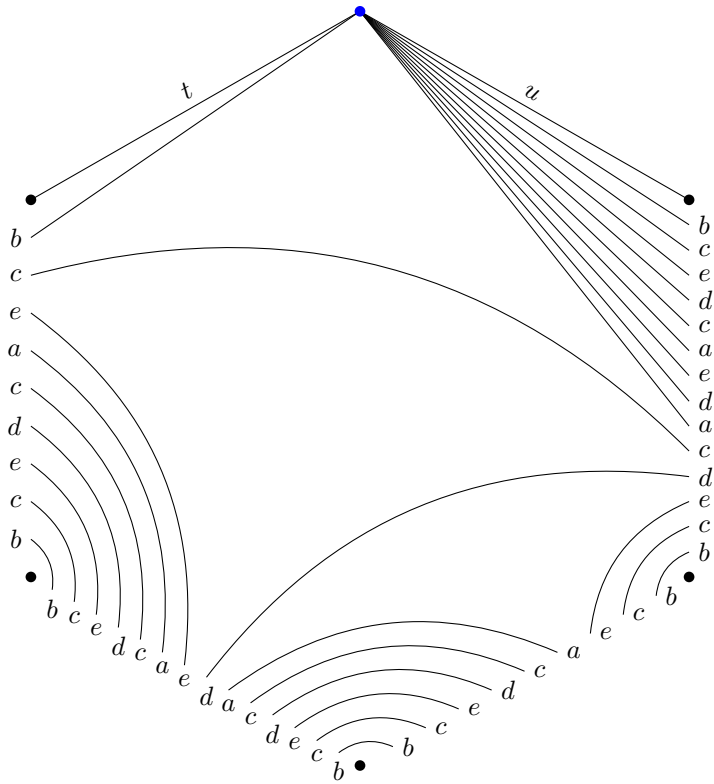
\[\begin{tikzpicture}[every edge quotes/.style={auto=right}]
    \begin{scope}[every node/.style={sloped,allow upside down}][every edge quotes/.style={auto=right}]
        \draw (-4.33,2.5) to node[above] {$t$} (0,5);
        \draw (0,5) to node[above] {$u$} (4.33,2.5);
        \draw (-4.33,2) node [left] {$b$} to (0,5);
        \draw (0,5) to (4.33,2.17) node [right] {$b$};
        \draw (0,5) to (4.33,1.83) node [right] {$c$};
        \draw (0,5) to (4.33,1.5) node [right] {$e$};
        \draw (0,5) to (4.33,1.17) node [right] {$d$};
        \draw (0,5) to (4.33,0.83) node [right] {$c$};
        \draw (0,5) to (4.33,0.5) node [right] {$a$};
        \draw (0,5) to (4.33,0.17) node [right] {$e$};
        \draw (0,5) to (4.33,-0.17) node [right] {$d$};
        \draw (0,5) to (4.33,-0.5) node [right] {$a$};
        \draw (3.897,-2.75) node [below,right] {$b$} to [bend left] (4.33,-2.17) node [right] {$b$};
        \draw (3.464,-3) node [below,right] {$c$} to [bend left] (4.33,-1.83) node [right] {$c$};
        \draw (3.031,-3.25) node [below,right] {$e$} to [bend left] (4.33,-1.5) node [right] {$e$};
        \draw (-0.28,-4.83) node [below] {$b$} to [bend left] (0.433,-4.75) node [below,right] {$b$};
        \draw (-0.57,-4.67) node [below] {$c$} to [bend left] (0.866,-4.5) node [below,right] {$c$};
        \draw (-0.86,-4.5) node [below] {$e$} to [bend left] (1.299,-4.25) node [below,right] {$e$};
        \draw (-1.15,-4.33) node [below] {$d$} to [bend left] (1.732,-4) node [below,right] {$d$};
        \draw (-1.44,-4.17) node [below] {$c$} to [bend left] (2.165,-3.75) node [below,right] {$c$};
        \draw (-1.73,-4) node [below] {$a$} to [bend left] (2.598,-3.5) node [below,right] {$a$};
        \draw (-4.33,-2) node [left] {$b$} to [bend left] (-4.05,-2.67) node [below] {$b$};
        \draw (-4.33,-1.5) node [left] {$c$} to [bend left] (-3.76,-2.83) node [below] {$c$};
        \draw (-4.33,-1) node [left] {$e$} to [bend left] (-3.47,-3) node [below] {$e$};
        \draw (-4.33,-0.5) node [left] {$d$} to [bend left] (-3.18,-3.17) node [below] {$d$};
        \draw (-4.33,0) node [left] {$c$} to [bend left] (-2.89,-3.33) node [below] {$c$};
        \draw (-4.33,0.5) node [left] {$a$} to [bend left] (-2.6,-3.5) node [below] {$a$};
        \draw (-4.33,1) node [left] {$e$} to [bend left] (-2.31,-3.67) node [below] {$e$};
        \draw (-4.33,1.5) node [left] {$c$} to [bend left] (4.33,-0.83) node [right] {$c$};
        \draw (-2.02,-3.83) node [below] {$d$} to [bend left] (4.33,-1.17) node [right] {$d$};
        \fill (-4.33,2.5) circle (2pt);
        \fill (-4.33,-2.5) circle (2pt);
        \fill (4.33,-2.5) circle (2pt);
        \fill (4.33,2.5) circle (2pt);
        \fill (0,-5) circle (2pt);
        \fill (0,5) [blue] circle (2pt);
    \end{scope}
\end{tikzpicture}\]

As we can see, the arcs $a_k$ to $e_k$ are no longer connected to the outside puncture.

By applying the lemma twice, after step 4(d), the subdiagram looks like this:

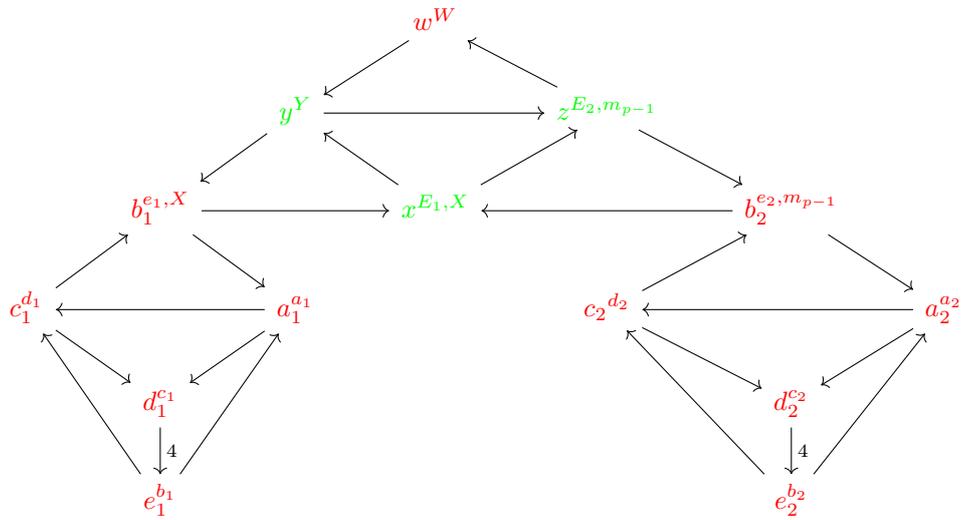
\[\begin{tikzcd}
    &&&\color{red}w^{W}\arrow[dl]\\&&
    \color{green}y^{Y}\arrow[rr]\arrow[dl]&&\color{green}z^{E_2,m_{p-1}}\arrow[dr]\arrow[ul]\\&
    \color{red}b_1^{e_1,X}\arrow[rr]\arrow[dr]&&\color{green}x^{E_1,X}\arrow[ul]\arrow[ur]&&\color{red}b_2^{e_2,m_{p-1}}\arrow[ll]\arrow[dr]\\
    \color{red}c_1^{d_1}\arrow[ur]\arrow[dr]&&\color{red}a_1^{a_1}\arrow[ll]\arrow[dl]&&\color{red}{c_2}^{d_2}\arrow[ur]\arrow[dr]&&\color{red}a_2^{a_2}\arrow[ll]\arrow[dl]\\
    &\color{red}d_1^{c_1}\arrow[d,"4"]&&&&\color{red}d_2^{c_2}\arrow[d,"4"]\\
    &\color{red}e_1^{b_1}\arrow[uur]\arrow[uul]&&&&\color{red}e_2^{b_2}\arrow[uur]\arrow[uul]
\end{tikzcd}\]

Step 5 consists of $(x,y,z,x)$. If we mutate at these vertices, we get:

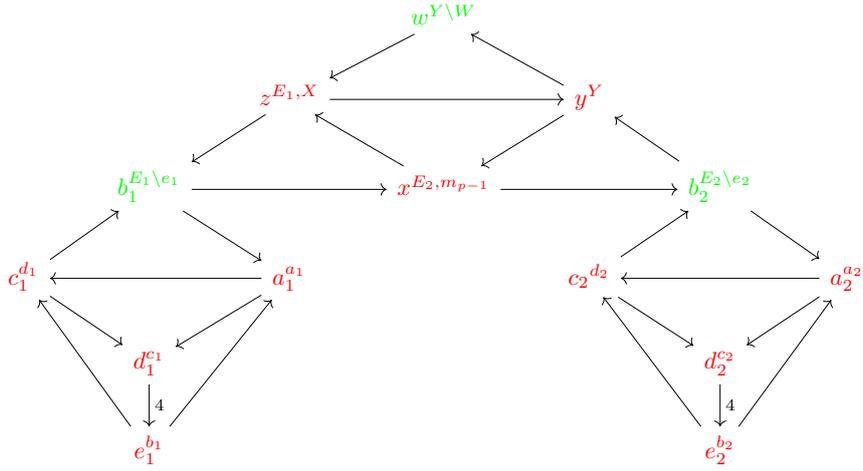
\[\adjustbox{scale=0.9}{\begin{tikzcd}
    &&&\color{green}w^{Y\setminus W}\arrow[dl]\\&&
    \color{red}z^{E_1,X}\arrow[rr]\arrow[dl]&&\color{red}y^{Y}\arrow[ul]\arrow[dl]\\&
    \color{green}b_1^{E_1\setminus e_1}\arrow[rr]\arrow[dr]&&\color{red}x^{E_2,m_{p-1}}\arrow[ul]\arrow[rr]&&\color{green}b_2^{E_2\setminus e_2}\arrow[ul]\arrow[dr]\\
    \color{red}c_1^{d_1}\arrow[ur]\arrow[dr]&&\color{red}a_1^{a_1}\arrow[ll]\arrow[dl]&&\color{red}{c_2}^{d_2}\arrow[ur]\arrow[dr]&&\color{red}a_2^{a_2}\arrow[ll]\arrow[dl]\\
    &\color{red}d_1^{c_1}\arrow[d,"4"]&&&&\color{red}d_2^{c_2}\arrow[d,"4"]\\
    &\color{red}e_1^{b_1}\arrow[uur]\arrow[uul]&&&&\color{red}e_2^{b_2}\arrow[uur]\arrow[uul]
\end{tikzcd}}\]

In terms of triangulation, after flipping the arc $x$, we have: 
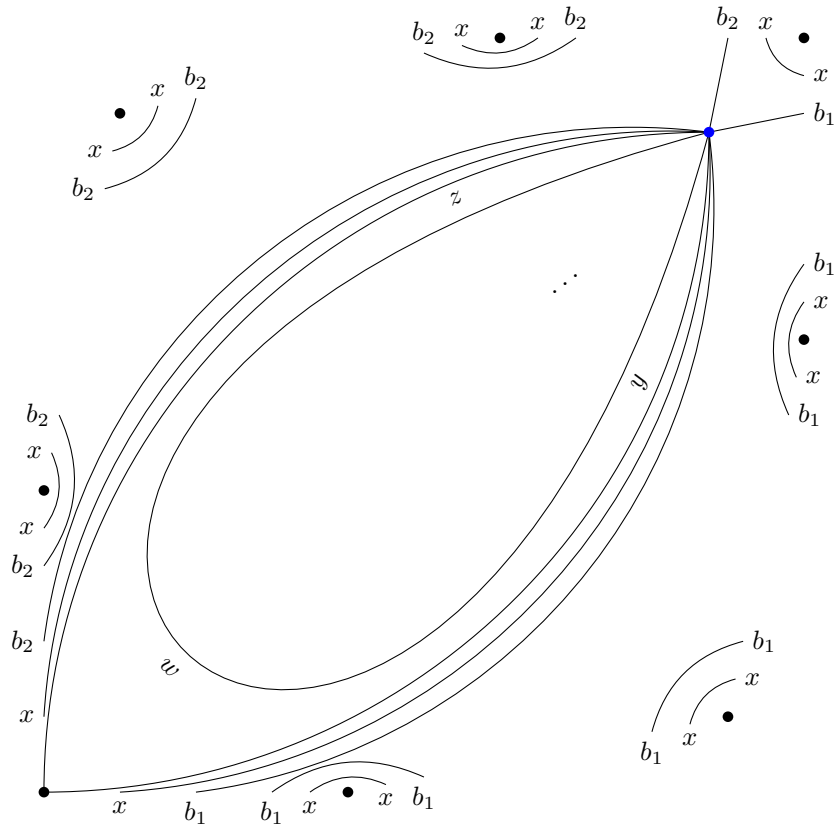
\label{flip_x}

\[\hspace*{-0.25\linewidth}\begin{tikzpicture}[every edge quotes/.style={auto=right}]
    \begin{scope}[every node/.style={sloped,allow upside down}][every edge quotes/.style={auto=right}]
        \draw (0,0) to [bend right=45,black] node[near end,above] {$y$} (8.75,8.75);
        \draw (0,0) to [bend left=45,black] node[near end,below] {$z$} (8.75,8.75);
        \draw (8.75,8.75) to [loop,in=255,out=195,min distance=150mm] node[below=0.1] {$w$} (8.75,8.75);
        \draw (3,0) node[below] {$b_1$} to [bend left] (5,0.2) node[below] {$b_1$};
        \draw (8,0.8) node[below] {$b_1$} to [bend left] (9.2,2) node [right] {$b_1$};
        \draw (9.8,5) node[right] {$b_1$} to [bend left] (10,7) node [right] {$b_1$};
        \draw (2,0) node[below] {$b_1$} to [bend right=45] (8.75,8.75);
        \draw (8.75,8.75) to (10,9) node [right] {$b_1$};
        \draw (0,3) node[left] {$b_2$} to [bend right] (0.2,5) node[left] {$b_2$};
        \draw (0.8,8) node[left] {$b_2$} to [bend right] (2,9.2) node [above] {$b_2$};
        \draw (5,9.8) node[above] {$b_2$} to [bend right] (7,10) node [above] {$b_2$};
        \draw (0,2) node[left] {$b_2$} to [bend left=45] (8.75,8.75);
        \draw (8.75,8.75) to (9,10) node [above] {$b_2$};
        \node at (6.85,6.85) {$\iddots$};
        \draw (1,0) node[below] {$x$} to [bend right=45] (8.75,8.75);
        \draw (0,1) node[left] {$x$} to [bend left=45] (8.75,8.75);
        \draw (3.5,0) node[below] {$x$} to [bend left] (4.5,0.1) node[below] {$x$};
        \draw (8.5,0.9) node[below] {$x$} to [bend left] (9.1,1.5) node [right] {$x$};
        \draw (9.9,5.5) node[right] {$x$} to [bend left] (10,6.5) node [right] {$x$};
        \draw (0,3.5) node[left] {$x$} to [bend right] (0.1,4.5) node[left] {$x$};
        \draw (0.9,8.5) node[left] {$x$} to [bend right] (1.5,9.1) node [above] {$x$};
        \draw (5.5,9.9) node[above] {$x$} to [bend right] (6.5,10) node [above] {$x$};
        \draw (9.5,10) node[above] {$x$} to [bend right] (10,9.5) node [right] {$x$};
        \fill (0,0) circle (2pt);
        \fill (10,10) circle (2pt);
        \fill [blue] (8.75,8.75) circle (2pt);
        \fill (4,0) circle (2pt);
        \fill (10,6) circle (2pt);
        \fill (9,1) circle (2pt);
        \fill (0,4) circle (2pt);
        \fill (6,10) circle (2pt);
        \fill (1,9) circle (2pt);
    \end{scope}
\end{tikzpicture}\]

The arc $x$ goes around the two pentagons, and we see that the arcs $x,w,y,z$ form a triangulated once-punctured digon, so flipping $y,z$ tags the outside puncture notched, and flipping $x$ again moves it back to the original position.

Step 6(b) consists of $\pi_1,\pi_2$. Just like in step 4(d), we shall show the result of the mutation sequence $\pi_k$ in a lemma:

\begin{lemma}\label{lemma_pi}
    Mutating the following diagram, where $T_0=\phi\text{ or }T$ (the blue colour of the vertex $u$ means that the greenness of $u$ does not matter):
    \[\adjustbox{scale=0.9}
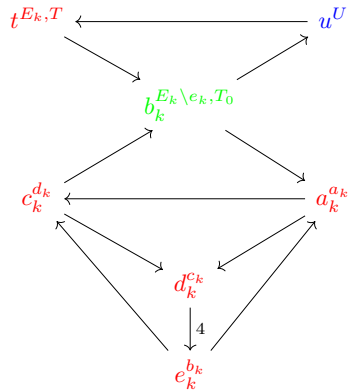
{\begin{tikzcd}
    \color{red}t^{E_k,T}\arrow[dr]&&\color{blue}u^{U}\arrow[ll]
    &&\\&
    \color{green}b_k^{E_k\setminus e_k,T_0}\arrow[ur]\arrow[dr]&&&&
    \\
    \color{red}c_k^{d_k}\arrow[ur]\arrow[dr]&&\color{red}a_k^{a_k}\arrow[ll]\arrow[dl]&&\\
    &\color{red}d_k^{c_k}\arrow[d,"4"]&&&&\\
    &\color{red}e_k^{b_k}\arrow[uur]\arrow[uul]&&&&
\end{tikzcd}}\]
with the mutation sequence $\pi_k$ gives the following diagram if $T_0=\phi$:

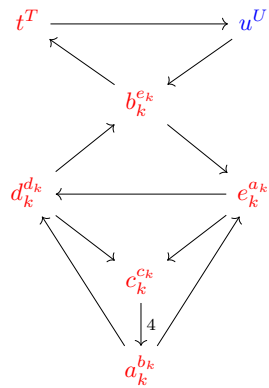
\[\adjustbox{scale=0.9}{\begin{tikzcd}
    \color{red}t^{T}\arrow[rr]&&\color{blue}u^{U}\arrow[dl]
    &&\\&
    \color{red}b_k^{e_k}\arrow[dr]\arrow[ul]&&&&
    \\
    \color{red}d_k^{d_k}\arrow[dr]\arrow[ur]&&\color{red}e_k^{a_k}\arrow[ll]\arrow[dl]&&\\
    &\color{red}c_k^{c_k}\arrow[d,"4"]&&&&\\
    &\color{red}a_k^{b_k}\arrow[uur]\arrow[uul]&&&&
\end{tikzcd}}\]

and the following diagram if $T_0=T$:

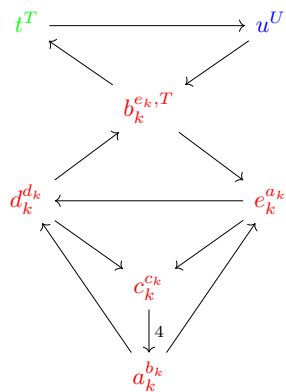
\[\adjustbox{scale=0.9}{\begin{tikzcd}
    \color{green}t^{T}\arrow[rr]&&\color{blue}u^{U}\arrow[dl]
    &&\\&
    \color{red}b_k^{e_k,T}\arrow[dr]\arrow[ul]&&&&
    \\
    \color{red}d_k^{d_k}\arrow[dr]\arrow[ur]&&\color{red}e_k^{a_k}\arrow[ll]\arrow[dl]&&\\
    &\color{red}c_k^{c_k}\arrow[d,"4"]&&&&\\
    &\color{red}a_k^{b_k}\arrow[uur]\arrow[uul]&&&&
\end{tikzcd}}\]
\end{lemma}

\begin{proof}
    Direct computation.
\end{proof}

In terms of triangulation, after the mutation sequence $\pi_k$, the triangulation becomes this:

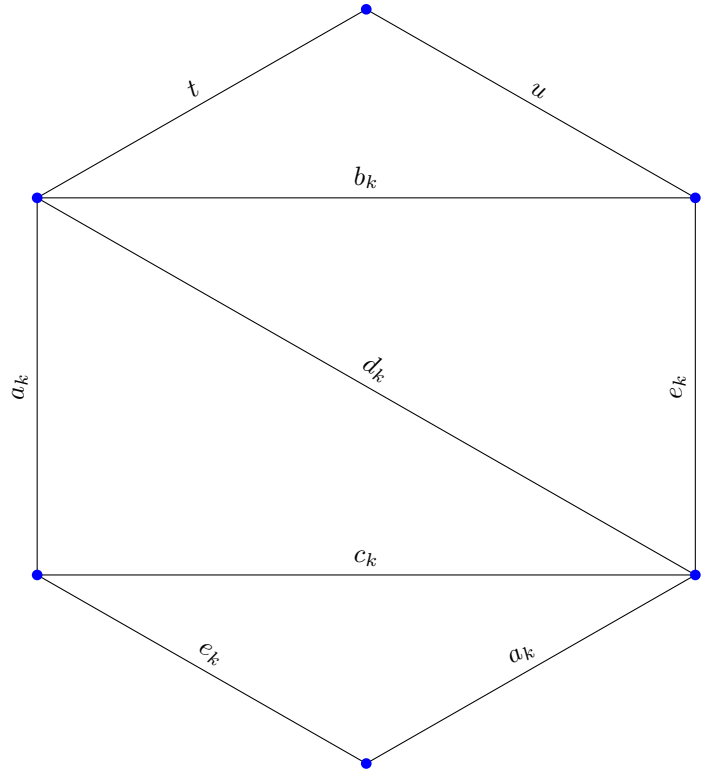
\[\begin{tikzpicture}[every edge quotes/.style={auto=right}]
    \begin{scope}[every node/.style={sloped,allow upside down}][every edge quotes/.style={auto=right}]
        \draw (-4.33,2.5) to node[above] {$t$} (0,5);
        \draw (0,5) to node[above] {$u$} (4.33,2.5);
        \draw (-4.33,2.5) to node[above] {$b_k$} (4.33,2.5);
        \draw (4.33,-2.5) to node[above] {$e_k$} (4.33,2.5);
        \draw (-4.33,-2.5) to node[above] {$a_k$} (-4.33,2.5);
        \draw (-4.33,2.5) to node[above] {$d_k$} (4.33,-2.5);
        \draw (-4.33,-2.5) to node[above] {$c_k$} (4.33,-2.5);
        \draw (0,-5) to node[above] {$a_k$} (4.33,-2.5);
        \draw (-4.33,-2.5) to node[above] {$e_k$} (0,-5);
        \fill [blue] (-4.33,2.5) circle (2pt);
        \fill [blue] (-4.33,-2.5) circle (2pt);
        \fill [blue] (4.33,-2.5) circle (2pt);
        \fill [blue] (4.33,2.5) circle (2pt);
        \fill [blue] (0,-5) circle (2pt);
        \fill [blue] (0,5) circle (2pt);
    \end{scope}
\end{tikzpicture}\]

As we can see, the arcs $a_k$ to $e_k$ are now connected to the outside puncture.

By applying the lemma twice, after step 6(b), the subdiagram looks like this:

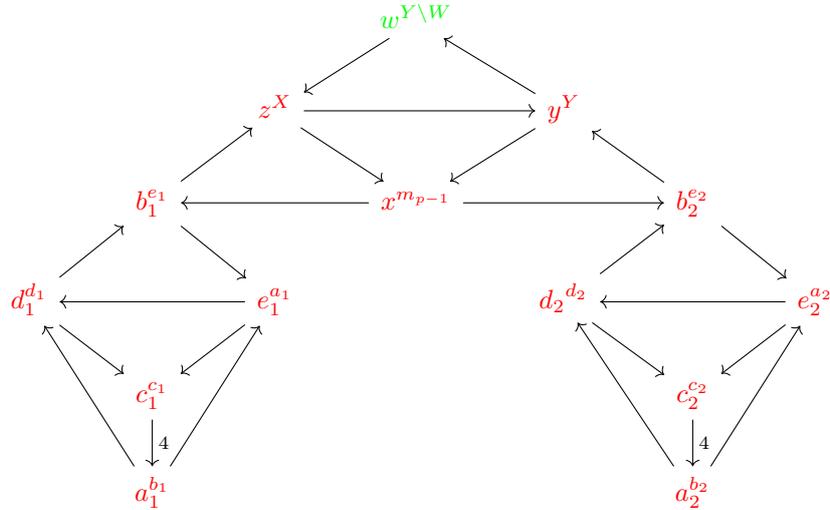
\[\begin{tikzcd}
    &&&\color{green}w^{Y\setminus W}\arrow[dl]\\&&
    \color{red}z^{X}\arrow[rr]\arrow[dr]&&\color{red}y^{Y}\arrow[ul]\arrow[dl]\\&
    \color{red}b_1^{e_1}\arrow[ur]\arrow[dr]&&\color{red}x^{m_{p-1}}\arrow[ll]\arrow[rr]&&\color{red}b_2^{e_2}\arrow[ul]\arrow[dr]\\
    \color{red}d_1^{d_1}\arrow[ur]\arrow[dr]&&\color{red}e_1^{a_1}\arrow[ll]\arrow[dl]&&\color{red}{d_2}^{d_2}\arrow[ur]\arrow[dr]&&\color{red}e_2^{a_2}\arrow[ll]\arrow[dl]\\
    &\color{red}c_1^{c_1}\arrow[d,"4"]&&&&\color{red}c_2^{c_2}\arrow[d,"4"]\\
    &\color{red}a_1^{b_1}\arrow[uur]\arrow[uul]&&&&\color{red}a_2^{b_2}\arrow[uur]\arrow[uul]
\end{tikzcd}\]

Since we obtained the desired diagram in Lemma \ref{one_puncture_summary}, the mutation sequence $\Delta_2$ is a maximal green sequence. $\blacksquare$

\subsubsection{The case where $n>2$}\label{n>2}

When $n>2$, the step 4(c) to 6(c) translate to

\begin{enumerate}[start=4]
    \item Moving arcs away from the outside puncture.
    \begin{enumerate}[start=3]
        \item Separating the $e_i$'s. $(f_1,f_2...,f_{n-2})$
        \item Moving arcs in each pentagon away from the outside puncture. $(\iota_1,\iota_2,...,\iota_n)$
        \item Moving $f_i$'s away from the outside puncture. $(f_{n-2},f_{n-1},...,f_1)$
    \end{enumerate}
    \item Tagging the outside puncture notched. $(\gamma)$
    \item Moving arcs back to the outside puncture.
    \begin{enumerate}
        \item Moving $f_i$'s back to the outside puncture. $(f_1,f_2...,f_{n-2})$
        \item Moving arcs in each pentagon back to the outside puncture. $(\pi_1,\pi_2,...,\pi_n)$
        \item Moving $f_i$'s back to their original positions. $(f_{n-2},f_{n-1},...,f_1)$
    \end{enumerate}
\end{enumerate}

where

\begin{itemize}
    \item $\iota_k=a_k,e_k,c_k,b_k,a_k,d_k,e_k,c_k,b_k$
    \item $\pi_k=b_k,c_k,e_k,d_k,b_k,a_k,c_k,e_k,b_k$
    \item $\gamma=\begin{cases}
        m_{p-2},l_{p-2},r_{p-2},m_{p-2}&\text{if }p>2,\\
        h_1,m_1,h_{q+1},h_1 & \text{if }p=2
    \end{cases}$
\end{itemize}

According to Lemma \ref{one_puncture_summary}, after step 1 to step 4(b), the full subdiagram containing vertices involved in step 4(c) to step 6(b) and adjacent vertices look like this:

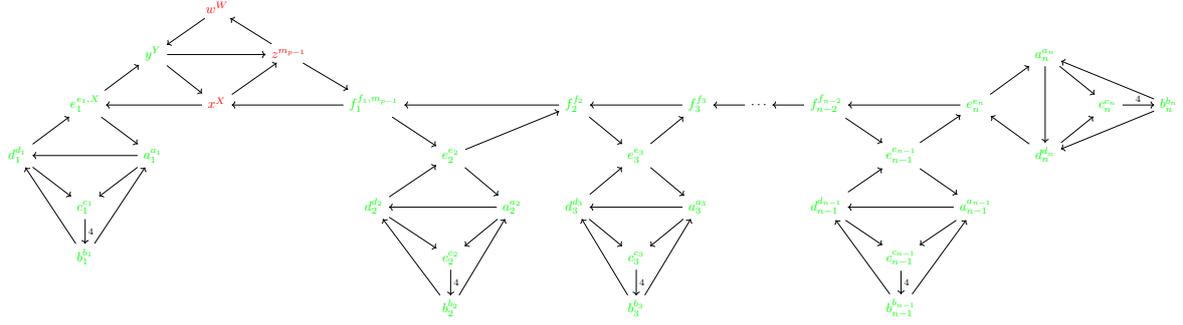
\[\adjustbox{scale=0.5}{\begin{tikzcd}
    &&&\color{red}w^{W}\arrow[dl]&&\\
    &&\color{green}y^{Y}\arrow[rr]\arrow[dr]&&\color{red}z^{m_{p-1}}\arrow[ul]\arrow[dr]&&&&&&&&&&&\color{green}a_n^{a_n}\arrow[dd]\arrow[dr]\\
    &\color{green}e_1^{e_1,X}\arrow[ur]\arrow[dr]&&\color{red}x^{X}\arrow[ll]\arrow[ur]&&\color{green}f_1^{f_1,m_{p-1}}\arrow[ll]\arrow[dr]&&&\color{green}f_2^{f_2}\arrow[dr]\arrow[lll]&&\color{green}f_3^{f_3}\arrow[ll]&\cdots\arrow[l]&\color{green}f_{n-2}^{f_{n-2}}\arrow[l]\arrow[dr]&&\color{green}e_n^{e_n}\arrow[ll]\arrow[ur]&&\color{green}c_n^{c_n}\arrow[r,"4"]&\color{green}b_n^{b_n}\arrow[ull]\arrow[dll]\\
    \color{green}d_1^{d_1}\arrow[ur]\arrow[dr]&&\color{green}a_1^{a_1}\arrow[ll]\arrow[dl]&&&&\color{green}e_2^{e_2}\arrow[dr]\arrow[urr]&&&\color{green}e_3^{e_3}\arrow[dr]\arrow[ur]&&&&\color{green}e_{n-1}^{e_{n-1}}\arrow[ur]\arrow[dr]&&\color{green}d_n^{d_n}\arrow[ul]\arrow[ur]\\
    &\color{green}c_1^{c_1}\arrow[d,"4"]&&&&\color{green}d_2^{d_2}\arrow[ur]\arrow[dr]&&\color{green}a_2^{a_2}\arrow[ll]\arrow[dl]&\color{green}d_3^{d_3}\arrow[ur]\arrow[dr]&&\color{green}a_3^{a_3}\arrow[ll]\arrow[dl]&&\color{green}d_{n-1}^{d_{n-1}}\arrow[ur]\arrow[dr]&&\color{green}a_{n-1}^{a_{n-1}}\arrow[ll]\arrow[dl]\\
    &\color{green}b_1^{b_1}\arrow[uur]\arrow[uul]&&&&&\color{green}c_2^{c_2}\arrow[d,"4"]&&&\color{green}c_3^{c_3}\arrow[d,"4"]&&&&\color{green}c_{n-1}^{c_{n-1}}\arrow[d,"4"]\\
    &&&&&&\color{green}b_2^{b_2}\arrow[uul]\arrow[uur]&&&\color{green}b_3^{b_3}\arrow[uul]\arrow[uur]&&&&\color{green}b_{n-1}^{b_{n-1}}\arrow[uul]\arrow[uur]
\end{tikzcd}}\]
where $w=\begin{cases}
    h_2&\text{if }p=2,q>1\\
    m_2&\text{if }p=3\\
    m_1&\text{if }p=4\\
    r_{p-4}&\text{otherwise}
\end{cases}, x=\begin{cases}
    h_1&\text{if }p=2,q>1\\
    m_{p-2}&\text{otherwise}
\end{cases},y=\begin{cases}
    m_1&\text{if }p=2,q>1\\
    l_{p-2}&\text{otherwise}
\end{cases},
z=\begin{cases}
    h_{q+1}&\text{if }p=2,q>1\\
    r_{p-2}&\text{otherwise}
\end{cases}$

$W,Y$ are sets of (frozen) vertices with $W\subset Y$, and $X=\begin{cases}
    h_1&\text{if }p=2,q>1\\
    l_{p-2}&\text{otherwise}
\end{cases}$.

The corresponding subtriangulation looks like this:

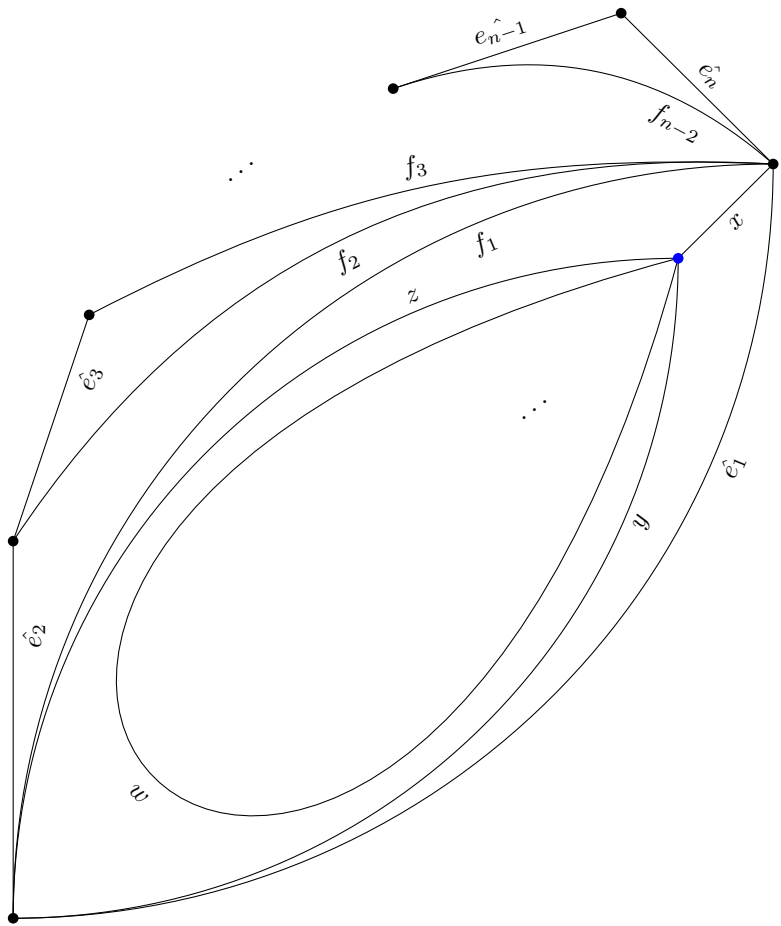
\[\hspace*{-0.25\linewidth}\begin{tikzpicture}[every edge quotes/.style={auto=right}]
    \begin{scope}[every node/.style={sloped,allow upside down}][every edge quotes/.style={auto=right}]

        \node at (6.85,6.85) {$\iddots$};
        
        \draw (8.75,8.75) to [loop,in=255,out=195,min distance=150mm] node[below=0.1] {$w$} (8.75,8.75);
        \draw (0,0) to [bend right=45,black] node[near end,below] {$y$} (8.75,8.75);
        \draw (0,0) to [bend left=45,black] node[near end,above] {$z$} (8.75,8.75);
        \draw (0,0) to [bend right=45,black] node[near end,below,black] {$\hat{e_1}$} (10,10);
        \draw (0,0) to node[near end,below,black] {$\hat{e_2}$} (0,5);
        \draw (0,5) to node[near end,below,black] {$\hat{e_3}$} (1,8);
        \draw (8,12) to node[above,black] {$\hat{e_n}$} (10,10);
        \draw (5,11) to node[above,black] {$\hat{e_{n-1}}$} (8,12);
        \draw (0,0) to [bend left=45,black] node[near end,below] {$f_1$} (10,10);
        \draw (0,5) to [bend left=30,black] node[below] {$f_2$} (10,10);
        \draw (1,8) to [bend left=15] node [above] {$f_3$} (10,10);
        \draw (5,11) to [bend left=30] node[near end,below] {$f_{n-2}$} (10,10);
        \draw (8.75,8.75) to node[below] {$x$} (10,10);
        \node at (3,10) {\reflectbox{$\ddots$}};
        \fill (0,0) circle (2pt);
        \fill (10,10) circle (2pt);
        \fill (0,5) circle (2pt);
        \fill (1,8) circle (2pt);
        \fill (8,12) circle (2pt);
        \fill [blue] (8.75,8.75) circle (2pt);
        \fill (5,11) circle (2pt);
    \end{scope}
\end{tikzpicture}\]

Step 4(c) consists of $(f_1,f_2,...,f_{n-2})$. Look at the vertices involved and adjacent vertices:

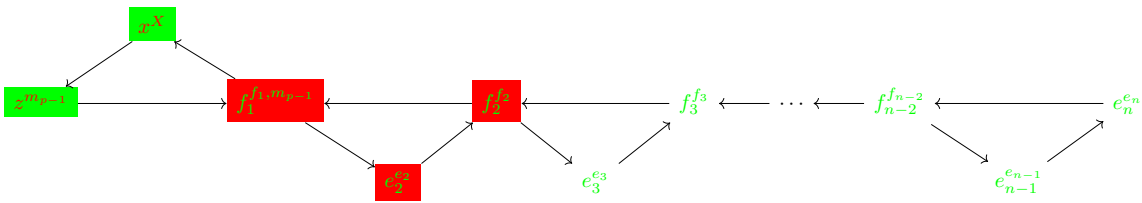
\[\adjustbox{scale=0.8}{\begin{tikzcd} 
    &|[color=red,fill=green]|x^{X}\arrow[dl]\\|[color=red,fill=green]|z^{m_{p-1}}\arrow[rr]&&|[color=green,fill=red]|f_1^{f_1,m_{p-1}}\arrow[ul]\arrow[dr]&&|[color=green,fill=red]|f_2^{f_2}\arrow[dr]\arrow[ll]&&\color{green}f_3^{f_3}\arrow[ll]&\cdots\arrow[l]&\color{green}f_{n-2}^{f_{n-2}}\arrow[l]\arrow[dr]&&\color{green}e_n^{e_n}\arrow[ll]\\&&&|[color=green,fill=red]|e_2^{e_2}\arrow[ur]&&\color{green}e_3^{e_3}\arrow[ur]&&&&\color{green}e_{n-1}^{e_{n-1}}\arrow[ur]
\end{tikzcd}}\]

Mutate at $f_1$:

\[\adjustbox{scale=0.8}
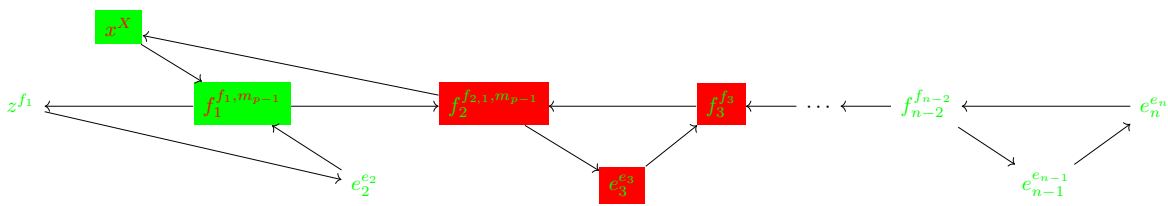
{\begin{tikzcd} 
    &|[color=red,fill=green]|x^{X}\arrow[dr]\\\color{green}z^{f_1}\arrow[drrr]&&|[color=red,fill=green]|f_1^{f_1,m_{p-1}}\arrow[ll]\arrow[rr]&&|[color=green,fill=red]|f_2^{f_{2,1},m_{p-1}}\arrow[dr]\arrow[ulll]&&|[color=green,fill=red]|f_3^{f_3}\arrow[ll]&\cdots\arrow[l]&\color{green}f_{n-2}^{f_{n-2}}\arrow[l]\arrow[dr]&&\color{green}e_n^{e_n}\arrow[ll]\\&&&\color{green}e_2^{e_2}\arrow[ul]&&|[color=green,fill=red]|e_3^{e_3}\arrow[ur]&&&&\color{green}e_{n-1}^{e_{n-1}}\arrow[ur]
\end{tikzcd}}\]

By observing the pattern at the highlighted vertices, after mutating at $f_{n-2}$, we have:

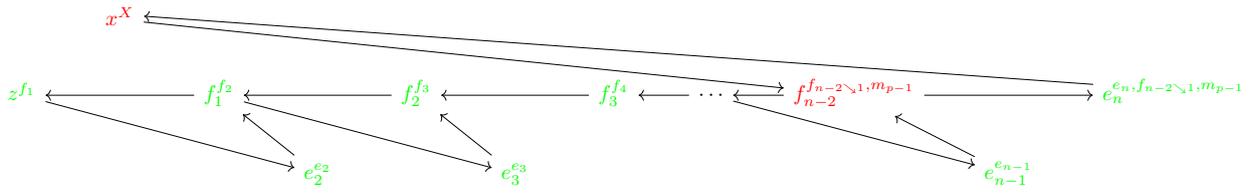
\[\adjustbox{scale=0.8}{\begin{tikzcd} 
    &\color{red}x^{X}\arrow[drrrrrrr]\\\color{green}z^{f_1}\arrow[drrr]&&\color{green}f_1^{f_2}\arrow[ll]\arrow[drrr]&&\color{green}f_2^{f_3}\arrow[ll]&&\color{green}f_3^{f_4}\arrow[ll]&\cdots\arrow[l]\arrow[drr]&\color{red}f_{n-2}^{f_{n-2\searrow1},m_{p-1}}\arrow[l]\arrow[rr]&&\color{green}e_n^{e_n,f_{n-2\searrow1},m_{p-1}}\arrow[ulllllllll,shift right]\\&&&\color{green}e_2^{e_2}\arrow[ul]&&\color{green}e_3^{e_3}\arrow[ul]&&&&\color{green}e_{n-1}^{e_{n-1}}\arrow[ul]
\end{tikzcd}}\]

In terms of triangulation, the subtriangulation becomes:

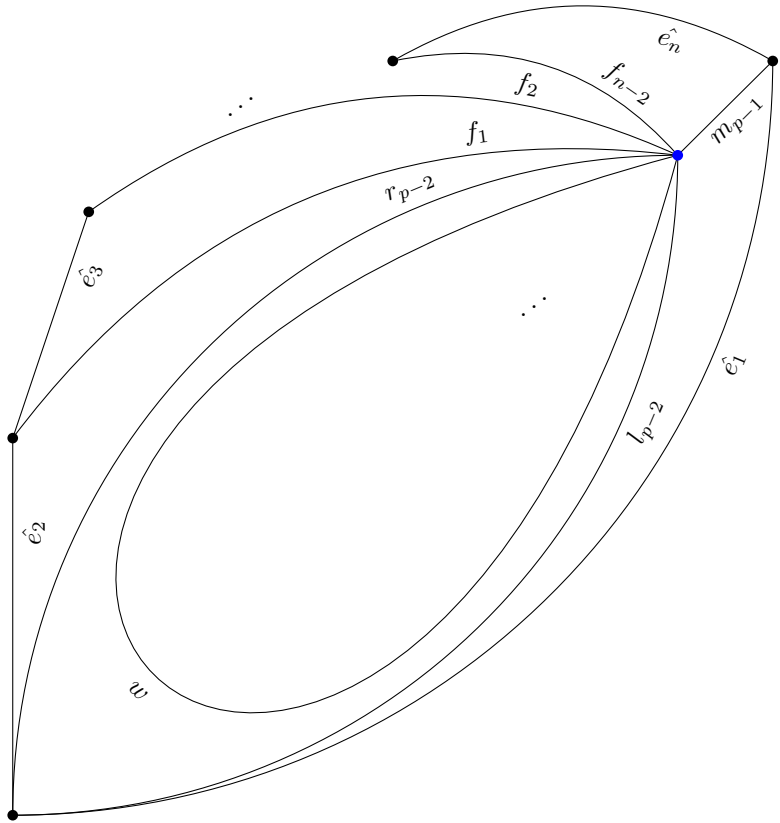
\[\begin{tikzpicture}[every edge quotes/.style={auto=right}]
    \hspace*{-0.15\linewidth}\begin{scope}[every node/.style={sloped,allow upside down}][every edge quotes/.style={auto=right}]
\node at (6.85,6.85) {$\iddots$};
        \draw (8.75,8.75) to [loop,in=255,out=195,min distance=150mm] node[below=0.1] {$w$} (8.75,8.75);
        \draw (0,0) to [bend right=45,black] node[near end,below] {$l_{p-2}$} (8.75,8.75);
        \draw (0,0) to [bend left=45,black] node[near end,above] {$r_{p-2}$} (8.75,8.75);
        \draw (0,0) to [bend right=45,black] node[near end,below,black] {$\hat{e_1}$} (10,10);
        \draw (0,0) to node[near end,below,black] {$\hat{e_2}$} (0,5);
        \draw (0,5) to node[near end,below,black] {$\hat{e_3}$} (1,8);
        \draw (5,10) to [bend left] node[near end,below,black] {$\hat{e_n}$} (10,10);
        \draw (0,5) to [bend left=30,black] node[near end,above] {$f_1$} (8.75,8.75);
        \draw (1,8) to [bend left=30,black] node[near end,above] {$f_2$} (8.75,8.75);
        \draw (5,10) to [bend left=30] node[near end,above] {$f_{n-2}$} (8.75,8.75);
        \draw (8.75,8.75) to node[below] {$m_{p-1}$} (10,10);
        \node at (3,9.5) {\reflectbox{$\ddots$}};
        \fill (0,0) circle (2pt);
        \fill (10,10) circle (2pt);
        \fill (0,5) circle (2pt);
        \fill (1,8) circle (2pt);
        \fill (5,10) circle (2pt);
        \fill [blue] (8.75,8.75) circle (2pt);
    \end{scope}
\end{tikzpicture}\]

The pentagons (represented by the "arcs" $\hat{e_k}$) are now separated by $f_1,f_2,...,f_{n-2}$.

Step 4(d) consists of $(\iota_1,\iota_2,...,\iota_n)$. Look at the entire subdiagram (with the locations of vertices rearranged):

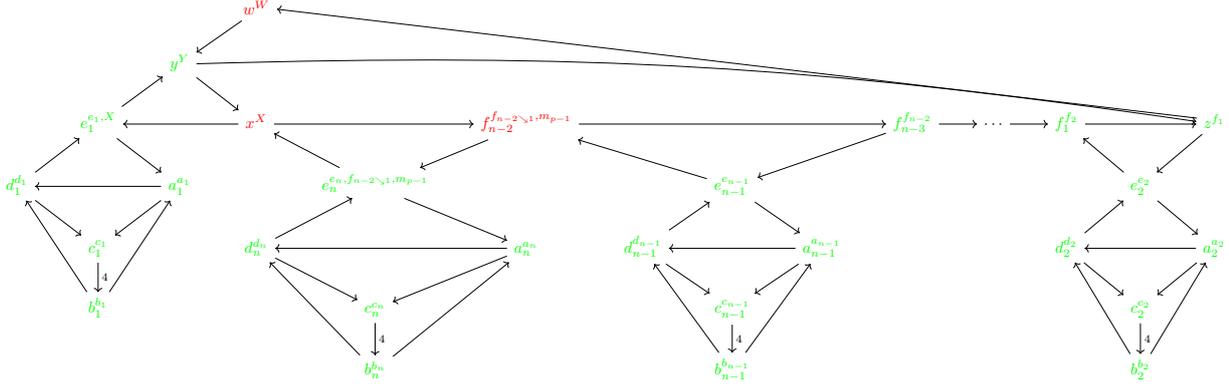
\[\adjustbox{scale=0.6}{\begin{tikzcd}
    &&&\color{red}w^{W}\arrow[dl]&&\\
    &&\color{green}y^{Y}\arrow[dr]\arrow[drrrrrrrrrrr,bend left=5]&&&&&&&&&&&&&&\\
    &\color{green}e_1^{e_1,X}\arrow[ur]\arrow[dr]&&\color{red}x^{X}\arrow[ll]\arrow[rr]&&\color{red}f_{n-2}^{f_{n-2\searrow1},m_{p-1}}\arrow[dl]\arrow[rrrr]&&&&\color{green}f_{n-3}^{f_{n-2}}\arrow[dll]\arrow[r]&\cdots\arrow[r]&\color{green}f_1^{f_2}\arrow[rr]&&\color{green}z^{f_1}\arrow[dl]\arrow[uullllllllll,shift right,near end]&&&&\\
    \color{green}d_1^{d_1}\arrow[ur]\arrow[dr]&&\color{green}a_1^{a_1}\arrow[ll]\arrow[dl]&&\color{green}e_n^{e_n,f_{n-2\searrow1},m_{p-1}}\arrow[dr]\arrow[ul]&&&\color{green}e_{n-1}^{e_{n-1}}\arrow[dr]\arrow[ull]&&&&&\color{green}e_2^{e_2}\arrow[ul]\arrow[dr]&&&&\\
    &\color{green}c_1^{c_1}\arrow[d,"4"]&&\color{green}d_n^{d_n}\arrow[dr]\arrow[ur]&&\color{green}a_n^{a_n}\arrow[dl]\arrow[ll]&\color{green}d_{n-1}^{d_{n-1}}\arrow[ur]\arrow[dr]&&\color{green}a_{n-1}^{a_{n-1}}\arrow[ll]\arrow[dl]&&&\color{green}d_2^{d_2}\arrow[ur]\arrow[dr]&&\color{green}a_2^{a_2}\arrow[ll]\arrow[dl]&&\\
    &\color{green}b_1^{b_1}\arrow[uur]\arrow[uul]&&&\color{green}c_n^{c_n}\arrow[d,"4"]&&&\color{green}c_{n-1}^{c_{n-1}}\arrow[d,"4"]&&&&&\color{green}c_2^{c_2}\arrow[d,"4"]&&\\
    &&&&\color{green}b_n^{b_n}\arrow[uul]\arrow[uur]&&&\color{green}b_{n-1}^{b_{n-1}}\arrow[uul]\arrow[uur]&&&&&\color{green}b_2^{b_2}\arrow[uul]\arrow[uur]&&
\end{tikzcd}}\]

By applying Lemma \ref{lemma_iota} $n$ times, we get:

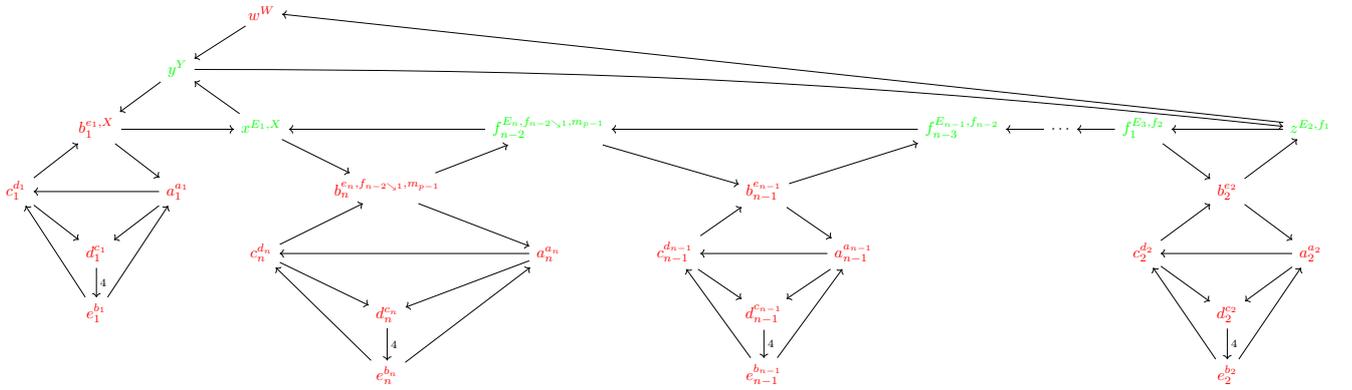
\[\adjustbox{scale=0.6}{\begin{tikzcd}
    &&&\color{red}w^{W}\arrow[dl]&&\\
    &&\color{green}y^{Y}\arrow[dl]\arrow[drrrrrrrrrrr,bend left=3]&&&&&&&&&&&&&&\\
    &\color{red}b_1^{e_1,X}\arrow[rr]\arrow[dr]&&\color{green}x^{E_1,X}\arrow[ul]\arrow[dr]&&\color{green}f_{n-2}^{E_n,f_{n-2\searrow1},m_{p-1}}\arrow[ll]\arrow[drr]&&&&\color{green}f_{n-3}^{E_{n-1},f_{n-2}}\arrow[llll]&\cdots\arrow[l]&\color{green}f_1^{E_3,f_2}\arrow[l]\arrow[dr]&&\color{green}z^{E_2,f_1}\arrow[ll]\arrow[uullllllllll,shift right,near end]&&&&\\
    \color{red}c_1^{d_1}\arrow[ur]\arrow[dr]&&\color{red}a_1^{a_1}\arrow[ll]\arrow[dl]&&\color{red}b_n^{e_n,f_{n-2\searrow1},m_{p-1}}\arrow[dr]\arrow[ur]&&&\color{red}b_{n-1}^{e_{n-1}}\arrow[dr]\arrow[urr]&&&&&\color{red}b_2^{e_2}\arrow[ur]\arrow[dr]&&&&\\
    &\color{red}d_1^{c_1}\arrow[d,"4"]&&\color{red}c_n^{d_n}\arrow[dr]\arrow[ur]&&\color{red}a_n^{a_n}\arrow[dl]\arrow[ll]&\color{red}c_{n-1}^{d_{n-1}}\arrow[ur]\arrow[dr]&&\color{red}a_{n-1}^{a_{n-1}}\arrow[ll]\arrow[dl]&&&\color{red}c_2^{d_2}\arrow[ur]\arrow[dr]&&\color{red}a_2^{a_2}\arrow[ll]\arrow[dl]&&\\
    &\color{red}e_1^{b_1}\arrow[uur]\arrow[uul]&&&\color{red}d_n^{c_n}\arrow[d,"4"]&&&\color{red}d_{n-1}^{c_{n-1}}\arrow[d,"4"]&&&&&\color{red}d_2^{c_2}\arrow[d,"4"]&&\\
    &&&&\color{red}e_n^{b_n}\arrow[uul]\arrow[uur]&&&\color{red}e_{n-1}^{b_{n-1}}\arrow[uul]\arrow[uur]&&&&&\color{red}e_2^{b_2}\arrow[uul]\arrow[uur]&&
\end{tikzcd}}\]

In terms of triangulation, just like the case where $n=2$, the arcs $a_k$ to $e_k$ are no longer connected to the outside puncture.

Step 4(d) consists of $(f_{n-2},f_{n-3},...,f_1)$. Look at the vertices involved and vertices adjacent to them:

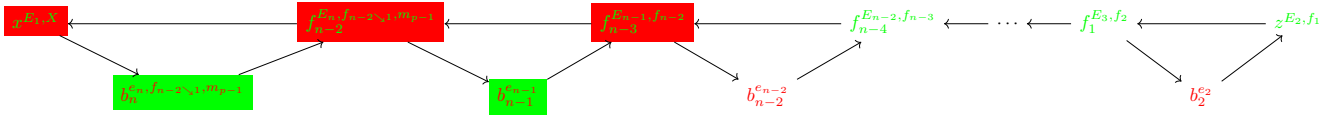
\[\adjustbox{scale=0.7}{\begin{tikzcd}
    |[color=green,fill=red]|x^{E_1,X}\arrow[dr]&&|[color=green,fill=red]|f_{n-2}^{E_n,f_{n-2\searrow1},m_{p-1}}\arrow[ll]\arrow[dr]&&|[color=green,fill=red]|f_{n-3}^{E_{n-1},f_{n-2}}\arrow[ll]\arrow[dr]&&\color{green}f_{n-4}^{E_{n-2},f_{n-3}}\arrow[ll]&\cdots\arrow[l]&\color{green}f_1^{E_3,f_2}\arrow[l]\arrow[dr]&&\color{green}z^{E_2,f_1}\arrow[ll]&&&&\\
    &|[color=red,fill=green]|b_n^{e_n,f_{n-2\searrow1},m_{p-1}}\arrow[ur]&&|[color=red,fill=green]|b_{n-1}^{e_{n-1}}\arrow[ur]&&\color{red}b_{n-2}^{e_{n-2}}\arrow[ur]&&&&\color{red}b_2^{e_2}\arrow[ur]&&&&
\end{tikzcd}}\]

Mutate at $f_{n-2}$:

\[\adjustbox{scale=0.7}
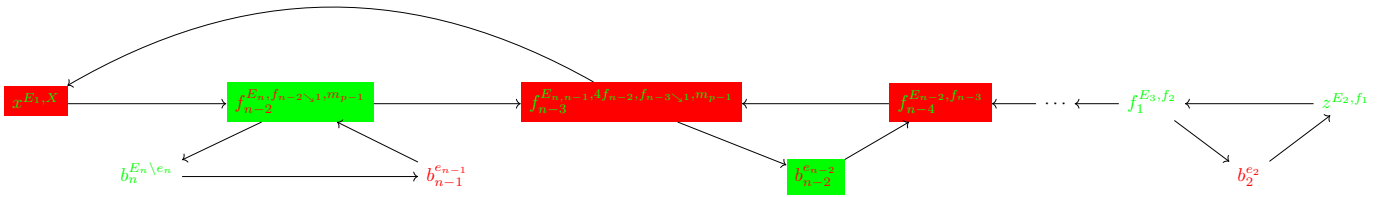
{\begin{tikzcd}
    |[color=green,fill=red]|x^{E_1,X}\arrow[rr]&&|[color=red,fill=green]|f_{n-2}^{E_n,f_{n-2\searrow1},m_{p-1}}\arrow[dl]\arrow[rr]&&|[color=green,fill=red]|f_{n-3}^{E_{n,n-1},4f_{n-2},f_{n-3\searrow1},m_{p-1}}\arrow[llll,bend right]\arrow[dr]&&|[color=green,fill=red]|f_{n-4}^{E_{n-2},f_{n-3}}\arrow[ll]&\cdots\arrow[l]&\color{green}f_1^{E_3,f_2}\arrow[l]\arrow[dr]&&\color{green}z^{E_2,f_1}\arrow[ll]&&&&\\
    &\color{green}b_n^{E_n\setminus e_n}\arrow[rr]&&\color{red}b_{n-1}^{e_{n-1}}\arrow[ul]&&|[color=red,fill=green]|b_{n-2}^{e_{n-2}}\arrow[ur]&&&&\color{red}b_2^{e_2}\arrow[ur]&&&&
\end{tikzcd}}\]

By observing the pattern at the highlighted vertices, after mutating at $f_1$, we have:

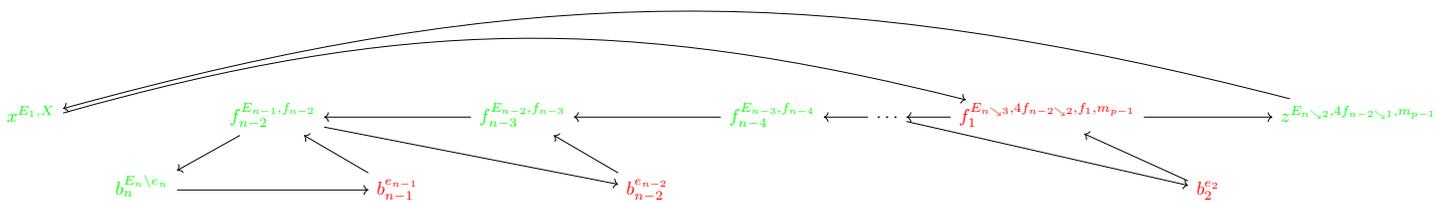
\[\adjustbox{scale=0.7}{\begin{tikzcd}
    \color{green}x^{E_1,X}\arrow[rrrrrrrr,bend left=15,shift right]&&\color{green}f_{n-2}^{E_{n-1},f_{n-2}}\arrow[dl]\arrow[drrr]&&\color{green}f_{n-3}^{E_{n-2},f_{n-3}}\arrow[ll]&&\color{green}f_{n-4}^{E_{n-3},f_{n-4}}\arrow[ll]&\cdots\arrow[l]\arrow[drr]&\color{red}f_1^{E_{n\searrow3},4f_{n-2\searrow2},f_1,m_{p-1}}\arrow[l]\arrow[rr]&&\color{green}z^{E_{n\searrow2},4f_{n-2\searrow1},m_{p-1}}\arrow[llllllllll,bend right=15] &&&&\\
    &\color{green}b_n^{E_n\setminus e_n}\arrow[rr]&&\color{red}b_{n-1}^{e_{n-1}}\arrow[ul]&&\color{red}b_{n-2}^{e_{n-2}}\arrow[ul]&&&&\color{red}b_2^{e_2}\arrow[ul]&&&&
\end{tikzcd}}\]

In terms of triangulation, just like in page \pageref{flip_x} where flipping $x$ makes it go around two pentagons, flipping $f_{k}$ makes it go around $n-k$ pentagons, and the arcs $f_k$ are no longer connected to the outside puncture.

Step 5 consists of $x,y,z,x$. Look at the vertices involved and vertices adjacent to them:

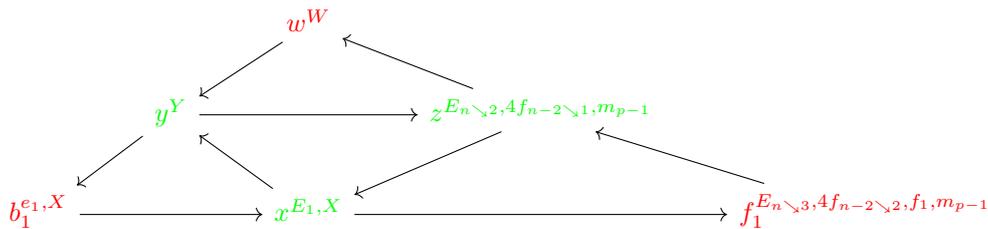
\[\adjustbox{scale=1}{\begin{tikzcd}
    &&\color{red}w^{W}\arrow[dl]\\
    &\color{green}y^{Y}\arrow[rr]\arrow[dl]&&\color{green}z^{E_{n\searrow2},4f_{n-2\searrow1},m_{p-1}}\arrow[ul]\arrow[dl]\\\color{red}b_1^{e_1,X}\arrow[rr]&&\color{green}x^{E_1,X}\arrow[ul]\arrow[rr]&&\color{red}f_1^{E_{n\searrow3},4f_{n-2\searrow2},f_1,m_{p-1}}\arrow[ul]
\end{tikzcd}}\]

Mutate at $x,y,z,x$, we get:

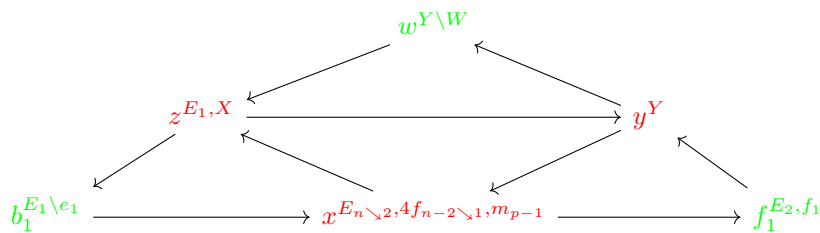
\[\adjustbox{scale=1}{\begin{tikzcd}
    &&\color{green}w^{Y\setminus W}\arrow[dl]\\
    &\color{red}z^{E_1,X}\arrow[rr]\arrow[dl]&&\color{red}y^{Y}\arrow[ul]\arrow[dl]\\\color{green}b_1^{E_1\setminus e_1}\arrow[rr]&&\color{red}x^{E_{n\searrow2},4f_{n-2\searrow1},m_{p-1}}\arrow[rr]\arrow[ul]&&\color{green}f_1^{E_{2},f_1}\arrow[ul]
\end{tikzcd}}\]

In terms of triangulation, flipping arc $x$ makes $x$ go around all $n$ pentagons, and the arcs $x,w,y,z$ form a triangulated once-punctured digon, so flipping $y,z$ tags the outside puncture notched, and flipping $x$ again moves it back to the original position.

Step 6(a) consists of $(f_1,f_2,...,f_{n-2})$. Look at the vertices involved and vertices adjacent to them:

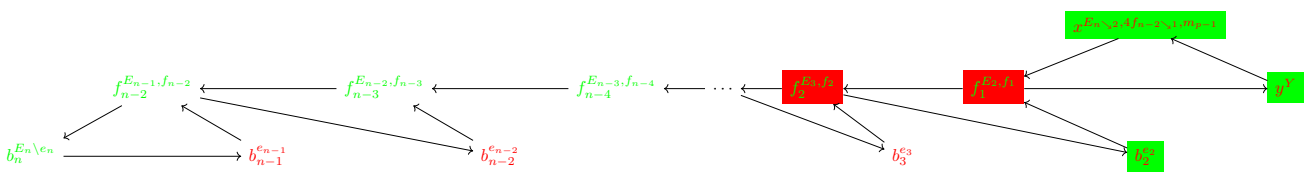
\[\adjustbox{scale=0.65}{\begin{tikzcd}
    &&&&&&&&&&|[color=red,fill=green]|x^{E_{n\searrow2},4f_{n-2\searrow1},m_{p-1}}\arrow[dl]\\&\color{green}f_{n-2}^{E_{n-1},f_{n-2}}\arrow[dl]\arrow[drrr]&&\color{green}f_{n-3}^{E_{n-2},f_{n-3}}\arrow[ll]&&\color{green}f_{n-4}^{E_{n-3},f_{n-4}}\arrow[ll]&\cdots\arrow[l]\arrow[drr]&|[color=green,fill=red]|f_2^{E_{3},f_2}\arrow[drrr]\arrow [l]&&|[color=green,fill=red]|f_1^{E_{2},f_1}\arrow[ll]\arrow[rr]&&|[color=red,fill=green]|y^{Y}\arrow[ul] &&&&\\
    \color{green}b_n^{E_n\setminus e_n}\arrow[rr]&&\color{red}b_{n-1}^{e_{n-1}}\arrow[ul]&&\color{red}b_{n-2}^{e_{n-2}}\arrow[ul]&&&&\color{red}b_3^{e_3}\arrow[ul]&&|[color=red,fill=green]|b_2^{e_2}\arrow[ul]&&&&
\end{tikzcd}}\]

Mutate at $f_1$:

\[\adjustbox{scale=0.65}
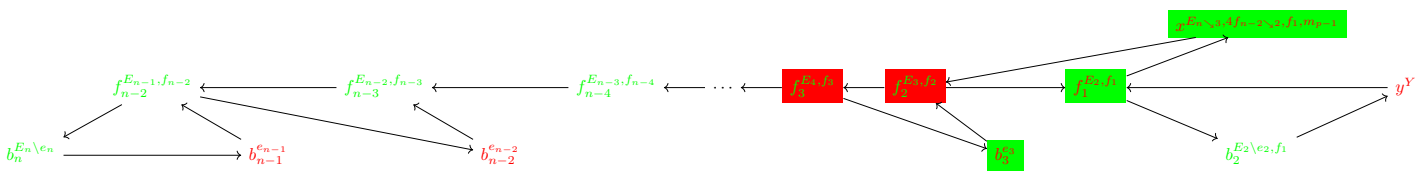
{\begin{tikzcd}
    &&&&&&&&&&&|[color=red,fill=green]|x^{E_{n\searrow3},4f_{n-2\searrow2},f_1,m_{p-1}}\arrow[dlll]\\&\color{green}f_{n-2}^{E_{n-1},f_{n-2}}\arrow[dl]\arrow[drrr]&&\color{green}f_{n-3}^{E_{n-2},f_{n-3}}\arrow[ll]&&\color{green}f_{n-4}^{E_{n-3},f_{n-4}}\arrow[ll]&\cdots\arrow[l]&|[color=green,fill=red]|f_3^{E_{4},f_3}\arrow[drr]\arrow[l]&|[color=green,fill=red]|f_2^{E_{3},f_2}\arrow[rr]\arrow [l]&&|[color=red,fill=green]|f_1^{E_{2},f_1}\arrow[ur]\arrow[dr]&&\color{red}y^{Y}\arrow[ll] &&&&\\
    \color{green}b_n^{E_n\setminus e_n}\arrow[rr]&&\color{red}b_{n-1}^{e_{n-1}}\arrow[ul]&&\color{red}b_{n-2}^{e_{n-2}}\arrow[ul]&&&&&|[color=red,fill=green]|b_3^{e_3}\arrow[ul]&&\color{green}b_2^{E_2\setminus e_2,f_1}\arrow[ur]&&&&
\end{tikzcd}}\]

By observing the pattern at the highlighted vertices, after mutating at $f_{n-2}$, we have:

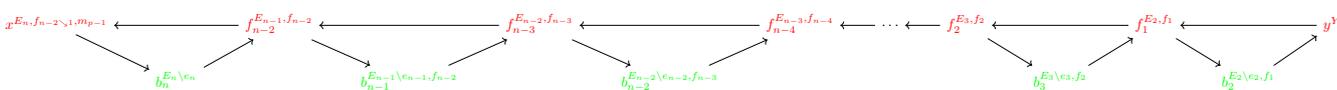
\[\adjustbox{scale=0.55}{\begin{tikzcd}
   \color{red}x^{E_n,f_{n-2\searrow1},m_{p-1}}\arrow[dr]&&\color{red}f_{n-2}^{E_{n-1},f_{n-2}}\arrow[ll]\arrow[dr]&&\color{red}f_{n-3}^{E_{n-2},f_{n-3}}\arrow[ll]\arrow[dr]&&\color{red}f_{n-4}^{E_{n-3},f_{n-4}}\arrow[ll]&\cdots\arrow[l]&\color{red}f_2^{E_{3},f_2}\arrow[dr]\arrow [l]&&\color{red}f_1^{E_{2},f_1}\arrow[ll]\arrow[dr]&&\color{red}y^{Y}\arrow[ll] &&&&\\
    &\color{green}b_n^{E_n\setminus e_n}\arrow[ur]&&\color{green}b_{n-1}^{E_{n-1}\setminus e_{n-1},f_{n-2}}\arrow[ur]&&\color{green}b_{n-2}^{E_{n-2}\setminus e_{n-2},f_{n-3}}\arrow[ur]&&&&\color{green}b_3^{E_3\setminus e_3,f_2}\arrow[ur]&&\color{green}b_2^{E_2\setminus e_2,f_1}\arrow[ur]&&&&
\end{tikzcd}}\]

In terms of triangulation, this is just undoing step 4(e) and therefore separates $b_k$'s (which are arcs from the original pentagons).

Step 6(b) consists of $(\pi_1,\pi_2,...,\pi_n)$. Look at the entire subdiagram:

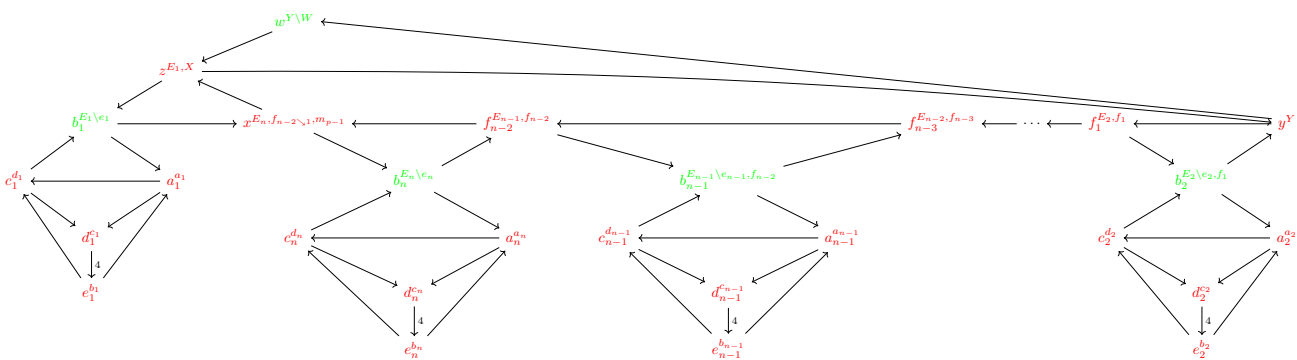
\[\adjustbox{scale=0.55}{\begin{tikzcd}
    &&&\color{green}w^{Y\setminus W}\arrow[dl]&&\\
    &&\color{red}z^{E_1,X}\arrow[dl]\arrow[drrrrrrrrrrr,bend left=3]&&&&&&&&&&&&&&\\
    &\color{green}b_1^{E_1\setminus e_1}\arrow[rr]\arrow[dr]&&\color{red}x^{E_{n},f_{n-2\searrow1},m_{p-1}}\arrow[ul]\arrow[dr]&&\color{red}f_{n-2}^{E_{n-1},f_{n-2}}\arrow[ll]\arrow[drr]&&&&\color{red}f_{n-3}^{E_{n-2},f_{n-3}}\arrow[llll]&\cdots\arrow[l]&\color{red}f_1^{E_{2},f_1}\arrow[l]\arrow[dr]&&\color{red}y^{Y}\arrow[ll]\arrow[uullllllllll,shift right,near end]&&&&\\
    \color{red}c_1^{d_1}\arrow[ur]\arrow[dr]&&\color{red}a_1^{a_1}\arrow[ll]\arrow[dl]&&\color{green}b_n^{E_n\setminus e_n}\arrow[dr]\arrow[ur]&&&\color{green}b_{n-1}^{E_{n-1}\setminus e_{n-1},f_{n-2}}\arrow[dr]\arrow[urr]&&&&&\color{green}b_2^{E_2\setminus e_2,f_1}\arrow[ur]\arrow[dr]&&&&\\
    &\color{red}d_1^{c_1}\arrow[d,"4"]&&\color{red}c_n^{d_n}\arrow[dr]\arrow[ur]&&\color{red}a_n^{a_n}\arrow[dl]\arrow[ll]&\color{red}c_{n-1}^{d_{n-1}}\arrow[ur]\arrow[dr]&&\color{red}a_{n-1}^{a_{n-1}}\arrow[ll]\arrow[dl]&&&\color{red}c_2^{d_2}\arrow[ur]\arrow[dr]&&\color{red}a_2^{a_2}\arrow[ll]\arrow[dl]&&\\
    &\color{red}e_1^{b_1}\arrow[uur]\arrow[uul]&&&\color{red}d_n^{c_n}\arrow[d,"4"]&&&\color{red}d_{n-1}^{c_{n-1}}\arrow[d,"4"]&&&&&\color{red}d_2^{c_2}\arrow[d,"4"]&&\\
    &&&&\color{red}e_n^{b_n}\arrow[uul]\arrow[uur]&&&\color{red}e_{n-1}^{b_{n-1}}\arrow[uul]\arrow[uur]&&&&&\color{red}e_2^{b_2}\arrow[uul]\arrow[uur]&&
\end{tikzcd}}\]

By applying Lemma \ref{lemma_pi} $n$ times, we get:

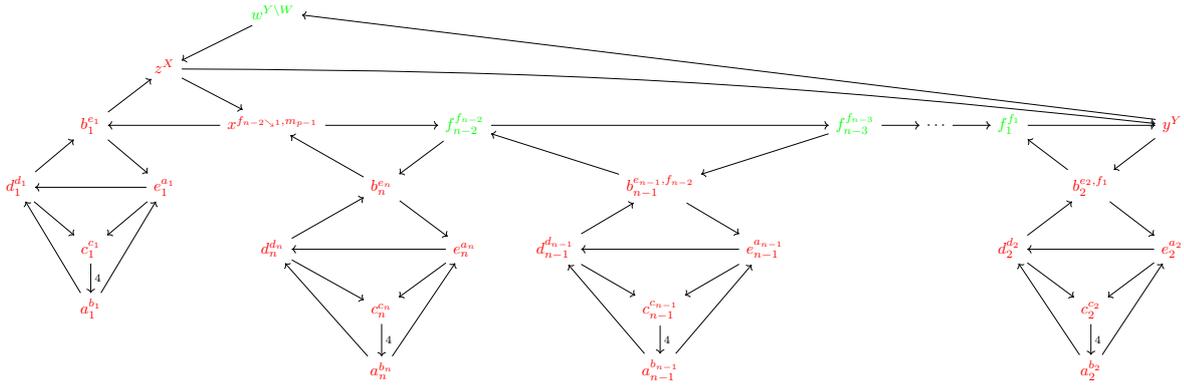
\[\adjustbox{scale=0.6}{\begin{tikzcd}
    &&&\color{green}w^{Y\setminus W}\arrow[dl]&&\\
    &&\color{red}z^{X}\arrow[dr]\arrow[drrrrrrrrrrr,bend left=3]&&&&&&&&&&&&&&\\
    &\color{red}b_1^{e_1}\arrow[ur]\arrow[dr]&&\color{red}x^{f_{n-2\searrow1},m_{p-1}}\arrow[ll]\arrow[rr]&&\color{green}f_{n-2}^{f_{n-2}}\arrow[rrrr]\arrow[dl]&&&&\color{green}f_{n-3}^{f_{n-3}}\arrow[dll]\arrow[r]&\cdots\arrow[r]&\color{green}f_1^{f_1}\arrow[rr]&&\color{red}y^{Y}\arrow[dl]\arrow[uullllllllll,shift right,near end]&&&&\\
    \color{red}d_1^{d_1}\arrow[ur]\arrow[dr]&&\color{red}e_1^{a_1}\arrow[ll]\arrow[dl]&&\color{red}b_n^{e_n}\arrow[dr]\arrow[ul]&&&\color{red}b_{n-1}^{e_{n-1},f_{n-2}}\arrow[dr]\arrow[ull]&&&&&\color{red}b_2^{e_2,f_1}\arrow[ul]\arrow[dr]&&&&\\
    &\color{red}c_1^{c_1}\arrow[d,"4"]&&\color{red}d_n^{d_n}\arrow[dr]\arrow[ur]&&\color{red}e_n^{a_n}\arrow[dl]\arrow[ll]&\color{red}d_{n-1}^{d_{n-1}}\arrow[ur]\arrow[dr]&&\color{red}e_{n-1}^{a_{n-1}}\arrow[ll]\arrow[dl]&&&\color{red}d_2^{d_2}\arrow[ur]\arrow[dr]&&\color{red}e_2^{a_2}\arrow[ll]\arrow[dl]&&\\
    &\color{red}a_1^{b_1}\arrow[uur]\arrow[uul]&&&\color{red}c_n^{c_n}\arrow[d,"4"]&&&\color{red}c_{n-1}^{c_{n-1}}\arrow[d,"4"]&&&&&\color{red}c_2^{c_2}\arrow[d,"4"]&&\\
    &&&&\color{red}a_n^{b_n}\arrow[uul]\arrow[uur]&&&\color{red}a_{n-1}^{b_{n-1}}\arrow[uul]\arrow[uur]&&&&&\color{red}a_2^{b_2}\arrow[uul]\arrow[uur]&&
\end{tikzcd}}\]

Step 6(c) consists of $(f_{n-2},f_{n-3},...,f_1)$. Look at the vertices involved and vertices adjacent to them.

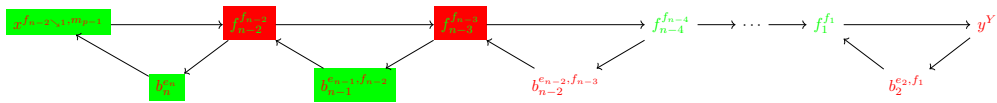
\[\adjustbox{scale=0.6}{\begin{tikzcd}
    |[color=red,fill=green]|x^{f_{n-2\searrow1},m_{p-1}}\arrow[rr]&&|[color=green,fill=red]|f_{n-2}^{f_{n-2}}\arrow[rr]\arrow[dl]&&|[color=green,fill=red]|f_{n-3}^{f_{n-3}}\arrow[dl]\arrow[rr]&&\color{green}f_{n-4}^{f_{n-4}}\arrow[r]\arrow[dl]&\cdots\arrow[r]&\color{green}f_1^{f_1}\arrow[rr]&&\color{red}y^{Y}\arrow[dl]\\
    &|[color=red,fill=green]|b_n^{e_n}\arrow[ul]&&|[color=red,fill=green]|b_{n-1}^{e_{n-1},f_{n-2}}\arrow[ul]&&\color{red}b_{n-2}^{e_{n-2},f_{n-3}}\arrow[ul]&&&&\color{red}b_2^{e_2,f_1}\arrow[ul]
\end{tikzcd}}\]

Mutate at $f_{n-2}$:

\[\adjustbox{scale=0.6}
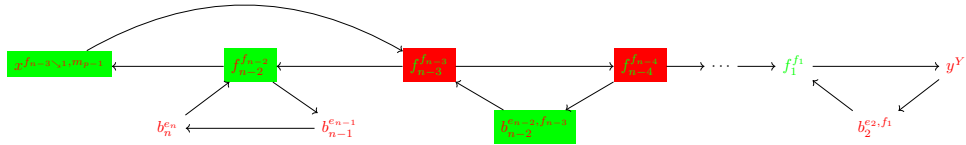
{\begin{tikzcd}
    |[color=red,fill=green]|x^{f_{n-3\searrow1},m_{p-1}}\arrow[rrrr,bend left]&&|[color=red,fill=green]|f_{n-2}^{f_{n-2}}\arrow[ll]\arrow[dr]&&|[color=green,fill=red]|f_{n-3}^{f_{n-3}}\arrow[ll]\arrow[rr]&&|[color=green,fill=red]|f_{n-4}^{f_{n-4}}\arrow[r]\arrow[dl]&\cdots\arrow[r]&\color{green}f_1^{f_1}\arrow[rr]&&\color{red}y^{Y}\arrow[dl]\\
    &\color{red}b_n^{e_n}\arrow[ur]&&\color{red}b_{n-1}^{e_{n-1}}\arrow[ll]&&|[color=red,fill=green]|b_{n-2}^{e_{n-2},f_{n-3}}\arrow[ul]&&&&\color{red}b_2^{e_2,f_1}\arrow[ul]
\end{tikzcd}}\]

By observing the pattern at the highlighted vertices, after mutating at $f_{1}$, we have:

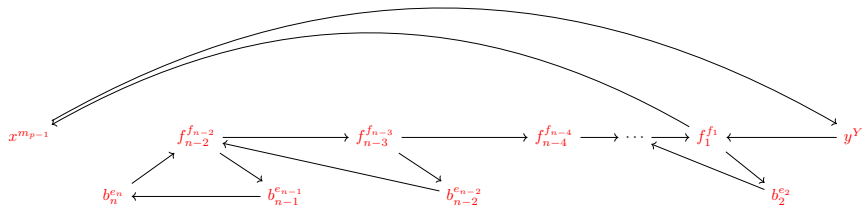
\[\adjustbox{scale=0.6}{\begin{tikzcd}
    \color{red}x^{m_{p-1}}\arrow[rrrrrrrrrr,bend left,shift left]&&\color{red}f_{n-2}^{f_{n-2}}\arrow[rr]\arrow[dr]&&\color{red}f_{n-3}^{f_{n-3}}\arrow[dr]\arrow[rr]&&\color{red}f_{n-4}^{f_{n-4}}\arrow[r]&\cdots\arrow[r]&\color{red}f_1^{f_1}\arrow[llllllll,bend right]\arrow[dr]&&\color{red}y^{Y}\arrow[ll]\\
    &\color{red}b_n^{e_n}\arrow[ur]&&\color{red}b_{n-1}^{e_{n-1}}\arrow[ll]&&\color{red}b_{n-2}^{e_{n-2}}\arrow[ulll]&&&&\color{red}b_2^{e_2}\arrow[ull]
\end{tikzcd}}\]

Now look at the entire subdiagram with the vertices relocated:

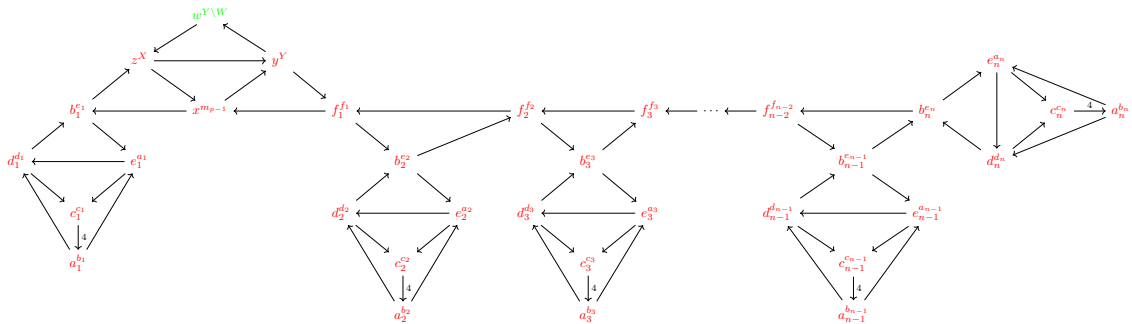
\[\adjustbox{scale=0.5}{\begin{tikzcd}
    &&&\color{green}w^{Y\setminus W}\arrow[dl]&&\\
    &&\color{red}z^{X}\arrow[rr]\arrow[dr]&&\color{red}y^{Y}\arrow[ul]\arrow[dr]&&&&&&&&&&&\color{red}e_n^{a_n}\arrow[dd]\arrow[dr]\\
    &\color{red}b_1^{e_1}\arrow[ur]\arrow[dr]&&\color{red}x^{m_{p-1}}\arrow[ll]\arrow[ur]&&\color{red}f_1^{f_1}\arrow[ll]\arrow[dr]&&&\color{red}f_2^{f_2}\arrow[dr]\arrow[lll]&&\color{red}f_3^{f_3}\arrow[ll]&\cdots\arrow[l]&\color{red}f_{n-2}^{f_{n-2}}\arrow[l]\arrow[dr]&&\color{red}b_n^{e_n}\arrow[ll]\arrow[ur]&&\color{red}c_n^{c_n}\arrow[r,"4"]&\color{red}a_n^{b_n}\arrow[ull]\arrow[dll]\\
    \color{red}d_1^{d_1}\arrow[ur]\arrow[dr]&&\color{red}e_1^{a_1}\arrow[ll]\arrow[dl]&&&&\color{red}b_2^{e_2}\arrow[dr]\arrow[urr]&&&\color{red}b_3^{e_3}\arrow[dr]\arrow[ur]&&&&\color{red}b_{n-1}^{e_{n-1}}\arrow[ur]\arrow[dr]&&\color{red}d_n^{d_n}\arrow[ul]\arrow[ur]\\
    &\color{red}c_1^{c_1}\arrow[d,"4"]&&&&\color{red}d_2^{d_2}\arrow[ur]\arrow[dr]&&\color{red}e_2^{a_2}\arrow[ll]\arrow[dl]&\color{red}d_3^{d_3}\arrow[ur]\arrow[dr]&&\color{red}e_3^{a_3}\arrow[ll]\arrow[dl]&&\color{red}d_{n-1}^{d_{n-1}}\arrow[ur]\arrow[dr]&&\color{red}e_{n-1}^{a_{n-1}}\arrow[ll]\arrow[dl]\\
    &\color{red}a_1^{b_1}\arrow[uur]\arrow[uul]&&&&&\color{red}c_2^{c_2}\arrow[d,"4"]&&&\color{red}c_3^{c_3}\arrow[d,"4"]&&&&\color{red}c_{n-1}^{c_{n-1}}\arrow[d,"4"]\\
    &&&&&&\color{red}a_2^{b_2}\arrow[uul]\arrow[uur]&&&\color{red}a_3^{b_3}\arrow[uul]\arrow[uur]&&&&\color{red}a_{n-1}^{b_{n-1}}\arrow[uul]\arrow[uur]
\end{tikzcd}}\]

Since we obtained the desired diagram in Lemma \ref{one_puncture_summary}, the mutation sequence $\Delta_n$ is a maximal green sequence. $\blacksquare$

\printbibliography

\end{document}